\documentclass[10pt]{article}

\usepackage[top = 2.5cm, bottom = 2.5cm, left = 2.5cm, right = 2.5cm]{geometry}

\usepackage{amsmath,amsthm}
\usepackage{bm,bbm,comment} 
\usepackage{amssymb,mathrsfs} 
\usepackage{amssymb}
\usepackage{amsfonts}
\usepackage{upgreek}
\usepackage{nicefrac}
\usepackage{authblk}
\usepackage{graphicx}
\usepackage{hyperref}
\hypersetup{
  colorlinks = true,
  citecolor = blue,
}
\usepackage{cite}
 \usepackage{color}
 \usepackage{algorithm2e}
\usepackage{stmaryrd}

\usepackage{esvect}

\usepackage[shortlabels]{enumitem}
\usepackage{url}
\usepackage{lmodern} %% improve pdf rendering
\usepackage{ifthen}
\usepackage{xargs}
\usepackage{todonotes}

\usepackage{aliascnt}
\usepackage{cleveref}
\usepackage{autonum}

\makeatletter

\newtheorem{theorem}{Theorem}
\crefname{theorem}{theorem}{Theorems}
\Crefname{Theorem}{Theorem}{Theorems}

\newaliascnt{proposition}{theorem}

\aliascntresetthe{proposition}
\crefname{proposition}{proposition}{propositions}
\Crefname{Proposition}{Proposition}{Propositions}

\newtheorem{thm}{Theorem}[section]
\crefname{theorem}{theorem}{Theorems}
\Crefname{Theorem}{Theorem}{Theorems}

\newtheorem{prop}[thm]{Proposition}
\crefname{proposition}{proposition}{Propositions}
\Crefname{Proposition}{Proposition}{Propositions}

\newtheorem{cor}[thm]{Corollary}
\crefname{corollary}{corollary}{Corollarys}
\Crefname{Corollary}{Corollary}{Corollarys}

\newtheorem{lemma}[thm]{Lemma}
\crefname{lemma}{lemma}{Lemmas}
\Crefname{Lemma}{Lemma}{Lemmas}

\newtheorem{remark}[thm]{Remark}
\crefname{remark}{remark}{Remarks}
\Crefname{Remark}{Remark}{Remarks}

\newtheorem{definition}[thm]{Definition}
\crefname{definition}{definition}{definitions}
\Crefname{Definition}{Definition}{Definitions}

\newtheorem{assumption}{\textbf{A}\hspace{-3pt}}
\crefformat{assumption}{{\textbf{A}}#2#1#3}

\crefformat{equation}{(#2#1#3)}
\renewcommand{\eqref}{\cref}

\crefformat{assumptionH}{{\textbf{H}}#2#1#3}

\numberwithin{equation}{section}
 \definecolor{forestgreen}{rgb}{0.13, 0.55, 0.13}

\newcommand{\ke}[1]{\color{purple} #1 \color{black}}

\newcommand{\IND}{\mathbf{1}}
\newcommand{\bbR}{\mathbb{R}}

\newcommand{\bbRD}{\mathbb{R}^d}
\newcommand{\N}{\mathbb{N}}

\newcommand{\cF}{\mathcal{F}}

\newcommand{\cG}{\mathcal{G}}

\newcommand{\cM}{\mathcal{M}}
\newcommand{\cO}{\mathcal{O}}
\newcommand{\cP}{\mathcal{P}}

\newcommand{\cX}{\mathcal{X}}

\newcommand{\cU}{\mathcal{U}}

\newcommand{\cL}{\mathcal{L}}
\newcommand{\scrH}{\mathscr{H}}

\newcommand{\scrL}{\mathscr{L}}
\newcommand{\bes}{\begin{equation}}
\newcommand{\ees}{\end{equation}}
\newcommand{\beas}{\begin{eqnarray}}
\newcommand{\eeas}{\end{eqnarray}}
\newcommand{\bea}{\begin{eqnarray}}
\newcommand{\eea}{\end{eqnarray}}
\newcommand{\be}{\begin{equation}}
\newcommand{\ee}{\end{equation}}
\newcommand{\bei}{\begin{itemize}}
\newcommand{\eei}{\end{itemize}}
\newcommand{\bec}{\begin{cases}}
\newcommand{\eec}{\end{cases}}
\newcommand{\ben}{\begin{enumerate}}
\newcommand{\een}{\end{enumerate}}

\newcommand{\bbP}{\mathbb{P}}
\newcommand{\bbE}{\mathbb{E}}

\newcommand{\bbl}{\begin{block}}
\newcommand{\ebl}{\end{block}}

\newcommand{\De}{\mathrm{d}}

\newcommand{\rmI}{\mathrm{I}}

\newcommand{\rme}{\mathrm{e}}
\newcommand{\bmu}{\bm{u}}

\newcommand{\ip}[2]{\langle #1,#2\rangle}

\def\mcp{\mathcal{P}}

\def\mcp{\mathcal{P}}

\def\msi{\mathsf{I}}

\def\msk{\mathsf{K}}

\def\msc{\mathsf{C}}

\def\PE{\mathbb{E}}

\def\plusinfty{+\infty}

\def\txts{\textstyle}

\def\rmC{\mathrm{C}}

\newcommand{\abs}[1]{\left\vert #1 \right\vert}
\newcommand{\absLigne}[1]{\vert #1 \vert}

\newcommandx{\Vnorm}[2][1=V]{\| #2 \|_{#1}}
\newcommandx{\normFr}[2][1=]{\ifthenelse{\equal{#1}{}}{\left\Vert #2 \right\Vert_{\mathrm{Fr}}}{\left\Vert #2 \right\Vert^{#1}_{\mathrm{Fr}}}}
\newcommandx{\norm}[2][1=]{\ifthenelse{\equal{#1}{}}{\left| #2 \right|}{\left| #2 \right|^{#1}}}
\newcommandx{\normLigne}[2][1=]{\ifthenelse{\equal{#1}{}}{\Vert #2 \Vert}{\Vert #2\Vert^{#1}}}

\newcommand{\parenthese}[1]{\left(#1 \right)}

\newcommand{\defEns}[1]{\left\lbrace #1 \right\rbrace }

\def\1{\mathbbm{1}}

\def\rset{\mathbb{R}}

\def\ie{\textit{i.e.}}

\def\eqsp{\;}

\newcommand{\ocint}[1]{\left(#1\right]}
\newcommand{\ooint}[1]{\left(#1\right)}
\newcommand{\ccint}[1]{\left[#1\right]}

\def\rmd{\mathrm{d}}

\def\Leb{\mathrm{Leb}}

\def\PP{\mathbb{P}}

\def\Leb{\mathrm{Leb}}

\def\transpose{\top}
\DeclareMathOperator{\Id}{\mathrm{Id}}
\DeclareMathOperator{\Idd}{\mathrm{I}_d}

\def\bfA{\mathbf{A}}

\DeclareMathOperator{\TV}{\mathrm{TV}}

\DeclareMathOperator*{\argmin}{arg\,min}
\DeclareMathOperator{\tr}{tr\,}

\def\Lip{\mathrm{Lip}}

\def\const{\mathrm{C}}
\def\frob{\mathrm{Fr}}
\def\kappaX{\kappa_{\beta}}
\def\hx{\hat{x}}

\def\distY{\varrho}

\def\sfd{{\sf d}}

\newcommand{\expe}[1]{\mathbb{E}\left[ #1\right]}

\newcommand{\expeLigne}[1]{\mathbb{E}[ #1]}

\def\hX{\hat{X}}
\def\bfe{\mathbf{e}}

\newcommand{\sequencetp}[1]{(#1_t)_{t \geq 0}}

\def\forany{\text{ for any }}
\def\Xdelta{X^{\delta}}

\def\tsigma{\check{\sigma}}
\def\fRU{f_R}
\def\akappa{a_{\bar\kappa}}

\usepackage{authblk}

\title{The exponential turnpike phenomenon for mean field game systems: weakly monotone drifts and small interactions}
\author[1]{Alekos Cecchin}
\author[2]{Giovanni Conforti}
\author[3]{Alain Durmus}
\author[4]{Katharina Eichinger}

\affil[1,2]{Università degli studi di Padova}
\affil[3,4]{IPParis} 

\date{\today}

\newcommand{\fl}{\mu_{\cdot}}
\newcommand{\hfl}{\hat\mu_{\cdot}}

\usepackage{mathrsfs}

\newcommand{\rcd}{\mathrm{rc}^{\delta}}
\newcommand{\scd}{\mathrm{sc}^{\delta}}

\newcommand{\lip}[1]{\|#1\|_{\mathrm{Lip}}}
\begin{document}

\maketitle

\abstract{This article aims at quantifying the long time behavior of solutions of mean field PDE systems arising in the theory of  Mean Field Games and McKean-Vlasov control. Our main contribution is to show well-posedness of the ergodic problem and the exponential turnpike property of dynamic optimizers, which implies exponential convergence to equilibrium for both optimal states and controls to their ergodic counterparts. In contrast with previous works that require some version of the Lasry-Lions monotonicity condition, our main assumption is a weak form of asymptotic monotonicity on the drift of the controlled dynamics and some basic regularity and smallness conditions on the interaction terms. Our proof strategy is probabilistic and based on the construction of contractive couplings between controlled processes and forward-backward stochastic differential equations. The flexibility of the coupling approach allows us to cover several interesting situations. For example, we do not need to restrict ourselves to compact domains and can work on the whole space $\bbR^d$, we can cover the case of non-constant diffusion coefficients and we can sometimes show turnpike estimates for the hessians of solutions to the backward equation. }

\tableofcontents

\paragraph{Acknowledgments:} 
A.C. is partially supported by INdAM-GNAMPA Project 2023, ``Stochastic Mean Field Models: Analysis and
Applications'', the PRIN 2022 Project 2022BEMMLZ ``Stochastic control and games and the role of
information'',  the PRIN 2022 PNRR Project P20224TM7Z ``Probabilistic methods for energytransition'', and the project MeCoGa ``Mean field control and games'' of the University of Padova through the program STARS@UNIPD - NextGenerationEU. 
G.C. and  K.E. acknowledge funding from the ANR project ANR-20-CE40-0014. K.E. acknowledges funding from the project PEPS JCJC 2023 of the INSMI and partial support through the University of Padova through the research project BIRD229791 ``Stochastic mean field control and the Schrödinger problem''. 

\section{Introduction}

The purpose of this work is to investigate the long time behavior of solutions and the exponential turnpike property of coupled forward-backward PDE systems arising in the theory  of mean field games (MFG) and McKean-Vlasov stochastic control. Mean field game theory is rooted in the works by Lasry and Lions \cite{lasry2006JeuxChampMoyen,lasry2006JeuxChampMoyena,lasry2007mean} and 
\cite{huang2006LargePopulationStochastic} 
who  gave a mean-field description  of Nash equilibria in differential games with infinitely many small and indistinguishable agents. On the other hand, McKean-Vlasov control problems arise naturally as the mean-field limit of classical high-dimensional stochastic control problems with symmetric costs. A rigorous probabilistic treatment of this class of non linear stochastic control problems began with the seminal paper of Carmona and Delarue \cite{carmona2015forward}, whereas differences and similarities between the two theories are explained in \cite{carmona2013control}. As of today, both MFGs and McKean-Vlasov control are thriving research fields, in part because of the widespread applicability of these models, for example in social sciences, economics and engineering, and in part because their rigorous analysis led to the development of new mathematical theories in connection with the study of Hamilton-Jacobi equations on the space of probability measures and the the so-called master equation \cite{cardaliaguet2019long}. In both these fields, one is naturally led to consider coupled system of PDEs, whose prototype is
\bes\label{eq:MF_PDE_intro}
\bec
\partial_t \varphi_t(x) + \frac{1}{2}\tr\left(\sigma(x)^{\top}\nabla^2\varphi_t(x)\sigma(x)\right) +H(x,\nabla\varphi_t(x)) + F(\mu_t,x)=0, \eqsp \varphi_T(x) = G(\mu_T,x),\\
 \partial_t\mu_t -\frac{1}{2}\tr(\nabla^2(\sigma^{\top}\sigma\mu_t)(x)) + \nabla\cdot(\partial_pH(x,\nabla\varphi_t(x))\mu_t(x))=0 \eqsp, \quad\mu_0=\mu.
 \eec
\ees
In the above, the backward equation is a Hamilton-Jacobi-Bellman (henceforth HJB) equation and the forward equation is a Fokker Planck equation. Here, the Hamiltonian function $H: \rset^d\times \rset^d\to \rset$ is defined through a running cost function $L:\rset^d \times \rset^d \to \rset$ and a drift field $b :\rset^d \to \rset^d$ through
\begin{equation}
H(x,p)=\inf_{u \in \bbR^d} \{L(x,u)+(b(x)+u)\cdot p \}\eqsp.\label{eq:1}
\end{equation}
In the context of mean field games, \eqref{eq:MF_PDE_intro} describes Nash equilibria for the game. 

The goal of the present work is to study the long time behavior of solutions to \eqref{eq:MF_PDE_intro} with a probabilistic approach. More precisely we prove that if the mean field interaction (in the cost) is not too strong, and the uncontrolled dynamics is sufficiently ergodic, then solutions satisfy the turnpike property, which we now briefly describe.  The turnpike property is a general principle postulating that the dynamics of optimal curves for control problems set on a large time horizon is divided in three phases. In the first phase solutions depart from the initial state to approach a steady state, called indeed turnpike. In the second long phase, the dynamics of solutions is localized around the turnpike, whereas in the third and last phase solutions move away from the turnpike to reach their final state. This behavior was first noticed in \cite{dorfman1987linear,mckenzie1976turnpike}
for problems arising in econometry, see the survey \cite{geshkovski2022turnpike} for further examples of modern applications of this property in engineering and machine learning.
In the current setup, a turnpike coincides with a solution to the ergodic system corresponding to \eqref{eq:MF_PDE_intro}, namely
\begin{equation}\label{eq:MF_PDE_ergodic_intro}
\bec
\eta + \frac{1}{2}\tr\left(\sigma(x)^{\top}\sigma(x)\nabla^2\varphi(x)\right) +  H(x,\nabla\varphi(x)) + F(\mu,x)=0,\\
\frac{1}{2}\tr(\nabla^2(\sigma(x)^{\top}\sigma(x)\mu)) - \nabla\cdot(\partial_p H(x,\nabla\varphi(x))\mu)=0.
\eec
\end{equation}
In this work we seek to establish the so-called exponential turnpike property of solutions, which implies in particular that the distance between dynamic and ergodic optimizers decays exponentially fast in time at timescales that are of the same order of the length of the time horizon over which the problem is set.
We now proceed to give an informal presentation of our contributions comparing them with results already available in the literature.

\paragraph{Literature review.}
The appearance of the work \cite{trelat2015turnpike} has renewed the interest around the long-time behavior of deterministic control problems, whereas literature on classical (i.e. finite-dimensional) stochastic control problems is comparatively smaller, and we refer the interested reader to \cite[Sec 1]{conforti2022coupling} for an overview of  relevant contributions. For mean field problems, most of the recent progresses are driven by the development of the theory of mean field games starting with the works \cite{cardaliaguet2012long, cardaliaguet2013long}. Here, the authors establish exponential turnpike theorems under the so called Lasry-Lions monotonicity condition, which may be viewed as a sort of (strong) convexity of the interaction term in the measure argument.
Subsequently, still working under the Lasry-Lions condition, Cardaliaguet and Porretta managed to study the long time-behavior of the master equation associated with the mean field PDE system in \cite{cardaliaguet2019master}. 
In \cite{cirant2021long}, the authors were able to exploit the smoothing properties of the Laplacian to go beyond the strict monotonicity condition and establish exponential turnpike estimates for second order MFGs under what they call a mild non-monotonicity condition. Similar ideas have been recently used in \cite{cesaroni2024stationary} to obtain local turnpike estimates for Kuramoto MFGs. The literature on mean field games includes also other works, such as \cite{cardaliaguet2013KAM,cardaliaguet2021ergodic,bardi2024long}, though they are less related to the scope of this article as they focus on first order mean field games, which are deterministic problems. 
Very recently, a general existence result is given in the preprint \cite{carmona2024probabilisticapproachdiscountedinfinite} by probabilistic methods, and an application of the turnpike property to numerical methods for mean field games is presented in \cite{carmona2024leveragingturnpikeeffectmean}.

The turnpike property  has been  less investigated for optimizers of  McKean-Vlasov control problems. The linear-quadratic case, for which many explicit calculations are possible, is studied in \cite{sun2024turnpike}. Existence for the ergodic problem is investigated in \cite{Cardaliaguet20203255}, and uniqueness and exponential convergence to the turnpike is establushed under convexity assumptions, while an explicit non-convex example where time-dependent value faunctions (in the space of probability measures) do not converge to the ergodic value function is given in \cite{Masoero2019}; we return to this in the next section.

\paragraph{Informal statement of the main results.}
Without any ambition of being rigorous or complete, let us offer a foretaste of our main results, referring to \Cref{sec:main-results} for precise statements. 
In a nutshell, we are able to establish well posedness for the ergodic system \eqref{eq:MF_PDE_ergodic_intro} and exponential turnpike estimates for solutions of \eqref{eq:MF_PDE_intro} under three different sets of assumptions, corresponding to items \ref{item:high_intro}-\ref{item:mild_intro}-\ref{item:low_intro} below. Roughly speaking, in going from \ref{item:high_intro} to \ref{item:low_intro} we lower the regularity requirements imposed on the interaction term and become more demanding on its growth.
\begin{theorem}[Informal statement]\label{thm:main_res_intro}
Assume that $b$ is asymptotically monotone in the sense that there exists $\rho,R > 0$, such that
$$
-\ip{b(x)-b(\hat{x})}{x-\hat{x}} \geq \rho |x-\hat{x}|^2 \quad \text{for } |x-\hat{x}| \geq R,
$$
$\sigma$ is uniformly elliptic, bounded and Lipschitz. Assume moreover that $L$ is uniformly convex in the control variable. Assume one of the following
\begin{enumerate}[label=(\roman*)]
\item\label{item:high_intro} $\partial_x L$ and $\partial_x F$ are bounded, and $\partial_xF$ is $W_1$-Lipschitz with small enough Lipschitz constant.
\item\label{item:mild_intro} $\partial_x L$ and $\partial_x F$ are bounded, and $F$ is $W_1$-Lipschitz with a small enough Lipschitz constant.
\item\label{item:low_intro}The oscillations of $L$ and $F$ in the space variable are bounded, and $F$ is Lipschitz with respect to the total variation distance with a small enough constant.
\end{enumerate}
Then, there exists a unique solution $\mu^{\infty},\varphi^{\infty}$ for the ergodic system \eqref{eq:MF_PDE_ergodic_intro} and a positive rate $\lambda$ independent of $T$ such that for any solution $(\mu_\cdot,\varphi_\cdot)$ of \eqref{eq:MF_PDE_intro} we have the following exponential turnpike property for \ref{item:high_intro}, \ref{item:mild_intro}
\begin{equation}\label{eq:intro_tp_est}
W_1(\mu_t,\mu^\infty)+\|\nabla\varphi_t-\nabla\varphi^\infty\|_{\infty} \lesssim e^{-\lambda t}+ e^{-\lambda (T-t)} \quad \forall t\in[0,T-1],
\end{equation}
and for \ref{item:low_intro}
\begin{equation}\label{eq:intro_tp_est_low}
\|\mu_t - \mu^\infty \|_{\TV}+\|\nabla\varphi_t-\nabla\varphi^\infty\|_{\infty} \lesssim e^{-\lambda t}+ e^{-\lambda (T-t)} \quad \forall t\in[1,T-1].
\end{equation}
Finally, if in case \ref{item:high_intro} $b$ and $D\sigma$ are Lipschitz continuous, then \eqref{eq:intro_tp_est} holds with a better rate $\lambda'> \lambda$ and for larger values of the $W_1$-Lipschitz constant of $\partial_x F$. Moreover, we also have a turnpike estimates for hessians
\begin{equation}
\|\nabla^2\varphi_t-\nabla^2\varphi^\infty\|_{\infty} \lesssim e^{-\lambda' t}+ e^{-\lambda' (T-t)} \quad \forall t\in[0,T-1].
\end{equation}
\end{theorem}
Note that $b$ is a confining drift, and we will take it to be also one-sided Lipschitz. Let us illustrate the significance of each of the three regimes in the simple though conceptually relevant example when
\bes
F(\mu,x)=W\ast\mu(x):=\int W(x-y)\mu(\De y)
\ees
for some interaction potential $W:\bbR^d\longrightarrow\bbR$. 
In this setup, \ref{item:high_intro} holds if $W$ is Lipschitz with arbitrarily large Lipschitz constant, but the Lipschitz constant of $\nabla W$ is small enough. \ref{item:mild_intro} holds if $W$ is Lipschitz with small enough Lipschitz constant. Finally, \ref{item:low_intro} applies when $W$ is simply bounded by a small enough constant. Thus, as we said above, there is a trade off between growth and regularity in going from one set of assumptions to the other. Our strategy consists in leveraging the weak monotonicity of the drift field to obtain uniform in time gradient and hessian bounds for the value function $\varphi$ in \eqref{eq:MF_PDE_intro}. These estimates are then used to obtain stability bounds for frozen versions of \eqref{eq:MF_PDE_intro}. Then, we show that if one of the above smallness conditions holds, these stability estimates are strong enough to show that the fixed-point iterations associated to the mean field PDE system are contractive with respect to a carefully chosen family of weighted norms. From this, we deduce the exponential turnpike property for solutions and a well posedness result for the mean field ergodic system. Deviating from all the above mentioned works, all our estimates are proven constructing suitable couplings between controlled diffusions processes. For example, we carry out a detailed analysis of coupling by reflection on systems of forward-backward stochastic differential equations (henceforth FBSDEs) and we generalise the construction of \emph{controlled} coupling by reflection recently introduced as a tool to obtain gradient estimates for HJB equations in \cite{conforti2022coupling}. This article constitutes one of the main sources of inspiration for the present work, together with the recent (and not so recent) successful applications of the coupling method (which is equivalent to Ishii-Lions doubling of variables method  in the PDE community: see \cite{priola2006gradient}) to prove either exponential ergodicity of non-controlled Markov processes \cite{lindvall1986coupling,eberle2016reflection,eberle2019couplings,eberle2019quantitative,eberle2019sticky,chen1997general,chen1997estimation,wang1994application,CHEN1997287,wang2020exponential} or gradient estimates for both linear and non-linear PDEs \cite{priola2006gradient,porretta2013GlobalLipschitzRegularizing}. We conclude this introductory section summarizing what appear to us as the most interesting contributions of this work.
\bei
\item We provide for the first time a general framework under which solutions to MFGs and McKean-Vlasov control problems exhibit the exponential turnpike property without imposing any form of the Lasry-Lions monotonicity condition on the cost functionals. The only form of monotonicity we require here is the asymptotic monotonicity of the drift $b$. Moreover, besides a smallness condition on the strength of the interaction and some growth conditions, there is no further assumption on the cost functionals, which can be quite general. More than that, in contrast with existing results we are able to cover the case of non-constant diffusion coefficients and we work on the whole space $\bbR^d$ instead of the torus, thus covering situations in which it is typically more difficult to obtain uniform in time and global estimates on the behaviour of Hamilton-Jacobi Bellman equations and diffusion processes.
\item We obtain, possibly for the first time, exponential turnpike estimates for the hessian of solutions to the backward equation in \Cref{eq:MF_PDE_intro}.
\item The uniform in time gradient and hessian estimates on solutions of HJB equations on $\rset^d$ as well as the stability estimates we establish at \Cref{sec:fin_dim_cont} are, to the best of our knowledge and understanding, sometimes more general than those previously available in the literature and of independent interest. For example, we are not aware of uniform in time Hessian estimates in the case of non-constant diffusion coefficients comparable to those in \Cref{sec:Hess_sigma_non_const}.
\item Our results open the door for obtaining uniform in time propagation of chaos estimates for McKean-Vlasov control problems in the same spirit in which coupling techniques made it possible to obtain such estimates for the convergence of systems of uncontrolled diffusion processes towards the McKean-Vlasov limit, see \cite{durmus2020elementary} for example. Moreover, the coupling approach developed here is robust enough to cover various generalizations of the current setup. For example, it can be adapted to incorporate a common noise in the dynamics of agents and to replace Brownian motion as a driving noise by a L\'evy process following either the probabilistic approach of \cite{majka2017coupling} or the more analytic viewpoint of \cite{porretta2024decay}. At the PDE level, this means working with the fractional Laplacian operator. We shall develop these research lines in future work.
\eei

Our results apply to mean field games, but have also a link with McKean-Vlasov optimal control. To clarify it, assume that ${F(\mu,x)=\frac{\delta\mathscr F}{\delta\mu}(\mu;x)}$,  where $\frac{\delta\mathscr F}{\delta\mu}$ denotes the linear functional derivative of the cost function $\mathscr{F}$. By taking a spatial gradient in the HJB equation in \eqref{eq:MF_PDE_intro}, one arrives at the Pontryagin optimality conditions for the McKean-Vlasov control problem
\be\label{eq:McKV_cont_intro}
\begin{split}
\inf_{u} \int_0^T\bbE&\big[L(X^u_t,u_t)\big]+\mathscr{F}(\cL(X^u_t))\,\De t+ \bbE\big[G(\cL(X^u_T),X^u_T)\big]\\
&\text{s.t.} \quad \De X^u_t = [b(X^u_t)+u_t]\De t+\sigma(X^u_t) \De B_t, \quad X^u_0\sim\mu.
\end{split}
\ee
In the above, $(u_t)_{t\in[0,T]}$ denotes a generic open loop square integrable control, and $\cL(X^u_t)$ is the marginal law at time $t$ of the controlled state $X^u_\cdot$. Thus, in order to apply the results of this paper to McKean-Vlasov control, one should prove, under our assumptions, existence of optimal controls and that the Pontryagin system gives necessary conditions for optmality. These results are proved in the literature, but under assumptions which do not cover ours: for instant, we have a space-dependent volatility and a drift only one-sided Lipschitz. Though it should not be difficult to prove this results, we do not pursue this direction here as it is not the scope of this paper and would greatly increase its length.
Here, we just focus on the exponential turnpike for the mean field game system of PDEs.

\paragraph{Organization. }
The paper is organized as follows. In \Cref{sec:main-results} we state our main results and hypotheses and provide some comments.  \Cref{sec:couplings} is devoted to the construction and analysis of coupling between diffusion processes.
In \Cref{sec:fin_dim_cont}, we apply the results \Cref{sec:couplings} to obtain uniform-in-time Lipschitz (\Cref{sec:fin_dim_Lip_est}) and Hessian estimates (\Cref{sec:fin_dim_Hess_est}) on the value function of classical optimal control problems. We also use our coupling construction to prove stability of classical optimal control problems under variations of the drift term in the dynamics, the running and the terminal cost in \Cref{sec:stab_lin_contr}. Finally, \Cref{sec:MF_PDE} provides the analysis of the mean field PDE system. We show uniform-in-time regularity estimates obtained for solutions of \eqref{eq:MF_PDE_intro}, prove existence and uniqueness for \eqref{eq:MF_PDE_ergodic_intro} and finally employ the stability estimates from \Cref{sec:stab_lin_contr} to obtain turnpike estimates  of the form \eqref{eq:intro_tp_est}. The different regimes of regularity are treated as follows: \ref{item:high_intro} in \Cref{sec:MF_high_reg}, \ref{item:mild_intro} in \Cref{sec:MF_mild_reg}, and \ref{item:low_intro} in \Cref{sec:MF_low_reg}.

\paragraph{Notation and convention.}
$\mcp_p(\rset^d)$ is the subset of $\mcp(\rset^d)$, the set of probability measures on $\rset^d$ endowed with its Borel $\sigma$-field $\mathcal{B}(\rset^d)$, with finite $p$-moment, $p\geq 1$. 
Let $\bbR_+ = [0,+\infty)$ and $\bbR_+^*=(0,+\infty)$.
Consider $\msi \subset \rset$ an interval of $\rset$. We denote by $C^{k,m,\beta}_{\mathrm{loc}}(\msi \times \bbR^{d})$ the space of functions which a $k$-times differentiable in time, $m$-times differentiable in space and whose derivatives (including the $0$-th derivative) are locally $\beta$-Hölder continuous in space and $\beta/2$ in time. Similarly, we define $ C^{m,\beta}_{\mathrm{loc}}(\bbR^d)$ and  $ C^{k,m,n,\beta}_{\mathrm{loc}}(\msi\times\bbR^d\times \bbR^d)$ and we set $C^{\beta_1,\beta_2}_{\mathrm{loc}}(\msi \times \bbR^d) = C^{0,0,\beta_1,\beta_2}_{\mathrm{loc}}(\msi \times \bbR^d)$. Lipschitz functions on $\bbR^d$ are denoted analogously by $C^{0,1}(\bbR^d)$, the Lipschitz seminorm is $\|f\|_\Lip = \sup_{x\neq y} \frac{|f(x)-f(y)|}{|x-y|}$, and $C^{0,1}_0(\bbR^d)$ is the subspace of Lipschitz functions satisfying $g(0) = 0$. The divergence and Hessian in the space variable of a function $f$ are denoted $\nabla \cdot f$ and $\nabla^2 f$, respectively.    

For a matrix  $\bfA\in\rset^{d\times n}$, with component $A_{i,j}$, $i\in\{1,\ldots,d\}$ and $j\in\{1,\ldots,n\}$, we denote the Frobenius norm of $\bfA$ by $\| \bfA \|_{\frob}^2 = \sum_{i,j} A_{i,j}^2$. We denote by $\succeq$ the order on symmetric matrices.

All along this paper, $\const^{\square}_{\triangle}$ denotes generic non-negative real constants which are relative to a function $\square$. In most cases $\const^f_{\triangle}$ corresponds to a uniform bound on $\partial_{\triangle} f$, i.e. $\| \partial_{\triangle} f\| \leq \const^f_{\triangle}$.

Finally, we adopt the usual convention $\inf\emptyset=+\infty$.

\section{Main results}\label{sec:main-results}

In order to turn \Cref{thm:main_res_intro} into a rigorous statement, let us spell out the assumptions we impose on the dynamics and on the cost functionals. One of our main requirements is that the uncontrolled dynamics, that is to say the SDE
\begin{equation}\label{eq:SDE_sigma_intro}
\De X_s = b(X_s)\De s +\sigma(X_s) \De B_s \eqsp,
\end{equation}
admits a unique solution and is exponentially ergodic. 
This typically happens if the drift $b$ satisfies some kind of monotonicity condition, for example if $b(\cdot)$ is strongly one-sided Lipschitz with negative constant. However, recent progresses in the coupling by reflection literature, see \cite{eberle2016reflection} which built upon \cite{chen1997general,chen1997estimation,wang1994application,CHEN1997287}, showed that this condition can be significantly relaxed if the diffusion coefficient is uniformly elliptic by actively using the diffusion part. Following this stream of work, we shall quantify the monotonicity properties of $b$ through its monotonicity profile, that we define at \eqref{eq:def_fun_kappa} below.

This leads to consider the following setup.
\begin{assumption}[Ergodicity]\label{ass:MF_drift_intro}
\ben[(i)]
\item \label{ass:MF_drift_intro_diffusion_coeff} The diffusion coefficient $\sigma: \bbR^d \rightarrow \bbR^{d \times d}$ is Lipschitz continuous and uniformly elliptic, \ie, there is $\Sigma \geq \sigma_0>0$ and $\const^\sigma_x<+\infty$ such that
\begin{align}
     2\Sigma^2 \mathrm{I} \succeq \sigma\sigma^{\top}(x) \succeq 2\sigma_0^2 \mathrm{I}, \quad \|\sigma(x)-\sigma(\hat x)\|_\mathrm{Fr}\leq |x-\hat x| \quad \forall x,\hat{x}\in\bbR^d\eqsp.
\end{align}
\item The drift field $b:\bbR^d\rightarrow\bbR^d$ is locally Lipschitz and of polynomial growth,  \ie $|b(x)|\leq K(1+|x|^p)$ for some $K,p<+\infty$. In addition, the monotonicity profile $\kappa_b$ with respect to $\sigma$ associated with $b$ is globally lower bounded by a possibly negative constant and satisfies $\kappa_b \in \msk$, where
\begin{equation}
  \label{eq:def_msk}
\txts  \msk = \{ \kappa \in \rmC( \ooint{0,\plusinfty}, \rset) \,: \,  \int_{0}^1 r (\kappa(r))_- \, \rmd r < \plusinfty \, , \, \liminf_{r \to \plusinfty} \kappa(r) > 0\} \eqsp.
\end{equation}
 The monotonicity profile $\kappa_{b}:\bbR_+\rightarrow\bbR$ of a given vector field $b :\bbR^d\rightarrow\bbR^d$ with respect to $\sigma$ is defined as
\begin{equation}\label{eq:def_fun_kappa}
\kappa_{b}(r) = \inf\left\{-\frac{\ip{b(x)-b(\hat{x})}{x-\hat{x}}}{|x-\hat{x}|^2} - \frac{\|\bar{\sigma}(x) - \bar{\sigma}(\hat x) \|_{\mathrm{Fr}}^2}{2|x-\hat{x}|^2}: |x-\hat{x}|=r  \right\}\eqsp, \quad r>0 \eqsp,
\end{equation}
with $\bar\sigma$ given by
\begin{equation}
  \label{eq:def_bar_sigma}
\bar{\sigma}(x) = \sqrt{\sigma(x)\sigma(x)^{\transpose} - \sigma_0^2\Idd},\quad x\in\bbR^d \eqsp.
\end{equation}
\een
\end{assumption}

In other words, this roughly means that $b$ is one-sided Lipschitz everywhere with a positive constant, and with a negative constant outside a large ball.
In \Cref{sec:couplings} we shall show that under \Cref{ass:MF_drift_intro} the semigroup associated with \eqref{eq:SDE_sigma_intro} has nice smoothing and ergodic properties. In particular, at least when $\sigma(\cdot)$ is constant, it is well-known \cite{eberle2016reflection} that if $\kappa_b\geq \bar\kappa\in\msk$, then there exist an exponential rate $\lambda_{\bar\kappa}$ and a multiplicative constant $C_{\bar\kappa}$, whose precise expression is given at \Cref{prop:Lyapunov_funct_coup_by_ref} below, such that if $X_\cdot,\hat{X}_\cdot$ are two solutions of
\eqref{eq:SDE_sigma} then
\bes
W_1(\cL(X_t),\cL(\hat X_t))\leq W_1(\cL(X_0),\cL(\hat X_0))C^{-1}_{\bar\kappa}e^{-\lambda_{\bar\kappa}t}.
\ees
Note that in the notation $\kappa_{b}$ we only make explicit the dependence on $b$ and do not include $\sigma$. The main reason is that in the whole paper the covariance function does not vary and is always given by $\sigma$, whereas the core of our proofs is based on convergence bounds between time-marginal distributions for time inhomogeneous diffusions with varying drift functions, in general the sum of the drift function $b$ and a suitable Markovian control $(\alpha_s)_{s \in\ccint{0,T}}$.

The second assumption we impose is a standard uniform convexity assumption on the running cost in the control variable.
\begin{assumption}[Coercivity]\label{ass:MF_coercivity_intro}
The running cost $L:\bbR^d\times\bbR^d\rightarrow\bbR$ is locally Lipschitz in the space variable and twice differentiable in the control variable. Moreover, the following holds.
\ben[(i)]
\item There exists $\rho^{L}_{uu}>0$ such that $\partial_{uu}L(x,u) \succeq \rho^{L}_{uu}\rmI \text{ for any }  x,u\in\bbRD$.
\item There exists  $\const^{L(\cdot,0)}_u <+\infty$ such that $\|\partial_u L(\cdot,0)\|_\infty \leq \const^{L(\cdot,0)}_{u}$.
\een
\end{assumption}

The three regularity regimes announced in \Cref{thm:main_res_intro} are captured in the following additional assumptions on the running cost $L:\bbR^d\times\bbR^d\longrightarrow\bbR$ and the interaction term $F:\cP_1(\bbR^d)\times\bbR^d\longrightarrow\bbR$, which are decreasing in regularity but become more stringent on the boundedness conditions on these functions.

\begin{assumption}[High regularity]
\label{ass:high_reg_intro}
\ben[(i)]
\item \label{ass:high_reg_intro_i} There exist $\const^L_x,\const^F_x <+\infty$ such that
\bes \sup_{u\in\bbR^d}\| L(\cdot,u)\|_\Lip \leq \const^L_x, \quad \sup_{\mu\in\cP_1(\bbR^d)}\| F(\mu,\cdot)\|_\Lip \leq \const^F_x.\ees
\item There exists $\const^F_{x\mu}<+\infty$ such that
\bes \|F(\mu,\cdot)-F(\hat\mu,\cdot) \|_\Lip \leq C^F_{x\mu}W_1(\mu,
\hat{\mu}) \quad \forall \mu,\hat\mu\in\cP_1(\bbR^d). \ees
\een
\end{assumption}

\begin{assumption}[Mild regularity]\label{ass:MF_mild_reg_intro}
\Cref{ass:high_reg_intro}-\ref{ass:high_reg_intro_i} holds. Moreover,  there exists $\const_\mu^F <+\infty$ such that
\bes
\|F(\mu,\cdot)-F(\hat\mu,\cdot) \|_\infty \leq C^F_{\mu}W_1(\mu,
\hat{\mu}), \quad \forall \mu,\hat\mu\in\cP_1(\bbR^d).\ees
\end{assumption}

\begin{assumption}[Low regularity]\label{ass:low_reg_in_proofs}
\begin{enumerate}[label=(\roman*)]
\item There exist $\const^{L},\const^{F}<+\infty$ such that
\bes
\sup_{u\in\bbR^d}\|L(\cdot,u)-L(0,u) \|_{\infty} \leq \const^{L},\quad \sup_{\mu\in\cP_1(\bbR^d)}\|F(\mu,\cdot) \|_{\infty}\leq \const^{F}.\ees
\item There exists $\const^F_{\mu,\TV} <+\infty$ such that \bes \|F(\mu,\cdot)-F(\hat{\mu},\cdot)\|_{\infty} \leq \const^{F}_{\mu,\mathrm{TV}} \| \mu - \hat{\mu}\|_{\TV} \quad \forall \mu,\hat\mu \in\cP_1(\bbR^d).\ees
\end{enumerate}
\end{assumption}
With this at hand, we are finally able to give a first precise statement ou our main results. To do so, we shall first specify what we mean by a solution of the mean field PDE systems. 

\begin{definition} 
We say that $(\mu_t,\varphi_t)_{t\in[0,T]}$ solves \eqref{eq:MF_PDE_intro} if $\varphi\in C([0,T] \times \bbR^{d})\cap C^{1,2,\beta}_{\mathrm{loc}}((0,T) \times \bbR^{d})$ is a classical solution of the Hamilton-Jacobi-Bellman equation and $t\mapsto\mu_t\in\cP_1(\bbR^d)$ is a weak solution of the Fokker-Planck equation. 

We say that $(\eta^{\infty},\varphi^{\infty},\mu^{\infty})$ is a solution to
\eqref{eq:MF_PDE_ergodic_intro} if 
$\varphi^{\infty} \in C^{2,\beta}_{\mathrm{loc}}(\bbR^{d})$, $(\eta^{\infty},\varphi^{\infty})$ solves the ergodic Hamilton-Jacobi equation and $\mu^{\infty}$ is the unique invariant measure for the forward equation. 
\end{definition}

We recall that, by the verification theorem, if $(\eta^{\infty},\varphi^{\infty},\mu^{\infty})$ is a solution to the ergodic MFG system, then 
\[
-\eta^\infty = \inf_{\beta, \mu} \int_{\bbR^d} \big[ L(x, \beta(x)) + F(\mu^\infty,x) \Big] \mu(dx) ,
\] 
 the infimum being over couples $(\beta,\mu)$, where $\beta$ is a Markovian controls and $\mu$ is the invariant measure of the SDE 
 $dX_t = b(X_t) +\beta(X_t) dt + \sigma(X_t) dB_t$; 
 further, the optimal control is give by $\partial_p H(x, \nabla \phi^\infty(x))$.

 Moreover, we shall impose some assumptions on the terminal cost $G:\cP_1(\bbR^d)\times\bbR^d\longrightarrow\bbR^d$ when present. For the scope of this paper it is sufficient to assume one of the following two conditions
\be\label{eq:G_ass}
\sup_{\mu\in\cP_1(\bbR^d)}\|G(\mu,\cdot) \|_{\Lip}<+\infty \quad \text{or} \sup_{\mu\in\cP_1(\bbR^d)}\|G(\mu,\cdot) \|_{\infty}<+\infty.
\ee
\begin{theorem}\label{thm:high_reg}
Assume that  \Cref{ass:MF_drift_intro}, \Cref{ass:MF_coercivity_intro} hold.
\begin{enumerate}[wide, labelwidth=!, labelindent=5pt, label=(\arabic*)]
\item \label{item:main-high} Let \Cref{ass:high_reg_intro} hold and define
\be\label{eq:bar_kappa_intro}
\bar\kappa(r) = \kappa_b(r) - \frac{2\const^{u}}{r}, \quad r>0, \quad \text{with} \quad
\const^{u}=\frac{1}{\rho^L_{uu}}\Big(\frac{2(\const_x^L+\const_x^F)}{C_{\kappa_b}\lambda_{\kappa_b}}+\const^{L(\cdot,0)}_u\Big).
\ee
Let $\lambda_{\kappa_b},C_{\kappa_b}$ and $\lambda_{\bar\kappa},C_{\bar\kappa}$ be the exponential rate and multiplicative constant defined via \Cref{prop:Lyapunov_funct_coup_by_ref}. If
\be\label{eq:high_reg_suff}
 \const^F_{x\mu}< \rho^L_{uu} \const^2_{\bar\kappa}\lambda^2_{\bar\kappa},
\ee
then the following hold.
\begin{enumerate}[ labelwidth=!, labelindent=20pt, label=(\alph*)]
\item\label{item:ergodic_main} The ergodic mean field PDE system \eqref{eq:MF_PDE_ergodic_intro} admits a unique solution $(\eta^{\infty},\varphi^{\infty},\mu^{\infty})$ in $\bbR\times C^{0,1}(\bbR^d)\times\cP_1(\bbR^d)$.
\item\label{item:tpike_main} There exists $\lambda>0$ such that if $(\mu_t,\varphi_t)_{t\in[0,T]}$ is a solution of the mean field PDE system \eqref{eq:MF_PDE_intro} with terminal cost $G$ satisfying \eqref{eq:G_ass} and $\mu_0\in\cP_p(\bbR^d)$, then for all $t\in[0,T-\tau]$ we have
\be\label{eq:tpike_est_final}
W_1(\mu_t,\mu^\infty)+\|\varphi_t-\varphi^{\infty} \|_{\Lip}\leq \const_i W_1(\mu_0,\mu^{\infty})e^{-\lambda t} +\const_f e^{-\lambda(T-t)}
\ee
where $\tau,\const_{i} ,\const_{f}<+\infty$ depend only on $\kappa_b$, the constants appearing in \Cref{ass:MF_drift_intro},\Cref{ass:MF_coercivity_intro},\Cref{ass:high_reg_intro}, $\lambda$ and $G$.
\een
\item\label{item:main_mild} If \Cref{ass:MF_mild_reg_intro} holds and if
\be\label{eq:turnpike_suff_cond_mild_intro}
 \const^F_{\mu}\leq \frac{\sqrt{\pi}}{4}{\rho^L_{uu}\sigma_0C^{2}_{\bar\kappa}\lambda^{3/2}_{\bar\kappa}},
\ee
then \ref{item:ergodic_main}-\ref{item:tpike_main} hold for $\bar\kappa$ as in \eqref{eq:bar_kappa_intro} and for constants $\tau,\const_{i} ,\const_{f}<+\infty$ depending only on $\kappa_b$, the constants appearing in \Cref{ass:MF_drift_intro},\Cref{ass:MF_coercivity_intro},\Cref{ass:MF_mild_reg_intro}, $\lambda$ and $G$.
\item\label{item:main_low}  Let  \Cref{ass:low_reg_in_proofs} holds and define
\bes\label{eq:kappa_low}
\bar\kappa(r) = \kappa_b(r) - \frac{2\const^{u}}{r}, \quad r>0, \quad \text{with} \quad
\const^{u}=\frac{1}{\rho^L_{uu}}\Big((8-2\sqrt{e})\frac{\const^L+\const^F}{C_{\kappa_b}\sqrt{\pi\lambda_{\kappa_b}}\sigma_0}+\const^{L(\cdot,0)}_u\Big).
\ees
If
\be\label{eq:turnpike_suff_cond_low_intro}
\const^F_{\mu,\TV}< \frac{\rho^{L}_{uu}\sqrt{\pi}\sigma_0\lambda_{\bar{\kappa}}C_{\bar{\kappa}}}{\sqrt{e} \max \left\{9, \left(4 + \frac{7}{\sqrt{\pi}C_{\kappa}\sigma_0}\right)\right\}},
\ee
then \ref{item:ergodic_main} holds . \Cref{item:tpike_main} also holds replacing $W_1(\mu_t,\mu^{\infty})$ by $\|\mu_t-\mu^\infty \|_{\TV}$, where $\tau,\const_{i} ,\const_{f}<+\infty$ depend only on $\kappa_b$, the constants appearing in \Cref{ass:MF_drift_intro},\Cref{ass:MF_coercivity_intro},\Cref{ass:low_reg_in_proofs}, $\lambda$ and $G$. 
 \een
\end{theorem}
As it will become clear in \Cref{sec:MF_PDE}, the profile $\bar\kappa$ can be interpreted as the monotonicity profile associated with the optimal drift in a stochastic control problem that corresponds to a frozen version of the mean field PDE system \Cref{eq:MF_PDE_intro}. In particular, $\bar\kappa$ can be used to obtain a quantiative turnpike theorem for the control problem in the spirit of \cite{conforti2022coupling}. Therefore we can interpret \eqref{eq:bar_kappa_intro} and \eqref{eq:turnpike_suff_cond_low_intro} as conditions asking that, in some sense, the strength of the interaction is small in comparison to the ergodic rates  of the frozen problem. 
\begin{proof} 
\Cref{thm:high_reg} is a by-product of more general results proven in \Cref{sec:MF_PDE}. In particular, \ref{item:main-high}-\ref{item:ergodic_main} is proven at \Cref{lem:MF_syst_high_reg}, whereas item \ref{item:main-high}-\ref{item:tpike_main} follows by taking $\mu_0=\mu^\infty,\hat{G}=\varphi^\infty$ in \Cref{prop:abstract_suff_cond_tpike_high} \ref{item:high_lipschitz_case}-\ref{item:hig_bdd_case}. \Cref{item:main_mild} (resp. item \ref{item:main_low}) is a consequence of \Cref{lem:MF_syst_mild_reg} (resp. \Cref{lem:erg_MF_syst_low_reg}) and \Cref{prop:abstract_suff_cond_tpike_mild}-\ref{item:mild_reg_tpike_est} (resp. \Cref{prop:abstract_suff_cond_tpike_low}-\ref{item:turnpike_low}) for the choice $\mu_0=\hat{\mu}_0$ and $\hat G=\varphi^\infty$.
\end{proof}
\begin{remark}
The proof of \Cref{thm:high_reg} can be adapted to cover the case when the drift $b$ depends on the measure argument, and the case when the running cost depends jointly on $x,u$ and $\mu$. This is why the stability estimates in \Cref{sec:stab_lin_contr} are written for stochastic control problems with different drifts. Another problem which might be possible to treat, but certainly more difficult, is a non-local dependence on $\nabla \phi$ in the backward equation, which arise from the optimality conditions of a mean field control problem with non-separated cost. 
However, we do not pursue this level of generality here as it would reduce the readability of an already long paper, and would distract from the main ideas. 
\end{remark}

\begin{remark}
Let us comment on the meaning of the smallness conditions 
\eqref{eq:high_reg_suff}, \eqref{eq:turnpike_suff_cond_mild_intro} or 
\eqref{eq:turnpike_suff_cond_low_intro}. Examples of drifts satisfying \Cref{ass:MF_drift_intro} are given in \cite{eberle2016reflection}: the typical one is the opposite of the gradient of a potential that is strongly convex outside a ball  of radiur $R$, while inside the ball can be e.g. double well or constantly zero. The smallness conditions mean that the Lipschitz constants of $F$ are small if $b$ and $\sigma$ are fixed. If the interaction $F$ is fixed, instead, they would mean that the constant $\lambda_{\kappa_b}$ has to be large, which in turn means that the radiur $R$ just mentioned has to be small enough. Note also that the ellipticity constant $\sigma_0$ enters in the definitions of $C_{\kappa_b}$ and $\lambda_{\kappa_b}$ in the next section. Henece the conditions could be rewritten in terms of the product $\lambda_{\kappa_b}\sigma_0$, which mean that the smallness conditions might be interpreted as requiring $\sigma_0$ large enough when fixing the drift $b$ and the cost $F$, thus partially recovering the conditions in the spirit of \cite{cirant2021long}. 

Note that in the torus there is no need of a drift $b$ to have a confinement property. This setup is considered by Masoero in \cite{Masoero2019}, where it is proved in a particular example that, if $\sigma_0$ is smaller than a precise threshold, then there exist multiple optimizers of the ergodic potential mean field game and the time-dependent value functions (of the mean field control problem) converge to a limit which is not the optimal value of the ergodic control problem. Compared to our setting, it means that a condition on $\sigma_0$ large (as given by our smallness conditions) is indeed necessary. 
\end{remark}

\begin{remark}
As the results of \Cref{sec:MF_PDE} show, the constant $\const_f$ in \eqref{eq:tpike_est_final} has a nice explicit expression if we only care about bounding $W_1(\mu_t,\mu^\infty)$. Indeed, in this case we have
\bes
\const_{\mathrm{f}}=\const'_{\mathrm{f}}\min\{(\const(\mu_0)+\|G(\mu_T,\cdot)\|_{\infty}),\|G(\mu_T,\cdot)-\varphi^\infty \|_\Lip\}
\ees
where $\const'_{f}<+\infty$ depends only on $\kappa_b$, the constants appearing in \Cref{ass:MF_drift_intro},\Cref{ass:MF_coercivity_intro},\Cref{ass:high_reg_intro}, $\lambda$ and $G$, and $\const(\mu_0)$ depends in addition on the first moment of $\mu_0$.
\end{remark}
Let us finally conclude by stating that by slightly reinforcing \Cref{ass:high_reg_intro}, we can control the second derivative of the solutions to the backward equation in \eqref{eq:MF_PDE_intro} and obtain a stronger form of \Cref{thm:high_reg}.
\begin{assumption}\label{ass:boost}

\ben[(i)]
\item We assume $b \in C^{1,\beta}_{\mathrm{loc}}(\bbR^d)$, and that there exist  $\const^b_x,\const^\sigma_{xx} <+\infty$ such that
\bes
\| b\|_\Lip\leq \const^b_x \eqsp, \quad \|\partial_{x_i}\sigma(x)-\partial_{x_i}\sigma(\hat{x})\|_{\mathrm{Fr}} \leq \const^\sigma_{xx}|x-\hat{x}| \quad \forall x,\hat{x}\in\bbR^d\eqsp.
\ees
\item We assume $L\in C^{1,2,\beta}_{\mathrm{loc}}(\bbR^d\times\bbR^d)$ and that there exists $\const^L_{xu} \geq 0$ such that  $\sup_{u\in\bbR^d}\|\partial_{u}L(\cdot,u) \|_\Lip\leq \const^L_{xu}$.
\een
\end{assumption}
With this extra assumption, we can have exponential turnpike estimates for stronger interactions (in the sense that the constant $\const^F_{x\mu}$ can be bigger), the solution of the ergodic system has a bounded hessian, and we can show exponential turnpike estimates for the hessian of solutions to the mean field PDE system. This is the content of the following theorem.
\begin{theorem}\label{thm:boost}
Let  \Cref{ass:MF_drift_intro}, \Cref{ass:MF_coercivity_intro}, \Cref{ass:high_reg_intro} and \Cref{ass:boost} hold.
 Then, there exists $\const_x^u<+\infty$ such that if we define
\be\label{eq:kappa_prime_def}
\bar\kappa'(r)=\kappa_b(r)-2\min\{\frac{\const^u}{r},\const^u_x\} \quad \forall r>0,
\ee
and if we suppose that
\be\label{eq:high_reg_suff_1}
 \const^F_{x\mu}< \rho^L_{uu} \const^2_{\bar\kappa'}\lambda^2_{\bar\kappa'},
 \ees
  holds, then we have
\ben[(a)]
\item $\lambda_{\bar\kappa'}>\lambda_{\bar\kappa}$ and $\const_{\bar\kappa'}>\const_{\bar\kappa}$. Therefore, \eqref{eq:high_reg_suff_1} is a weaker condition than \eqref{eq:high_reg_suff}.
\item\label{item:thm_boost_ergodic} The ergodic mean field PDE system \eqref{eq:MF_PDE_ergodic_intro} admits a unique solution $(\eta^{\infty},\varphi^{\infty},\mu^{\infty})$ in $\bbR\times C^{0,1}(\bbR^d)\times\cP_1(\bbR^d)$. Moreover, there exists $\const^{\psi}_{xx}<+\infty$ depending on $\kappa_b$, all constants in \Cref{ass:MF_drift_intro},\Cref{ass:MF_coercivity_intro},\Cref{ass:high_reg_intro} and on all constants in \Cref{ass:boost} except $\const^L_{xu}$ such that $\|\nabla^2\varphi^\infty\|_\infty\leq\const^{\psi}_{xx}$.
\item\label{item:thm_boost_tpike} There exists $\lambda>0$ such that if $(\mu_t,\varphi_t)_{t\in[0,T]}$ is a solution of the mean field PDE system \eqref{eq:MF_PDE_intro} with $G$ satisfying the first condition in \eqref{eq:G_ass} and $\mu_0\in\cP_p(\bbR^d)$,  then for all $t\in[0,T-\tau]$ we have
\bes
\begin{split}
W_1(\mu_t,\mu^\infty)
+\|\varphi_t-\varphi^{\infty} \|_{\Lip}+\|\nabla\varphi_t-\nabla\varphi^{\infty} \|_{\Lip}
&\leq W_1(\mu,\mu^{\infty})\const'_i e^{-\lambda t} +\const'_f e^{-\lambda(T-t)},
\end{split}
\ees
with $\tau,\const'_i,\const'_f$ depending on $\kappa_b$, all constants in \Cref{ass:MF_drift_intro},\Cref{ass:MF_coercivity_intro},\Cref{ass:high_reg_intro} and on all constants in \Cref{ass:boost}, $\lambda$ and $G$.
\een
\end{theorem}
\begin{proof} 
\Cref{thm:boost} is a by-product of more general results proven in \Cref{sec:MF_PDE}. In particular, \ref{item:ergodic_main} is proven at \Cref{lem:MF_syst_high_reg}, whereas item \ref{item:thm_boost_tpike} follows by taking $\mu_0=\mu^\infty$ in \Cref{prop:abstract_suff_cond_tpike_high} \ref{item:high_lipschitz_case_and_boost}-\ref{item:tpike_for_hessians}.\end{proof}

\section{Estimates on diffusion processes}\label{sec:couplings}
In this section, we consider a complete filtered probability space $(\Omega,\cF,(\cF_s)_{s\geq0},\bbP)$ satisfying the usual conditions supporting 4 standard $\cF_s$-adapted independent $d$-dimensional Brownian motions $\{(B^i_s)_{s\geq0} \, :\, i = 0,1,2,3\}$,
 setting $B_s = B_s^0$.

\subsection{Time-regularity and coupling for non-homogeneous diffusions}
\label{sec:intro_coupling}
We aim in this section to estimate certain properties of the (non-homogeneous) semigroup associated with stochastic differential equations (SDEs) of the form
\begin{equation}\label{eq:SDE_sigma}
\De X_s = \beta_s(X_s)\De s +\sigma(X_s) \De B_s \eqsp,
\end{equation}
starting from $X_0$ independent of $\sequencetp{B}$ and where $\beta_{\cdot} : \rset_+\times \rset^d \to \rset^d$, and
$\sigma :  \rset^d \to \rset^{d\times d}$. 

Our main objective is to show quantitative geometric convergence of the non-homogeneous equation \eqref{eq:SDE_sigma} and suitable controlled versions using coupling techniques. For this, we extend results from
\cite{eberle2016reflection} which only considers the case where
$\sigma$ is a constant function and \cite{wang2020exponential} which considers non constant $\sigma$ but only homogeneous drift.  Moreover, as announced above, we shall work with controlled versions of \eqref{eq:SDE_sigma} and show that we can couple two controlled SDEs in such a way that the average distance decreases exponentially fast. In this, setting, the coupling we consider, first introduced in \cite{conforti2022coupling}, is \emph{not} coupling by reflection, and we name it controlled coupling by reflection. However, the contractive properties of the two couplings are established in a similar way. The assumptions we impose on \eqref{eq:SDE_sigma} are a plain adaptation of \Cref{ass:MF_drift_intro} to the notation of this setting.

To this end, we extend
the definition of the monotonicity profile given in \eqref{eq:def_fun_kappa} to the time-inhomogeneous vector field $\beta_\cdot$ in the natural way
\begin{equation}\label{eq:def_fun_kappa_inhom}
\kappaX(r) = \inf\left\{-\frac{\ip{\beta_s(x)-\beta_s(\hat{x})}{x-\hat{x}}}{|x-\hat{x}|^2} - \frac{\|\bar{\sigma}(x) - \bar{\sigma}(\hat x) \|_{\mathrm{Fr}}^2}{2|x-\hat{x}|^2}: |x-\hat{x}|=r, \, s \in \bbR^+  \right\}\eqsp, \quad r>0 \eqsp.
\end{equation}
We assume the following on the drift.
\begin{assumption}\label{ass:coupling}
\begin{enumerate}[wide, labelwidth=!, labelindent=0pt,label=(\roman*)]
\item  For every $s\geq 0$ the vector field $x \mapsto \beta_s(x)$ is continuous and uniformly one-sided Lipschitz continuous:
\begin{equation}
\sup \{\ip{\beta_s(x)-\beta_s(\hat{x})}{x-\hat{x}}/\norm{x-\hx}^2 \,: \, x,\hx \in\rset^d \, , \,  x \neq \hx\} < \infty \eqsp .
\end{equation}

In addition, for any compact set $\msc \subset \rset^d$, there exists a locally integrable function $\upphi_{\msc} : \rset_+ \to \rset_+$ such that $\sup_{x \in\msc} \norm{\beta_t(x)} \leq \upphi_{\msc}(t)$ for any $t \in\rset_+$.
\item The function $\sigma:  \bbR^d \rightarrow \bbR^{d \times d}$ satisfies \Cref{ass:MF_drift_intro}-\ref{ass:MF_drift_intro_diffusion_coeff}.
\item\label{ass:_kappa} We suppose that there exists $\bar{\kappa} \in \msk$ such that for all $r> 0$ $\kappa_\beta(r) \geq \bar{\kappa}(r)$.
\end{enumerate}
\end{assumption}

In particular under \Cref{ass:coupling}, there exists unique strong
solutions for \eqref{eq:SDE_sigma} by
\cite[Corollary 2.6]{gyongy1996ExistenceStrongSolutions}.

In the sequel, the contractive properties of the couplings we shall consider will be expressed either in the total variations norm or in some modification of the $1$-Wasserstein distance $W_1$, with the exception of \Cref{sec:interpolation}, where we shall work with suitable modifications of the $2$-Wasserstein distance $W_2$. Let us now proceed to the construction of these twisted version of $W_1$, which we borrow from \cite{eberle2016reflection}, leaving to the proof of \Cref{prop:contr_same_drift_interp} all considerations about $W_2$.

For $\kappa\in \msk$, $\tsigma > 0$ consider
\begin{equation}\label{eq:aux_fun_eberle_2}
\txts    R_0=\inf\{R\geq 0: \inf_{r \geq R}\kappa(r)\geq 0 \} \eqsp, \quad  R_1=\inf\defEns{ R\geq R_0: \inf_{r \geq R} \{ \kappa(r)R(R-R_0) \} \geq 4\tsigma^2 } \eqsp,
  \end{equation}
and define
\begin{equation}
\phi_{\kappa,\tsigma}(r)=\exp\Big(-\frac{1}{2 \tsigma^2}\int_0^r s(\kappa(s))^-\De s\Big)\eqsp ,\quad \Phi_{\kappa,\tsigma}(r)=\int_0^r\phi_{\kappa,\tsigma}(s)\De s\eqsp,\quad
 g_{\kappa,\tsigma}(r)=1- \frac{\int_0^{r \wedge R_1}\Phi_{\kappa,\tsigma}(s)/\phi_{\kappa,\tsigma}(s) \De s}{2\mathcal{Z}_{\kappa,\tsigma}} \eqsp,
\end{equation}
where $(\kappa(s))^- = \max\{-\kappa(s),0\}$, and $\mathcal{Z}_{\kappa,\tsigma} = \int_0^{R_1}\Phi_{\kappa,\tsigma}(s)/\phi_{\kappa,\tsigma}(s) \De s$.
With this notation at hand we are ready to construct the afore-mentioned modifications of the Wasserstein distance to achieve contraction properties.
\begin{definition}\label{def:twisted_metric}
For $\kappa \in \msk$, $\tsigma > 0$ define $f_{\kappa,\tsigma}, f_\kappa : \rset_+ \to \rset_+$  as
\begin{equation}
  \label{eq:def_f_kappa}
f_{\kappa,\tsigma}(r) =\int_0^r\phi_{\kappa,\tsigma}(s)g_{\kappa,\tsigma}(s)\rmd s \eqsp, \quad f_{\kappa}(r) = f_{\kappa,\tsigma}(r)\eqsp, \quad r \geq 0  \eqsp.
\end{equation}
For $\kappa \in \msk$, $\tsigma > 0$, we define for any $\mu,\hat{\mu}\in\cP_1(\bbR^d)$
\bes
W_{f_\kappa}(\mu,\hat{\mu}) = \inf_{\pi\in\Pi(\mu,\hat{\mu})} \int f_\kappa(|x-\hat{x}|) \pi(\De x \De \hat{x}),
\ees
where $\Pi(\mu,\hat\mu)$ is the set of couplings of $\mu$ and $\hat\mu$
\bes
\Pi(\mu,\hat\mu)=\{ \pi\in\cP(\bbR^{2d}): \pi(A\times\bbR^d)=\mu(A),\,\,\pi(\bbR^d\times A)=\nu(A) \quad \forall A\subseteq \mathcal{B}(\bbR^d) \}.
\ees
\end{definition}
 
 \begin{prop}\label{prop:Lyapunov_funct_coup_by_ref}
 For $\kappa\in \msk$, $\tsigma > 0$ we define  $\lambda_{\kappa,\tsigma}, \lambda_{\kappa}>0$ and $C_{\kappa,\tsigma}, C_{\kappa} \geq 0$,
\begin{equation}
  \label{def:lambda_kappa_C_kappa}
  \lambda_{\kappa, \tsigma}=\frac{\tsigma^{2}}{\mathcal{Z}_{\kappa,\tsigma}} \eqsp, \quad \lambda_{\kappa} = \lambda_{\kappa,\sigma_0}\eqsp, \quad
   C_{\kappa,\tsigma}= \frac{\phi_{\kappa,\tsigma}(R_0)}{2} \eqsp, \quad C_{\kappa} = C_{\kappa,\sigma_0}.
\end{equation}
Then the following holds
\ben[(i), ref=\theprop(\roman*)]
\item\label{item_1:Lyapunov_funct_coup_by_ref} $f_{\kappa, \tsigma}$ is concave and continuously differentiable such that $f'_{\kappa,\tsigma}$ is absolutely continuous. Furthermore, it is equivalent to the identity $r \mapsto r$ on $\rset_+$:
\begin{equation}\label{eq:Lyapunov_funct_coup_by_ref_2}
  C_{\kappa,\tsigma}\, r\leq f_{\kappa,\tsigma}(r) \leq r\eqsp, \quad C_{\kappa,\tsigma}\leq f'_{\kappa,\tsigma}(r) \leq 1\eqsp,
\end{equation}
\item\label{item_2:Lyapunov_funct_coup_by_ref}  The differential inequality
\begin{equation}\label{eq:Eberle_funct_2}
2\tsigma^{2}f''_{\kappa,\tsigma}(r)-r\kappa(r)f'_{\kappa,\tsigma}(r)\leq -\lambda_{\kappa,\tsigma} f_{\kappa,\tsigma}(r) \eqsp.
\end{equation}
holds for all $r>0$.
\item\label{item_3:Lyapunov_funct_coup_by_ref}  $\lambda_{\kappa,\tsigma},C_{\kappa,\tsigma}$ are monotone in the following sense: if $\kappa,\kappa'\in K$ are such that 
 \bes
\kappa(r)\geq\kappa'(r) \quad \forall r>0,
 \ees
 then 
 \bes
 \lambda_{\kappa,\tsigma}\geq\lambda_{\kappa',\tsigma}, \quad C_{\kappa,\tsigma}\geq C_{\kappa',\tsigma}.
 \ees
 \end{enumerate}
 \end{prop}

We refer to \cite[Section 2.1]{durmus2020elementary} and the corresponding appendix for the proof of these properties.

We preface this section by relatively standard time regularity properties of the family of distributions $\sequencetp{\mu}$ associated with solutions of \eqref{eq:SDE_sigma}. They will be useful to set up the fixed-point iterations which we employ to construct solutions to the mean field PDE system \eqref{eq:MF_PDE}.
\begin{prop}\label{prop:Holder_in_time}
Assume \Cref{ass:coupling}. Suppose in addition that for $p \geq 1$
$$\sup_{s \in \ccint{0,T}, \,x \in \rset^d} \{\norm{\beta_s(x)}/(1+\norm{x}^{p}) \} =: \const^\beta_p< \plusinfty,$$ 
and that the initial point $X_0$ satisfies $\expeLigne{\norm{X_0}^p} < \plusinfty$ . Let $(X_t)_{t\in[0,T]}$ be the unique strong solution to \eqref{eq:SDE_sigma} and denote by $\mu_t$ the distribution of $X_t$ for any $t \geq 0$. Let $T \geq 0$ be fixed. Then the following holds.
\begin{enumerate}[label=(\roman*)]
\item\label{item:time_Holder_Wf} There exists $\const \geq 0$ such that for all $s,t \in [0,T]$, $W_{1}(\mu_s,\mu_t) \leq \const |t-s|^{\frac{1}{2}}$, where $\const$ depends only on $\Sigma$,$\kappa_\beta,\bbE[|X_0|^p],$ $\const^\beta_p$ and $T$.
\item\label{item:time_Holder_TV} If $s \mapsto \beta_s(x)$ is Hölder continuous, uniformly in $x$, \ie,   for $\upgamma \in \ocint{0,1}$,
  \begin{equation}
    \sup_{x\in\rset^d} \sup_{t,s \in\ccint{0,T}}\{\norm{\beta_t(x) - \beta_s(x)}/\abs{t-s}^{\upgamma}\} =:\const^\beta_H < \plusinfty \eqsp,
  \end{equation}
then for all $\varepsilon \in[0,T]$, there exists  $\const\geq 0$  such that for all $s,t \in [\varepsilon,T]$,
  $\|\mu_s-\mu_t\|_{\mathrm{TV}} \leq \const |t-s|^{\upgamma \wedge 1/2}$, where $\const$ depends only on $\Sigma$,$\kappa_\beta,\bbE[|X_0|^p],$ and $\const^\beta_p$, $\const^\beta_H$, $\varepsilon$ and $T$.
\end{enumerate}
\end{prop}
The proof of this statement along with all proofs for results in this section are deferred to \Cref{sec:appendix_coupling}. For \Cref{prop:Holder_in_time} see \Cref{sec:proof_Holder_in_time}.

We can now start the construction and the analysis of the couplings we shall employ throughout the whole paper.

\subsection{Controlled coupling by reflection}\label{sec:controlled_coupling}

The  first system of SDEs that we consider is the following. For an $\cF_0$-measurable random variable $\zeta=(\zeta_1,\zeta_2) \in \rset^{2d}$,  consider for $t < T_0 = \inf\{s \geq 0 \,: \, X_s \neq \hX_s \}$,
\begin{equation}\label{eq:coup_by_ref_contraction_2}
  \begin{aligned}
\De X_t& = \{ \beta_t(X_t) + {\alpha}_t(X_t) \}\De t+\sigma_0 \De B^1_t + \bar{\sigma}(X_t)  \De B^3_t , \eqsp X_0 = \zeta_1 \eqsp, \\
\De \hat{X}_t& =\{ \beta_t(\hat{X}_t) +{\alpha}_t(X_t) \}\De t+ \sigma_0 \De \hat{B}^1_1 + \bar{\sigma}(\hat{X}_t)  \De B^3_t , \eqsp \hat{X}_0 = \zeta_2 \eqsp,
\end{aligned}
\end{equation}
where $\bar{\sigma}$ is given in \eqref{eq:def_bar_sigma},
\begin{equation}
  \label{eq:def_mirror_brownian}
 \De\hat{B}^1_t= (\Idd-2 \,\bfe_t \cdot \bfe^{\transpose}_t)\cdot \De B^1_t \eqsp, \quad \bfe_t = (X_t-\hat{X}_t)/|X_t-\hat{X}_t| \eqsp,
\end{equation}
and for $t \geq T_0$,
\begin{equation}\label{eq:coup_by_ref_contraction_2_2}
  \De X_t = \beta_t(X_t) + \alpha_t(X_t)\De t+\sigma_0 \De B^1_t + \bar{\sigma}(X_t)  \De B^3_t  \eqsp, \quad \hX_t = X_t \eqsp.
\end{equation}

\begin{prop}
  Assume \Cref{ass:coupling} and that
  $\alpha$ is bounded and locally Lipschitz continuous, \ie,
  $\sup \{\alpha_s(x) \,: \, s \in \rset_+ \, , \, x \in\rset^d\} <
  \plusinfty$ and for any $t\in\rset_+$ and any compact set $\msc$,
  for any $x,\hat{x}\in\msc$,
  $\norm{\alpha_t(x) - \alpha_t(\hat{x}) } \leq
  \uppsi_{\msc}(t)|x - \hat{x}|$, for some locally integrable function
  $\uppsi_{\msc} : \rset_+ \to \rset_+$.
  Then the system \eqref{eq:coup_by_ref_contraction_2}-\eqref{eq:coup_by_ref_contraction_2_2} admits a strong

\end{prop}

\begin{remark}
Defining for $t\in[0,T]$, $\bar{M}_t := \int_0^t \sigma_0 \De B^1_s + \int_0^t\bar{\sigma}(X_s)  \De B^3_s$ and $ \bar{B}_t := \int_0^t \sigma^{-1}(X_s) \De \bar{M}_s$,

we have under \Cref{ass:coupling} that $(\bar{M}_t)_{t \geq 0}$ is a square integrable martingale and $(\bar{B}_t)_{t \geq 0}$ is a Brownian motion by Lévy's characterization. Similarly, setting $\tilde{M}_t := \int_0^t \sigma_0 \De \hat{B}^1_s + \int_0^t\bar{\sigma}(X_s)  \De B^3_s$ and $
\tilde{B}_t := \int_0^t \sigma^{-1}(X_s) \De \tilde{M}_s$,
we have that $(\tilde{M}_t)_{t \geq 0}$ is a square integrable martingale and $(\tilde{B}_t)_{t \geq 0}$ is a Brownian motion.
Therefore, if $\alpha\equiv 0$, $(X_s,\hat{X}_s)_{s\geq 0}$ is indeed a coupling of the diffusion process \eqref{eq:SDE_sigma} with initial laws given by $(\zeta_1, \zeta_2)$.
\end{remark}

 Note that if $\alpha\equiv0$, \eqref{eq:coup_by_ref_contraction_2}-\eqref{eq:coup_by_ref_contraction_2_2} recovers usual coupling by reflection for non-homogeneous drifts. When $\alpha\not \equiv 0$ this is not a coupling in the classical sense because, for example, the drift of $(\hat{X}_t)_{t \geq 0}$ may depend on $\sequencetp{X}$ and therefore $(\hat{X}_t)_{t \geq 0}$ may not even be a Markov process. However, this construction will be key to obtain uniform in time Lipschitz estimates for the value function of stochastic control problems. In this framework, $\sequencetp{X}$ and $(u_t=\alpha_t(X_t))_{t \geq 0}$ represent  an optimal process and an optimal control for a given stochastic control problem. On the other hand, $\sequencetp{\hat{X}}$ is an admissible process relative to the suboptimal control $\sequencetp{u}$ for a stochastic control problem with different initial conditions. The usefulness of this construction has first been highlighted in \cite{conforti2022coupling} and we shall further demonstrate here its interest in a more general framework.

 Let us turn now to the contractive properties we can prove using this coupling by reflection.
\begin{prop}\label{prop:contr_same_drift}
Assume \Cref{ass:coupling} that
  $\alpha$ is bounded and locally Lipschitz continuous, \ie,
  $\sup \{\alpha_s(x) \,: \, s \in \rset_+ \, , \, x \in\rset^d\} <
  \plusinfty$ and for any $t\in\rset_+$ and any compact set $\msc$,
  for any $x,\hat{x}\in\msc$,
  $\norm{\alpha_t(x) - \alpha_t(\hat{x}) } \leq
  \uppsi_{\msc}(t)|x - \hat{x}|$, for some locally integrable function
  $\uppsi_{\msc} : \rset_+ \to \rset_+$.
  Suppose that $\zeta$ has a finite first order moment. Let
$(X_s,\hat{X}_s)_{s\geq 0}$ be a solution of  \eqref{eq:coup_by_ref_contraction_2}-\eqref{eq:coup_by_ref_contraction_2_2} and denote by $\mu_t$ and $\hat{\mu}_t$ the distribution of $X_t$ and $\hat{X}_t$ respectively for $t \geq 0$.
Let $\bar{\kappa} \in \msk$ such that for any $r>0$ $\kappa_\beta(r) \geq \bar{\kappa}(r)$. The following holds.
\begin{enumerate}[label=(\roman*)]
\item\label{item_2:contraction_coup_by_ref}
For any $t \geq 0$,
\begin{equation}
\bbE[f_{\bar{\kappa}}(|X_t-\hat{X}_t|)]\leq \exp(-\lambda_{\bar{\kappa}} t) \bbE[ f_{\bar{\kappa}}(|X_0-\hat{X}_0|)] \eqsp.
\end{equation}
 As a corollary, we have $W_{f_{\bar{\kappa}}}(\mu_t,\hat{\mu}_t)\leq e^{-\lambda_{\bar{\kappa}} t} W_{f_{\bar{\kappa}}}(\mu_0,\hat{\mu}_0)$.
\item\label{item:final_TV_coup_by_ref}
For any $t \geq 0$,
    \begin{equation}
        \bbP[X_t\neq\hat{X}_t]\leq q^{\bar{\kappa}}_t  \bbE[ f_{\bar{\kappa}}(|X_0-\hat{X}_0|)]    \label{eq:1} \eqsp,
      \end{equation}
 where for $\kappa\in \msk$ we define for any $\tsigma,\lambda>0,t\geq 0$
\begin{equation}\label{q_kappa_t_def}
q^{{\kappa},\lambda,\tsigma}_t =
\begin{cases}
\frac{1}{\sqrt{2\pi t}C_{{\kappa}}\tsigma} & t < \frac{1}{2\lambda}\\
    \frac{\sqrt{\lambda e}}{{\sqrt{\pi}C_{{\kappa}}\tsigma}} e^{-\lambda t} &  t \geq  \frac{1}{2\lambda}
\end{cases}, \quad \text{and} \quad q^{\kappa,\lambda}_t= q^{\kappa,\lambda,\sigma_0}_t, \quad q^{\kappa}_t= q^{\kappa,\lambda_\kappa,\sigma_0}_t.
\end{equation} 
  As a corollary, we have $\|\mu_t-\hat\mu_t \|_{\TV}\leq q^{\bar\kappa}_t W_{f_{\bar{\kappa}}}(\mu_0,\hat{\mu}_0)$.
\end{enumerate}
    
  \end{prop}
The proof is given in \Cref{sec:proof_contr_same_drift}.
  \subsection{Interpolation between reflection coupling and synchronous coupling}\label{sec:interpolation}

For the Hessian estimates of the value function of optimal control problems with non-constant diffusion coefficient we also employ the following interpolation between (uncontrolled) reflection coupling and (uncontrolled) synchronous coupling. 
For an $\cF_0$-measurable random variable $\zeta=(\zeta_1,\zeta_2) \in \rset^{2d}$, consider for $t < T_0 = \inf\{s \geq 0 \,: \, X_s \neq \hX_s \}$,
\begin{equation}\label{eq:coup_by_ref_contraction_3}
  \begin{aligned}
\De X_t& = \beta_t(X_t) \rmd t+\frac{1}{\sqrt{2}}\sigma_0 \De B^1_t + \frac{1}{\sqrt{2}}\sigma_0 \De B^2_t + \bar{\sigma}(X_t)  \De B^3_t , \eqsp X_0 = \zeta_1 \eqsp,\\
\De \hat{X}_t& = \beta_t(\hat{X}_t) \rmd t+ \frac{1}{\sqrt{2}} \sigma_0 \De \hat{B}^1_1 + \frac{1}{\sqrt{2}}\sigma_0 \De B^2_t+ \bar{\sigma}(\hat{X}_t)  \De B^3_t  , \eqsp \hat{X}_0 = \zeta_2 \eqsp,
\end{aligned}
\end{equation}
where $\hat{B}^1_t$ is defined in  \eqref{eq:def_mirror_brownian}
and for $t \geq T_0$,
\begin{equation}\label{eq:coup_by_ref_contraction_2_3}
  \De X_t = \beta_t(X_t)\De t+\frac{1}{\sqrt{2}}\sigma_0 \De B^1_t + \frac{1}{\sqrt{2}}\sigma_0 \De B^2_t+ \bar{\sigma}(X_t)  \De B^3_t  \eqsp, \quad \hX_t = X_t \eqsp.
\end{equation}

The difference with coupling by reflection lies in the additional Brownian motion $ B^2_t$, which is the same for both processes. This is why we can view \eqref{eq:coup_by_ref_contraction_3} as an interpolation between reflection and  synchronous coupling. Note that this coupling is also different from the approximate coupling by reflection studied e.g., in \cite{eberle2016reflection,eberle2019couplings,durmus2020elementary}, and which we will also consider in the next subsection (see \eqref{eq:delta_coup_by_ref}).

For this coupling, a well posedness result is completely analogous to \eqref{eq:coup_by_ref_contraction_2}-~\eqref{eq:coup_by_ref_contraction_2_2} and we shall not write it in detail to avoid repetitions.

The interpolated coupling can also be shown to be a contraction: this is the content of the upcoming \Cref{prop:contr_same_drift_interp}. The convergence rates in this case are worse than those at \Cref{prop:contr_same_drift}. However, while studying controlled processes, we sometimes need to analyze the effect of the coupling on processes which are not necessarily $\sequencetp{X}$ and $\sequencetp{\hat{X}}$. More precisely we will use \eqref{eq:coup_by_ref_contraction_2}-\eqref{eq:coup_by_ref_contraction_2_2} to construct a coupling of forward-backward differential equations, namely the Pontryagin systems
of two different control problems. In this context, adding the synchrounous term will be very helpful in obtaining uniform in time estimates on the distance between the adjoint processes.
\begin{prop}\label{prop:contr_same_drift_interp}
  Assume \Cref{ass:coupling} and suppose that $\zeta$ has a finite first order moment. Let
$(X_s,\hat{X}_s)_{s\geq 0}$ be a  process solution of  \eqref{eq:coup_by_ref_contraction_2}-\eqref{eq:coup_by_ref_contraction_2_2} and denote by $\mu_t$ and $\hat{\mu}_t$ the distribution of $X_t$ and $\hat{X}_t$ respectively for $t \geq 0$.
Let $\bar{\kappa} \in \msk$ such that for any $r>0$ $\kappa_\beta(r) \geq \bar{\kappa}(r)$.
Set  $\tilde{f}_{\bar{\kappa}}=f_{{\bar{\kappa},\sigma_0/\sqrt{2}}}$,  $\tilde{\lambda}_{\bar{\kappa}}=\lambda_{\bar{\kappa},\sigma_0/\sqrt{2}}$ and $\tilde{C}_{\bar{\kappa}} = C_{\bar{\kappa},\sigma_0/\sqrt{2}}$ where $f_{\bar{\kappa},\sigma_0/\sqrt{2}},\lambda_{\bar{\kappa},\sigma_0/\sqrt{2}}$ and $C_{\bar{\kappa},\sigma_0/\sqrt{2}}$ are defined in \eqref{eq:def_f_kappa} and \eqref{def:lambda_kappa_C_kappa}.  The following holds.
\begin{enumerate}[label=(\roman*)]
\item\label{item_2:contraction_coup_by_ref_interp}
For any $t \geq 0$, $\bbE[\tilde{f}_{\bar{\kappa}}(|X_t-\hat{X}_t|)]\leq \exp(-\tilde{\lambda}_{\bar{\kappa}} t) \bbE[ \tilde{f}_{\bar{\kappa}}(|X_0-\hat{X}_0|)] $.

\item\label{item:final_TV_coup_by_ref_interp}
We have for any $t \geq 0$, $\bbP[X_t\neq\hat{X}_t]\leq \tilde{q}^{\bar{\kappa}}_t  \bbE[ \tilde{f}_{\bar{\kappa}}(|X_0-\hat{X}_0|)]   $,

      where  $\tilde{q}^{\bar{\kappa}} = q^{\bar{\kappa},\lambda_{\bar{\kappa},\sigma_0/\sqrt{2}},\sigma_0/\sqrt{2}}$ is defined in \eqref{q_kappa_t_def}.
    \item\label{item:h_dominates_square}
There exists a function $\tilde{f}_{\bar{\kappa},2}\in C^1(\bbR_+,\bbR_+)$ such that  $\tilde{f}_{\bar{\kappa},2}'$ is absolutely continuous and positive constants $\tilde{C}_{\bar{\kappa},2},\tilde{\lambda}_{\bar{\kappa},2}$  such that
\begin{equation}
\tilde{f}_{\bar{\kappa},2}(0)=0,\quad \tilde{f}'_{\bar{\kappa},2}(0)>0,\quad \tilde{f}_{\bar{\kappa},2}(r)\geq \tilde{C}_{\bar{\kappa},2} \eqsp r^2 \quad \forany r>0 \eqsp,
\end{equation}
and for any $t \geq 0$,
\begin{equation}
\bbE[\tilde{f}_{\bar{\kappa},2}(|X_t-\hat{X}_t|)]\leq  e^{-\tilde{\lambda}_{\bar{\kappa},2}t}\bbE[\tilde{f}_{\bar{\kappa},2}(|X_0-\hat{X}_0|)] \eqsp.
\end{equation}
    \end{enumerate}
  \end{prop}
Note that contraction w.r.t. distances which dominate (but are not equivalent to) $W_2$ have already been obtained using coupling by reflection in \cite{luo2016ExponentialConvergenceWasserstein}.
The proof of \Cref{prop:contr_same_drift_interp} is given in \Cref{sec:proof_contr_same_drift_interp}.

\subsection{Approximate coupling by reflection}\label{sec:delta_coup}

The third and last coupling we consider is an approximate coupling by
reflection. Again for a constant diffusion matrix this coupling has
often been studied in the literature, for example as an approximation
of the so called sticky coupling
\cite{eberle2019couplings,durmus2020elementary,eberle2019sticky}. The
interest in this object lies in the fact that it allows to couple two diffusion processes with two different drift functions $\beta$ and $\hat{\beta}$ satisfying \Cref{ass:coupling}.
More precisely, for an $\cF_0$-measurable random variable $\zeta=(\zeta_1,\zeta_2) \in \rset^{2d}$ consider for $\delta >0$,
\begin{equation}\label{eq:delta_coup_by_ref}
  \begin{aligned}
\De \Xdelta_t& = \beta_t(\Xdelta_t) \rmd t+\sigma_0 \rcd(r^{\delta}_s)\De B^1_t + \sigma_0  \scd(r^{\delta}_s) \De B^2_t + \bar{\sigma}(\Xdelta_t)  \De B^3_t  , \eqsp X_0 = \zeta_1,\\
\De \hat{\Xdelta}_t& = \beta_t(\hat{\Xdelta}_t) \rmd t+  \sigma_0 \rcd(r^{\delta}_s) \De \hat{B}^{1,\delta}_1 + \sigma_0  \scd(r^{\delta}_s)\De B^2_t+ \bar{\sigma}(\hat{\Xdelta}_t)  \De B^3_t , \eqsp \hat{X}_0 = \zeta_2,
\end{aligned}
\end{equation}
where
\begin{equation}
  \label{eq:def_mirror_brownian_delta_delta}
 \De\hat{B}^{1,\delta}_t= (\Idd-2 \,\bfe_t^{\delta} \cdot (\bfe^{\delta}_t)^{\transpose})\cdot \De B^1_t \eqsp, \quad \bfe_t^{\delta} = (X_t^{\delta}-\hat{X}_t^{\delta})/r_t^{\delta}\eqsp, \quad r^{\delta}_t = |X^\delta_t-\hat{X}^\delta_t| \eqsp.
\end{equation}
and $\rcd,\scd $ are non-negative globally Lipschitz functions such that
\begin{equation}\label{eq:properties_of_rc}
\rcd(u)^2 + \scd(u)^2 = 1\eqsp, \quad \rcd(u) = 0  \text{ for } u \leq \delta /2 \eqsp,\quad
\rcd(u) = 1 \text{ for } u \geq \delta \eqsp.
\end{equation}

As previously emphasized, the difference with respect to the previous constructions lies in the fact that we consider two different drift fields $\beta,\hat\beta$ that satisfy \Cref{ass:coupling} and that the reflection coefficient is a function of the distance between the processes, namely the function $\rcd$. In the sequel, we will employ this coupling to obtain stability estimates for the solutions of stochastic control problems with different objective functions and controlled dynamics.

\begin{remark}
Similarly to \eqref{eq:SDE_sigma}, under \Cref{ass:coupling}, there exists unique strong
solutions for \eqref{eq:delta_coup_by_ref} for any $\delta >0$ by
\cite[Corollary 2.6]{gyongy1996ExistenceStrongSolutions}.
We call this process ($\delta$-)approximate coupling by reflection. In addition,
defining for $t\geq 0$, $M_t^\delta := \int_0^t \sigma_0 \rcd(r^\delta_s) \De B^1_s + \int_0^t \sigma_0 \scd(r^\delta_s) \De B^2_s + \int_0^t \bar{\sigma}(X_s^\delta)  \De B^3_s, \quad
\bar{B}_t^\delta := \int_0^t \sigma^{-1}(X_s^\delta) \De M_s^\delta$,
we have that $\sequencetp{M^\delta}$ is a square integrable martingale and $\sequencetp{\bar{B}^\delta}$ is a standard Brownian motion, by Lévy's characterization. Similarly, setting
$\tilde{M}_t^\delta := \int_0^t \sigma_0 \rcd(r^\delta_s) \De \hat{B}^{1,\delta}_s + \int_0^t \sigma_0 \scd(r^\delta_s) \De B^2_s + \int_0^t \bar{\sigma}(\hat{X}_s^\delta)  \De B^3_s, \quad
\tilde{B}_t^\delta := \int_0^t \sigma^{-1}(\hat{X}_s^\delta) \De \tilde{M}_s^\delta$,
we have that $\sequencetp{\tilde{M}}$ is a square integrable martingale and $\sequencetp{\tilde{B}}$ is a Brownian motion.

Therefore, for any $\delta >0$, $\sequencetp{X^\delta}$ is a solution of \eqref{eq:SDE_sigma} and $\sequencetp{\hat{X}^\delta}$ also with $\hat{\beta}$ replaced by $\beta$.
\label{prop:delta_well_posed}
\end{remark}

In our next result, we bound in Wasserstein and total variation distances, the time marginals of the distributions of solutions of \eqref{eq:SDE_sigma} and solutions of \eqref{eq:SDE_sigma} with $\hat{\beta}$ replacing $\beta$. In this setting, in contrast to results of the previous sections, we do not get exponential bounds but non-vanishing  estimates because the drift field $\beta,\hat\beta$ do not necessarily coincide.
\begin{prop}\label{prop:contr_delta_coup}
  Assume \Cref{ass:coupling} and suppose that $\zeta$ has a finite first order moment. Let $T \geq 0$ be fixed.
  In addition, define $\const^{\beta_s}  = \sup_{x \in\rset^d} \norm{\beta_s(x)-\hat{\beta}_s(x)}$ and suppose that $\sup_{s \in \ccint{0,T}}\const^{\beta_s} < \plusinfty$.
Let
$(X_s^\delta,\hat{X}_s^\delta)_{s\geq 0}$ be a strong solution of  \eqref{eq:delta_coup_by_ref} for any $\delta >0$, and denote by $\mu_t$ and $\hat{\mu}_t$ the distribution of $X_t^\delta$ and $\hat{X}_t^\delta$ respectively for $t \geq 0$ (which do not depend on $\delta$ by \Cref{prop:delta_well_posed}).
 Let $\bar{\kappa} \in \msk$ such that for any $r >0$,  $\kappa_\beta(r) \geq \bar{\kappa} (r)$, and  recall  $f_{\bar\kappa},\lambda_{\bar\kappa}$ and $C_{\bar\kappa}$ are defined in \eqref{eq:def_f_kappa} and \eqref{def:lambda_kappa_C_kappa}

Then, the following holds.
\ben[(i), ref=(\roman*)]
\item\label{item:W_1_contr_delta_coup}
For any $t \in\ccint{0,T}$, we have for any $\delta >0$,
\begin{equation}
\bbE[f_{\bar{\kappa}}(|X_t^\delta-\hat{X}_t^\delta|)]\leq \exp(-\lambda_{\bar{\kappa}} t)\bbE[f_{\bar{\kappa}}(|X_0-\hat{X}_0|)] + \int_0^te^{-\lambda_{\bar{\kappa}}(t-s)}\const^{\delta\beta_s} \,\De s + \cO(\delta)\eqsp.
\end{equation}
where $\cO(\delta)$ denotes a function such that $|\cO(\delta)|\leq \const \delta$ for all $\delta>0$ and $\const$ a constant depending only on the coefficients in \Cref{ass:coupling}, $T$ and $(\const^{\delta\beta_s})_{s\in[0,T]}$.

As a corollary, it holds $W_{f_{\bar{\kappa}}}(\mu_t,\hat{\mu}_t)\leq e^{-\lambda_{\bar{\kappa}} t} W_{f_{\bar{\kappa}}}(\mu_0,\hat{\mu}_0)+\int_0^te^{-\lambda_{\bar{\kappa}}(t-s)}\const^{\delta\beta_s} \,\De s$.
\item \label{item:TV_contr_delta_coup} For all $0\leq t_0 < t\leq T$  we have for any $\delta >0$,
\begin{equation}
\|\mu_t-\hat\mu_t\|_{\TV}\leq \bbP[X_t^\delta\neq \hat X_t^\delta] \leq q^{\bar{\kappa}}_{t-t_0}\eqsp\bbE[f_{\bar\kappa}(|X^\delta_{t_0}-\hat{X}^\delta_{t_0}|)] + \frac{1}{\sqrt{2}} \left(\int_{t_0}^{t} (\const^{\delta\beta_s}) ^2 \De s \right)^{1/2} \eqsp,
\end{equation}
where $q^{\bar{\kappa}}_t$ is given by \eqref{q_kappa_t_def}. As a corollary, it holds
\begin{equation}
\|\mu_t-\hat\mu_t\|_{\TV} \leq q^{\bar{\kappa}}_{t-t_0}\eqsp\left(e^{-\lambda_{\bar{\kappa}} t_0} W_{f_{\bar{\kappa}}}(\mu_0,\hat{\mu}_0)+\int_0^{t_0}e^{-\lambda_{\bar{\kappa}}(t_0-s)}\const^{\delta\beta_s} \,\De s \right) + \frac{1}{\sqrt{2}} \left(\int_{t_0}^{t} (\const^{\delta\beta_s}) ^2 \De s \right)^{1/2} \eqsp,
\end{equation}
\end{enumerate}

\end{prop}
The proof is given in \Cref{sec:proof_contr_delta_coup}.

\begin{remark}
Let us remark that the existing results on $\delta$-approximate coupling by reflection or sticky coupling, see e.g. \cite{durmus2020elementary,eberle2016reflection, eberle2019sticky}, all require uniform continuity of the drift field in the form of 
the assumption $\lim_{r \rightarrow 0} r\kappa_\beta(r)^- = 0$. We managed to drop this requirement in \Cref{prop:contr_delta_coup}.
\end{remark}

\section{Estimates for finite-dimensional stochastic control problems}\label{sec:fin_dim_cont}
The goal of this section is to derive gradient bounds, hessian bounds and stability estimates for a class of classical stochastic control problems that we now briefly describe. To this aim, we recall (see e.g. \cite[Ch. IV]{fleming2006controlled}) that $(\Omega,\cF,(\cF_s)_{s\geq0},\bbP,(B_s)_{s\geq0})$ is a reference probability system if $(\Omega,\cF,\bbP)$ is a probability space equipped with an augmented filtration $(\cF_s)_{s\geq0}\subseteq \cF$ and $(B_s)_{s\geq0}$ is an $\cF_s$-adapted standard $d$-dimensional Brownian motion. Given a reference probability system, we consider a stochastic control problem in which the dynamics is given by a diffusion process with controlled drift and non-constant diffusion coefficient and  the cost function consists of a dynamic part, encoded in the running cost $\ell_s$, and of a terminal cost $g$. More precisely, we define the set of admissible controls $\cU_{t,T}$ as the set of all $\cF_s$-adapted processes such that 
\bes
\bbE[\int_t^T|u_s|^2\De s]<+\infty.
\ees
The stochastic control problem we study is then defined by
\begin{equation}\label{eq:classical_control_problem}
\varphi^{T,g}_t(x)=\inf_{u\in\cU_{t,T}} J_{t,x}^{T,g}(u),  
\ee
where for any admissible control $u$ the corresponding dynamics is
\be
 \De X^{u}_s= [b_s(X^{u}_s) + u_s]\De s +\sigma(X^{u}_s)\cdot\De B_s, \quad X^u_t=x,
\end{equation}
and the cost functional is given by
\bes J_{t,x}^{T,g}(u)=\bbE[\int_t^T\ell_s(X^u_s,u_s)\,\De s + g(X^u_T)].\ees
In the above, we denoted $\varphi^{T,g}_t(x)$ the optimal value of the control problem, also known as value function, when viewed as function of $(t,x)$.
Let us now list the assumptions we shall impose on the coefficients. Whereas the assumptions on the drift and on the convexity of the running cost in the control variable do not change throughout the section, the requirements we impose on the behavior of the cost with respect to variations in the space variable changes depending on the context.
\begin{assumption}\label{ass:drift_field_fin_dim_control}
We impose the following conditions on $b(\cdot)$ and  $\sigma(\cdot)$.
\ben[(i)]
\item The diffusion coefficient $\sigma: \bbR^d \rightarrow \bbR^{d \times d}$ is Lipschitz continuous and uniformly elliptic, \ie, there is $\Sigma \geq \sigma_0>0$ and $\const_x^\sigma$ such that
\begin{align}
     2\Sigma^2 \mathrm{I} \succeq \sigma\sigma^{\top}(x) \succeq 2\sigma_0^2 \mathrm{I},\quad \|\sigma(x)-\sigma(\hat{x})\|_{\rm{Fr}}\leq \const^{\sigma}_x |x-\hat{x}| \quad \forall x,\hat{x}\in\bbR^d.
\end{align}
\item  We assume that the drift field $b:[0,T]\times\bbR^d\rightarrow\bbR^d$ is continuous, locally Lipschitz, of polynomial growth and that $[0,T]\times\bbR_+\ni(s,r)\mapsto\kappa_{b_s}(r)$ is bounded below by a possibly negative constant. Moreover, there exists $\kappa_{b}\in \msk$, where $\msk$ is defined in \eqref{eq:def_msk}, such that
\bes
\kappa_{b_s}(\cdot) \geq \kappa_{b}(\cdot) \quad \forall s\in[t,T] \eqsp .
\ees
\een
\end{assumption}

Note that \Cref{ass:drift_field_fin_dim_control} ensures that the SDE 
\bes
\De X^u_s=[b(X^u_s)+u_s]\De s + \De B_s, \quad X_t=x,
\ees
admits a pathwise unique solution for all admissible control $u.$

\begin{assumption}\label{ass-coercivity}
The running cost $\ell:[0,T]\times\bbR^d\times\bbR^d\longrightarrow\bbR$ is locally Lipschitz in the space variable, locally $\beta-$H\"older in time and twice differentiable in the control variable. Moreover, the following holds.
\ben[(i)] 
\item There exists $\rho^{\ell}_{uu}>0$ such that for all $s\in[t,T]$
\bes
\partial_{uu}\ell_s(x,u) \succeq \rho^{\ell}_{uu}\mathrm{I} \quad \forall x,u\in\bbRD.
\ees
\item There exists a finite constant  $\const^{\ell(\cdot,0)}_u$ such that for all $s\in[t,T]$
\bes
\|\partial_u \ell_s(\cdot,0)\| \leq \const^{\ell(\cdot,0)}_{u}.
\ees
\een
\end{assumption}
We now present the two alternative settings we consider for the space regularity and growth of the cost functions.

\begin{assumption}\label{ass:mild-reg}
There exist finite constants $(\const^{\ell_s}_x)_{s\in[t,T]},\const^{g}_x$ uniformly bounded in $s\in[0,T]$ and such that
\bes
\|\ell_s(\cdot,u)\|_{\Lip}  \leq \const^{\ell_s}_x,\quad
\lip{g} \leq \const^{g}_x \quad\forall u\in\bbR^d,s\in[0,T].
\ees

\end{assumption}
\begin{assumption}\label{ass:low-reg}
There exist finite constants $(\const^{\ell_s})_{s\in[t,T]},\const^{g}_x$ uniformly bounded in $s\in[0,T]$ and such that
\bes
\|\ell_s(\cdot,u) - \ell_s(0,u)\|_{\infty} \leq \const^{\ell_s}, \quad\|g\|_{\Lip}\leq \const^g_x \quad\forall u\in\bbR^d,s\in[0,T].
\ees
\end{assumption}
In going from \Cref{ass:mild-reg} to \Cref{ass:low-reg} there is a tradeoff between growth and regularity. While the former allows for unbounded costs, the latter requires less smoothness.

\subsection{Lipschitz estimates and Hamilton-Jacobi-Bellman equations}\label{sec:fin_dim_Lip_est}
The goal of this section is twofold. First, we obtain uniform in time Lipschitz estimates for the value function using the coupling constructions developed at \Cref{sec:couplings}. Second, we show that the value function is the unique classical solution of the Hamilton-Jacobi-Bellman equation for \eqref{eq:classical_control_problem} and prove a version of the verification Theorem adapted to the current setup. This is done by slightly adapting classical results. To state the next proposition, we introduce the following function space
\bes\label{eq:def_calX}
\cX= \{ \varphi \in {C([0,T] \times \bbR^{d})\cap C^{1,2,\beta}_{\mathrm{loc}}((0,T) \times \bbR^{d})} : \,\, \sup_{t\in[0,T]}\|\varphi_t\|_\Lip<+\infty \}
\ees

\begin{prop}\label{prop:properties}
Assume \Cref{ass:drift_field_fin_dim_control}, \Cref{ass-coercivity} and one among \Cref{ass:mild-reg} or \Cref{ass:low-reg}.
\begin{enumerate}[label=(\roman*)]
\item\label{item:fin_dim_HJB} The value function  $(t,x)\mapsto \varphi^{T,g}_{t}(x)$ defined by \eqref{eq:classical_control_problem}  is the unique classical solution in $\cX$ of the HJB equation
\begin{equation}\label{eq:fin_dim_HJB}
\bec\partial_s\varphi_s(x) + \frac{1}{2}\tr\left(\sigma^{\top}\eqsp\nabla^2\varphi_s\eqsp\sigma\right)(x) +  h_s(x,\nabla\varphi_s(x))=0, \\
\varphi_T(x)= g(x)
\eec
\end{equation}

where the Hamiltonian $h_s$ is given by
\begin{equation}\label{eq:Hamiltonian_fin_dim}
h_s(x,p) = \inf_{u\in\bbR^d} \ell_s(x,u)+(b_s(x)+u)\cdot p 
\end{equation}
\item\label{item:fin_dim_OptContr} The map 
\bes
    (s,x) \mapsto w_s(x,\nabla \varphi^{T,g}_s(x))\eqsp, 
\ees
where $w_s$ is defined by
\begin{equation}\label{eq:optimal_policy_def}
w_s(x,p)=\arg\min_{u\in\bbR^d}\ell_s(x,u)+(b_s(x)+u)\cdot p,
\end{equation}
is an optimal Markov control policy in the sense that if for any given $(t,x) \in\rset_+\times \rset^{d}$ we define the process $(X_s)_{s\in[t,T]}$ as the unique strong solution of the SDE
\begin{equation}\label{eq:fin_dim_OptDyn}
    \De X_s =[b_s(X_s) + w_s(X_{s},\nabla \varphi^{T,g}_s(X_{s}))] \De s  + \sigma(X_s)\cdot \De B_s \eqsp,\quad X_t=x
\end{equation}
and we set
\begin{equation}
    u_{s}:= w_s(X_{s},\nabla \varphi^{T,g}_s(X_{s}))\eqsp, 
\end{equation}
then $u$ is an optimal control, i.e. a minimizer in \eqref{eq:classical_control_problem}
and the process $(X_s)_{s\in[t,T]}$ coincides with $(X^{u}_s)_{s\in[t,T]}$ a.s..
\item\label{item:fin_dim_system}  The value function $\varphi^{T,g}$ does not depend on the reference probability system. 

\end{enumerate}
\end{prop}

In the next lemma we show uniform in time Lipschitz estimates on the value function. Instead of the classical Lipschitz norm we shall often use an equivalent modification defined in terms of a concave function $f$, taken to be equivalent to the identity.
\bes\label{eq:f_Lip_norm}
\|\varphi\|_{f}:=\sup_{x\neq\hat{x}}\frac{|\varphi(x)-\varphi(\hat{x})|} {f(|x-\hat{x}|)}.
\ees
The reason why we introduce these norms is that they are naturally associated with the twisted metric $W_f$ used to express the strict contractivity of coupling by reflection.
\begin{lemma}\label{lemma:gradient_estimate_linearized problem}
Let \Cref{ass:drift_field_fin_dim_control}, \Cref{ass-coercivity}
hold and $0 \leq t\leq T$.
\begin{enumerate}[label=(\roman*)]
\item\label{item:grad_est_Lip} If \Cref{ass:mild-reg} holds, then
\begin{equation}\label{eq:grad_est_high_reg}
\begin{split}
\|\varphi_t^{T,g}\|_{f_{\kappa_{b}}}=: \const^{\varphi_t}_x&\leq \int_{t}^T   \frac{\const^{\ell_s}_x}{C_{\kappa_{b}}} e^{-\lambda_{\kappa_{b}} (s-t)} \De s+ \| g\|_{f_{\kappa_{b}}}e^{-\lambda_{\kappa_{b}}(T-t)} , \\\|w_t(\cdot,\nabla\varphi^{T,g}_t(\cdot))\|_\infty=:\const^{u_t}&\leq \frac{\const^{\varphi_t}_x+\const^{\ell(\cdot,0)}}{\rho^{\ell}_{uu}},
\end{split}
\end{equation}
where $C_{\kappa_b},\lambda_{\kappa_b}$ have been defined at \Cref{prop:Lyapunov_funct_coup_by_ref}.

\item\label{item:grad_est_bdd+Lip} If \Cref{ass:low-reg} holds, then
\begin{equation}\label{eq:grad_est_mild_reg}
\begin{split}
\|\varphi^{T,g}_t\|_{f_{\kappa_{b}}}=\const^{\varphi_t}_x&\leq \int_t^{T} 2\const^{\ell_s}q^{\kappa_{b}}_{s-t}\eqsp\De s+ \| g\|_{f_{\kappa_b}}e^{-\lambda_{\kappa_b}(T-t)}  \\
\|w_t(\cdot,\nabla\varphi^{T,g}_t(\cdot))\|_\infty=\const^{u_t}&\leq \frac{\const^{\varphi_t}_x+\const^{\ell(\cdot,0)}}{\rho^{\ell}_{uu}}
\end{split}
\end{equation}
where $q^{\kappa_{b}}_t$ has been defined at \eqref{q_kappa_t_def}. 

\end{enumerate} 
\end{lemma}

By slightly modifying the assumption on the terminal condition in \Cref{ass:low-reg} we obtain the following Corollary of \Cref{lemma:gradient_estimate_linearized problem}.
\begin{cor}\label{cor:grad_est_g_bdd}
Let \Cref{ass:drift_field_fin_dim_control}, \Cref{ass-coercivity} hold. If we replace in \Cref{ass:mild-reg} or \Cref{ass:low-reg} the condition $\|g \|_\Lip\leq\const^g_x$ with $\|g \|_\infty\leq \const^g,$ then the conclusions of \Cref{prop:properties} remain true. Furthermore:
\begin{enumerate}[label=(\roman*)]
\item If the version of \Cref{ass:mild-reg}, obtained replacing the condition $\|g \|_\Lip\leq\const^g_x$ with $\|g \|_\infty\leq \const^g,$ holds, then the following modification of \eqref{eq:grad_est_high_reg} holds
\begin{equation}\label{eq:grad_est_high_reg_mod}
\begin{split}
\|\varphi^{T,g}_t\|_{f_{\kappa_{b}}}&\leq \int_t^{T} \frac{\const^{\ell_s}_x}{C_{\kappa_{b}}} e^{-\lambda_{\kappa_{b}} (s-t)} \De s + \|g \|_\infty q^{\kappa_b}_{T-t} , \\\|w_t(\cdot,\nabla\varphi^{T,g}_t(\cdot))\|_\infty&\leq \frac{\const^{\varphi_t}_x+\const^{\ell(\cdot,0)}}{\rho^{\ell}_{uu}}:=\const^{u_t}_x,
\end{split}
\end{equation}
where $q^{\kappa_{b}}_t$ has been defined at \eqref{q_kappa_t_def}. 
\item If the version of \Cref{ass:low-reg}, obtained replacing the condition $\|g \|_\Lip\leq\const^g_x$ with $\|g \|_\infty\leq \const^G,$ holds, then the following modification of \eqref{eq:grad_est_mild_reg} holds
\begin{equation}\label{eq:grad_est_mild_reg_mod}
\begin{split}
\|\varphi^{T,g}_t\|_{f_{\kappa_{b}}}&\leq \int_t^{T} 2\const^{\ell_s}q^{\kappa_{b}}_{s-t}\De s+ \|g \|_\infty q^{\kappa_b}_{T-t}, \\\|w_t(\cdot,\nabla\varphi^{T,g}_t(\cdot))\|_\infty&\leq \frac{\const^{\varphi_t}_x+\const^{\ell(\cdot,0)}}{\rho^{\ell}_{uu}}:=\const^{u_t}_x,
\end{split}
\end{equation}
where $q^{\kappa_{b}}_t$ has been defined at \eqref{q_kappa_t_def}. 
\end{enumerate}
\end{cor}

\subsubsection{Proofs}
The proof of \Cref{prop:properties} and \Cref{lemma:gradient_estimate_linearized problem} are carried out as follows: we first prove \Cref{prop:properties} under the Lipschitzianity assumption \Cref{ass:mild-reg}, then prove \Cref{lemma:gradient_estimate_linearized problem} in full, and eventually come to complete the proof of \Cref{prop:properties} under \Cref{ass:low-reg}.
\begin{proof}[Proof of \Cref{prop:properties} under \Cref{ass:mild-reg}]
Define for $K\geq0$
\begin{equation}\label{eq:truncated_contro_pb}
    \varphi^{K}_t(x) :=\inf\left\{ J_{t,x}^{T,g}((u_s)_{s\in[t,T]}) \, :\, {(u_s)_{s\in[t,T]} \in \cU^{K}_{t,T}} \right\}\eqsp, 
\end{equation}
where $\cU^K_{t,T}$ is the restriction of $\cU_{t,T}$ to controls that are almost surely bounded by $K$
\begin{equation}
  \cU^{K}_{t,T} = \defEns{(u_s)_{s\in[t,T]} \,: \, |u_s(\omega)| \leq K \text{ for }\Leb \otimes \PP\text{ almost every } (s,\omega) }\eqsp.
\end{equation}
We now proceed to show a non sharp Lipschitz estimate on $\varphi^K_t$ which is uniform in $t\in[0,T]$ and $K$. To this aim fix $t,x,\hat{x},K$ and let $u$ be a $\varepsilon$-optimal control in \eqref{eq:truncated_contro_pb} for initial condition $x$, i.e. $\varphi_t(x) \geq J^{T,g}_{t,x}(u) - \varepsilon$. Moreover, let $\hat{X}^u$ be the controlled dynamics for the control $u$ and initial condition $\hat{X}^u_t=x.$ Then, denoting by $\rho_b$ the global lower bound (which might be negative) on $\kappa_{b_s}$, i.e. $\kappa_{b_s} \geq \rho_b$, a standard calculation using synchronous coupling gives
\be\label{eq:stupid_coupling}
\bbE[|X^u_s-\hat{X}^u_s|]\leq e^{-\rho_b(s-t)}|x-\hat{x}| \quad \forall s\in[t,T].
\ee
Using $u$ a suboptimal control in the problem defining $\varphi^K_t(\hat{x})$ gives
\bes
\varphi_t(\hat{x})-\varphi_t(x) \leq \bbE[\int_t^T\ell_s(X^u_s,u_s)-\ell_s(\hat{X}^u_s,u_s)\De s + g(X^u_T)-g(\hat{X}^u_T)]+\varepsilon
\ees
From here, \Cref{ass:mild-reg} ,\eqref{eq:stupid_coupling} and letting $\varepsilon\rightarrow 0$ we obtain the Lipschitz estimate

\begin{equation}\label{eq:stupid_grad_est}
   \sup_{t\in[0,T]} \|\varphi^{K}_t \|_{\Lip} \leq \int_0^T e^{-\rho_b (s-T)}\const^{\ell_s}_x\De s +\const^g_x =: \tilde{\const}^{\varphi,T}_x.
\end{equation}

Moreover, under the standing assumptions we can invoke \cite[Thm 3.1, Rmk 2.3]{rubio2011ExistenceUniquenessCauchy} (since we have a moment bound thanks to Prop. \ref{prop_existence_moment_estimate}) to conclude that  $\varphi^K$ is the unique classical solution in ${C([0,T] \times \bbR^{d})\cap C^{1,2}_{\mathrm{loc}}((0,T) \times \bbR^{d})}$ of the HJB equation 
\bes
\bec
\partial_t \varphi_t(x)+\frac12\tr(\sigma^{\top}\eqsp\nabla^2\varphi_t\eqsp\sigma)(x)+h^K_t(x,\nabla\varphi_t(x))=0,\\
\varphi_T\equiv g
\eec
\ees
where $h^K$ is the truncated Hamiltonian
\begin{align}\label{eq:trunc_Hamiltonian}
h_t^{K}(x,p) = \inf_{|u|\leq K} \ell_s(x,u)+(b_s(x)+u)\cdot p.
\end{align}
Thanks to the Lipschitz estimate \eqref{eq:stupid_grad_est} and \Cref{ass-coercivity} we know that $\varphi^K\in\cX$ and with \Cref{lem:Est_Contr_Hamiltonian}
\bes
\|w_t(\cdot,\nabla\varphi^{K}_t(\cdot))\|_\infty \leq \frac{\tilde{\const}^{\varphi,T}_x+\const^{\ell(\cdot,0)}}{\rho^{\ell}_{uu}} \quad \forall t\in[0,T], \eqsp K\geq 0,
\ees
where $w$ is the Markov control policy associated to $h_t$ defined at \eqref{eq:optimal_policy_def}. This implies
\bes h^K_t(x,\nabla\varphi^K_t(x))=h_t(x,\nabla\varphi^K_t(x)) \quad \forall \,\,K\geq \frac{\tilde{\const}^{\varphi,T}_x+\const^{\ell(\cdot,0)}}{\rho^{\ell}_{uu}}, \,x\in\bbR^d,t\in[0,T].
\ees
But then, $\varphi^K$ is a solution of the original HJB equation \eqref{eq:fin_dim_HJB}. Given this, in order to establish \ref{item:fin_dim_HJB} and \ref{item:fin_dim_OptContr} it is sufficient to prove that any classical solution $\varphi$ of \eqref{eq:fin_dim_HJB} in $\cX$ coincides with $\varphi^{T,g}$ and that it provides with an optimal Markov policy as described in \ref{item:fin_dim_OptContr}. 

To this aim take any such solution $\varphi$ and fix $t\leq T$, $x\in\bbR^d$ and a control $u \in \cU_{t,T}$. Then, denoting by $(X^u_s)_{s\in[t,T]}$ the controlled process, we find using the definition of $h_s$ that for all $t\leq s\leq T$:
\bes
\begin{split}
\ell_s(X^{u}_s,u_s)&=-[b_s(X^{u}_s)+u_s]\cdot \nabla\varphi_s(X^{u}_s)+\Big(\ell_s(X^{u}_s,u_s)+[b_s(X^{u}_s)+u_s]\cdot \nabla\varphi_s(X^{u}_s)\Big)\\
&\geq -[b_s(X^{u}_s)+u_s]\cdot \nabla\varphi_s(X^{u}_s)+h_s(X^{u}_s,\nabla\varphi_s(X^u_s))
\end{split}
\ees 
This gives
\bes\label{eq:verifi_1}
\begin{split}
J^{T,g}_{t,x}(u_\cdot) &\geq \bbE[\int_t^T-[b(X^u_s)+u_s]\cdot\nabla\varphi_s(X^u_s)+h_s(X^u_s,\nabla\varphi^{T,g}_s(X^u_s))\eqsp\De s +g(X^u_{T})]\\
&\stackrel{\eqref{eq:fin_dim_HJB}}{=} \bbE[-\int_t^T\partial_s\varphi_s(X^u_s)+[b(X^u_s)+u_s]\cdot\nabla\varphi_s(X^u_s)+\frac12\tr(\sigma^\top\nabla^2\varphi_s\sigma)(X^u_s)\eqsp\De s+g(X^u_T)]\\
&= \varphi_t(x),
\end{split}
\ee
where to obtain the last expression we used It\^o's formula, whose application is justified by $\varphi\in\cX$.
Since the choice of $u$ is arbitrary, we obtain that $\varphi_t(x)\leq\varphi^{T,g}_t(x)$. Next, we observe that thanks to the condition $\sup_{s\in[t,T]} \|\varphi_s\|_\Lip<+\infty$  we can invoke \cite[Thm 2.8]{gyongy1996ExistenceStrongSolutions}, which provides with  a strong solution $(X_s)_{s\in[t,T]}$ for \eqref{eq:fin_dim_OptDyn}. Moreover, this process coincides by strong uniqueness with $X^{\tilde{\bmu}}_\cdot$ a.s. where
\bes
\tilde{\bmu}_s=w_s(X_s,\nabla\varphi_s(X_s)), \quad \forall s\in[t,T].
\ees
Using that 
\bes
\ell_s(X^{\tilde{\bmu}}_s,\tilde\bmu_s)=h_s(X^{\tilde{\bmu}}_s,\nabla\varphi_s(X^{\tilde{\bmu}}_s))-[b_s(X^{\tilde{\bmu}}_s)+\tilde\bmu_s]\cdot \nabla\varphi_s(X^{\tilde{\bmu}}_s), \quad \forall s\in[t,T]
\ees 
and arguing as in \eqref{eq:verifi_1}, we obtain that 
$\varphi_t(x)=J^{T,g}_{t,x}(\tilde{\bmu})\geq\varphi^{T,g}_t(x)$. Since the converse inequality has already been established, we obtain that  $\varphi_t(x)=\varphi^{T,g}_t(x)$ and that $(s,x)\mapsto w_s(x,\nabla\varphi_s(x))$ is an optimal Markov control policy as desired. This completes the proof of \ref{item:fin_dim_HJB} and \ref{item:fin_dim_OptContr}. Item \ref{item:fin_dim_system}  is a direct consequence of \Cref{prop:properties}-\ref{item:fin_dim_HJB}. Indeed, for each  reference probability system, the associated value function is the unique solution to the HJB equation \eqref{eq:fin_dim_HJB}.
\end{proof}

\begin{proof}[Proof of \Cref{lemma:gradient_estimate_linearized problem}]
In both cases, once the Lipschitz estimate on $\varphi^{T,g}$ has been established, the upper bound on the optimal control policy follows from a routine application of \Cref{ass-coercivity}, as detailed in \Cref{lem:Est_Contr_Hamiltonian}. For this reason, we shall only focus on the proof of the Lipschitz estimates here.
\bei 
\item \emph{Proof of \ref{item:grad_est_Lip}}.

Fix $t,x,\hat{x}$  and consider the optimal Markov control policy $(s,x)\mapsto w_s(x,\nabla\varphi^{T,g}_s(x))$ given by \Cref{prop:properties}, whose validity under \Cref{ass:mild-reg} has already been established.
Now let $(X_s,\hat{X}_s)_{s \in [t,T]}$, $(X_t,\hat{X}_t)=(x,\hat{x})$, be coupling by reflection as in \eqref{eq:coup_by_ref_contraction_2}-\eqref{eq:coup_by_ref_contraction_2_2} for the choices 
\be\label{eq:coup_choices}
\begin{split}
\beta_s(x) =b_s(x),  \quad \alpha_s(x) =  w_s(x,\nabla\varphi^{T,g}_s(x)).
\end{split}
\ee
Next, we observe that the law of $X$ is the law of the optimally controlled process for \eqref{eq:classical_control_problem} because of \Cref{prop:properties}, item \ref{item:fin_dim_OptContr}. Moreover, choosing a different probability system and using the fact that the definition of the value function does not depend on this choice (item \ref{item:fin_dim_system} in \Cref{prop:properties}), we can view the process $u_s=w_s(X_s,\nabla\varphi^{T,g}_s(X_s))$ as a suboptimal control for the problem defining $\varphi^{T,g}_t(\hat{x})$ and $\hat{X}_\cdot$ as the corresponding controlled dynamics. This leads to
\bes
\begin{split}
\varphi^{T,g}_t(\hat{x})-\varphi^{T,g}_t(x)&\leq\bbE[\int_t^T\ell_s(\hat{X}_s,u_s)\,\De s + g(\hat{X}_T)]-\bbE[\int_t^T\ell_s(X_s,u_s)\,\De s + g(X_T)] \\
&\leq \bbE[\int_t^T C^{-1}_{\kappa_{b}} \const^{\ell_s}_xf_{\kappa_{b}}(|X_s-\hat{X}_s|)\De s + \| g\|_{f_{\kappa_{b}}} f_{\kappa_{b}}(|X_T-\hat{X}_T|)].
\end{split}
\ees
Using \Cref{prop:contr_same_drift}-\ref{item_2:contraction_coup_by_ref} and observing that the choice of $x,\hat{x}$ is arbitrary we obtain the desired result.
\item 
  \emph{Proof of \ref{item:grad_est_bdd+Lip}}

  We begin by regularizing the running cost: for $\varepsilon>0$, let 
  \begin{equation}
    \ell^\varepsilon_s(x,u) = \int_{\bbR^d}\ell(y,u)\gamma_\varepsilon(y-x)\eqsp\De y \eqsp, \quad \gamma_\varepsilon(z) = (2\uppi \varepsilon)^{-d/2}\exp(-|z|^2 / (2\varepsilon)) \eqsp,
  \end{equation} 
and define $\varphi^{\varepsilon}$ as the value function obtained replacing $\ell$ with $\ell^\varepsilon$ in \eqref{eq:classical_control_problem}.
Next, we note that the functions $(\ell^{\varepsilon},g)$ satisfy \Cref{ass:low-reg}, \Cref{ass:drift_field_fin_dim_control}, \Cref{ass-coercivity} with the same constants as $(\ell,g)$but also satisfy \Cref{ass:mild-reg}, for which \Cref{prop:properties} has already been established. Thus, for any fixed $\varepsilon >0$, $t\in\ccint{0,T}$, $x,\hat{x} \in \rset^d$ we repeat the same coupling construction \eqref{eq:coup_choices} employed in the proof of \ref{item:grad_est_Lip} for the choices 
\be
\beta_s(x) =b_s(x)\eqsp, \quad \alpha_s(x)=  w^\varepsilon_s(x,\nabla\varphi^{\varepsilon}_s(x))\eqsp,
\ee
with $w^{\varepsilon}_s(x,p) = \argmin_{u \in\rset^d} \ell_s^{\varepsilon}(x,u) + (b_s(x)+u)\cdot p$.
Then, considering $(X^{\varepsilon}_s,\hat{X}_s^{\varepsilon})_{s\in\ccint{t,T}}$ the solution of \eqref{eq:coup_by_ref_contraction_2}-\eqref{eq:coup_by_ref_contraction_2_2} for this choice of vector fields and $X_t^{\varepsilon} = x$, $\hat{X}_t^{\varepsilon} = \hat{x}$, we obtain by definition, 
\begin{equation}
\begin{split}
\varphi^{\varepsilon}_t(\hat{x})-\varphi^{\varepsilon}_t(x)\leq  \expe{\int_t^T\ell^\varepsilon_s(\hat{X}^{\varepsilon}_s,u^\varepsilon_s)\,\De s + g(\hat{X}^{\varepsilon}_T)}-\expe{\int_t^T\ell^\varepsilon_s(X^{\varepsilon}_s,u^\varepsilon_s)\,\De s + g(X^{\varepsilon}_T)}
\end{split}
\end{equation}
with  $u^\varepsilon_s = w^\varepsilon_s(X^{\varepsilon}_s,\nabla\varphi_s^\varepsilon(X^{\varepsilon}_s))$, $s \in \ccint{t,T}$.

The proof differs from that of \ref{item:grad_est_Lip} in how the last expression is bounded. Here, using \Cref{ass:low-reg}, we obtain
\begin{multline}
\expe{\int_t^T\ell^\varepsilon_s(\hat{X}^{\varepsilon}_s,u^\varepsilon_s) - \ell^\varepsilon_s(X^{\varepsilon}_s,u^\varepsilon_s)\,\De s + g(\hat{X}^{\varepsilon}_T) - g(X^{\varepsilon}_T)}\\
\leq  \int_{t}^{T}2\const^{\ell_s} \bbP[X^{\varepsilon}_s\neq\hat{X}^{\varepsilon}_s]\De s+ \|g\|_{f_{\kappa_b}}\expe{f_{\kappa_b}(|X^{\varepsilon}_T-\hat{X}^{\varepsilon}_T|)} \eqsp.
\end{multline}
Using \Cref{prop:contr_same_drift}-\ref{item:final_TV_coup_by_ref} and \Cref{prop:contr_same_drift}-\ref{item_2:contraction_coup_by_ref}  we obtain that for any $t \in \ccint{0,T}$,
\begin{equation}
  \label{eq:bound_vareps_sup_t}
  \sup_{\varepsilon} \|\varphi^\varepsilon_t\|_{f_{\kappa_b}} \leq A_1 = \int_t^{T} 2\const^{\ell_s}q^{\kappa_{b}}_{s-t}\eqsp\De s+ \| g\|_{f_{\kappa_b}}e^{-\lambda_{\kappa_b}(T-t)} \eqsp,
\end{equation}
since the choice of $x,\hat{x}$ was arbitrary.

To conclude the proof it suffices to show for any $t \in \ccint{0,T}$ and $x \in\rset^d$,  $\lim_{\varepsilon\to 0}\varphi^\varepsilon(x)= \varphi^{T,g}_t(x)$. This is done via standard arguments which we now report. We fix $t,x$ and first show that $\lim\sup_{\varepsilon\to 0}\varphi^\varepsilon_t(x) \leq\varphi^{T,g}_t(x)$. To see this, let $u^\delta\in\mathcal{U}_{t,T}$ be a $\delta-$optimal control for the problem defining $\varphi^{T,g}_t(x)$, \ie,
$J^{T,g}_{t,x}(u^\delta)\leq\varphi^{T,g}_t(x)+\delta$. Since $\ell^\varepsilon$ converges to $\ell$ pointwise, \Cref{ass:drift_field_fin_dim_control}, \Cref{ass-coercivity} and \Cref{ass:low-reg} we can apply the dominated convergence theorem to obtain 
\begin{equation}
\limsup_{\varepsilon\to 0}\varphi^\varepsilon_t(x)\leq\limsup_{\varepsilon\to 0}\expe{\int_t^T\ell^{\varepsilon}_s(X^{u^\delta}_s,u^\delta_s)\De s +g(X^{u^\delta}_T)} = J^{T,g}_{t,x}(u^\delta)\leq\varphi^{T,g}_t(x)+\delta
\end{equation}
and since $\delta$ is arbitrary, the proof that $\limsup_{\varepsilon\to 0}\varphi^\varepsilon\leq \varphi^{T,g}$ is concluded. To finish, it only remains to show that $\liminf_{\varepsilon \rightarrow 0} \varphi^{\varepsilon}\geq \varphi^{T,g}$. To this end, consider $u^\varepsilon_s = w^\varepsilon_s(X^{\varepsilon}_s,\nabla\varphi_s^\varepsilon(X^{\varepsilon}_s))$, $s \in \ccint{t,T}$,
$w^{\varepsilon}_s(x,p) = \argmin_{u \in\rset^d} \ell_s^{\varepsilon}(x,u) + (b_s(x)+u)\cdot p$ again. Using \eqref{eq:bound_vareps_sup_t} and \Cref{lemma:gradient_estimate_linearized problem} under \Cref{ass:mild-reg}, we get that almost surely for any $s \in \ccint{t,T}$,
\begin{equation}
  \label{eq:9}
  \sup_{\varepsilon >0} \| u^\varepsilon_s \|_{\infty} \leq \frac{A_1+\const^{\ell(\cdot,0)}}{\rho^{\ell}_{uu}} 
\end{equation}
But then, \Cref{prop_existence_moment_estimate} gives that there exists $M \geq 0$ such that 
\begin{equation}
\sup_{\varepsilon,s\in[t,T]}\expe{|X^{u^\varepsilon}_s|}\leq M\eqsp.
\end{equation}
We are now going to use this bound to conclude. Indeed, observe that for any $R>0$
\begin{align}
\varphi^{T,g}_t(x)-\varphi^\varepsilon_t(x)&\leq \expe{\int_t^T\ell_s(X^{u^\varepsilon}_s,u^\varepsilon_s)-\ell^\varepsilon_s(X^{u^\varepsilon}_s,u^\varepsilon_s)\,\De s} \\
&\stackrel{\Cref{ass:low-reg}}{\leq} (T-t)\sup_{s\in[t,T]} \parenthese{\sup_{|x|\leq R,|u|\leq {\const}^u} |\ell^\varepsilon_s-\ell_s(x,u)| +2 \const^\ell_s \bbP[|X^{u^\varepsilon}_T \geq R]}\\
&\leq (T-t)\sup_{s\in[t,T]} \parenthese{\sup_{|x|\leq R,|u|\leq {\const}^u} |\ell^\varepsilon_s-\ell_s(x,u)| + \frac{2 \const^{\ell_s} M}{R}} \eqsp.
\end{align}
Since $\ell$ is locally Lipschitz,  $\ell^\varepsilon$ converges uniformly on any compact sets to $\ell$. Therefore, letting $\varepsilon\rightarrow0$, we obtain for any $R >0$,
\begin{equation}
\limsup_{\varepsilon\to 0}\varphi^{T,g}_t(x)-\varphi^{\varepsilon}_t(x)\leq \frac{2M \sup_{s\in[t,T]}\const^{\ell_s}}{R} \eqsp.
\end{equation}
Taking $R\to \plusinfty$ completes the proof of  the pointwise convergence of $\varphi^\varepsilon$ toward $\varphi^{T,g}$.  
\end{itemize}
\end{proof}
\begin{proof}[Proof of \Cref{prop:properties} under \Cref{ass:low-reg}]
The proof is the same as in the case when \Cref{ass:mild-reg} holds, except for how the Lipschitz estimate on the value functions $\varphi^K$ is derived. Here, we rely on \Cref{lemma:gradient_estimate_linearized problem}-\ref{item:grad_est_bdd+Lip} noting that it remains valid even if we restrict minimization in \eqref{eq:classical_control_problem} to controls that are almost surely bounded by a given constant.
\end{proof}
\subsection{Pontryagin optimality conditions and hessian bounds}\label{sec:fin_dim_Hess_est}
This section is devoted towards establishing hessian bounds of the value function associated to the control problem \eqref{eq:classical_control_problem}. As a first step, we establish the link to the Pontryagin optimalilty conditions under the following additional assumption.
\begin{assumption}\label{ass-Pontryagin}
We assume $b\in C^{0,1,\beta}_{\mathrm{loc}}((0,T)\times\bbR^d)$, $\ell\in C^{0,1,2,\beta}_{\mathrm{loc}}((0,T)\times\bbR^d\times\bbR^d)$, $g\in C^{1,\beta}_{\mathrm{loc}}(\bbR^d)$.
\begin{enumerate}[label=(\roman*),wide]
\item  There exist constants $\const^{b_s}_x,\const^{\ell_s}_{xu},\const^{\sigma}_{xx}$ uniformly bounded in $s\in[0,T]$ and such that 
\bes
\begin{split}
\|b_s \|_\Lip\leq \const_x^{b_s}, \quad 
\|\partial_{xu}\ell_s(\cdot,\cdot) \|_\infty\leq \const^{\ell}_{xu}\\
\|\partial_{x_i}\sigma(x)-\partial_{x_i}\sigma(\hat{x})\|_{\mathrm{Fr}} \leq \const^\sigma_{xx}|x-\hat{x}| \quad \forall x,\hat{x}\in\bbR^d.
\end{split}
\ees
\end{enumerate}
\end{assumption}

 \begin{prop}\label{prop:Pontryagin}
Assume \Cref{ass:drift_field_fin_dim_control}, \Cref{ass-coercivity}, \Cref{ass:mild-reg}, \Cref{ass-Pontryagin} and fix $(t,x)$. If $(X_s)_{s\in[t,T]}$ is the optimal process defined at \Cref{prop:properties}-\ref{item:fin_dim_OptContr} and we set 
 \be\label{eq:def_adjoint}
Y_{s}:= \nabla \varphi^{T,g}_s(X_s), \quad Z_{s} := \nabla^2\varphi^{T,g}_{s}(X_{s}), \quad s\in[t,T],\ee
 then $(X_\cdot,Y_\cdot,Z_\cdot)$ form a solution to the FBSDE system
\begin{equation}\label{eq:Pontryagin}
\begin{split}
\bec \De X_s &= \partial_{p} h_s(X_s, Y_s)\De s + \sigma(X_{s})\,\De B_s \eqsp,  \\
\De Y_s &= -[\partial_{x} h_s(X_s, Y_s) + \tr\left(\partial_x\sigma(x)^{\top}Z_s \sigma(X_{s})\right)]\De s
    + Z_s \,\sigma(X_{s}) \, \De B_s\eqsp,\\
X_t &= x\eqsp,\quad  Y_T = \nabla g(X_T) \eqsp .
\eec
\end{split}
\end{equation}
where $h_s$ is defined in \eqref{eq:Hamiltonian_fin_dim} and 
\begin{align}\label{eq:dphs}
    \partial_{p}h_s(x,p) = b_s(x)\cdot p + w_s(x,p).\quad
    \partial_{x}h_s(x,p) = \partial_x b_s(x) \cdot p + \partial_{x}\ell_s(x,w_s(x,p)),
\end{align}
and where the trace has to be understood as
\begin{equation}
\tr\left(\partial_x\sigma^{\top}q\sigma\right)_i(x) = \tr\left(\partial_{x^i}\sigma^{\top}q\sigma\right)(x).
\end{equation}

 \end{prop}
 To prove this, result, we rely on classical results by Krylov \cite{Krylov:holder}, which allow to improve the regularity of the solutions of parabolic PDEs under additional regularity assumptions on the coefficients.
 \begin{proof}
We begin by showing that adding \Cref{ass-Pontryagin} implies that $\varphi^{T,g}$ enjoys more regularity than implied by \Cref{prop:properties}. To this aim, observe that using that $\varphi^{T,g}\in\cX$, where $\cX$ has been defined in \eqref{eq:def_calX}, the Lipschitz estimate from \Cref{lemma:gradient_estimate_linearized problem} and \Cref{ass-Pontryagin} we have that 
\bes
(t,x)\mapsto \partial_x\big( h_t(\cdot,\nabla\varphi^{T,g}_t(\cdot)\big)(x)
\ees
is of class $C^{\beta/2,\beta}_{\mathrm{loc}}((0,T)\times\bbR^d)$. This, together with $\varphi^{T,g}\in\cX$ and \Cref{ass-Pontryagin} entitle us to apply \cite[Theorem 8.12.1]{Krylov:holder}, which gives that $\varphi^{T,g} \in C^{1,3}((0,T) \times \bbR^{d})$.
Thanks to the improved regularity, we can write the equation satisfied by $\varphi^{i} := \partial_{x_i} \varphi^{T,g}$ by differentiating the HJB equation \eqref{eq:fin_dim_HJB}. We obtain
\begin{equation}\label{eq:HJB_linearized_der}
\begin{aligned}
\partial_s \varphi^i_s(x) &+ \frac{1}{2}\tr\left(\sigma^{\top}\nabla^2\varphi^{i}_s\sigma\right)(x) +\tr\left(\partial_{x_i}\sigma^{\top}\nabla^2\varphi^{T,g}_s\sigma\right)(x) \\
&+ \partial_{x_i} h_s(x,\nabla\varphi^{T,g}_s(x)) + \partial_p h_s(x,\nabla\varphi^{T,g}_s(x))\cdot\nabla\varphi^i_s(x)  = 0
\end{aligned}
\end{equation}
Again thanks of the improved regularity estimate on $\varphi^{T,g}$ we can apply Itô's formula to $Y^i_{s} = \varphi^i_s(X_{s})$ to obtain
\begin{align}
\De Y^i_{s} = \Big[ &\partial_s \varphi^{i}_s(X_s) +  \partial_p h_s(X_s,\nabla \varphi^{T,g}_s(X_s))\cdot \nabla\varphi^{i}_s(X_s) +\frac{1}{2}\tr\left(\sigma(X_s)^{\top}\nabla^2 \varphi^{i}_s(X_s)\sigma(X_s)\right)\Big] \De s + Z^i_s\cdot \sigma(X_s)\De B_s \\
\overset{\eqref{eq:HJB_linearized_der}}{=} \ &-[\partial_{x_i}h_s(X_s,\nabla \varphi^{T,g}_s(X_s)) + \tr\left(\partial_{x_i}\sigma(X_s)^{\top}Z_s \sigma(X_{s})\right)] \De s+ Z^i_s \cdot \sigma(X_{s})\De B_s,
\end{align}
where $Z^i_s$ is the $i$-th line of $Z_s$.
\end{proof}
Let us now pass to the proof of the hessian bounds of the value function via the Pontryagin optimality conditions. To quantify them, given a differentiable function $\varphi$ let us introduce the notation 
\bes
\|\nabla\varphi\|_{\mathrm{Lip}}:=\sup_{x\neq\hat{x}}\frac{|\nabla\varphi^{T,g}_t(x)-\nabla\varphi^{T,g}_t(\hat{x})|}{|x-\hat{x}|}.
\ees
For readability, the proofs are divided into the cases of a constant diffusion coefficient depicted in \Cref{sec:Hess_sigma_const} and position-dependent diffusion coefficient in \Cref{sec:Hess_sigma_non_const}.
\subsubsection{Constant diffusion coefficient}\label{sec:Hess_sigma_const}
\begin{lemma}\label{lem:hessian_bound_sigma_constant}
Let $\sigma=\rmI$  and assume that \Cref{ass:drift_field_fin_dim_control}, \Cref{ass-coercivity}, \Cref{ass:mild-reg}, \Cref{ass-Pontryagin} hold.
\begin{enumerate}[label=(\roman*)]
\item\label{item:linearized_hess_est_1} For all $t\leq T$ we have
\bes\label{eq:hessian_bound_val_fun_linear}
\lip{\nabla\varphi^{T,g}_t}=:\const^{\varphi_t}_{xx}\leq \int_t^T 2q^{\bar\kappa}_{s-t}(\const_x^{b_s}\const_x^{\varphi_s}+\const^{\ell_s}_x)\De s+ \min\{\ke{2}\const_x^g\, q^{\bar\kappa}_{T-t} ,C^{-1}_{\bar\kappa}\const^g_{xx}e^{-\lambda_{\bar\kappa}(T-t)}\},
\ee
where $\const_x^{\varphi_s}$ is given in \eqref{eq:grad_est_high_reg}, $\const^g_{xx}=\|\nabla g\|_\Lip$, and for all $r>0$
\begin{equation}\label{eq:kappa_hessian_estimate}
\bar\kappa(r) = \kappa_b(r) - r^{-1}\sup_{s\in[0,T]} 2\const^{u_s} 
\end{equation}
and $\const^{u_s}$ the supremum norm of the optimal control, given by \eqref{eq:grad_est_high_reg}, which is uniformly bounded in time.
\item\label{item:linearized_hess_est_2} If we further assume that there exists $\rho_b>0$ such that
\bes
-\partial_xb_s(x)\succeq\rho_b\mathrm{I} \quad \forall x\in\bbR^d,s\in[0,T],\ees
then we have that for all $t\leq T$
\bes
\|\nabla\varphi^{T,g}_t\|_{\mathrm{Lip}}=\const^{\varphi_t}_{xx}\leq\int_t^T 2e^{-\rho_b(s-t)}q^{\bar\kappa}_{s-t}(\const_x^b\const_x^{\varphi_s}+\const^{\ell_s}_x)\De s+ \const_x^g e^{-\rho_b(T-t)} q^{\bar\kappa}_{T-t}.
\ees
Moreover, we have
\be\label{eq:hessian_bound_control}
\| w_t(\cdot,\nabla\varphi^{T,g}_t(\cdot)) \|_\Lip =:\const^{u_t}_x\leq \frac{{\const^{\varphi_t}_{xx}}+\const^{\ell_t}_{xu}}{\rho^{\ell}_{uu}}.
\ees
\end{enumerate}
\end{lemma}
As it was the case for Lipschitz estimates, the Lipschitz bound on the optimal control policy at item \ref{item:linearized_hess_est_2} is a direct consequence of the definition of $h_t$ and \Cref{ass-coercivity}, see \Cref{lem:Est_Contr_Hamiltonian} for details. 
\begin{proof}
We start with the proof of \ref{item:linearized_hess_est_1}. For this, fix $t,x,\hat{x}$ and consider coupling by reflection of the optimal processes corresponding to the intital conditions $(t,x)$ and $(t,\hat{x})$. That is to say, we consider the pair $(X_s,\hat{X}_s)_{s\in[t,T]}$ given by \eqref{eq:coup_by_ref_contraction_2}-\eqref{eq:coup_by_ref_contraction_2_2} for the choices 
\bes
\begin{split}
\beta_s(x) =b_s(x)+w_s(x,\nabla\varphi^{T,g}_s(x)),\quad 
\alpha_s(x)  =0 \quad 
\quad \forall x \in\bbR^d,s\in[t,T],
\end{split}
\ees
where $w_s(x,p)$ is the optimal Markov policy defined at \eqref{eq:optimal_policy_def}. 
Moreover, \Cref{lemma:gradient_estimate_linearized problem}-\ref{item:grad_est_Lip} implies that
\begin{equation}\label{eq:kappa_1_majorization}
\kappa_{\beta_s}(r) \geq \kappa_b(r)-\frac{2\const^{u_s}}{r}\geq \bar\kappa(r) \quad\forall s\in[t,T],
\end{equation}
with $\bar\kappa$ as in \eqref{eq:kappa_hessian_estimate}.
If we define the adjoint processes $(Y_\cdot,Z_\cdot),(\hat{Y}_\cdot,\hat{Z}_\cdot)$ as in \eqref{eq:def_adjoint}, \Cref{prop:Pontryagin} gives
\be\label{eq:adj_coupl}
\bec
\De Y_s &= -[\partial_{x} \ell_s(X_s,w_s(X_s,Y_s)) +\partial_x b_s(X_s) Y_s]\De s
    + Z_s \, \De B_s\eqsp,\quad Y_T=\nabla g(X_T),\\
    \De \hat{Y}_s &= -[\partial_{x} \ell_s(\hat{X}_s,w_s(\hat{X}_s,\hat{Y}_s)) +\partial_x b_s(\hat{X}_s) \hat{Y}_s]\De s
    + \hat{Z}_s \, \De \hat{B}_s\eqsp,\quad
 \hat{Y}_T=\nabla g (\hat{X}_T).
\eec
\ee
Defining $e^{Y}_s=|Y_s-\hat{Y}_s|^{-1} (Y_s-\hat{Y}_s)$ for $s<T_0$ 
and $e^{Y}_s=v$ with $v$ some vector of unitary length for $s>T_0$, we find, applying It\^o formula to the convex function $|\cdot|_a=\sqrt{|\cdot|^2+a}$, using \eqref{eq:adj_coupl} and eventually letting $a\downarrow 0$ 
\begin{equation}\label{eq:diff_of_Y_Ito_differential}
\begin{split}
 \bbE[|Y_s-\hat{Y}_s|]-\bbE[|Y_{s'}-\hat{Y}_{s'}|] \leq \int_s^{s'}\bbE[( \partial_xb_r(X_r) Y_r -  \partial_xb_r(\hat{X}_r) \hat{Y}_r)\cdot \rme^Y_r\\+ |\partial_x\ell_r(X_r,Y_r) -\partial_x\ell_r(\hat{X}_r,\hat{Y}_r)| ] \De r, \quad \forall s,s'\in[t,T].
\end{split}
\end{equation}
 To prove the first of the two estimates we start by setting $s=t,s'=T$ and observe that
 \bes
 \begin{aligned}
|( \partial_xb_r(X_r)\cdot Y_r -  \partial_xb_r(\hat{X}_r)\hat{Y}_r)\rme^Y_r| &\leq 2 \const_x^{b_s}\const_x^{\varphi_r}\IND_{\{X_r\neq \hat{X}_r\}}\\
|\partial_x\ell_r(X_r,Y_r) -\partial_x\ell_r(\hat{X}_r,\hat{Y}_r)| &\leq 2\const_x^{\ell_r}\IND_{\{X_r\neq\hat{X}_r\}}\\
|\partial_xg(X_T)-\partial_xg(\hat{X}_T)|&\leq\min\{ 2\const^g_x\IND_{\{X_T\neq\hat{X}_T\}}, \const^g_{xx}|X_T - \hat{X}_T|\}.
 \end{aligned}
 \ees
Plugging these bounds in \eqref{eq:diff_of_Y_Ito_differential} gives
\bes
|\nabla\varphi^{T,g}_t(x)-\nabla\varphi^{T,g}_t(\hat{x})|\leq \int_t^T 2(\const_x^b\const_x^{\varphi_r}+\const^{\ell_r}_x)P[X_r\neq \hat{X}_r]\,\De r +\min\{2\const_x^g\,P[X_T\neq \hat{X}_T],\const^g_{xx}\bbE_P[|X_T - \hat{X}_T|]\}
\ees
The desired bound \eqref{eq:hessian_bound_val_fun_linear} now follows from \eqref{eq:kappa_1_majorization} and \Cref{prop:contr_same_drift}. The proof of \ref{item:linearized_hess_est_1} is now complete and we can proceed to the proof of \eqref{item:linearized_hess_est_2}. To this aim, we resume from \eqref{eq:diff_of_Y_Ito_differential} and observe that thanks to the assumption on $\partial_xb_r $ we have
\bes
\begin{split}
(\partial_xb_r(X_r) Y_r -  \partial_xb_r(\hat{X}_r)\hat{Y}_r)\cdot\rme^Y_r &= (\partial_xb(X_r)(Y_r -  \hat{Y}_r))\cdot \rme^Y_r\\
&+((\partial_xb(X_r)-\partial_xb(\hat{X}_r))\hat{Y}_r)\cdot\rme^Y_r \\
&\leq -\rho_b|Y_r-\hat{Y}_r|+2\const^{b_r}_{x}\const^{\varphi_r}_x\IND_{\{X_r\neq\hat{X}_r\}}
\end{split}
\ees
Handling the other terms in the right hand side of \eqref{eq:diff_of_Y_Ito_differential} as we did in the previous case, we obtain that, upon setting $\zeta(s)=\bbE[|Y_s-\hat{\bar{Y}}_s|]$ we have that 
\bes
 \zeta(s)\leq -\rho_b\int_{s}^{s'}\zeta(r)\De r+ \int_s^{s'}2(\const^{b_r}_{x}\const^{\varphi_r}_x+\const_x^{\ell_r})P[X_r\neq\hat{X}_r]\De r+ \zeta(s') \quad \forall s,s'\in[t,T].
\ees
The desired conclusion now follows invoking \Cref{prop:contr_same_drift}-\ref{item:final_TV_coup_by_ref}
and an application of a suitable version of Grönwall's lemma, see \Cref{lem:Gronwall_weak_form}.

\end{proof}

\subsubsection{Non-constant diffusion coefficient}\label{sec:Hess_sigma_non_const}

\begin{lemma}\label{thm:lin_Hess_est}
Assume \Cref{ass:drift_field_fin_dim_control}, \Cref{ass-coercivity},\Cref{ass:mild-reg}, \Cref{ass-Pontryagin}. Moreover, set
\be\label{eq:hess_bound_gen_form}
\const^{\ell,T}_x=\sup_{t\in[0,T]}\const^{\ell_s}_x, \quad\const^{b,T}_x=\sup_{t\in[0,T]} \const^{b_s}_x.
\ee
\begin{enumerate}[label=(\roman*)]
\item\label{item:Hess_est_mild_Hessbdd} There exists a finite positive constant $\const_{xx}(\cdot)$ depending only on $\const^{\sigma}_{x},\kappa_b,\rho^{\ell}_{uu},\const^{\ell(\cdot,0)}_u,\const^{b,T}_{x},\const^{\ell,T}_x,\sup_{s\leq T}\const^{\varphi_s}_x,\const^{\sigma}_{xx}$ and such that for all $t<T$
\begin{equation}\label{eq:gen_hess_est_1}
\begin{split}
    \| \nabla \varphi^{T,g}_t\|_{\Lip}  =\const^{\varphi_t}_{xx}\leq \const_{xx}(\const^{\sigma}_{x},\kappa_b,\rho^{\ell}_{uu},\const^{\ell(\cdot,0)}_u,\const^{b,T}_{x},\const^{\ell,T}_x,\sup_{s\leq T}\const^{\varphi_s}_x,\const^{\sigma}_{xx})
    +4(\const_{x}^g)^2\tilde{C}^{-1}_{\bar\kappa}e^{-\tilde\lambda_{\bar\kappa}(T-t)}\\
    +(2\const_x^g+4(\const^g_x)^2)\tilde{q}^{\bar\kappa}_{T-t},
    \end{split}
    \end{equation}
where $\const_x^{\varphi_s}$ is given in \eqref{eq:grad_est_high_reg}.
Moreover, if $\|\nabla g\|_\Lip=\const^g_{xx}<+\infty$ we can replace $(2\const_x^g+4(\const^g_x)^2)\tilde{q}^{\bar\kappa}_{T-t}$ in the above by
\bes 2(\const^g_x+1)\const^{g}_{xx}\tilde{C}^{-1}_{\bar\kappa}e^{-\tilde\lambda_{\bar\kappa}(T-t) }\ees
\item\label{item:Hess_est_mild_Gradbdd}  
We have
\be\label{eq:hessian_bound_control}
\| w_t(\cdot,\nabla\varphi^{T,g}_t(\cdot)) \|_\Lip =\const^{u_t}_x\leq \frac{{\const^{\varphi_t}_{xx}}+\const^{\ell_t}_{xu}}{\rho^{\ell}_{uu}}
\ees
\end{enumerate}
\end{lemma}
As before, the proof is based on a coupling between the optimal processes relative to the initial conditions $(t,x)$ and $(t,\hat{x})$. However, deviating from what we have done so far we will not use standard coupling by reflection, but rather work with the interpolation between reflection and synchronous coupling we described at \Cref{sec:interpolation} for the choice
\bes \beta_s(x)=b(x)+w_s(x,\nabla\varphi^{T,g}_s(x)), \quad  \forall s\in[t,T],x\in\bbR^d.
\ees
The reason is roughly the following. While standard coupling by reflection creates precisely noise in the only direction associated to a non-positive (actually zero) eigenvalue of second derivative of the Euclidean distance $x \mapsto |x|$, helping us gain contractivity in the twisted metrics defined by $f_\kappa$, this is not enough for the adjoint dynamics $(Y_s)_s$ because we do not have information on the geometry induced by the associated diffusion coefficient $(Z_s)_s$. In that case the straightforward thing which seems to help is a synchronous coupling in order to account for all directions in the same way.
More precisely, given three independent Brownian motions $(B^1_{s},B^2_{s},B^3_{s})_{s\in[t,T]}$ we consider a weak solution of  

\begin{equation}\label{eq:coup_by_ref_int_4}
\begin{cases}
\De X_s = \partial_p h_s(X_s,Y_s)\De s+\frac{1}{\sqrt{2}}\sigma_0 \De B^1_s +\frac{1}{\sqrt{2}}\sigma_0 \De B^2_s + \bar{\sigma}(X_s) \De B^3_s, \, X_t = x \eqsp,\\
\De \hat{X}_s = \partial_p h_s(\hat{X}_s,\hat{Y}_s)\De s+ \frac{1}{\sqrt{2}}\sigma_0 \De \hat{B}^1_s +\frac{1}{\sqrt{2}}\sigma_0 \De B^2_s+  \bar{\sigma}(\hat{X}_s) \De B^3_s, \, \hat{X}_t = \hat{x} \eqsp ,
\end{cases}
\end{equation}
with 
\bes 
Y_s=\nabla\varphi^{T,g}_s(X_s),\quad \hat{Y}_s=\nabla\varphi^{T,g}_s(\hat{X}_s).
\ees
In accordance with the notation used so far, we have
\bes
\bar{\sigma}(x) =\big(\sigma(x)\sigma(x)^{\top} - \sigma_0^2\mathrm{I}\big)^{1/2}, \quad \De\hat{B}^1_s= (\mathrm{I}-2 \,\rme^X_s \cdot (\rme^X)^{\top}_s \IND_{T_0 > s})\cdot \De B^1_s
\ees
and
\bes
T_0 = \inf\{s \geq 0: X_s = \hat{X}_s \}, \quad \rme^X_s = \frac{X_s-\hat{X}_s}{|X_s-\hat{X}_s|} .
\ees
We also introduce the process 
\bes
\De B^{1,\perp}_s =\De B^1_s - \rme_s \De W_s, \quad \De W_s= \rme_s\cdot\De B^{1}_s,\quad s\in[t,T],
\ees
and note that $W_\cdot$ is a one-dimensional Brownian  motion and $(B^{1,\perp}_s,W_s,B^2_s,B^3_s)_{s\in[t,T]}$ are independent processes by construction.
Moreover, we recall that thanks to \Cref{prop:Pontryagin} the dynamics of the adjoint processes is given for $i=1,\ldots,d$ by
\be\label{eq:pontryagin_sec_der_sigma}
    \begin{cases}
        \De Y^i_s= -\partial_{x_i}h_s(X_s,Y_s)\De s-\tr(\partial_{x^i} \sigma(X_s)^{\top}  Z_s \sigma(X_s))\De s + Z^i_s \sigma(X_s) \De B_s,\\
        \De \hat{Y}^i_s= -\partial_{x_i}h_s(\hat{X}_s,\hat{Y}_s)\De s -\tr(\partial_{x^i} \sigma(\hat{X}_s)^{\top}  \hat{Z}_s \sigma(\hat{X}_s))\De s + \hat{Z}^i_s\sigma(\hat{X}_s)\De \hat{B}_s,
    \end{cases}
 \ee
 where 
 \bes
Z_s=\nabla^2\varphi_s^{T,g}(X_s), \quad\hat{Z}_s=\nabla^2\varphi_s^{T,g}(\hat{X}_s),
 \ees
 the Brownian motions $B_\cdot,\hat{B}_\cdot$ are defined by
\bes
\begin{split}
\De B_s &:= \sigma(X_s)^{-1}\Big(\frac{1}{\sqrt{2}}\sigma_0 \De B^1_s +\frac{1}{\sqrt{2}}\sigma_0 \De B^2_s + \bar{\sigma}(X_s) \De B^3_s\Big),\\
\De \hat B_s &:= \sigma(\hat X_s)^{-1}\Big(\frac{1}{\sqrt{2}}\sigma_0 \De \hat B^1_s +\frac{1}{\sqrt{2}}\sigma_0 \De B^2_s + \bar{\sigma}(\hat X_s) \De B^3_s\Big).
\end{split}
\ees
From now on, to increase readability, given a non-negative stochastic process $(\xi_s)_{s \in [t,T]}$ on we write $\cO(\xi_s)$ to denote a stochastic process $(\chi_s)_{s\in[t,T]}$ such that almost surely we have 
\bes 
|\chi_s| \leq C\xi_s \quad \forall s\in[t,T],
\ees
where
$C>0$ is a constant depending only on the quantities appearing at \Cref{ass:drift_field_fin_dim_control}, \Cref{ass-coercivity} \Cref{ass:mild-reg} and \Cref{ass-Pontryagin} and in particular independent of time. In order to prove \Cref{thm:lin_Hess_est} we need the following preliminary lemma.

\begin{lemma}\label{lemm:hess_bound_sigma_corrector}
Under the hypothesis of \Cref{thm:lin_Hess_est}, define \bes
\eta_s: =Z_s-\hat{Z}_s,\,\, \zeta_s:=Z_s+\hat{Z}_s,\,\, r^Y_s := |Y_s - \hat{Y}_s|, r^X_s=|X_s - \hat{X}_s| \quad s\in [t,T].
\ees
Then, the following holds.
\begin{enumerate}[label=(\roman*)]
\item\label{item:Est_dyn_distY} Let $\distY:\bbR_{\geq0}\longrightarrow\bbR_{\geq0}$ be given by $\distY(r)=r+r^2$. Then we have
\be\label{eq:dynamics_convex_modification_adjoint}
\begin{split}
\De \distY(r^Y_s) &\geq   \cO(\IND_{\{r^X_s>0\}})\De s+\cO(\tr(\eta^{\top}_s\eta_s)^{1/2}+r^X_s\tr(\zeta^{\top}_s\zeta_s)^{1/2})\De s +\frac{\sigma^2_0}{2}\tr(\eta^{\top}_s\eta_s)+ \De M^{\distY}_s,
\end{split}
\ee
where  $M^{\distY}_\cdot$ is a square integrable martingale.

\item\label{item:Est_Dyn_Y+hatY} We have
\bes
\begin{split}
\De |Y_s+\hat Y_s|^2 \geq \big( \cO(1)+\cO(\tr(\eta^{\top}_s\eta_s)^{1/2}+\tr(\zeta^{\top}_s\zeta_s)^{1/2})\big)\De s + \frac{\sigma^2_0}{2}\tr(\zeta^{\top}_s\cdot\zeta_s)\De s+2 ( Y_s+\hat Y_s)\cdot \De M^{Y+\hat Y}_s ,
\end{split}
\ees
where  $M^{Y+\hat Y}_\cdot$ is a square integrable martingale.
\item\label{item:d_rX} We have
\be\label{eq:d_rX}
\begin{split}
\De r^X_s&= \cO(r^X_s+\IND_{r^X_s>0})\De s+\sqrt{2}\sigma_0 \De W_s + \rme^{\top}_s( \bar\sigma(X_s)-\bar\sigma(\hat{X}_s))\De B^3_s.
\end{split}
\ee
\item\label{item:Est_Quad_Var_rX_Y+hatY} The quadratic covariation between $|Y_\cdot+\hat{Y}_\cdot|^2_\cdot$ and $r^{X}_\cdot$ satisfies
\be\label{eq:dynamics_corrector_term}
\De [|Y_\cdot+\hat{Y}_\cdot|^2,r^X_\cdot]_s 
=\cO\big(\tr(\eta^T_s\eta_s)^{1/2}+r^X_s\tr(\zeta^T_s \zeta_s)^{1/2}\big)\De s.
\ee
\end{enumerate}
\end{lemma}

\begin{proof}

\begin{itemize}[wide]
\item \textbf{Proof of \ref{item:Est_dyn_distY}}

From \eqref{eq:pontryagin_sec_der_sigma}, we see that the difference $Y_s-\hat{Y}_s$ satisfies
\be\label{eq:d_deltaY}
\begin{split}
\De (Y_s-\hat{Y}_s)&= [-\partial_{x}h_s(X_s,Y_s) + \partial_{x}h_s(\hat{X}_s,\hat{Y}_s)]\De s  \\
&-\tr(\nabla\sigma(X_s)^{\top}\cdot Z_s\cdot\sigma(X_s)-\nabla\sigma(\hat{X}_s)^{\top}\cdot \hat{Z}_s\cdot\sigma(\hat{X}_s))\De s+\De M^{Y-\hat{Y}}_s
\end{split}
\ee
where
\bes
\begin{split}
\De M^{Y-\hat{Y}}_s= \frac{\sigma_0}{\sqrt{2}} \eta_s\De B^{1,\perp}_s+\frac{\sigma_0}{\sqrt{2}} \zeta_s\rme^X_s\De W_s +\frac{\sigma_0}{\sqrt{2}}\eta_s\cdot\De B^2_s+(Z_s\cdot\bar{\sigma}(X_s)-\hat{Z}_s\bar{\sigma}(\hat{X}_s)) \De B^3_s. 
\end{split}
\ees
It follows from \Cref{lemma:gradient_estimate_linearized problem} and \Cref{ass:mild-reg} that 
\be\label{eq:h_x_bound}
[-\partial_{x}h_s(X_s,Y_s) + \partial_{x}h_s(\hat{X}_s,\hat{Y}_s)]=\cO(\IND_{\{r^X_s>0\}})
\ee
Next, observe that for all $i$
\bes
\begin{split}
\tr&\big( \partial_{x_i}\sigma(X_s)^{\top}Z_s\sigma(X_s)-\partial_{x_i}\sigma(\hat{X}_s)^{\top}\hat{Z}_s\cdot\sigma(\hat{X}_s) \big)\\
&=\tr\big((\sigma^{\top}\partial_{x_i}\sigma)(X_s) Z_s-(\sigma^{\top}\partial_{x_i}\sigma)(\hat{X}_s) \hat{Z}_s \big)\\
&=\frac12\tr\big((\sigma^{\top}\partial_{x_i}\sigma(X_i)+\sigma^{\top}\partial_{x_i}\sigma(\hat{X}_i)) \eta_s + (\sigma^{\top}\partial_{x_i}\sigma(X_i)-\sigma^{\top}\partial_{x_i}\sigma(\hat{X}_i)) \zeta_s \big)
\end{split}
\ees
whence, thanks to \Cref{ass:drift_field_fin_dim_control},\Cref{ass-Pontryagin}
\be\label{eq:trace_bound}
\tr\big( \partial_{x_i}\sigma(X_s)\cdot Z_s\cdot\sigma(X_s)-\partial_{x_i}\sigma(\hat{X}_s)\cdot \hat{Z}_s\cdot\sigma(\hat{X}_s) \big) = \cO(\tr(\eta^{\top}_s\eta_s)^{1/2}+r^X_s\tr(\zeta^{\top}_s\zeta_s)^{1/2}).
\ee
Plugging \eqref{eq:trace_bound} and \eqref{eq:h_x_bound} in \eqref{eq:d_deltaY} we obtain that 
\bes
\De (Y_s-\hat Y_s) = [\cO(\IND_{\{r^X_s>0\}})+\cO(\tr(\eta^{\top}_s\eta_s)^{1/2}+r^X_s\tr(\zeta^{\top}_s\zeta_s)^{1/2})]\De s +\De M^{Y-\hat Y}_s
\ees
In view of computing the It\^o differential of $(r^Y_\cdot)^2$ observe that the independence of $(B^{1,\perp}_\cdot,W_\cdot,B^{2}_\cdot,B^3_\cdot)$ implies that, denoting $M^{Y-\hat Y,i}$ the $i$-th component of  $M^{Y-\hat Y}$, 
\bes
\De [M^{Y-\hat Y,i}_\cdot]_s \geq \frac{\sigma^2_0}{2} \sum_{j=1}^d(\eta^{ij}_{s})^2\De s \quad \forall i=1,\ldots,d.
\ees
But then, thanks to \Cref{lemma:gradient_estimate_linearized problem}
\bes
\begin{split}
\De (r^Y_s)^2 &= 2(Y_s-\hat Y_s)\De (Y_s-\hat Y_s) + \sum_{i=1}^d\De[M^{Y-\hat Y,i}]\\
&{\geq}  \cO(\IND_{\{r^X_s>0\}})\De s+\cO(\tr(\eta^{\top}_s\eta_s)^{1/2}+r^X_s\tr(\zeta^{\top}_s\zeta_s)^{1/2})\De s \\
& +\frac{\sigma^2_0}{2}\tr(\eta^{\top}_s\eta_s)\De s+ 2(Y_s-\hat{Y}_s)  \De M^{Y-\hat Y}_s.
\end{split}
\ees
Likewise, using the convexity of $|\cdot|$ and setting $\rme^Y_s=(r^Y_s)^{-1}(Y_s-\hat{Y}_s)$
\bes
\begin{split}
\De r^Y_s &\geq \rme^Y_s\cdot  \De (Y_s-\hat Y_s)  \geq \cO(\IND_{\{r^X_s>0\}})\De s+\cO(\tr(\eta^{\top}_s\eta_s)^{1/2}+r^X_s\tr(\zeta^{\top}_s\zeta_s)^{1/2})\De s +\rme^Y_s\cdot \De M^{Y-\hat{Y}}_s.
\end{split}
\ees 

Gathering the last two bounds, we obtain
\bes
\begin{split}
\De \distY(r^Y_s) \geq  \cO(r^X_s+\IND_{r^X_s>0})\De s\cO(\tr(\eta^{\top}_s\eta_s)^{1/2}+r^X_s\tr(\zeta^{\top}_s\zeta_s)^{1/2})\De s
+\frac{\sigma^2_0}{2}\tr(\eta^{\top}_s\eta_s)\De s + \De M^{\distY}_s
\end{split}
\ees
where $M^{\distY}_s$ with
\bes
\De M^{\distY}_s = (\rme^Y_s\cdot +2(Y_s-\hat Y_s))\cdot \De M^{Y-\hat Y}_s.
\ees
\item \textbf{Proof of \ref{item:Est_Dyn_Y+hatY}}

 We start with the observation that \Cref{ass:mild-reg} and the Lipschitz estimate from \Cref{lemma:gradient_estimate_linearized problem}
\bes\label{eq:SDE_Y+hatY}
\De (Y_s+\hat{Y}_s) = \cO(1)\De s
+\tr(\sigma^{\top} \partial_x \sigma(X_s) Z_s+\sigma^{\top} \partial_x \sigma(\hat{X}_s) \hat{Z}_s)\De s +\De M^{Y+\hat Y}_s
\ees
with
\bes\label{eq:mart_Y+hatY}
\begin{split}
\De M^{Y+\hat Y}_s&=\frac{\sigma_0}{\sqrt{2}} \zeta_s\De B^{1,\perp}_s+\frac{\sigma_0}{\sqrt{2}} \eta_s\rme^X_s\De W_s +\frac{\sigma_0}{\sqrt{2}}\zeta_s\De B^2_s+(Z_s\cdot\bar{\sigma}(X_s)+\hat{Z}_s\cdot\bar{\sigma}(\hat{X}_s)) \De B^3_s. 
\end{split}
\ees
Next, observe that
\bes
\begin{split}
&\tr\big((\sigma^{\top}\partial_{x_i}\sigma)(X_s)t Z_s+(\sigma^{\top}\partial_{x_i}\sigma)(\hat{X}_s)\hat{Z}_s \big)\\
=&\frac12\tr\big((\sigma^{\top}\partial_{x_i}\sigma(X_s)+\sigma^{\top}\partial_{x_i}\sigma(\hat{X}_s)) \zeta_s + (\sigma^{\top}\partial_{x_i}\sigma(X_s)-\sigma^{\top}\partial_{x_i}\sigma(\hat{X}_s))\eta_s \big)
\end{split}
\ees
whence 
\be\label{eq:trace_bound_2}
\tr\big( \partial_{x_i}\sigma(X_s)\cdot Z_s\cdot\sigma(X_s)+\partial_{x_i}\sigma(\hat{X}_s)\cdot \hat{Z}_s\cdot\sigma(\hat{X}_s) \big) = \cO(\tr(\eta^{\top}_s\eta_s)^{1/2}+\tr(\zeta^{\top}_s\zeta_s)^{1/2}).
\ee
The independence of $(B^{1,\perp}_\cdot,W_\cdot,B^{2}_\cdot,B^3_\cdot)$ implies that, denoting $M^{Y+\hat Y,i}$ the $i$-th component of  $M^{Y+\hat Y}$, 
\bes
\De [M^{Y-\hat Y,i}_\cdot]_s \geq \frac{\sigma^2_0}{2} \sum_{j=1}^d(\zeta^{ij}_{s})^2\De s \quad \forall i=1,\ldots,d.
\ees
But then,
\bes
\begin{split}
\De |Y_s+\hat Y_s|^2 &=2(Y_s+\hat Y_s)\cdot\De(Y_s+\hat{Y}_s)+\sum_{i=1}^d \De [M^{Y+\hat Y,i}_\cdot]_s\\
&\geq -\cO(1)\De s-\cO(\tr(\eta^{\top}_s\eta_s)^{1/2}+\tr(\zeta^{\top}_s\zeta_s)^{1/2})\De s + {\frac{\sigma_0^2}{2}}\tr(\zeta^{\top}_s\cdot\zeta_s)\De s+2 ( Y_s+\hat Y_s)\cdot \De M^{Y+\hat Y}_s.
\end{split}
\ees

\item \textbf{Proof of \ref{item:d_rX}} An application of It\^o formula gives
\bes
\begin{split}
\De r^X_s&= (\rme^X_s)^\top(\partial_ph_s(X_s,Y_s)-\partial_ph_s(\hat{X}_s,\hat{Y}_s))\De s+\sqrt{2}\sigma_0 \De W_s+ \rme^{\top}_s( \bar\sigma(X_s)-\bar\sigma(\hat{X}_s))\De B^3_s.
\end{split}
\ees
The desired result now follows from \Cref{ass-Pontryagin} and \Cref{lemma:gradient_estimate_linearized problem}.
\item \textbf{Proof of \ref{item:Est_Quad_Var_rX_Y+hatY}}
  Combining \ref{item:d_rX} with \eqref{eq:SDE_Y+hatY} and \eqref{eq:mart_Y+hatY} this gives that the quadratic covariation between $|Y_s+\hat Y_s|^2$ and $r^X_s$ is 
\begin{align}\label{eq:quad_cov_1}
\De [|Y_\cdot+\hat{Y}_\cdot|^2,r^X_\cdot]_s  
= \Big({2}\sigma^2_0(Y_s+\hat{Y}_s)^{\top}\eta_s\rme^X_s+2(Y_s+\hat{Y}_s)^{\top}(\bar\sigma(X_s)Z_s
+\bar\sigma(\hat{X}_s)\hat Z_s)(\bar\sigma^{\top}(X_s)-\bar\sigma^{\top}(\hat{X}_s)) \rme^X_s\Big)\De s
\end{align}
Using \Cref{lemma:gradient_estimate_linearized problem} and Cauchy-Schwartz inequality we obtain
\begin{equation}
2\sigma^2_0(Y_s+\hat{Y}_s)^{\top}\cdot \eta_s\cdot \rme^X_s = \cO(\tr(\eta^{\top}_s\eta_s)^{1/2}).
\end{equation}
In light of this, it only remain to show that the second term in \eqref{eq:quad_cov_1} is $\cO(r^X_s\tr(\zeta^{\top}_s\zeta_s)^{1/2}+\tr(\eta^{\top}_s\eta_s)^{1/2})$. To to this, observe that writing $Z_s=(\zeta_s+\eta_s)/2,\hat{Z}_s=(\zeta_s-\eta_s)/2$, we find the equivalent form 
\bes
\begin{split}
&(Y_s+\hat{Y}_s)^{\top}(\bar\sigma(X_s)+\bar\sigma(\hat{X}_s))\zeta_s\cdot (\bar\sigma^{\top}(X_s)-\bar\sigma^{\top}(\hat{X}_s))\rme^X_s +(Y_s+\hat{Y}_s)^{\top}(\bar\sigma(X_s)-\bar\sigma(\hat{X}_s)) \eta_s(\bar\sigma^{\top}(X_s)-\bar\sigma^{\top}(\hat{X}_s)) \rme^X_s 
\end{split}
\ees

Using \Cref{lemma:gradient_estimate_linearized problem}, the Lipschitzianity of $\sigma $ to bound $\bar\sigma^\top(X_s)-\bar\sigma^\top(\hat{X}_s)$, the boundedness of $\sigma$ and Cauchy-Schwartz inequality we see that the first term in the above display is $\cO(r^X_s\tr(\zeta^{\top}_s\zeta_s)^{1/2})$. Bounding the second term in same way except for the fact that the difference $\bar\sigma^\top(X_s)-\bar\sigma^\top(\hat{X}_s)$ is bounded using the boundedness of $\sigma$ we obtain that it is $\cO(\tr(\eta^{\top}_s\eta_s)^{1/2})$.
\end{itemize}

\end{proof}

\begin{proof}[Proof of \Cref{thm:lin_Hess_est}]
Observing that $\eta_s=\eta_s\IND_{r^X_s>0}$ we obtain directly from \Cref{lemm:hess_bound_sigma_corrector} that

\be\label{eq:hess_bound_sigma_ito_1}
\begin{split}
\De \big(\distY(r^Y_s)+ |Y_s+\hat{Y}_s|^2r^X_s\big) &= \cO(r^X_s+\IND_{r^X_s>0})+\cO(\IND_{r^X_s>0}+r^X_s)\tr(\eta^T_s\eta_s)^{1/2}+r^X_s\tr(\zeta^T_s\zeta_s)^{1/2}) \\
&+\frac{\sigma^2_0}{2}\tr(\eta^T_s\eta_s)+\frac{\sigma^2_0}{2}r^X_s\tr(\zeta^T_s\zeta_s) + \De M_s,
\end{split}
\ee
where $M_\cdot$ is a square integrable martingale.
Quadratic minimization in the variables $\tr(\eta^T_s\eta_s)^{1/2}$ and respectively $\tr(\zeta^T_s\zeta_s)^{1/2}$ implies that the right hand side is bounded below by
\bes
\De \big(\distY(r^Y_s)+ |Y_s+\hat{Y}_s|^2r^X_s\big) \geq \cO\left((r^X_s)^2+r^X_s+\IND_{\{r^X_s>0\}}\right) \De s + \De M_s,
\ees
We can now integrate the above inequality, take expectation on both sides and apply \Cref{prop:contr_same_drift_interp} (for the choice $\beta_s(x)=b(x)+w_s(x,\nabla\varphi^T,g_s(x))$) to obtain
\bes
\begin{split}
\bbE[\distY(r^Y_t)+|Y_s+\hat{Y}_s|^2r^X_t]&\leq \bbE[\distY(r^Y_T)+|Y_T+\hat{Y}_T|^2r^X_T]+\const_0|x-\hat{x}|\int_{t}^T(\tilde\const_{\bar\kappa,2}e^{-\tilde{\lambda}_{\kappa,2}(s-t)}+\tilde\const_{\bar\kappa} e^{-\tilde\lambda_{\bar\kappa}(s-t)}+\tilde{q}^{\bar\kappa,\tilde{\lambda}_{\bar\kappa}}_{s-t})\De s,
\end{split}
\ees
where $\bar\kappa$ is as in \eqref{eq:kappa_hessian_estimate}  and $\const_0$ is some constant independent of time, and depending only on the quantities appearing at \Cref{ass:drift_field_fin_dim_control}, \Cref{ass-coercivity} \Cref{ass:mild-reg} and \Cref{ass-Pontryagin}.
Since 
\bes
\bbE[|Y_T+\hat{Y}_T|^2r^X_T]\leq 4(\const^g_x)^2\bbE[r^X_T] \leq 4(\const^g_x)^2 \tilde{\const}_{\bar\kappa} e^{-\tilde\lambda_{\bar\kappa}(T-t)} |x-\hat{x}|
\ees
and
\bes \bbE[\distY(r^Y_T)]\leq (2\const^x_g+(4\const_x^g)^2) \,\tilde{q}^{\bar\kappa,\tilde{\lambda}_{\bar\kappa}}_{T-t}\,|x-\hat{x}| \ees 

we obtain \eqref{eq:gen_hess_est_1}. To conclude, we observe that if we assume $\const_{xx}^g<+\infty$ we have estimate differently $\bbE[\distY(r^Y_T)] $ as follows
\bes
\bbE[\distY(r^Y_T)] \leq (2\const_x^g+1)\const^g_{xx}\tilde\const_{\bar\kappa}e^{-\tilde\lambda_{\bar\kappa} (T-t)}
\ees
\end{proof}

\subsection{Stability estimates}\label{sec:stab_lin_contr}

In this section we aim at comparing the value functions of two different stochastic control problems and the laws of the corresponding optimal processes. To this aim, we introduce an auxiliary stochastic control problem whose structure is the same as \eqref{eq:classical_control_problem}. Namely, for given running  and terminal cost $\hat\ell,\hat g$, drift field $\hat b$ we consider

\begin{equation}\label{eq:classical_control_problem_hat}
\hat\varphi^{T,\hat{g}}_t(x)=\inf_{u\in\cU_{t,T}} \hat{J}_{t,x}^{T,\hat{g}}(u),  
\ee
where for any admissible control $u$ the corresponding dynamics is
\be
 \De X^{u}_s= [\hat b_s(X^{u}_s) + u_s]\De s +\sigma(X^{u}_s)\cdot\De B_s, \quad X^u_t=x,
\end{equation}
and $\sigma(\cdot)$ is the same diffusion coefficient of \eqref{eq:classical_control_problem}. Finally, we define the cost functional in the obvious way

\bes \hat J_{t,x}^{T,\hat g}(u)=\bbE[\int_t^T\hat \ell_s(X^u_s,u_s)\,\De s + \hat g(X^u_T)].\ees

We shall consider two scenarios. In the first, one controls the difference between the drift fields and running costs in Lipschitz norm, in the second, the bound is expressed in supremum norm. In both cases, the difference between terminal costs is expressed in Lipschitz seminorm though this requirement will be later relaxed in \Cref{sec:MF_PDE} by an early stopping argument.
  \begin{assumption}\label{ass:stab_high_reg}
For all $s\in[0,T]$ there exist a function  $\delta\ell_s:\bbR^d\rightarrow \bbR$ such that 
\bes 
\ell_s(x,u)-\hat\ell_s(x,u)=\delta\ell_s(x) \quad \forall x,u\in\bbR^d.
\ees
We assume $b_s=\hat{b}_s$ for all $s\in[0,T]$ Moreover, there exist finite constants $\const^{\delta\ell_s}_x,\const^g_x$ uniformly bounded in $s\in[0,T]$ and such that 
\begin{equation}\label{eq:stab_est_W1_ass}
\begin{split}
\|\delta\ell_s\|_\Lip \leq \const^{\delta \ell_s}_x, \quad  \|g-\hat{g}\|_\Lip \leq \const^{\delta g}_x.
\end{split}
\end{equation}
\end{assumption}
  
\begin{assumption}\label{ass:stab_mild_reg}
There exist finite constants $\const^{\delta b_s},\const^{\delta \ell_s},\const^{\delta\ell_s}_u,\const^{\delta g}_x$ uniformly bounded in $s\in[0,T]$ and such that
\begin{equation}\label{eq:stab_est_TV_ass}
\begin{split}
&\|b_s-\hat{b}_s\|_{\infty} \leq \const^{\delta b_s} , \quad \|\ell_s(\cdot,u)-\hat{\ell}_s(\cdot,u)\|_{\infty} \leq \const^{\delta \ell_s} ,\\
&\| \ell_s(x,\cdot)-\hat{\ell}_s(x,\cdot)\|_\Lip\leq \const^{\delta\ell_s}_u,\quad \|g-\hat{g}\|_{\Lip} \leq \const^{\delta g}_x 
\end{split}
\end{equation}
\end{assumption}
As anticipated above, we want to bound the distance between the laws of the optimal processes for \eqref{eq:classical_control_problem} and \eqref{eq:classical_control_problem_hat}. This of course means that we fix two initial distributions $\mu,\hat\mu\in\cP_1(\bbR^d)$ and compare $\mu_t=\cL(X_t)$ and $\hat{\mu}_t=\cL(\hat{X}_t)$ with 
\be\label{eq:opt_flow}
\De X_s = \partial_ph_s(X_s,\nabla\varphi^{T,g}_s(X_s))\De s+\sigma(X_s)\De B_s, \quad X_0\sim\mu
\ee
and
\be\label{eq:opt_flow_hat}
\De \hat{X}_s = \partial_ph_s(\hat{X}_s,\nabla\hat\varphi^{T,\hat g}_s(\hat X_s))\De s+\sigma(\hat X_s)\De B_s,\quad \hat X_0\sim\hat\mu.
\ee
The key to bounding the distance between the marginal laws is hence to obtain bounds on $\nabla\varphi^{T,g}_s - \nabla\hat\varphi^{T,\hat g}_s$.
\begin{lemma}\label{lem:stability_estimate_linear_pb}
Assume that \Cref{ass:drift_field_fin_dim_control}, \Cref{ass-coercivity} hold. Let also $0 \leq t\leq T$.
\begin{enumerate}[label=(\roman*)]
\item\label{item:stab_mild_reg} Assume that \Cref{ass:mild-reg} and \Cref{ass:stab_mild_reg} hold for both problem \eqref{eq:classical_control_problem} and problem \eqref{eq:classical_control_problem_hat}.
Then for all $t\leq T$ 
\begin{equation}\label{eq:stab_est_mild_reg_linear}
\|\varphi^{T,g}_t-\hat\varphi^{T,\hat{g}}_t\|_{\Lip}=:\const^{\delta\varphi_t}_x\leq 2\int_t^{T}\big(\const^{\delta \ell_s} +\const^{\delta b_s}\const^{\varphi_s}_x\big) q^{\tilde\kappa}_{s-t}\De s
+ e^{-\lambda_{\tilde\kappa}(T-t)}\|g-\hat{g} \|_{f_{\tilde\kappa}}, 
\end{equation}
with any profile $\tilde\kappa \in \msk$, where $\msk$ is defined in \eqref{eq:def_msk},  satisfying 
\be\label{eq:kappa_choice_new}
\tilde\kappa \leq \min_{t\leq s \leq T} \{\kappa_{\partial_p h_s(\cdot,\nabla\varphi^{T,g}_s(\cdot))},\kappa_{\partial_p h_s(\cdot,\nabla\hat\varphi^{T,\hat g}_s(\cdot))}\}.
\ee
In particular, we can choose
\begin{equation}\label{eq:kappa_2}
\tilde\kappa(r)=\bar\kappa(r):=\kappa_{b}(r) - \frac{2}{r}\sup_{t \leq  s \leq T}\max\{\const^{u_s},\const^{\hat{u}_s}\}
 \quad \forall r>0.
\end{equation}
and $\const^{u_s},\const^{\hat{u}_s}$ are the respective bounds in supremum norm on the optimal controls given by \eqref{eq:grad_est_high_reg}. As a consequence, 
\bes
\|\partial_ph_s(\cdot,
\nabla\varphi^{T,g}_s(\cdot))-\partial_p\hat{h}_s(\cdot,
\nabla\hat\varphi^{T,\hat g}_s(\cdot))\|_\infty=:\const^{\delta u_s}_{x}\leq \const^{\delta b_s}+\frac{(1+\const^{\delta\ell_s}_u)}{\rho^{\ell}_{uu}}\const^{\delta\varphi_s}_x. 
\ees
Moreover, we have that for all $0\leq t_0\leq t\leq T$ and all initial conditions $\mu,\hat\mu\in\cP_1(\bbR^d)$ the distance between the optimal laws $(\mu_t,\hat{\mu}_t)_{t\in[0,T]}$ can be bounded as follows

\begin{equation}\label{eq:stability_between_laws}
\begin{split}
W_{f_{\tilde\kappa}}(\mu_t,\hat{\mu}_t)&\leq e^{-\lambda_{\tilde\kappa}t} W_{f_{\tilde\kappa}}(\mu_0,\hat{\mu}_0)+ \int_0^t e^{-\lambda_{\tilde\kappa} (t-s)} \const^{\delta u_s}_x\De s, \\
\|\mu_t - \hat{\mu}_t\|_{\TV} &\leq q^{\tilde\kappa}_{t-t_0}e^{-\lambda_{\tilde\kappa}t_0} W_{f_{\tilde\kappa}}(\mu_0, \hat{\mu}_0)+ q^{\tilde\kappa}_{t-t_0}\int_{0}^{t_0} e^{-\lambda_{\tilde\kappa}(t_0 - s)} \const^{\delta u_s}\De s+ \frac{1}{\sqrt{2}} \left(\int_{t_0}^{t}\left(\const^{\delta u_s}\right)^2 \De s \right)^{\frac{1}{2}} .
\end{split}
\end{equation}
\item\label{item_stab_low_reg_new} Assume that \Cref{ass:low-reg} and \Cref{ass:stab_mild_reg} hold for both problem \eqref{eq:classical_control_problem} and problem \eqref{eq:classical_control_problem_hat}.
Then \eqref{eq:stab_est_mild_reg_linear} and \eqref{eq:stability_between_laws} hold for all $0\leq t_0\leq t\leq T$ with $\bar\kappa$ as in \eqref{eq:kappa_2} and $\const^{u_s},\const^{\hat{u}_s}$ as in \eqref{eq:grad_est_mild_reg}.

\item\label{item:stab_high_reg}  Assume that \Cref{ass:mild-reg} and \Cref{ass:stab_high_reg} hold for both problem \eqref{eq:classical_control_problem} and problem \eqref{eq:classical_control_problem_hat}.
Then for all $t\leq T$
\be\label{eq:stab_est_high_reg_linear}
\begin{split}
    \|\varphi^{T,g}_t-\hat\varphi^{T,\hat{g}}_t\|_{f_{\bar\kappa}}=\const^{\delta\varphi_t}_x\leq \int_t^{T} \frac{\const_x^{\delta\ell_s} }{C_{\tilde\kappa}} e^{-\lambda_{\tilde\kappa}(s-t)}\De s+ \frac{\const^{\delta g}_x}{C_{\tilde\kappa}} e^{-\lambda_{\tilde\kappa}(T-t)}
\end{split}
\ee
and the estimates \eqref{eq:stability_between_laws} hold as well with 
\bes
\const^{\delta u_s} =\const^{\delta b_s}+ \frac{\const^{\delta\varphi_s}_x}{\rho^\ell_{uu}}
\ees
\item \label{item:linear_stability_boost} Assume that \Cref{ass:mild-reg}, \Cref{ass:stab_high_reg} and \Cref{ass-Pontryagin} hold for both problem \eqref{eq:classical_control_problem} and problem \eqref{eq:classical_control_problem_hat}. Furthermore, assume that $\sigma(\cdot)$ is constant, i.e. $\sigma(\cdot)=2\sigma_0\mathrm{I}$. Fix $T''<T$, then we have 
\bes
\|\nabla\varphi^{T,g}_t-\nabla\hat\varphi^{T,\hat g}_t \|_\Lip \leq \int_t^{T''}2(\const^{\delta\ell_s}_x+\const^{\delta\varphi_s}_x\const^b_x)q^{\tilde\kappa}_{s-t} \De s+2q^{\tilde\kappa}_{T''-t} \|\varphi^{T,g}_{T''}-\varphi^{T,\hat g}_{T''}\|_\Lip.
    \ees
    where $\tilde\kappa$ is any profile satisfying the condition obtained by  replacing $T$ with $T''$ in \eqref{eq:kappa_choice_new}.
\end{enumerate}
\end{lemma}
\begin{proof}
Let $h_s,\hat{h}_s$ the Hamiltonians of the two problems,
\bes
h_s(x,p) = \inf_{u\in\bbR^d} \ell_s(x,u)+(b_s(x)+u)\cdot p  ,\quad 
\hat{h}_s(x,p) = \inf_{u\in\bbR^d} \hat{\ell}_s(x,u)+(\hat{b}_s(x)+u)\cdot p .
\ees
 Next, observe that using a Taylor expansion we know that for any $x,p,\hat{p}$ 
\bes
h_s(x,\hat{p})-h_s(x,p) = \int_0^1 \partial_p h_s(x, r \hat{p} + (1-r) p)  \De r\cdot (\hat{p}-p)
\ees
Thus, thanks to \Cref{prop:properties} we obtain that if we set and $\gamma=\hat{\varphi}^{T,\hat{g}}-\varphi^{T,g}$, then $\gamma$ is a classical solution to
\bes\label{eq:HJB_linearized}
\bec
\partial_s\gamma_s(x) + \frac{1}{2}\tr\left(\sigma^{\top}\nabla^2\gamma_s\sigma\right)(x) + \tilde{b}_s(x)\cdot\nabla\gamma_s(x) + v_s(x)=0\\
\gamma_{T}(x) = \hat{g}-g(x)
\eec
\ees
with 
\bes
\begin{split}
\tilde{b}_s(x)&=\int_0^1 \partial_p h_s(x,r\nabla\hat{\varphi}^{T,\hat{g}}_s(x) + (1-r) \nabla\varphi^{T,g}_s(x)))\De r, \\
v_s(x)&=[\hat{h}_s-h_s](x,\nabla\hat{\varphi}^{T,\hat{g}}_s(x)).
\end{split}
\ees
We thus have the Feynman-Kac representation
\begin{equation}\label{eq:FK_repr}
\gamma_t(x) = \bbE[\int_t^Tv_s(\tilde{X}^{t,x}_s)\De s+(\hat{g}-g)(\tilde{X}^{t,x}_T)],
\end{equation}
where 
\bes\label{eq:tilde_X}
\De \tilde{X}^{t,x}_s=\tilde{b}_s(\tilde{X}^{t,x}_s)\De s +\sigma(\tilde{X}^{t,x}_s)\De B_s,\quad \tilde{X}^{t,x}_t=x.
\ees

We shall work with the representation \eqref{eq:FK_repr} to establish the stability bounds.
\bei[wide]
\item \emph{Proof of \ref{item:stab_mild_reg}} 
We begin observing that by definition $\kappa_{\tilde b_s}\geq\tilde\kappa$ for all $s\in[t,T]$. Next, note that using $w_s(x,p)$ as a suboptimal control in the variational problem defining $\hat{h}_s(x,p)$, we find
\bes
\begin{split}
\hat{h}_s(x,p)-{h}_s(x,p)&\leq \hat\ell_s(x,w_s(x,p))-\ell_s(x,w_s(x,p)) +(\hat b_s(x)-{b}_s(x))\cdot p\\
&\leq \const^{\delta \ell_s}+\const^{\delta b_s}|p|.
\end{split}
\ees
Exchanging the roles of $h_s,\hat{h}_s$ yields 
\be\label{eq:bound_diff_ham}
\begin{split}
v_s(x)&=|\hat{h}_s(x,\nabla\hat\varphi^{T,\hat{g}}_s(x))-h_s(x,\nabla\hat\varphi^{T,\hat{g}}_s(x))|\\
&\leq \const^{\delta \ell_s}+\const^{\delta b_s}\const^{\hat\varphi_s}_x \quad \forall x\in \bbR^d
\end{split}
\ee

Now let $(X_s, \hat{X}_s)_s$ be a coupling by reflection  of two solutions of \eqref{eq:tilde_X} with $X_t = x, \hat{X}_t = \hat{x}$, i.e. a solution of \eqref{eq:coup_by_ref_contraction_2}-\eqref{eq:coup_by_ref_contraction_2_2} for the choices $\beta_s=\tilde{b}_s, \alpha_s=0$. We have, using \eqref{eq:FK_repr}

\be\label{eq:lipschitz_bound_gamma}
\begin{split}
|\gamma_t(x)-\gamma_t(\hat{x})|\stackrel{\eqref{eq:bound_diff_ham}}{\leq} \int_t^T(\const^{\delta \ell_s}+\const^{\delta b_s}\const^{\hat\varphi_s}_x)\bbP[X_s\neq \hat{X}_s]\De s
+ 
\|g-\hat{g} \|_{f_{\tilde\kappa}} \bbE[f_{\tilde\kappa}(|X_T- \hat{X}_T|)]
\end{split}
\ee
Invoking \Cref{prop:contr_same_drift} \ref{item_2:contraction_coup_by_ref} and \ref{item:final_TV_coup_by_ref}  we obtain \eqref{eq:stab_est_mild_reg_linear}. To prove \eqref{eq:stability_between_laws} fix two initial conditions $\mu,\hat{\mu}$ and consider approximate $\delta$-coupling by reflection of \eqref{eq:opt_flow} and \eqref{eq:opt_flow_hat}. That is to say, we consider the strong solution of \eqref{eq:delta_coup_by_ref} for the choices $\beta_s(x)=\partial_ph_s(x,
\nabla\varphi^{T,g}_s(x)),\,\hat\beta_s(x)=\partial_p\hat{h}_s(x,
\nabla\hat\varphi^{T,\hat g}_s(x)).$ In view of applying \Cref{prop:contr_delta_coup}, we proceed to bound
\bes
\begin{split}
|\partial_ph_s(x,\nabla\varphi^{T,g}_s(x))-\partial_p\hat{h}_s(x,\nabla\hat\varphi^{T,\hat g}_s(x))|
&\leq \const^{\delta b_s}+|w_s(x,\nabla\varphi^{T,g}_s(x))-w_s(x,\nabla\hat\varphi^{T,\hat g}_s(x))| + |w_s-\hat{w}_s|(x,\nabla\hat\varphi^{T,\hat g}_s(x))\\
&\leq \const^{\delta b_s}+\frac{(1+\const^{\delta\ell_s}_u)}{\rho^{\ell}_{uu}}\const^{\delta\varphi_s}_x,
\end{split}
\ees
where to pass from the first to the second line in the above display we used \Cref{lem:ham_bound}, which shall be proven separately below. At this point, we can apply \Cref{prop:contr_delta_coup}-\ref{item:W_1_contr_delta_coup} with $\const^{\delta\beta_s}=\const^{\delta b_s}+\frac{(1+\const^{\delta\ell_s}_u)}{\rho^{\ell}_{uu}}\const^{\delta\varphi_s}_x, \kappa_{\beta}=\tilde\kappa$ and eventually let $\delta\rightarrow0$ to obtain the first estimate in \eqref{eq:stability_between_laws}. To prove the estimate in total variation norm, we fix $t_0\leq t$ and apply \Cref{prop:contr_delta_coup}-\ref{item:TV_contr_delta_coup}. To finish it only remains to show that we can choose $\tilde{\kappa}=\bar\kappa$, with $\bar\kappa$ given by \eqref{eq:kappa_2}. To see this, observe that the standing assumptions  allow to invoke \Cref{lemma:gradient_estimate_linearized problem}-\ref{item:grad_est_high} to obtain $\|w_s(\cdot,\nabla\varphi^{T,g}_s(\cdot))\|_\Lip\leq\const^{u_s},\|w_s(\cdot,\hat\varphi^{T,g}_s(\cdot))\|_\Lip\leq\const^{\hat u_s}$ with $\const^{u_s},\const^{\hat{u}_s}$ given by \eqref{eq:grad_est_high_reg}. This immediately implies $\kappa_{\tilde b}\geq \bar\kappa$ as desired. 

\item\emph{Proof of \ref{item_stab_low_reg_new}}  The proof is identical to the one for \ref{item:stab_mild_reg} except for how the lower bound on $\kappa_{\tilde{b}_s}$ is obtained. In this case, the standing assumptions allow to invoke \Cref{lemma:gradient_estimate_linearized problem}-\ref{item:grad_est_bdd+Lip} (instead of \ref{item:grad_est_high} as we previously did) to we have that $\kappa_{\tilde{b}_s}(\cdot)\geq \bar\kappa(\cdot)$ with $\const^{u_s},\const^{\hat{u}_s}$ given by \eqref{eq:grad_est_mild_reg}.
\item \emph{Proof of \ref{item:stab_high_reg}} 
The proof is basically the same except for how the difference $h_s-\hat{h}_s$ is bounded. In this case, \Cref{ass:stab_high_reg} gives 
\be
|\partial_x [\hat{h}_s-h_s]|(x,\nabla\hat\varphi^{T,\hat{g}}_s(x)) \leq \const^{\delta\ell_s}_x, \quad \forall x,p\in\bbR^d.
\ee
Consequently, the Lipschitz estimate on $\gamma_t$ is obtained in a slightly different way. 
\bes 
|\gamma_t(x)-\gamma_t(\hat{x})|{\leq} \int_t^T\frac{\const^{\delta \ell_s}_x}{\const_{\bar{\kappa}}}\bbE[f_{\bar\kappa}(|X_s-\hat{X}_s|)]\De s
+ \|g-\hat{g} \|_{f_{\tilde\kappa}}\bbE[f_{\bar{\kappa}}(|X_T- \hat{X}_T|)]
\ees
Invoking \Cref{prop:contr_same_drift}-\ref{item_2:contraction_coup_by_ref},  we obtain \eqref{eq:stab_est_mild_reg_linear}.
\item\emph{Proof of \ref{item:linear_stability_boost}}
Observe that if $(X_s,\hat{X}_s)$ is coupling by reflection of two solutions of \eqref{eq:tilde_X}, and if we define $Y_s=\nabla\gamma_s(X_s),\hat{Y}_s=\nabla\gamma_s(\hat{X}_s)$ for $s\in[t,T]$, then we obtain by applying Itô's formula together with \eqref{eq:HJB_linearized} and \eqref{eq:tilde_X} that
\bes
\De Y_s=[ -\partial_x b_s(X_s) \cdot Y_s-\partial_x v_s(X_s)]\De s+\De M_s, \quad \De \hat{Y}_s= [-\partial_x b_s(X_s) \cdot \hat{Y}_s-\partial_x v_s(X_s)]\De s+ \De \hat{M}_s,
\ees
where $M_\cdot$ and $\hat M_{\cdot}$ are two square integrable martingales. Using that 
\bes
\bbE[|\partial_x v_s(\hat X_s)-\partial_x v_s(X_s)|]\leq 2 \const^{\delta\ell_s}_x q^{\tilde\kappa}_{s-t},\quad  \bbE[|\partial_x g(\hat X_T)-\partial_x g(X_T)|]\leq 2 \const^{g}_xq^{\tilde\kappa}_{T-t}
\ees
and 
\bes
\bbE[|\partial_x b_s(X_s) \cdot Y_s-\partial_x b_s(\hat X_s) \cdot \hat Y_s|]\leq 2\const_x^{b}\,\const^{\delta \varphi_s}_x
q^{\tilde\kappa}_{s-t}
\ees
we obtain the inequality
\bes
|Y_t-\hat Y_t| \leq \int_t^T2(\const_x^{b}\,\const^{\delta \varphi_s}_x
+ \const^{\delta\ell_s}_x )q^{\tilde\kappa}_{s-t}\De s+2 \const^{g}_xq^{\tilde\kappa}_{T-t}
\ees
Since $Y_t=\partial_x\gamma_t(x),\hat Y_t=\partial_x\gamma_t(\hat x)$, and the choice of $x,\hat x, t$ was arbitrary, the conclusion follows.
\eei
\end{proof}

\begin{lemma}\label{lem:ham_bound}  Under the hypothesis of \Cref{lem:stability_estimate_linear_pb}-\ref{item:stab_high_reg} we have
\begin{equation}
|w_s-\hat{w}_s|(x,p)\leq \frac{\const^{\delta\ell}_u}{\rho^{\ell}_{uu}}
\end{equation}
\end{lemma}
\begin{proof}
In the proof we write $w$  for $w_s(x,p)$ and $\hat w$ for $\hat w_s(x,p)$ for the sake of brevity.  We observe that \Cref{ass-coercivity} gives
\bes
\begin{split}
\frac{\rho^\ell_{uu}}{2}|w-\hat w|^2 &\leq \ell_s(x,\hat w)+(b_s(x)+\hat w)\cdot p-h_s(x,p)\\
&\leq [\ell_s-\hat\ell_s](x,\hat w)-[b_s(x)-\hat{b}_s(x)]\cdot p+\hat{h}_s(x,p)-h_s(x,p)
\end{split}
\ees
By symmetry we obtain
\bes
\frac{\rho^\ell_{uu}}{2}|w-\hat w|^2 \leq [\hat\ell_s-\ell_s](x,w)-[\hat b_s(x)-{b}_s(x)]\cdot p+h_s(x,p)-\hat{h}_s(x,p)
\ees
Summing the two bounds and using \Cref{ass:stab_mild_reg} we obtain the desired result.
\end{proof}

\section{Mean Field PDE systems}\label{sec:MF_PDE}
In this section, we undertake the study the long time behavior of the PDE system \eqref{eq:MF_PDE_intro}, which we repeat here for the readers' convenience
\begin{equation}\label{eq:MF_PDE}
\bec
\partial_t \varphi_t + \frac{1}{2}\tr\left(\sigma(x)^{\top}\sigma(x)\nabla^2\varphi_t(x)\right) + H(x,\nabla\varphi_t(x)) + F(\mu_t,x)=0, \eqsp \varphi_T(x) = G(\mu_T,x),\\
\partial_t\mu_t -\frac{1}{2}\tr(\nabla^2(\sigma^{\top}\sigma\mu_t(x))(x)) + \nabla\cdot(\partial_p H(x,\nabla\varphi_t(x))\mu_t)=0.
\eec
\end{equation}
Observe that \eqref{eq:MF_PDE} can be seen as a fixed point problem. Namely, given a flow $\fl\in\Gamma$, where
\begin{equation} 
\Gamma:=\left\{  \mu:[0,T] \rightarrow \mathcal{P}_1(\mathbb{R}^d): \sup_{0\leq s < t\le T}\frac{W_1(\mu_s,\mu_t)}{|t-s|^{1/2}}<+\infty, 
\quad \mu_0\in \mathcal{P}_p(\mathbb{R}^d) \right\},
\end{equation}
we can define the iteration $(\psi^{T,G}_{\cdot}[\mu_\cdot],\nu^{T,G}_{\cdot}[\mu_\cdot])$ as solution to the frozen problem
\begin{equation}\label{eq:linearized_PDE_system}
\bec
\partial_t \psi_t + \frac{1}{2}\tr\left(\sigma(x)^{\top}\sigma(x)\nabla^2\psi_s(x)\right) + H(x,\nabla\psi_t(x)) + F(\mu_t,x)=0, \quad \psi_T=G(\mu_T,x),\\
\partial_t\nu_t -\frac{1}{2}\tr(\nabla^2(\sigma^{\top}\sigma\nu_t)(x)) + \nabla\cdot( \partial_p H(x,\nabla\psi_t(x))\nu_t(x))=0, \quad \nu_0=\mu_0.
\eec
\end{equation}
Clearly, $(\varphi_t,\mu_t)_{t \in [0,T]}$ solves \eqref{eq:MF_PDE} if and only if $\mu_\cdot = \nu^{T,G}_{\cdot}[\mu_\cdot]$ and $\varphi_\cdot = \psi^{T,G}_{\cdot}[\mu_\cdot]$.
Observe that (at least formally for now) for $\mu_{\cdot} \in \Gamma$ given, $\psi^{T,G}_{\cdot}[\mu_\cdot]$ is the value function to the optimal control problem \eqref{eq:classical_control_problem} 
and $\nu^{T,G}_{\cdot}[\mu_\cdot]$ is the corresponding marginal flow of the associated optimal dynamics \eqref{eq:fin_dim_OptDyn} when setting $b_s=b$, $\ell_s(x,u) = L(x,u) + F(\mu_s,x)$, $g(x)=G(\mu_T,x)$. 
We plan to exploit this relationship in order to apply the results from \Cref{sec:fin_dim_cont} to establish existence and uniqueness of solutions of \eqref{eq:MF_PDE} and an exponential turnpike property towards the (unique) solution $(\varphi^\infty,\mu^\infty) \in C^{0,1}(\bbR^d) \times \cP_{1}(\bbR^d)$ of the ergodic system 
\begin{equation}\label{eq:MF_PDE_ergodic}
\bec
\eta + \frac{1}{2}\tr\left(\sigma(x)^{\top}\sigma(x)\nabla^2\varphi(x)\right) +  H(x,\nabla\varphi(x)) + F(\mu,x)=0,\\
\frac{1}{2}\tr(\nabla^2(\sigma^{\top}\sigma\mu)(x)) - \nabla\cdot(\partial_p H(x,\nabla\varphi(x))\mu(x))=0.
\eec
\end{equation}
We start with a rather abstract lemma, yielding a turnpike estimate on solutions of the PDE system \eqref{eq:MF_PDE}, once we have established two inequalities for suitable metrics on the space of flows $\Gamma$. This gives us tractable sufficient conditions for proving turnpike estimates in the following subsections.
\subsubsection*{An abstract turnpike estimate}
We assume that we are given a metric $\sfd : \cP_1(\bbR^d) \times \cP_1(\bbR^d) \rightarrow \bbR_+$ and a family of continuous functions $\upsilon^\lambda: \bbR_+^* \rightarrow \bbR_+^*$, parameterized by $\lambda > 0$, and define the following metrics for $\mu_\cdot, \hat{\mu}_\cdot \in \Gamma$
\begin{align}
&\overrightarrow{d_{\lambda, \kappa}^{T}}(\fl,\hfl) := \sup_{0 < s\leq T}(\upsilon^\lambda(s))^{-1}\sfd(\mu_s,\hat\mu_s), \quad \overleftarrow{d_{\lambda,\kappa}^{T}}(\fl,\hfl) := \sup_{0 \leq s\leq T}e^{\lambda (T-s)}\sfd(\mu_s,\hat\mu_s).
\end{align}

\begin{lemma}[Abstract turnpike estimate]\label{lem:abs_turnpike_est}
Let $G:\cP_1(\bbR^d) \times \bbR^d \rightarrow \bbR$, $\hat{G}: \bbR^d \rightarrow \bbR$ be two terminal conditions and assume that $\nu^{T,G}$ and $\nu^{T,\hat{G}}$ are well-defined on $\Gamma$. Suppose that there exists $\bar{\kappa} \in \msk$, $0 < \lambda < \lambda_{\bar{\kappa}}$, $0\leq \varepsilon(\lambda)< 1$ such that for all $\mu_\cdot, \hat{\mu}_\cdot \in \Gamma$
\begin{equation}\label{eq:abs_trunpike_cond_1}
\overrightarrow{d_{\lambda,\bar{\kappa}}^{T}}(\nu^{T,\hat{G}}_\cdot[\mu_\cdot],\nu^{T,\hat{G}}_\cdot[\hat\mu_\cdot])\leq W_{f_{\bar{\kappa}}}(\mu_0,\hat{\mu}_0)+ \varepsilon(\lambda) \overrightarrow{d_{\lambda,\bar{\kappa}}^{T}}(\mu_\cdot,\hat\mu_\cdot),
\end{equation}
and suppose furthermore, that if $\mu_0 = \hat{\mu}_0$
\begin{equation}\label{eq:abs_trunpike_cond_2}
\overleftarrow{d_{\lambda,\bar{\kappa}}^{T}}(\nu^{T,\hat{G}}_\cdot[\mu_\cdot],\nu^{T,\hat{G}}_\cdot[\hat\mu_\cdot])\leq \varepsilon(\lambda) \overleftarrow{d_{\lambda,\bar{\kappa}}^{T}}(\mu_\cdot,\hat\mu_\cdot).
\end{equation}
and that there exists a solution for the mean field PDE system for the initial condition $\mu_0$ and terminal condition $\hat{G}$.
Then, if $(\mu^{T,G}_\cdot, \varphi^{T,G}_\cdot)$   is a solution of \eqref{eq:MF_PDE} for the terminal cost $G$ and initial condition $\mu_0$, and $(\mu^{T,\hat{G}}_\cdot, \varphi^{T,\hat{G}}_\cdot)$ a solution of \eqref{eq:MF_PDE} for terminal cost $\hat{G}:\bbR^d \rightarrow \bbR$ and initial condition $\hat{\mu}_0$
\begin{equation}\label{eq:abstract-tpike}
\sfd(\mu^{T,G}_t, \mu^{T,\hat{G}}_t) \leq \frac{\upsilon^\lambda(t)}{1 - \varepsilon(\lambda)} W_{f_{\bar{\kappa}}}(\mu_0, \hat{\mu}_0) + 
\frac{e^{-\lambda(T-t)}}{1 - \varepsilon(\lambda)} \overleftarrow{d_{\lambda,\bar{\kappa}}^{T}}(\nu^{T,\hat{G}}_\cdot[\mu^{T,G}_\cdot],\nu^{T,G}_\cdot[\mu^{T,G}_\cdot]) \eqsp \quad \forall t\in[0,T].
\end{equation}

\end{lemma}
\begin{proof}

Denote by $\tilde{\mu}_\cdot$ a fixed point of $\nu^{T,\hat{G}}$ with initial condition $\mu_0$. Let us also write $\mu = \mu^{T,G}$ and $\hat{\mu} = \mu^{T,\hat{G}}$ for readability. By the triangular inequality
\begin{equation}\label{eq:part_forward_backward}
\sfd(\mu_t, \hat{\mu}_t) \leq \sfd(\mu_t, \tilde{\mu}_t)+\sfd(\tilde{\mu}_t, \hat{\mu}_t)   \eqsp.
\end{equation}
By assumption, we have
\begin{align}
\overrightarrow{d_{\lambda,\bar{\kappa}}^{T}}(\nu^{T,\hat{G}}_\cdot[\tilde{\mu}_\cdot],\nu^{T,\hat{G}}_\cdot[\hat{\mu}_\cdot])\leq W_{f_{\bar{\kappa}}}(\mu_0,\hat{\mu}_0)+\varepsilon(\lambda) \overrightarrow{d_{\lambda,\bar{\kappa}}^{T}}(\tilde{\mu}_\cdot,\hat{\mu}_\cdot) \eqsp,
\end{align}
which gives since $\tilde{\mu}_\cdot$ and $\hat{\mu}_\cdot$ are fixed points of $\nu^{T,\hat{G}}_\cdot$
\begin{equation}\label{eq:forward_stab}
\sfd(\tilde{\mu}_t, \hat{\mu}_t) \leq \frac{\upsilon^\lambda(t)}{1 - \varepsilon(\lambda)} W_{f_{\bar{\kappa}}}(\mu_0, \hat{\mu}_0) \eqsp.
\end{equation}
In order to bound the first term in \eqref{eq:part_forward_backward}, observe that again by assumption we have
\begin{align}
\overleftarrow{d_{\lambda,\bar{\kappa}}^{T}}(\nu^{T,\hat{G}}_\cdot[\tilde{\mu}_\cdot],\nu^{T,G}_\cdot[\mu_\cdot])
&\leq \overleftarrow{d_{\lambda,\bar{\kappa}}^{T}}(\nu^{T,\hat{G}}_\cdot[\tilde{\mu}_\cdot],\nu^{T,\hat{G}}_\cdot[\mu_\cdot]) + \overleftarrow{d_{\lambda,\bar{\kappa}}^{T}}(\nu^{T,\hat{G}}_\cdot[\mu_\cdot],\nu^{T,G}_\cdot[\mu_\cdot]) \\
&\leq \varepsilon(\lambda) \overleftarrow{d_{\lambda,\bar{\kappa}}^{T}}(\tilde{\mu}_\cdot,\mu_\cdot) + \overleftarrow{d_{\lambda,\bar{\kappa}}^{T}}(\nu^{T,\hat{G}}_\cdot[\mu_\cdot],\nu^{T,G}_\cdot[\mu_\cdot]) \eqsp.
\end{align}
Thanks to the fixed point property of $\tilde{\mu}_\cdot$ for $\nu^{T,\hat{G}}$ and $\mu_\cdot$ for $\nu^G$ we obtain
\begin{align}\label{eq:est-final-big-metric}
\overleftarrow{d_{\lambda,\bar{\kappa}}^{T}}(\tilde{\mu}_\cdot,\mu_\cdot) \leq \frac{1}{1 - \varepsilon(\lambda)} \overleftarrow{d_{\lambda,\bar{\kappa}}^{T}}(\nu^{T,\hat{G}}_\cdot[\mu_\cdot],\nu^{T,G}_\cdot[\mu_\cdot]) \eqsp.
\end{align}
Now use the definition of $\overleftarrow{d_{\lambda,\bar{\kappa}}^{T}}$ and \eqref{eq:forward_stab}, \eqref{eq:part_forward_backward} to conclude.

\end{proof}

In the following sections, we distinguish three different regimes of regularity, where the main goal is to prove existence and uniqueness of the ergodic PDE system \eqref{eq:MF_PDE_ergodic} and to verify the hypothesis of \Cref{lem:abs_turnpike_est} in order to conclude turnpike estimates both for the flows and value functions associated to \eqref{eq:MF_PDE} under different regularity and growth assumptions on the final cost $G$.

\subsection{High regularity}\label{sec:MF_high_reg}

Let us start by making rigorous the relationship of the fixed-point iterations \eqref{eq:linearized_PDE_system} and the optimal control problems considered in \Cref{sec:fin_dim_cont}. To this aim, we introduce the constant
\be\label{def:c_x_psi}
\const_x^\psi:=\frac{\const^L_x+\const^F_x}{\lambda_{\kappa_b}C_{\kappa_b}}.
\ees

\begin{lemma}\label{lem:basics_MF_high}
Assume \Cref{ass:MF_drift_intro}, \Cref{ass:MF_coercivity_intro}, \Cref{ass:high_reg_intro}. Let $\mu_\cdot \in \Gamma$, $G: \cP_1(\bbR^d) \times \bbR^d \rightarrow \bbR$ locally Lipschitz continuous in the second variable such that $G(\mu,\cdot)$ is of linear growth for all $\mu\in\cP_1(\bbR^d)$. The following holds.
\begin{enumerate}[label=(\roman*)] 
\item\label{item:HJB_ok_high} There is a unique solution to \eqref{eq:linearized_PDE_system}, denote it by $(\psi^{T,G}_{\cdot}[\mu_\cdot],\nu^{T,G}_{\cdot}[\mu_\cdot])$. 

Moreover, setting $b_s=b$, ${\ell_s(x,u) = L(x,u) + F(\mu_s,x)}$, $g(x)=G(\mu_T,x)$, $\psi^{T,G}_{\cdot}[\mu_\cdot]$ is the value function of \eqref{eq:classical_control_problem} and  $\nu^{T,G}_{\cdot}[\mu_\cdot] \in \Gamma$ is the corresponding optimal flow \eqref{eq:opt_flow}.
\item\label{item:grad_est_MF_high} The following quantitative bounds hold true
\begin{align}
\|\psi^{T,G}_{t}[\mu_\cdot]\|_{f_{\kappa_{b}}} &\leq \const_x^\psi (1 - e^{-\lambda_{\kappa_{b}}(T-t)}) +\min\left\{ \| G(\mu_T,\cdot)\|_{f_{\kappa_{b}}}e^{-\lambda_{\kappa_{b}}(T-t)},\|G(\mu_T,\cdot)\|_{\infty} q^{\kappa_b}_{T-t}  \right\}  ,  \\
\|w(\cdot,\nabla\psi^{T,G}_{t}[\mu_\cdot](\cdot))\|_\infty &\leq \frac{\|\psi^{T,G}_{t}[\mu_\cdot]\|_{f_{\kappa_{b}}}+\const^{L(\cdot,0)}_u}{\rho^{L}_{uu}},
\end{align}
where
\begin{equation}
w(x,p) := \argmin_{u \in \bbR^d} \left\{ L(x,u)+(b(x)+u)\cdot p\right\}  \eqsp.
\end{equation}

\item\label{item:stab_MF_high} Suppose now that $G$ is Lipschitz in the space variable and that $\hat{G}: \cP_1(\bbR^d) \times \bbR^d \rightarrow \bbR$ is another terminal cost satisfying the same assumptions as $G$. Given another flow $\hat{\mu}_\cdot \in \Gamma$, we have
\begin{align}
\|\psi^{T,G}_{t}[\mu_\cdot]-\psi^{T,\hat{G}}_{t}[\hat{\mu}_\cdot]\|_{f_{\tilde{\kappa}}}
&\leq \const^{\delta \psi_t}_x+ \|G(\mu_T,\cdot) - \hat{G}(\hat{\mu}_T,\cdot) \|_{f_{\tilde{\kappa}}}e^{-\lambda_{\tilde{\kappa}}(T-t)} \\
W_{f_{\tilde{\kappa}}}(\nu^{T,G}_{t}[\mu_\cdot],\nu^{T,\hat{G}}_{t}[\hat{\mu}_\cdot])
&\leq e^{-\lambda_{\tilde{\kappa}}t} W_{f_{\tilde{\kappa}}}(\mu_0,\hat{\mu}_0)+ \frac{1}{\rho^L_{uu}}\int_0^t e^{-\lambda_{\tilde{\kappa}} (t-s)} \const^{\delta \psi_s}_x\De s \\
& \quad+ \frac{1}{2\rho^L_{uu}\lambda_{\tilde{\kappa}}}e^{-\lambda_{\tilde{\kappa}}(T-t)}\|G(\mu_T,\cdot) - \hat{G}(\hat{\mu}_T,\cdot) \|_{f_{\tilde{\kappa}}},
\end{align}
where \bes\label{c_delta_psi_x} \const^{\delta \psi_t}_x = \int_t^T\frac{\const^F_{x\mu}}{C_{\tilde{\kappa}}^2}W_{f_{\tilde{\kappa}}}(\mu_s,\hat{\mu}_s)e^{-\lambda_{\tilde{\kappa}}(s-t)}\De s\ees and $\tilde{\kappa} \in \msk$ is any profile satisfying

\begin{equation}\label{eq:tilde_kappa_cond}
\tilde{\kappa} \leq  \min_{0\leq s \leq T} \{\kappa_{\partial_p H(\cdot,\nabla\psi^{T,G}_{s}[\mu_\cdot](\cdot))},\kappa_{\partial_p H(\cdot,\nabla\psi^{T,\hat{G}}_{s}[\hat{\mu}_\cdot](\cdot))}\}.
\end{equation}

In particular, if $\max\{\|\hat G\|_{f_{\kappa_\beta}}(\mu_T,\cdot),\| G(\mu_T,\cdot)\|_{f_{\kappa_\beta}}\}\leq 2\const^\psi_x$ we can choose $\tilde{\kappa}=\bar\kappa$, with $\bar\kappa$ as in \eqref{eq:bar_kappa_intro}.

\end{enumerate}
\end{lemma}
\begin{proof}\Cref{item:HJB_ok_high} is a consequence of \Cref{prop:properties} in combination with \Cref{prop:Holder_in_time}, where we refer to \Cref{cor:grad_est_g_bdd} for the bounded case. Uniqueness for the Fokker-Planck equation follows by \cite[Thm 1]{bogachev2007UniquenessSolutionsWeak}. In the same way, the Lipschitz estimates in \ref{item:grad_est_MF_high} follow by a direct application of \Cref{lemma:gradient_estimate_linearized problem} and \Cref{cor:grad_est_g_bdd}. \Cref{item:stab_MF_high} follows from \Cref{lem:stability_estimate_linear_pb}. 
\end{proof}

With these tools at hand we can prove existence and uniqueness of \eqref{eq:MF_PDE_ergodic}.
\begin{thm}\label{lem:MF_syst_high_reg}
Assume \Cref{ass:MF_drift_intro}, \Cref{ass:MF_coercivity_intro}, \Cref{ass:high_reg_intro} and that \eqref{eq:high_reg_suff} with $\bar{\kappa}$ from \eqref{eq:bar_kappa_intro} holds. Then, there there exists a unique solution $(\mu^\infty,\varphi^\infty,\eta^\infty) \in \cP_{1}(\bbR^d) \times C^{0,1}(\bbR^d) \times \bbR$ to the ergodic mean field PDE system \eqref{eq:MF_PDE_ergodic} satisfying $\varphi^\infty(0)=0$ . Moreover, we have 
\begin{equation}\label{eq:gr_est_tpike_high}
\|\varphi^\infty \|_{f_{\kappa_b}}\leq \const_x^\psi.
\end{equation}
and $\mu^{\infty}\in\cP_p(\bbR^d)$. If we furthermore assume that \Cref{ass:boost} also holds, existence and uniqueness of a solution to the ergodic mean field PDE system holds under the relaxed condition \eqref{eq:high_reg_suff_1} for the choices
\be\label{eq:hessian_choices_1}
\const^u_x =\frac{1}{\rho^L_{uu}}\big(\const^{L}_{ux}+\const^\psi_{xx} \big), \const^{\psi}_{xx}= \const_{xx}(\const^{\sigma}_{x},\kappa_b,\rho^{L}_{uu},\const^{L(\cdot,0)}_u,\const^{b}_{x},\const^{L}_x,\const^{\psi}_x,\const^{\sigma}_{xx}),
\ee
where $\const_{xx}(\cdot)$ is as in \eqref{eq:gen_hess_est_1}.

\end{thm}

\begin{proof}
\bei[wide]
\item \underline{Step 1:Existence and uniqueness for the frozen ergodic system \eqref{eq:linearized_ergodic_PDE_system}.}

\noindent In this step, we show hat for any $\mu \in \cP_{1}(\bbR^d)$ we can find a unique classical solution $(\eta^{\infty}[\mu],\psi^{\infty}[\mu],\nu^{\infty}[\mu])$ to the frozen ergodic system
\begin{equation}\label{eq:linearized_ergodic_PDE_system}
\bec
-\eta^\infty + \frac{1}{2}\tr\left(\sigma(x)^{\top}\sigma(x)\nabla^2\psi(x)\right) + H(x,\nabla\psi(x)) + F(\mu,x)=0, \quad \\
 -\frac{1}{2}\tr(\nabla^2(\sigma^{\top}\sigma\nu)) + \nabla\cdot( \partial_p H(x,\nabla\psi_s(x))\nu(x))=0.
\eec
\end{equation}
For this, fix $T>0$, $\mu \in \cP_{1}(\bbR^d)$ and define
\begin{align}\label{def:iterations_linearized_tpike}
\Phi^{\mu}_T: C^{0,1}(\bbR^d) \longrightarrow C^{0,1}(\bbR^d), \qquad &g\mapsto \psi^{T,g}_0[\mu], \\
\overline{\Phi}^{\mu}_T:  \cG^{\const^\psi_x}_{0} \longrightarrow C^{0,1}_0(\bbR^d), \qquad &g\mapsto \psi^{T,g}_0[\mu] - \psi^{T,g}_0[\mu](0),
\end{align}
where $\psi^{T,g}_0[\mu]$ denotes the solution to the HJB equation of \eqref{eq:linearized_PDE_system} and with a slight abuse of notation, for a given $\mu$ we write $ \psi^{T,g}_t[\mu]$ instead of $\psi^{T,g}_t[\mu_\cdot]$ when $\mu_\cdot$ is the flow constantly equal to $\mu$. Furthermore,we recall that $C^{0,1}(\bbR^d)$ is the space of Lipschitz continuous functions and $C^{0,1}_0(\bbR^d)$ its subspace of functions satisfying $g(0) = 0$, and $\cG^{\const}_{0} $ is defined as follows
\begin{equation}
\cG^{\const}_{0} := \left\{g \in C^{0,1}_0(\bbR^d): \eqsp \| g \|_{f_{\kappa_b}} \leq \const \right\}
\end{equation}
 
We are going to show that $\overline{\Phi}^{\mu}_T$ is a contraction and conclude by Banach's fixed point theorem. To this aim, we first observe that \Cref{lem:basics_MF_high}-\ref{item:grad_est_MF_high} ensures that
\begin{equation}
\overline{\Phi}^{\mu}_T(\cG^{\const^\psi_x}_{0}) \subseteq \cG^{\const^\psi_x}_{0} \quad \forall T>0.
\end{equation}
Moreover, if \Cref{ass:boost} also holds, we can also apply \Cref{thm:lin_Hess_est} to get
\be\label{eq:boosted_invariant_set}
\overline{\Phi}^{\mu}_T(\cG^{\const^\psi_x}_{0}) \subseteq \cG^{\const^\psi_x}_{0}\cap \{\varphi:\|\nabla\varphi\|_{\Lip}\leq \const^{\psi}_{xx} +\varepsilon(T)\},
\ee
where $\const^{\psi}_{xx}$ is as in \eqref{eq:hessian_choices_1} and 
\bes
    \varepsilon(T)=4(\const_x^\psi)^2\tilde{C}^{-1}_{\bar\kappa}e^{-\tilde\lambda_{\bar\kappa}T}+(2\const_x^\psi+4(\const^g_x)^2)\tilde{q}^{\bar\kappa}_{T}.
\ees

Next, fix $g,\hat{g}\in\cG^{\const^\psi_x}_0$. 
Then, \Cref{lem:basics_MF_high}-\ref{item:stab_MF_high} gives
\begin{equation}\label{eq:contraction_ergodic_map_lin}
\| \psi^{T,g}_0[\mu] - \psi^{T,\hat{g}}_0[\mu]\|_{f_{\bar{\kappa}}} \leq e^{-\lambda_{\bar{\kappa}} T} \|g - \hat{g} \|_{f_{\bar{\kappa}}},
\end{equation}
with $\bar{\kappa}$ defined in \eqref{eq:bar_kappa_intro}. 
Hence, the map $\overline{\Phi}^\mu_T$ is a contraction in $\cG^{\const^\psi_x}_{0}$
equipped with the norm $\| \cdot \|_{f_{\bar{\kappa}}}$ defined in \eqref{eq:f_Lip_norm}. From this, we deduce existence and uniqueness of a fixed point $\psi^{\infty,T}[\mu]$ thanks to Banach's fixed point Theorem, and we set $\eta^{\infty,T}=\psi^{\infty,T}[\mu](0)$. We now proceed to show that, setting $\psi^{\infty}[\mu]:=\psi^{\infty,1}[\mu]$, $\eta^{\infty}[\mu]:=\eta^{\infty,1}[\mu]$ 
\begin{equation}\label{eq:indep_of_T} \psi^{ \infty,T}[\mu]= \psi^{\infty}[\mu], \quad \eta^{\infty,T}[\mu] = \eta^{\infty}[\mu] T\end{equation}
holds for all $T\in[0,1]$. By the dynamic programming principle, for any $T$ and $g$ such that ${||g||_{f_{\bar{\kappa}}} \leq \const^\psi_x}$, we have
\[
	\Phi^\mu_T(g) = \Phi^\mu_{T/2} (\Phi^\mu_{T/2}(g)), 
\]
which, together with uniqueness of the fixed point for $\overline{\Phi}^\mu_T$ , implies that
\[
	\psi^{\infty,T}[\mu] = \psi^{\infty, T/2}[\mu], \qquad 
	\eta^{\infty,T}[\mu] =2 \eta^{\infty, T/2}[\mu]. 
\]

Iterating this argument, we find that if $T\in[0,1]$ is a dyadic number, then \eqref{eq:indep_of_T} holds.
Since the map $T \mapsto \Phi^\mu_T(\psi^{\infty,1}[\mu])(x)$ is continuous for all $x\in \bbR^d$, we can extend \eqref{eq:indep_of_T} to $T\in[0,1]$. To conclude observe that  
\begin{equation}
\psi^{\infty}[\mu]+\eta^\infty[\mu](1-t)\stackrel{\eqref{eq:indep_of_T}}{=}\Phi^{\mu}_{1-t}(\psi^{\infty}[\mu]) \stackrel{\eqref{def:iterations_linearized_tpike}}{=} \psi^{1-t, \psi^{\infty}[\mu] }_0[\mu] =\psi^{1, \psi^{\infty}[\mu] }_t[\mu],
\end{equation} 
and $t\mapsto \psi^{1, \psi^{\infty}[\mu] }_t[\mu]$ is a classical solution to the HJB equation 
\begin{equation}
\partial_t \varphi_t + \frac{1}{2}\tr\left(\sigma(x)^{\top}\sigma(x)\nabla^2\varphi_t(x)\right) + H(x,\nabla\varphi_t(x)) + F(\mu,x)=0, \eqsp \varphi_1(x) = \psi^{\infty}[\mu](x),
\end{equation}
thanks to \Cref{prop:properties}-\ref{item:fin_dim_HJB}. 
But then, $(\eta^\infty[\mu],\psi^{\infty}[\mu])$ is a classical solution to the ergodic HJB equation in \eqref{eq:linearized_ergodic_PDE_system}. Moreover we note that $\psi^{\infty}[\mu]\in\cG^{\const^\psi_x}_0$ by construction, and if also \Cref{ass:boost} holds, we obtain thanks to \eqref{eq:boosted_invariant_set} that $\|\nabla\psi^{\infty}[\mu]\|_\Lip\leq \const^{\psi}_{xx}$ by letting $T\rightarrow+\infty$ in \eqref{eq:boosted_invariant_set}.
Next, note that \Cref{prop:contr_same_drift}-\ref{item_2:contraction_coup_by_ref} implies that the SDE 
\begin{equation}
    \De X_s = \partial_p H(X_s,\nabla \psi^{\infty}[\mu](X_s)) \De s + \sigma(X_s) \De B_s
\end{equation}
admits a unique invariant measure $\nu^{\infty}[\mu]$. Thus, $\nu^{\infty}[\mu]$ is a weak solution of the Fokker-Planck equation in the sense of \cite[(1.5)]{bogachev2007UniquenessSolutionsWeak} in \eqref{eq:linearized_ergodic_PDE_system}. Moreover, since we $\kappa_{\partial_pH(\cdot,\nabla\psi^{\infty}[\mu](\cdot))}\in\msk$, we also have $\nu^{\infty}[\mu]\in\cP_p(\bbR^d)$ thanks to \Cref{prop_existence_moment_estimate}.  and $(\eta^{\infty}[\mu],\nu^{\infty}[\mu],\psi^{\infty}[\mu])$ form a solution to \eqref{eq:linearized_ergodic_PDE_system}. 
For uniqueness consider $g, \hat{g} \in C_{\Lip,f_{\kappa_b}}(\bbR^d)$, and observe that by defining
\begin{equation}
\tilde{\kappa}(r) := \kappa_b(r) - \frac{2}{\rho^L_{uu}r} \left(\const^{L(\cdot,0)} + \max\left\{\const^\psi_x, \|g \|_{f_{\kappa_b}},  \|\hat{g} \|_{f_{\kappa_b}} \right\} \right)
\end{equation}
applying once again \Cref{lem:basics_MF_high}-\ref{item:stab_MF_high}  gives
\begin{equation}\label{eq:contraction_ergodic_map_lin}
\| \psi^{T,g}_0[\mu] - \psi^{T,\hat{g}}_0[\mu]\|_{f_{\tilde{\kappa}}} \leq e^{-\lambda_{\tilde{\kappa}} T} \|g - \hat{g} \|_{f_{\tilde{\kappa}}}.
\end{equation}
In particular plugging in two fixed points and letting $T \rightarrow \infty$, we deduce uniqueness for the value functions. Uniqueness of the invariant measure follows by \cite[Thm 1]{bogachev2007UniquenessSolutionsWeak}.
\item \underline{Step 2: Existence and uniqueness for the mean field ergodic system \eqref{eq:MF_PDE_ergodic}.} 

\noindent Let $\mu,\hat\mu \in \cP_1(\bbR^d)$. Thanks to the fixed point properties of $\psi^{\infty}[\mu]$ and $\psi^{\infty}[\hat\mu]$ and to $\psi^\infty[\mu],\psi^{\infty}[\hat\mu]\in\cG_0^{\const^\psi_x}$(which was defined at \eqref{def:iterations_linearized_tpike}) we can apply \Cref{lem:basics_MF_high}-\ref{item:stab_MF_high} with $\tilde\kappa=\bar\kappa$ to get
\begin{equation}\label{eq:ergodic_stability}
\|\nabla\psi^{\infty}[\mu]-\nabla\psi^{\infty}[\hat\mu]\|_{\infty} \leq \frac{\const^F_{x\mu}}{C^2_{\bar{\kappa}}} W_{f_{\bar{\kappa}}}(\mu,\hat{\mu}) \int_0^Te^{-\lambda_{\bar{\kappa}}s} \De s + e^{-\lambda_{\bar{\kappa}}T} 2\const^\psi_x. 
\end{equation}

Letting $T\rightarrow+\infty$ we arrive at 
\begin{equation}\label{eq:optimal_ergodic_policy_bound_high}
\|w(\cdot,\nabla\psi^{\infty}[\mu](\cdot))-w(\cdot,\nabla\psi^{\infty}[\hat\mu](\cdot))\|_{\infty}\leq \frac{\const^F_{x\mu}}{\rho^L_{uu}C^2_{\bar{\kappa}}\lambda_{\bar{\kappa}}} W_{f_{\bar{\kappa}}}(\mu,\hat{\mu}) 
\end{equation}
But then, using a standard coupling argument for bounding the distance between invariant measures we obtain
\begin{equation}
W_{f_{\bar{\kappa}}}(\nu^{\infty}[\mu],\nu^{\infty}[\hat\mu]) \leq \frac{\|w(\cdot,\nabla\psi^{\infty}[\mu](\cdot))-w(\cdot,\nabla\psi^{\infty}[\hat\mu](\cdot))\|_{\infty}}{\lambda_{\bar{\kappa}}}{\leq} \frac{\const^F_{x\mu}}{\rho^L_{uu}C^2_{\bar{\kappa}}\lambda^2_{\bar{\kappa}}} W_{f_{\bar{\kappa}}}(\mu,\hat{\mu}) 
\end{equation}
From this bound, we deduce that \eqref{eq:suff_cond_tpik_high_reg_in_proofs} implies that  the map
\begin{equation}
    \cP_1(\bbR^d)\ni\mu \mapsto \nu^{\infty}[\mu]
\end{equation}
is a contraction and we conclude with Banach's fixed point theorem existence and uniqueness of a unique fixed point $\mu^{\infty}$ in $\cP_1(\bbR^d)$. It then follows that $(\eta^{\infty}[\mu^\infty],\psi^{\infty}[\mu^{\infty}],\mu^{\infty})$ is the unique solution to the ergodic mean field PDE system \eqref{eq:MF_PDE_ergodic}. Moreover, since $\psi^{\infty}[\mu]\in\cG^{\const^\psi_x}_0$ by construction, we have that \eqref{eq:gr_est_tpike_high} holds. {Finally, we note that if \Cref{ass:boost} holds, we can profit from the additional information $\max\{\|\nabla\psi^\infty[\mu_\cdot]\|_\Lip,\|\nabla\psi^\infty[\hat\mu_\cdot]\|_\Lip\}\leq \const^\psi_{xx}$ which implies that $\min\{\kappa_{\partial_pH(\cdot,\nabla\psi^\infty[\mu](\cdot))},\kappa_{\partial_pH(\cdot,\nabla\psi^\infty[\hat\mu](\cdot))}\}\geq\bar\kappa'$ with $\bar\kappa'$ as in \eqref{eq:high_reg_suff_1} and $\const^u_x$ given by \eqref{eq:hessian_choices_1}. But then, thanks to \Cref{lem:basics_MF_high}, we can replace  $\bar\kappa$ with $\bar\kappa'$ in \eqref{eq:ergodic_stability} and in all the subsequent identities, leading to existence and uniqueness of solutions for the ergodic mean field PDE system under the relaxed condition \eqref{eq:high_reg_suff_1}{. Finally, $\|\nabla\varphi^\infty\|_\Lip\leq \const^\psi_{xx}$ follows from $\|\nabla\psi^{\infty}[\mu]\|_\Lip\leq \const^{\psi}_{xx}$ for all $\mu \in \cP_1(\bbR^d)$.} Lastly, observe that since $\mu^{\infty}$ is the invariant distribution of the SDE with drift $\partial_pH(\cdot,\nabla\varphi^\infty(\cdot))$ and this drift satisfy a Lyapunov condition thanks to $\kappa_{\partial_pH(\cdot,\nabla\varphi^\infty(\cdot))} \in \msk$, we hence get $\mu^\infty\in\cP_p(\bbR^d)$ by  \Cref{prop_existence_moment_estimate}.
}
\eei
\end{proof}
\begin{cor}\label{cor:boost}
In the same setting and notation of \Cref{lem:basics_MF_high}, if  $\hat G= \varphi^\infty$ and \Cref{ass:boost} holds we have
\ben[(i)] 
\item\label{item:cor:boost_i} Item \ref{item:stab_MF_high} from  \Cref{lem:basics_MF_high} holds for the choice
\begin{equation}
\tilde{\kappa} \leq \min\{\bar\kappa',\inf_{t\leq T} {\kappa_{\partial_pH(\cdot,\nabla\psi^{T,G}_t[\mu_\cdot](\cdot))}}\color{forestgreen}\} .
\end{equation}
with $\bar\kappa'$ as in \eqref{eq:high_reg_suff_1} for the choice 
\be\label{eq:iterates_choices}
\const^u_x = \frac{1}{\rho^L_{uu}}(\const^L_{xu}+2(\const^\psi_x+1)\const^\psi_{xx}+4(\const^\psi_x)^2)\tilde{C}^{-1}_{\bar\kappa},
\ee
where $\const^{\psi}_{xx}$ is as in \eqref{eq:hessian_choices_1}.
\item\label{item:cor:boost_ii}If additionally $\sigma(\cdot)=2\sigma_0\mathrm{I}$, and $ T''\leq T$ is such that 
\be\label{eq:cor_boost_extra_ass}
\inf_{t\leq T''} {\kappa_{\partial_pH(\cdot,\nabla\psi^{T,G}_t[\mu_\cdot](\cdot))}}\geq \bar\kappa'
\ee
then
\be\label{eq:stability_iterates_high_hessians}
\|\nabla\psi^{T,G}_t[\mu_\cdot]-\nabla\psi^{T,\hat G}_t[\hat\mu_\cdot]\|_\infty \leq \int_t^{T''}2\frac{\const^b_x\const^{\delta\psi_s}_x+\const^F_{x\mu} W_{f_{\tilde\kappa}}(\mu_s,\hat\mu_s)}{C_{\tilde\kappa}}q^{\tilde\kappa}_{s-t} \De s+ q^{\tilde\kappa}_{T''-t} \| \psi^{T,G}_{T''}[\mu_\cdot]- \psi^{T,\hat G}_{T''}[\hat\mu_\cdot]\|_\Lip,
\ee
where $\const^{\delta\psi_s}_x$ has been defined at \eqref{c_delta_psi_x}.
\een
\end{cor}
\begin{proof}
First we observe that from \Cref{lem:MF_syst_high_reg} we have $\|\ke{\nabla}\varphi^\infty\|_{\Lip}\leq\const^\psi_{xx}$ with $\const^\psi_{xx}$ given by \eqref{eq:hessian_choices_1}. Next, we invoke \Cref{thm:lin_Hess_est} to obtain that for all $t\leq T$ and all flow $\mu_\cdot$ we have
\bes 
\| \psi^{T,\varphi^\infty}_t[\mu_\cdot]\|_\Lip\leq \const^{\psi}_{xx}+ ((2\const^\psi_x+1)\const^\psi_{xx}+4(\const^\psi_x)^2)\tilde{C}^{-1}_{\bar\kappa}, \quad \|w(\cdot,\psi^{T,\varphi^\infty}_t[\mu_\cdot](\cdot)\|_\Lip\leq \const^u_x
\ees
with $\const^u_x$ as in \eqref{eq:iterates_choices}. But then we have \be\label{eq:kappa'_bound}
\kappa_{\partial_pH(\cdot,\nabla\psi^{T,\varphi^\infty}_t[\mu_\cdot])} \geq \bar\kappa'
\ee for all $t\leq T$. Using this information in \Cref{lem:basics_MF_high}-\ref{item:stab_MF_high} proves \ref{item:cor:boost_i}. To prove \ref{item:cor:boost_ii} we first invoke the dynamic programming principle and then  \Cref{lem:stability_estimate_linear_pb}-\ref{item:linear_stability_boost} for the choice $\tilde\kappa=\ke{\bar\kappa'}$. This choice can be made  thanks to \eqref{eq:kappa'_bound} and the extra assumption \eqref{eq:cor_boost_extra_ass}.
\end{proof}

\noindent Let us now come to the proof of the turnpike property in the high regularity regime. In order to so, we precise the two auxiliary metrics needed to conclude with \Cref{lem:abs_turnpike_est}
\begin{equation}\label{eq:def_abs_met_high_mild}
\begin{aligned}
&\overrightarrow{d_{\lambda, \kappa}^{T}}(\fl,\hfl) := \sup_{0 \leq s\leq T}e^{\lambda s}W_{f_\kappa}(\mu_s,\hat\mu_s), \quad \overleftarrow{d_{\lambda,\kappa}^{T}}(\fl,\hfl) := \sup_{0 \leq s\leq T}e^{\lambda (T-t)}W_{f_\kappa}(\mu_s,\hat\mu_s).
\end{aligned}
\end{equation}
We also introduce the notation of the following norms
\bes
\|G\|_{f_{\kappa_b}}:= \sup_{\mu \in \cP_1(\bbR^d)}\| G(\mu,\cdot)\|_{f_{\kappa_b}}, \quad \|G\|_\infty=\sup_{\mu\in\cP_1(\bbR^d)}\|G(\mu,\cdot) \|_\infty.
\ees
\begin{thm}\label{prop:abstract_suff_cond_tpike_high}
Assume \Cref{ass:MF_drift_intro}, \Cref{ass:MF_coercivity_intro}, \Cref{ass:high_reg_intro} and let 
$\hat{G}: \bbR^d \rightarrow \bbR^d$ be such that $\| \hat{G}\|_{f_{\kappa_b}}  \leq 2 \const_x^\psi.$
\ben[label=(\roman*)] 
\item\label{item:grad_est_high} For any $\mu_\cdot,  \in\Gamma$ and $T>0$ the estimate
\begin{equation}\label{eq:gest_high_last}
\begin{split}
\|\nabla\psi^{T,\hat{G}}_t[\mu_\cdot] \|_\infty \leq 2 \const_x^\psi , \quad
\end{split}
\end{equation}
holds for all $t\in[0,T]$ with $\const^\psi_x$ given by \eqref{def:c_x_psi}. Moreover, for any $\hat\mu_\cdot\in\Gamma$ such that $\hat\mu_0\in\cP_p(\bbR^d)$, we have $\nu^{T,\hat G}[\hat\mu_\cdot]_\cdot\in \Gamma_{\hat\mu_0,\const}\subseteq\Gamma$ where
\bes
\Gamma_{\hat\mu_0,\const} =\{\tilde\mu : [0,T] \rightarrow \mathcal{P}_1(\mathbb{R}^d) : \tilde\mu_0=\hat\mu_0, \eqsp\sup_{0\leq s<t\leq T}\frac{W_1(\tilde\mu_s,\tilde\mu_t)}{|t-s|^{1/2}}\leq \const \}
\ees
and $\const$ depends only on $\Sigma,\int|x|^p\mu_0(\De x),\kappa_b,\const^\psi_x$ and $T$.

\item\label{item:stability_iterate_high} Let $\bar\kappa$ be given by \eqref{eq:bar_kappa_intro}.
Then, for any $\mu_\cdot, \hat{\mu}_\cdot \in \Gamma$, $\lambda<\lambda_{\bar{\kappa}}$ we have
\begin{equation}\label{eq:iterate_stability_high_1}
\overrightarrow{d_{\lambda,\bar{\kappa}}^{T}}(\nu^{T,\hat{G}}_\cdot[\mu_\cdot],\nu^{T,\hat{G}}_\cdot[\hat\mu_\cdot])\leq W_{f_{\bar{\kappa}}}(\mu_0,\hat{\mu}_0)+ \varepsilon(\lambda) \overrightarrow{d_{\lambda,\bar{\kappa}}^{T}}(\mu_\cdot,\hat\mu_\cdot),
\end{equation}
with
\begin{equation}\label{eq:epsilon_lambda_high}
\varepsilon(\lambda)=\frac{\const^F_{x\mu}}{\rho^L_{uu}C^2_{\bar{\kappa}}(\lambda^2_{\bar{\kappa}}-\lambda^2)}.
\end{equation}
If $\mu_0 = \hat{\mu}_0$, then we also have
\begin{equation}\label{eq:iterate_stability_high_T}
\overleftarrow{d_{\lambda,\bar{\kappa}}^{T}}(\nu^{T,\hat{G}}_\cdot[\mu_\cdot],\nu^{T,\hat{G}}_\cdot[\hat\mu_\cdot])\leq \varepsilon(\lambda) \overleftarrow{d_{\lambda,\bar{\kappa}}^{T}}(\mu_\cdot,\hat\mu_\cdot).
\end{equation}
\item\label{item:contraction_iterate_high}  If 
\begin{equation}\label{eq:suff_cond_tpik_high_reg_in_proofs}
\const^F_{x\mu}<\rho^L_{uu}C^2_{\bar{\kappa}}\lambda^2_{\bar{\kappa}},
\end{equation}
then for any $\lambda<\lambda^*$ with
\bes
\lambda^*=(\lambda_{\bar{\kappa}}-\frac{\const^F_{x\mu}}{\rho^L_{uu}C^2_{\bar{\kappa}}})^{1/2},
\ees 
we have $\varepsilon(\lambda)<1$.
In particular, for any initial condition $\hat\mu_0\in\cP_{p}(\bbR^d)$ the mean field PDE system \eqref{eq:MF_PDE} with terminal condition $\hat{G}$ and initial condition $\hat{\mu}_0$ has a unique solution $(\mu^{T,\hat{G}}_\cdot, \varphi^{T,\hat{G}}_\cdot)$ in $\Gamma_{\hat\mu_0,\const}$ with $\const$ as in \ref{item:grad_est_high}. 
\item\label{item:epsilon_boost} If additionally \Cref{ass:boost} holds and $\hat{G}=\varphi^\infty$ then items \ref{item:stability_iterate_high} and \ref{item:contraction_iterate_high} hold replacing $\bar\kappa$ with $\bar\kappa'$ as in \eqref{eq:kappa_prime_def} and $\const^u_x$ as in \eqref{eq:iterates_choices}.
\item\label{item:high_reg_tpike_est} For any $\lambda<\lambda^*$ and any solution $(\mu^{T,G}_\cdot, \varphi^{T,G}_\cdot)$  to \eqref{eq:MF_PDE} with terminal condition $G:\cP_1(\bbR^d) \times \bbR^d \rightarrow \bbR$ and initial condition $\mu_0\in\cP_p(\bbR^d)$ the following holds.
\ben[(a)]
\item\label{item:high_lipschitz_case} If $\| G\|_{f_{\kappa_b}}<+\infty$ then, defining 
\bes\label{eq:mult_constants_high_3}
\tau(G)= \frac{\log\big(\|G\|_{f_{\kappa_b}}-\const^{\psi}_x)/\const^\psi_x\big)}{\lambda_{\kappa_b}}\vee 0 , \quad \kappa_G(r) = \kappa_b(r) -\frac{2(\const^{L(\cdot,0)}_u+ \max\{2\const^\psi_x, \|G \|_{f_{\kappa_b}}\})}{\rho^{L}_{uu} r},
\ees
we have that for all $T>\tau(G)$
\begin{equation}\label{eq:tp_flow_Lip_high}
W_{f_{\bar{\kappa}}}(\mu^{T,G}_t, \mu^{T,\hat{G}}_t) \leq \frac{1}{1-\varepsilon(\lambda)}\Big(W_{f_{\bar\kappa}}(\mu_0,\hat{\mu}_0) e^{-\lambda t}  + \frac{e^{\lambda_{\bar\kappa}\tau(G) }}{2\rho^{L}_{uu}C_{\kappa_G}\lambda_{\kappa_G}}\|\hat{G} - G(\mu^{T,G}_T,\cdot) \|_{f_{\kappa_G}} e^{-\lambda(T-t)}\Big)
\end{equation}
holds for all $t\leq T$. Moreover, for all $t \leq T-\tau(G)$ we have
\bes\label{eq:tp_control_lip_high}
\|\varphi^{T,G}_t-\varphi^{T,\hat G}_t \|_\Lip\leq \frac{\const^F_{x\mu}}{C^2_{\bar\kappa}(1-\varepsilon(\lambda))} \Big(\frac{W_{f_{\bar\kappa}}(\mu_0,\hat\mu_0)}{\lambda+\lambda_{\bar\kappa}}e^{-\lambda t}+\frac{2\const^\psi_xe^{\lambda_{\bar\kappa}\tau(G)}}{\lambda_{\bar\kappa}C_{\bar\kappa}(\lambda_{\bar{\kappa}}-\lambda)\rho^L_{uu}}e^{-\lambda(T-t)}\Big)+\frac{4\const^\psi_x}{C_{\bar\kappa}}e^{\lambda_{\bar\kappa}\tau(G)}e^{-\lambda_{\bar\kappa}(T-t)}
\ees

\item\label{item:hig_bdd_case} If  $\|G \|_{\infty} < + \infty$, if we define 
\footnote{ We have the explicit bound
\bes
\tau'(G)\leq  \max\left\{ \frac{1}{\lambda_{\kappa_b}}\log\const^G, \frac{1}{2\lambda_{\kappa_b}} \right\}> 0, \text{ with } \const^G = \frac{\frac{\sqrt{\lambda_{\kappa_b}e}}{\sqrt{\pi}C_{\kappa_b}\sigma_0}\| G\|_{\infty}  - \const^\psi_x}{\const^\psi_x}
\ees} \bes
\tau'(G)=\inf\{\tau>0: q^{\kappa_b}_\tau \|G\|_\infty\leq \const^\psi_x\},\quad T'_0:=T-\tau'(G),
\ees

then the turnpike estimate for marginal flows  \eqref{eq:tp_flow_Lip_high} holds for $t\leq T'_0$ setting $\tau(G)=\tau'(G),\kappa_G=\bar\kappa$ and replacing $\|\hat{G} - G(\mu^{T,G}_T,\cdot) \|_{f_{\kappa_G}}$ with $4C_{\bar\kappa}^{-1}\const^\psi_x$. Moreover, for $t\in[T'_0,T]$ we have 
\bes
W_{f_{\bar\kappa}}(\mu^{T,G}_t,\mu^{T,\hat G}_t) \leq M_1(\|G\|_\infty)+\tilde{M}_1(2\const^\psi_x)
\ees
with $M_1(\|G\|_{\infty}),\,\tilde{M}_1(2\const^{\psi}_x)$ as in \eqref{eq:moment_bound}. Moreover, the turnpike estimate \eqref{eq:tp_control_lip_high} holds for all $t\leq T'_0$ replacing $\tau(G)$ with $\tau'(G)$ .
\item \label{item:high_lipschitz_case_and_boost} If additionally \Cref{ass:boost} holds, $\hat G=\varphi^{\infty},\|G\|_{f_{\kappa_b}}<+\infty$ and if we define $\tau''(G)<+\infty$ by
\bes
\tau''(G)=\inf\{\tau\geq0:\kappa_{\partial_pH(\cdot,\nabla\varphi^{T,G}_{T-s}(\cdot))} \geq \bar\kappa' \,\,\,\forall s\geq\tau \},\quad T''_0:=T-\tau''(G),
\ees
then the turnpike estimate for marginal flows  \eqref{eq:tp_flow_Lip_high} holds for $t\leq T$ replacing $\bar\kappa$ with $\bar\kappa'$ and  $\tau(G)$ with $\tau''(G)$. Moreover, the turnpike estimate \eqref{eq:tp_control_lip_high} holds for $t\leq T''_0$ replacing $\bar\kappa$ with $\bar\kappa'$ and  $\tau(G)$ with $\tau''(G)$.
\item\label{item:tpike_for_hessians} If \Cref{ass:boost} holds, $\hat G=\varphi^\infty$ and $\sigma(\cdot)=2\sigma_0\mathrm{I}$, then for all $t\leq T''_0$
\be\label{eq:tpike_hess_in_proof}
\|\nabla\varphi^{T,G}_t-\nabla\varphi^{T,\hat G}_t\|_\Lip\leq \const'_{\mathrm{i}}e^{-\lambda t}+\const'_{\mathrm{f}}e^{-\lambda(T-t)},
\ee
where $\const'_{\mathrm{i}},\const'_{\mathrm{f}}$ depend on $\kappa_b$, all constants in \Cref{ass:MF_drift_intro},\Cref{ass:MF_coercivity_intro},\Cref{ass:high_reg_intro},\Cref{ass:boost}, $\lambda$ and $G$.

\een
\end{enumerate}
\end{thm}
\begin{proof}

\begin{itemize}[wide]
\item \underline{Proof of \ref{item:grad_est_high}:} 
The first claim follows directly from  \Cref{lem:basics_MF_high}-\ref{item:grad_est_MF_high}. The second claim follows applying \Cref{prop:Holder_in_time}-\ref{item:time_Holder_Wf} with $\beta_t(x)=\partial_pH(x,\nabla  \psi^{T,\hat{G}}_t [\hat\mu_\cdot](x))$ using the gradient estimate \eqref{eq:gest_high_last} to find a lower bound for $\kappa_{\beta}$.

\item \underline{Proof of \ref{item:stability_iterate_high}:} 
In order to ease notation, we write $\hat{\nu}= \nu^{T,\hat{G}}$, $\nu = \nu^{T,G}$, $\hat{\psi} = \psi^{T,\hat{G}}$, $\psi = \psi^{T,G}$.
Using
\begin{equation}
W_{f_{\bar{\kappa}}}(\mu_s,\hat{\mu}_s) \leq e^{-\lambda s} \overrightarrow{d_{\lambda,\bar{\kappa}}^{T}}(\mu_\cdot,\hat\mu_\cdot),
\end{equation}
and the hypothesis on $\|\hat{G}\|_{f_{\kappa_b}}$, we can apply \Cref{lem:basics_MF_high}-\ref{item:stab_MF_high} with $\tilde\kappa=\bar\kappa$ to get
\begin{equation}\label{eq:forward_norm_1}
\|\nabla\hat{\psi}_t[\mu_\cdot]-\nabla\hat{\psi}_t[\hat\mu_\cdot]\|_{\infty}= \const^{\delta\psi_t}_x\leq \frac{\const^F_{x\mu}}{C^2_{\bar{\kappa}}}  \overrightarrow{d_{\lambda,\bar{\kappa}}^{T}}(\mu_\cdot,\hat\mu_\cdot) \int_t^Te^{-\lambda s-\lambda_{\bar{\kappa}}(s-t)} \De s\leq \frac{\const^F_{x\mu}}{C^2_{\bar{\kappa}}}  \frac{e^{-\lambda t}}{\lambda_{\bar{\kappa}}+\lambda} \overrightarrow{d_{\lambda,\bar{\kappa}}^{T}}(\mu_\cdot,\hat\mu_\cdot)
\end{equation}
and
\begin{equation}\label{eq:high_reg_iteration_stability_1}
\begin{split}
W_{f_{\bar{\kappa}}}(\hat{\nu}_t[\mu_\cdot],\hat{\nu}_t[\hat\mu_\cdot])&\leq e^{-\lambda_{\bar{\kappa}}t}W_{f_{\bar{\kappa}}}(\mu_0,\hat{\mu}_0)+ \frac{1}{\rho^L_{uu}}\int_0^te^{-\lambda_{\bar{\kappa}}(t-s)}\const_x^{\delta\psi_s}\De s\\
&\stackrel{\eqref{eq:forward_norm_1}}{\leq} e^{-\lambda_{\bar{\kappa}}t}W_{f_{\bar{\kappa}}}(\mu_0,\hat{\mu}_0)+ \frac{\const^F_{x\mu}}{\rho^L_{uu}C^2_{\bar{\kappa}}} \frac{e^{-\lambda t}}{\lambda_{\bar{\kappa}}^2 -\lambda^2}\overrightarrow{d_{\lambda,\bar{\kappa}}^{T}}(\mu_\cdot,\hat\mu_\cdot).
\end{split}
\end{equation}

From this bound, we obtain the first inequality by multiplying by $e^{\lambda t}$, maximizing over $t \in [0,T]$ and using $\lambda < \lambda_{\bar{\kappa}}$.
For the second bound in $\overleftarrow{d_{\lambda,\bar{\kappa}}^{T}}$ suppose now $\mu_0 = \hat{\mu}_0$. Using again
\Cref{lem:basics_MF_high}-\ref{item:stab_MF_high} for $\tilde\kappa=\bar\kappa$ and the definition of $\overleftarrow{d_{\lambda,\bar{\kappa}}^{T}}$ gives
\begin{equation}\label{eq:high_reg_iteration_stability_2}
\|\nabla\hat{\psi}_t[\mu_\cdot]-\nabla\hat{\psi}_t[\hat\mu_\cdot]\|_{\infty}=\const^{\delta\psi_t}_x\leq \frac{\const^F_{x\mu}}{C^2_{\bar{\kappa}}}  \overleftarrow{d_{\lambda,\bar{\kappa}}^{T}}(\mu_\cdot,\hat\mu_\cdot) \int_t^Te^{-\lambda (T-s)}e^{-\lambda_{\bar{\kappa}}(s-t)} \De s\leq \frac{\const^F_{x\mu}}{C^2_{\bar{\kappa}}}  \frac{e^{-\lambda (T-s)}}{\lambda_{\bar{\kappa}}-\lambda} \overleftarrow{d_{\lambda,\bar{\kappa}}^{T}}(\mu_\cdot,\hat\mu_\cdot)
\end{equation}
and
\begin{equation}
W_{f_{\bar{\kappa}}}(\hat{\nu}_t[\mu_\cdot],\hat{\nu}_t[\hat\mu_\cdot])\leq  \frac{1}{\rho^L_{uu}}\int_0^te^{-\lambda_{\bar{\kappa}}(t-s)}\const_x^{\delta\psi_s}\De s \stackrel{\label{eq:high_reg_iteration_stability_2}}{\leq} \frac{\const^F_{x\mu}}{\rho^L_{uu}C^2_{\bar{\kappa}}} \frac{e^{-\lambda (T-t)}}{\lambda_{\bar{\kappa}}^2 -\lambda^2}\overleftarrow{d_{\lambda,\bar{\kappa}}^{T}}(\mu_\cdot,\hat\mu_\cdot).
\end{equation}
The second inequality now follows by multiplying by $e^{\lambda (T-t)}$, maximizing over $t \in [0,T]$ and using $\lambda < \lambda_{\bar{\kappa}}$.
\item \underline{Proof of \ref{item:contraction_iterate_high}:} 
The fact that we have $\varepsilon(\lambda) < 1$ under these conditions follows by standard computations. But then, thanks to \ref{item:grad_est_high}, we have that for any $\hat\mu_0\in\cP_p(\bbR^d)$ and $\const$ given in \ref{item:grad_est_high}, $\nu^{T,\hat G}[\cdot]$ maps the complete metric space $\Gamma_{\hat\mu_0,\const}$ onto itself and is a strict contraction for the distance $\overrightarrow{d_{\lambda,\bar{\kappa}}^{T}}$. We can therefore invoke Banach's fixed point Theorem which gives us existence of a unique fixed point.
Now, observe that the existence and uniqueness result follows by Banach's fixed point theorem. 
\item \underline{Proof of \ref{item:epsilon_boost}} The proof is identical to the proof of \ref{item:stability_iterate_high}-\ref{item:contraction_iterate_high} with the only difference that \Cref{lem:basics_MF_high}-\ref{item:stab_MF_high} is applied for the choice $\tilde\kappa=\bar\kappa'$. We are allowed to make such choice because of  \Cref{cor:boost}-\ref{item:cor:boost_i}, which we can invoke thanks to the extra assumptions \Cref{ass:boost} and $\hat G=\varphi^\infty$.
\item \underline{Proof of \ref{item:high_reg_tpike_est}}
Thanks to the estimates we have just established in \ref{item:stability_iterate_high}, \Cref{lem:abs_turnpike_est} gives for $\sfd = W_{f_{\bar{\kappa}}}$ and $\upsilon^\lambda(t) = e^{-\lambda t}$
\begin{equation}\label{eq:high_tpike_1}
W_{f_{\bar{\kappa}}}(\mu^{T,G}_t, \mu^{T,\hat{G}}_t) \leq \frac{e^{-\lambda t}}{1 - \varepsilon(\lambda)} W_{f_{\bar{\kappa}}}(\mu_0, \hat{\mu}_0) + 
\frac{e^{-\lambda(T-t)}}{1 - \varepsilon(\lambda)} \overleftarrow{d_{\lambda,\bar{\kappa}}^{T}}(\hat{\nu}_\cdot[\mu^{T,G}_\cdot],\nu_\cdot[\mu^{T,G}_\cdot]) \eqsp.
\end{equation}
We now proceed to bound $\overleftarrow{d_{\lambda,\bar{\kappa}}^{T}}(\hat{\nu}_\cdot[\mu^{T,G}_\cdot],\nu_\cdot[\mu^{T,G}_\cdot])$ for $G$ Lipschitz in the space variable and $G$ bounded.
\begin{itemize}[wide]
\item  \underline{Case $\| G\|_{f_{\kappa_b}} < +\infty  $}  
We distinguish between two subcases depending on whether or not $\tau(G)>0$. If $\tau(G)=0$, i.e. if $\|G\|_{f_{\kappa_b}}\leq 2\const^\psi_x$, we know from \Cref{lem:basics_MF_high}-\ref{item:grad_est_MF_high} that $\|\varphi^{T,G}_t\|_{f_{\kappa_b}} \leq 2 \const^{\psi}_x$ for $t\leq T$. This bound entitles us to apply \Cref{lem:basics_MF_high}-\ref{item:stab_MF_high} with $\tilde\kappa = \bar{\kappa}$  to obtain
\begin{align}
W_{f_{\bar{\kappa}}}(\hat{\nu}_t[\mu^{T,G}_\cdot],\nu_t[\mu^{T,G}_\cdot]) \leq e^{-\lambda_{\bar{\kappa}}(T-t)} \frac{1}{2\lambda_{\bar{\kappa}}\rho^L_{uu}} \|\hat{G} - G(\mu^{T,G}_T,\cdot) \|_{f_{\bar{\kappa}}} \eqsp,
\end{align}
which yields
\begin{equation}
\overleftarrow{d_{\lambda,\bar{\kappa}}^{T}}(\hat{\nu}_\cdot[\mu^{T,G}_\cdot],\nu_\cdot[\mu^{T,G}_\cdot]) \leq \frac{1}{2\lambda_{\bar{\kappa}}\rho^L_{uu}} \|\hat{G} - G(\mu^{T,G}_T,\cdot) \|_{f_{\bar{\kappa}}} \eqsp.
\end{equation}
Plugging this result in \eqref{eq:high_tpike_1} gives
\begin{equation}\label{eq:tpike_better_G_Lip_small}
W_{f_{\bar{\kappa}}}(\mu^{T,G}_t, \mu^{T,\hat{G}}_t) \leq \frac{1}{1-\varepsilon(\lambda)}\Big(W_{f_{\bar\kappa}}(\mu_0,\hat{\mu}_0) e^{-\lambda t}  + \frac{1}{2\lambda_{\bar\kappa}\rho^{L}_{uu}}\|\hat{G} - G(\mu^{T,G}_T,\cdot) \|_{f_{\bar\kappa}} e^{-\lambda(T-t)}\Big),
\end{equation}
which implies the turnpike estimate \eqref{eq:tp_flow_Lip_high} since $\bar{\kappa} \geq \kappa_G$ and $C_{\kappa_G} \leq 1$. In turn, using \eqref{eq:tp_flow_Lip_high} in the first estimate of \Cref{lem:basics_MF_high}-\ref{item:stab_MF_high} gives, after some basic calculations that 
\bes\label{eq:high_tpike_control_easy}
\|\varphi^{T,G}_t-\varphi^{T,\hat G}_t \|_\Lip\leq \frac{\const^F_{x\mu}}{C^2_{\bar\kappa}(1-\varepsilon(\lambda))} \Big(\frac{W_{f_{\bar\kappa}}(\mu_0,\hat\mu_0)}{\lambda+ \lambda_{\bar{\kappa}}}e^{-\lambda t}+\frac{\|\hat G - G(\mu_T,\cdot)\|_{f_{\bar\kappa}}}{2\lambda_{\bar\kappa}(\lambda_{\bar{\kappa}}-\lambda)\rho^L_{uu}}e^{-\lambda(T-t)}\Big)+\|\hat G - G(\mu_T,\cdot)\|_{f_{\bar\kappa}}e^{-\lambda_{\bar\kappa}(T-t)},
\ees
which implies \eqref{eq:tp_control_lip_high} in the current setting where $\tau(G)=0$ by bounding $\|\hat G - G(\mu_T,\cdot)\|_{f_{\bar\kappa}}$ using the equivalence between $\|\cdot\|_{f_{\bar\kappa}}$,  $\|\cdot\|_{f_{\kappa_b}}$ and a triangular inequality.
On the other hand, if $\|G(\mu,\cdot) \|_{f_{\kappa_b}} > 2 \const^{\psi}_x$ then, defining $T_0 = T -\tau(G)$, we have $\|\psi^{T,G}_{t}[\mu^{T,G}_\cdot] \|_{f_{\kappa_b}} \leq 2 \const^\psi_x$ for all $t \leq T_0$ thanks to \Cref{lem:basics_MF_high}-\ref{item:grad_est_MF_high}.
But then, arguing on the basis of the dynamic programming principle we can apply the result obtained in the previous case. We get that for $t\leq T_0$ 
\bes
W_{f_{\bar{\kappa}}}(\hat{\nu}_t[\mu^{T,G}_\cdot],\nu_t[\mu^{T,G}_\cdot]) 
\leq e^{-\lambda_{\bar{\kappa}}(T_0-t)} \frac{1}{2\lambda_{\bar{\kappa}}\rho^L_{uu}} \|\hat{\psi}_{T_0}[\mu^{T,G}_\cdot] - \psi_{T_0}[\mu^{T,G}_\cdot] \|_{f_{\bar{\kappa}}}.
\ees
Using the definition of $\tau(G)$ and 
applying 
\Cref{lem:basics_MF_high}-\ref{item:stab_MF_high} with $\tilde\kappa=\kappa_G$ and observing that $\bar\kappa\geq\kappa_G$ implies
\bes
\|\hat{\psi}_{T_0}[\mu^{T,G}_\cdot] - \psi_{T_0}[\mu^{T,G}_\cdot] \|_{f_{\bar{\kappa}}}\leq \|\hat{\psi}_{T_0}[\mu^{T,G}_\cdot] - \psi_{T_0}[\mu^{T,G}_\cdot] \|_{f_{\kappa_G}} \leq \|\hat{G}-G(\mu_t,\cdot)\|_{f_{\kappa_G}},
\ees 
we get using again $\bar\kappa\geq\kappa_G$ and $C_{\kappa_G}\leq1$
\begin{equation}\label{eq:DPP_estimate}
\begin{split}
W_{f_{\bar{\kappa}}}(\hat{\nu}_t[\mu^{T,G}_\cdot],\nu_t[\mu^{T,G}_\cdot]) 
\leq \,e^{\lambda_{\bar\kappa}\tau(G)}e^{-\lambda_{\bar{\kappa}}(T-t)}\frac{\|\hat{G}-G(\mu_t,\cdot)\|_{f_{\kappa_G}}}{2\lambda_{\kappa_G}C_{{\kappa_G}}\rho^L_{uu}}.
\end{split}
\end{equation}
It remains to consider the case when $t\in[T_0,T]$. Here, invoking again \Cref{lem:basics_MF_high}-\ref{item:stab_high_reg} with $\tilde\kappa=\kappa_G$ provides with
\bes\label{eq:early_stopping_high_1}
\begin{split}
W_{f_{\kappa_G}}(\hat{\nu}_t[\mu^{T,G}_\cdot],\nu_t[\mu^{T,G}_\cdot]) 
\leq e^{-\lambda_{\kappa_G}(T-t)} \frac{\|\hat G- G(\mu_t,\cdot)\|_{f_{\kappa_G}} }{2\lambda_{\kappa_G}\rho^L_{uu}} \\\leq e^{\lambda_{\bar\kappa}\tau(G)}e^{-\lambda_{\bar{\kappa}}(T-t)} \frac{\|\hat{G}-G(\mu_t,\cdot)\|_{f_{\kappa_G}}}{2\lambda_{\kappa_G}\rho^L_{uu}} .
\end{split}
\ees
Using $W_{f_{\bar\kappa}}\leq C^{-1}_{\kappa_G}W_{f_{\kappa_G}}$ and gathering  the resulting bound with \eqref{eq:early_stopping_high_1}  yields
\begin{align}
\overleftarrow{d_{\lambda,\bar{\kappa}}^{T}}(\hat{\nu}_\cdot[\mu^{T,G}_\cdot],\nu_\cdot[\mu^{T,G}_\cdot]) \leq \frac{e^{\lambda_{\bar\kappa}\tau(G)}}{2\lambda_{\kappa_G}C_{\kappa_G} \rho^{L}_{uu}} \|\hat{G} - G(\mu^{T,G}_T,\cdot) \| _{f_{\kappa_G}} \eqsp.
\end{align}
Plugging this bound back in \eqref{eq:high_tpike_1} gives the turnpike estimate \eqref{eq:tp_control_lip_high}. We are left with the proof of the turnpike estimates for the value functions. To this aim, observe that since we have that $\|\varphi^{T,\hat G}_{t}\|_{f_{\kappa_b}},\|\varphi^{T, G}_{t}\|_{f_{\kappa_b}}\leq 2\const^\psi_x$ for all $t \leq T_0$ thanks to \Cref{lem:basics_MF_high}-\ref{item:grad_est_MF_high}. But then, thanks to the dynamic programming principle,  we can apply the results proven in the former case if we restrict to the time interval $[0,T_0]$. In particular, \eqref{eq:high_tpike_control_easy}, implies that for $t\leq T_0$
\bes
\|\varphi^{T,G}_t-\varphi^{T,\hat G}_t \|_\Lip\leq\frac{\const^F_{x\mu}}{C^2_{\bar\kappa}(1-\varepsilon(\lambda))} \Big(\frac{W_{f_{\bar\kappa}}(\mu_0,\hat\mu_0)}{\lambda+\lambda_{\bar{\kappa}}}e^{-\lambda t}+\frac{\|\varphi^{T,\hat G}_{T_0} - \varphi^{T,G}_{T_0}\|_{f_{\bar\kappa}}}{2\lambda_{\bar\kappa}(\lambda_{\bar{\kappa}}-\lambda)\rho^L_{uu}}e^{-\lambda(T_0-t)}\Big)+\|\varphi^{T,\hat G}_{T_0} - \varphi^{T,G}_{T_0}\|_{f_{\bar\kappa}}e^{-\lambda_{\bar\kappa}(T_0-t)}.
\ees
Rewriting $T_0=T-\tau(G)$ and bounding $\|\varphi^{T,\hat G}_{T_0} - \varphi^{T,G}_{T_0}\|_{f_{\bar\kappa}}$ using the equivalence between $\|\cdot\|_{f_{\bar\kappa}}$,  $\|\cdot\|_{f_{\kappa_b}}$ and triangular inequality gives the desired result.
\item \underline{Case $\| G\|_{\infty}  < +\infty$. } 
We have thanks to \Cref{lem:basics_MF_high}-\ref{item:grad_est_MF_high} that $\|\varphi^{T,G}_t \|_{f_{\kappa_b}} \leq 2 \const^\psi_x$ for $t\leq T'_0$. Thus, invoking the dynamic programming principle, we can use the results obtained under the assumption ${\| G\|_{f_{\kappa_b}}<+\infty}$ if we restrict to the time-interval $[0,T'_0]$. This means that the bounds \eqref{eq:tp_flow_Lip_high}-\eqref{eq:tp_control_lip_high} hold setting 
\bes 
\tau(G)=\tau'(G),\,\kappa_G=\bar\kappa, \quad \text{and} \quad \hat{G}=\varphi^{T,\hat G}_{T'_0},G(\mu_T,\cdot)=\varphi^{T,G}_{T'_0}.\ees Since the difference $\| \varphi^{T,G}_{T'_0}-\varphi^{T,\hat{G}}_{T_0}\|_{f_{\bar\kappa}}$ can be bounded with $4C^{-1}_{\bar\kappa}\const^\psi_x$ by a triangular inequality and the definition of $T'_0$, we obtain the desired turnpike estimate for the flows and the sought turnpike estimate for the value functions in $t\in[0,T'_0]$. To cover the case $t\in[T'_0,T]$ for the flows we use the trivial bound $W_{f_{\bar\kappa}}(\mu^{T,\hat G}_t,\mu^{T, G}_t)\leq \int|x|\mu^{T,G}_t+\int|x|\mu^{T,\hat G}_t$, and then bound these quantities with Lemma \ref{lem:moment_bound}.
\item \underline{Proof of \ref{item:high_lipschitz_case_and_boost}} The proof is basically same used for \ref{item:high_lipschitz_case}, with the only differences that here we  consider the subcases $\tau''(G)=0$, $\tau''(G)>0$ instead of considering the subcases $\tau(G)=0$, $\tau(G)>0$ as we did there, and all applications of \Cref{lem:basics_MF_high}-\ref{item:stab_MF_high} for $\tilde\kappa=\bar\kappa$ are now with $\tilde\kappa=\bar\kappa'$. Note that we are allowed to so because of \Cref{cor:boost}-\ref{item:cor:boost_i}. Finally, let us remark that the fact that $\tau''(G)<+\infty$ is a consequence of the fact that the dependence on $\| G\|_{f_{\kappa_b}}$ in the gradient and hessian bounds in \Cref{lemma:gradient_estimate_linearized problem}, \Cref{lem:hessian_bound_sigma_constant}, \Cref{thm:lin_Hess_est} decreases exponentially for large values of $T-t$,
 thanks to the definition of $q^\kappa_{T-t}$, see \eqref{q_kappa_t_def}.
\item\underline{Proof of \ref{item:tpike_for_hessians}} Using the fixed-point property of solutions we bound $\|\nabla\varphi^{T,G}_t-\nabla\varphi^{T,\hat G}_t\|_{\Lip}$ with \Cref{cor:boost}-\ref{item:cor:boost_ii}, which we apply with $T''=T''_0$. Indeed, this choice ensures that \eqref{eq:cor_boost_extra_ass} holds.  To obtain \eqref{eq:tpike_hess_in_proof}, we use the turnpike estimate for flows \eqref{eq:tp_flow_Lip_high} to estimate $W_{f_{\bar\kappa}}(\mu_s,\hat\mu_s)$, the turnpike estimate for the value functions \eqref{eq:tp_control_lip_high} to estimate $\const^{\delta\psi_s}_x$, and finally upper bound the resulting integrals with the help of \Cref{lem:boring_calculations_mild}.
\end{itemize}
\end{itemize}
\end{proof}

\begin{lemma}\label{lem:moment_bound}
Let $(\mu^{T,G}_\cdot,\varphi^{T,G}_\cdot)$ be a solution to \eqref{eq:MF_PDE} and $\mu^b$ be the invariant measure for the uncontrolled dynamics \eqref{eq:SDE_sigma_intro}. Then, we have
\bes
\int_{\bbR^d}|x|\,\mu^{T,G}_t\leq \min\{M_1(\|G\|_\infty),\tilde{M}_1(\|G\|_{f_{\kappa_b}})\}
\ees
with 
\begin{equation}\label{eq:moment_bound}
\begin{aligned}
M_1(\|G\|_\infty) &:= (1+C_{\kappa_b}^{-1} )\int_{\bbR^d}|x|\mu^b + C_{\kappa_b}^{-1} \left( \int_{\bbR^d}|x|\mu_0+\frac{1}{\rho^L_{uu}\lambda_{\kappa_b}}(\const^\psi_x+ \const^{L(\cdot,0)}_u)+ \frac{3\|G\|_\infty}{2\rho^L_{uu}\sqrt{\pi\lambda_{\kappa_b}}C_{\kappa_b}\sigma_0} \right)\\
\tilde{M}_1(\|G\|_{f_{\kappa_b}}) &:=
(1+C_{\kappa_b}^{-1} )\int_{\bbR^d}|x|\mu^b + C_{\kappa_b}^{-1} \left(\int_{\bbR^d}|x|\mu_0 + \frac{1}{\rho^L_{uu}\lambda_{\kappa_b}} (\const^{\psi}_{x}+\|G\|_{f_{\kappa_b}} + \const^{L(\cdot,0)}_u) \right)
\end{aligned}
\end{equation}
\end{lemma}
The proof can be found in \Cref{sec:proof_lem:moment_bound}.

\subsection{Mild regularity}\label{sec:MF_mild_reg}

Let us summarize some properties which we will constantly use in this section, as in \Cref{sec:MF_high_reg}. Recall that we define $\const^\psi_x = (\const^L_x + \const^F_x)/(\lambda_{\kappa_b} C_{\kappa_b})$ in \eqref{def:c_x_psi}.
\begin{lemma}\label{lem:basics_MF_mild}
Assume \Cref{ass:MF_drift_intro}, \Cref{ass:MF_coercivity_intro}, \Cref{ass:MF_mild_reg_intro}. Let $\mu_\cdot \in \Gamma$, $G: \cP_1(\bbR^d) \times \bbR^d \rightarrow \bbR$ locally Lipschitz continuous in the second variable such that $G(\mu,\cdot)$ is of linear growth for all $\mu\in\cP_1(\bbR^d)$.
The following holds.
\begin{enumerate}[label=(\roman*)] 
\item\label{item:HJB_ok_mild} There is a unique solution to \eqref{eq:linearized_PDE_system}, denote it by $(\psi^{T,G}_{\cdot}[\mu_\cdot],\nu^{T,G}_{\cdot}[\mu_\cdot])$. 

Moreover, setting $b_s=b$, $\ell_s(x,u) = L(x,u) + F(\mu_s,x)$, $g(x)=G(\mu_T,x)$, $\psi^{T,G}_{\cdot}[\mu_\cdot]$ is the value function of \eqref{eq:classical_control_problem} and the marginal laws of its optimally controlled dynamics are given by $\nu^{T,G}_{\cdot}[\mu_\cdot] \in \Gamma$.
\item\label{item:grad_est_MF_mild} The following quantitative bounds hold true
\begin{align}
\|\psi^{T,G}_{t}[\mu_\cdot]\|_{f_{\kappa_{b}}} &\leq \const^\psi_x (1 - e^{-\lambda_{\kappa_{b}}(T-t)}) + \min \left\{\| G(\mu_T,\cdot)\|_{f_{\kappa_{b}}}e^{-\lambda_{\kappa_{b}}(T-t)}, \| G(\mu_T,\cdot)\|_{\infty} q^{\kappa_b}_{T-t} \right\} , \\
\|w(\cdot,\nabla\psi^{T,G}_{t}[\mu_\cdot](\cdot))\|_\infty &\leq \frac{\|\psi^{T,G}_{t}[\mu_\cdot]\|_{f_{\kappa_{b}}}+\const^{L(\cdot,0)}}{\rho^{L}_{uu}},
\end{align}
where
\begin{equation}
w(x,p) := \argmin_{u \in \bbR^d} \left\{ L(x,u)+(b(x)+u)\cdot p\right\}  \eqsp.
\end{equation}
\item\label{item:stab_MF_mild} Suppose now that $G$ is Lipschitz in the space variable and that $\hat{G}: \cP_1(\bbR^d) \times \bbR^d \rightarrow \bbR$ is another terminal cost satisfying the same assumptions as $G$. Given another flow $\hat{\mu}_\cdot \in \Gamma$, we have
\begin{align}
\|\psi^{T,G}_{t}[\mu_\cdot]-\psi^{T,\hat{G}}_{t}[\hat{\mu}_\cdot]\|_{f_{\tilde\kappa}}
&\leq \const^{\delta \psi_t}_x+ \|G(\mu_T,\cdot) - \hat{G}(\hat{\mu}_T,\cdot) \|_{f_{\tilde\kappa}}e^{-\lambda_{\tilde\kappa}(T-t)} \\
W_{f_{\tilde\kappa}}(\nu^{T,G}_{\cdot}[\mu_\cdot],\nu^{T,\hat{G}}_{\cdot}[\hat{\mu}_\cdot])
&\leq e^{-\lambda_{\tilde\kappa}t} W_{f_{\tilde\kappa}}(\mu_0,\hat{\mu}_0)+ \frac{1}{\rho^L_{uu}}\int_0^t e^{-\lambda_{\tilde\kappa} (t-s)} \const^{\delta \psi_s}_x\De s \\
& \quad+ \frac{1}{2\rho^L_{uu}\lambda_{\tilde\kappa}}e^{-\lambda_{\tilde\kappa}(T-t)}\|G(\mu_T,\cdot) - \hat{G}(\hat{\mu}_T,\cdot) \|_{f_{\tilde\kappa}},
\end{align}
where 
$$\const^{\delta \psi_t}_x = \int_t^T\frac{2\const^F_{\mu}}{C_{\tilde\kappa}}W_{f_{\tilde\kappa}}(\mu_s,\hat{\mu}_s)q^{\tilde\kappa}_{s-t}\De s,$$
and $\tilde{\kappa} \in \msk$ is any profile satisfying \eqref{eq:tilde_kappa_cond}. In particular, if $\max\{\|\hat G(\mu_T,\cdot)\|_{f_{\kappa_\beta}},\| G(\mu_T,\cdot)\|_{f_{\kappa_\beta}}\}\leq 2\const^\psi_x$ we can choose $\tilde{\kappa}=\bar\kappa$, with $\bar\kappa$ as in \eqref{eq:bar_kappa_intro}.

\end{enumerate}
\end{lemma}
\begin{proof}
\Cref{item:HJB_ok_mild} is a consequence of \Cref{prop:properties} in combination with \Cref{prop:Holder_in_time}, where we refer to \Cref{cor:grad_est_g_bdd} for the bounded case. In the same way, \ref{item:grad_est_MF_mild} follows by a direct application of \Cref{lemma:gradient_estimate_linearized problem} and \Cref{cor:grad_est_g_bdd}. Finally, \ref{item:stab_MF_mild} follows from \Cref{lem:stability_estimate_linear_pb}.
\end{proof}

\begin{thm}\label{lem:MF_syst_mild_reg}
Assume \Cref{ass:MF_drift_intro}, \Cref{ass:MF_coercivity_intro}, \Cref{ass:MF_mild_reg_intro} and that \eqref{eq:turnpike_suff_cond_mild_intro} holds. Then, there there exists a unique solution $(\mu^\infty,\varphi^\infty,\eta^\infty)\in\cP_1(\bbR^d)\times C^{0,1}(\bbR^d)\times\bbR$ to the ergodic mean field PDE system \eqref{eq:MF_PDE_ergodic} satisfying $\varphi^\infty(0)=0$ . Moreover, we have 
\begin{equation}\label{eq:gr_est_tpike_mild}
\|\varphi^\infty \|_{\kappa_b}\leq \const^\psi_x,
\end{equation}
and $\mu^\infty\in\cP_p(\bbR^d)$.
\end{thm}

\begin{proof}
\bei 
\item \underline{Step 1: Existence and uniqueness for the frozen ergodic system}
The proof of this claim is omitted since it is identical the the one given at \Cref{lem:MF_syst_high_reg}. As before, we denote such solution $(\eta^\infty[\mu],\psi^\infty[\mu],\nu^\infty[\mu])$.
\item \underline{Step 2: Existence and uniqueness for the mean field ergodic system \eqref{eq:MF_PDE_ergodic}} 
The proof is again similar to the one in \Cref{lem:MF_syst_high_reg}. The sole difference is that the bound between the gradients of two ergodic value functions for $\mu, \hat{\mu} \in \cP_1(\bbR^d)$ is given by
\begin{equation}
\|\nabla\psi^{\infty}[\mu]-\nabla\psi^{\infty}[\hat\mu]\|_{\infty} \leq 2\frac{\const^F_\mu}{C_{\bar{\kappa}}} W_{f_{\bar{\kappa}}}(\mu,\hat{\mu}) \int_0^Tq^{\bar{\kappa}}_s \De s + e^{-\lambda_{\bar{\kappa}}T} 2\const^\psi_x. 
\end{equation}
This yields, using the explicit calculations at \Cref{lem:boring_calculations_mild} below,
\begin{equation}
W_{f_{\bar{\kappa}}}(\nu^{\infty}[\mu],\nu^{\infty}[\hat\mu]) 
\leq 
W_{f_{\bar{\kappa}}}(\mu,\hat{\mu})\frac{4\const^F_\mu }{\sqrt{\pi}\rho^L_{uu}C^2_{\bar{\kappa}}\lambda^{3/2}_{\bar{\kappa}}\sigma_0},
\end{equation}
and we have again a contraction under \eqref{eq:suff_cond_tpik_mild_reg_in_proofs}.
\end{itemize}
\end{proof}
\noindent In order to prove the turnpike estimates, we use again the two auxiliary metrics defined in \eqref{eq:def_abs_met_high_mild}.
\begin{thm}\label{prop:abstract_suff_cond_tpike_mild}
Assume \Cref{ass:MF_drift_intro}, \Cref{ass:MF_coercivity_intro}, \Cref{ass:MF_mild_reg_intro}. 
Let 
$\hat{G}: \bbR^d \rightarrow \bbR^d$ be such that
$\| \hat{G}\|_{f_{\kappa_b}}  \leq 2 \const_x^\psi.$
\ben[label=(\roman*)] 
\item\label{item:grad_est_mild} \Cref{item:grad_est_high} from \Cref{prop:abstract_suff_cond_tpike_high} holds.
\item\label{item:stability_iterate_mild} 
Let $\bar{\kappa}$ be as in \eqref{eq:bar_kappa_intro}. 

Then for any $\mu_\cdot, \hat{\mu}_\cdot \in \Gamma$, $\lambda<\lambda_{\bar{\kappa}}$ we have
\begin{equation}\label{eq:iterate_stability_mild_1}
\overrightarrow{d_{\lambda,\bar{\kappa}}^{T}}(\nu^{T,\hat{G}}_\cdot[\mu_\cdot],\nu^{T,\hat{G}}_\cdot[\hat\mu_\cdot])\leq W_{f_{\bar{\kappa}}}(\mu_0,\hat{\mu}_0)+ \varepsilon(\lambda) \overrightarrow{d_{\lambda,\bar{\kappa}}^{T}}(\mu_\cdot,\hat\mu_\cdot),
\end{equation}
with
\begin{equation}\label{eq:epsilon_lambda_mild}
\varepsilon(\lambda)=\frac{2\const^F_{\mu}\sqrt{e}}{\sqrt{\pi}C^2_{\bar{\kappa}}\sigma_0}  \bigg( \frac{\sqrt{\lambda_{\bar{\kappa}}}}{\lambda_{\bar{\kappa}}^2-\lambda^2}+\frac{1}{\sqrt{\lambda_{\bar{\kappa}}}({\lambda_{\bar{\kappa}}-\lambda)}}\bigg).
\end{equation}
If $\mu_0 = \hat{\mu}_0$, then we also have
\begin{equation}\label{eq:iterate_stability_mild_T}
\overleftarrow{d_{\lambda,\bar{\kappa}}^{T}}(\nu^{T,\hat{G}}_\cdot[\mu_\cdot],\nu^{T,\hat{G}}_\cdot[\hat\mu_\cdot])\leq \varepsilon(\lambda) \overleftarrow{d_{\lambda,\bar{\kappa}}^{T}}(\mu_\cdot,\hat\mu_\cdot).
\end{equation}
\item\label{item:contraction_iterate_mild} If 
\begin{equation}\label{eq:suff_cond_tpik_mild_reg_in_proofs}
\const^F_{\mu}<\frac{\sqrt{\pi}}{4\sqrt{e}}\rho^L_{uu}\sigma_0C^2_{\bar{\kappa}}\lambda^{3/2}_{\bar{\kappa}},
\end{equation}
then for any $0<\lambda < \lambda^*$ with
\begin{equation} 
\lambda^*= - \frac{\alpha}{2\sqrt{\lambda_{\bar{\kappa}}}}+\frac{1}{2} \sqrt{\frac{\alpha^2}{4\lambda_{\bar{\kappa}}}-(2\lambda^{1/2}_{\bar{\kappa}}\alpha-\lambda^2_{\bar{\kappa}}) }, \quad \text{with} \quad \alpha= \frac{\sqrt{\pi}C^2_{\bar{\kappa}}\sigma_0\rho^{L}_{uu}}{2\const^F_{\mu}\sqrt{e}}
\end{equation}
we have $\varepsilon(\lambda)<1$. 
In particular, for any initial condition $\mu_0\in\cP_1(\bbR^d)$ the mean field PDE system \eqref{eq:MF_PDE} with terminal condition $\hat{G}$ and initial condition $\hat{\mu}_0$ has a unique solution $(\mu^{T,\hat{G}}_\cdot, \varphi^{T,\hat{G}}_\cdot)$ in $\Gamma_{\hat\mu_0,\const}$ with $\const$ as in \ref{item:grad_est_mild}. 
\item\label{item:mild_reg_tpike_est}
For any $\lambda < \lambda^*$ and any solution $(\mu^{T,G}_\cdot, \varphi^{T,G}_\cdot)$ to \eqref{eq:MF_PDE} with terminal condition $G:\cP_1(\bbR^d) \times \bbR^d \rightarrow \bbR$ and initial condition $\mu_0\in\cP_p(\bbR^d)$  we have the following turnpike estimates:
\ben[(a)]
\item\label{eq:mild_lipschitz_case}
If $\| G\|_{f_{\kappa_b}}<+\infty$ then, defining 
\bes\label{eq:mult_constants_mild_3}
\tau(G)= \frac{\log\big( (\|G\|_{f_{\kappa_b}}-\const^{\psi}_x)/\const^\psi_x\big)}{\lambda_{\kappa_b}}\vee 0 , \quad \kappa_G(r) = \kappa_b(r) -\frac{2(\const^{L(\cdot,0)}_u+ \max\{2\const^\psi_x, \|G \|_{f_{\kappa_b}}\})}{\rho^{L}_{uu} r},
\ees
we have that 
\begin{equation}\label{eq:tp_flow_Lip_mild}
W_{f_{\bar{\kappa}}}(\mu^{T,G}_t, \mu^{T,\hat{G}}_t) \leq \frac{1}{(1-\varepsilon(\lambda))}\Big(W_{f_{\bar\kappa}}(\mu_0,\hat{\mu}_0) e^{-\lambda t}  + \frac{e^{\lambda_{\bar\kappa}\tau(G) }}{2\lambda_{\kappa_G}C_{\kappa_G}\rho^{L}_{uu}}\|\hat{G} - G(\mu^{T,G}_T,\cdot) \|_{f_{\kappa_G}} e^{-\lambda(T-t)}\Big)
\end{equation}
holds for all $t\leq T$. Moreover, for all $t \leq T-\tau(G)$ we have
\begin{equation}\label{eq:tp_control_lip_mild}
\begin{aligned}
\|\varphi^{T,G}_t-\varphi^{T,\hat G}_t \|_\Lip
&\leq \frac{\const^F_{\mu}}{2C_{\bar\kappa}(1-\varepsilon(\lambda))} \Big(\frac{ W_{f_{\bar\kappa}}(\mu_0,\hat\mu_0)}{\sqrt{\pi} C_{\bar{\kappa}\sigma_0}}\left(\frac{1}{\sqrt{\lambda_{\bar\kappa}}} + \frac{\sqrt{\lambda_{\bar\kappa}}}{\lambda + \lambda_{\bar\kappa}} \right) e^{-\lambda t} \\
&+ \frac{2\sqrt{e}}{\sqrt{\pi}}\frac{\const^\psi_x e^{\lambda_{\bar\kappa}\tau(G)}}{\lambda_{\bar\kappa}C_{\bar{\kappa}}^2 \sigma_0\rho^L_{uu}}\left(\frac{1}{\sqrt{\lambda_{\bar\kappa}}} + \frac{\sqrt{\lambda_{\bar\kappa}}}{\lambda_{\bar\kappa} - \lambda} \right) e^{-\lambda(T-t)}\Big)+\frac{4\const^\psi_x}{C_{\bar{\kappa}}}e^{\lambda_{\bar\kappa}\tau(G)}e^{-\lambda_{\bar\kappa}(T-t)}
\end{aligned}
\end{equation}
\item\label{eq:mild_bdd_case} If  $\|G \|_{\infty} < + \infty$, if we define 
\footnote{ We have the explicit bound
\bes
\tau'(G)\leq  \max\left\{ \frac{1}{\lambda_{\kappa_b}}\log\const^G, \frac{1}{2\lambda_{\kappa_b}} \right\}> 0, \text{ with } \const^G = \frac{\frac{\sqrt{\lambda_{\kappa_b}e}}{\sqrt{\pi}C_{\kappa_b}\sigma_0}\| G\|_{\infty}  - \const^\psi_x}{\const^\psi_x}
\ees} \bes
\tau'(G)=\inf\{\tau>0: q^{\kappa_b}_\tau \|G\|_\infty\leq \const^\psi_x\},\quad T'_0:=T-\tau'(G),
\ees

then the turnpike estimate for marginal flows  \eqref{eq:tp_flow_Lip_mild} holds for $t\leq T'_0$ setting $\tau(G)=\tau'(G),\kappa_G=\bar\kappa$ and replacing $\|\hat{G} - G(\mu^{T,G}_T,\cdot) \|_{f_{\kappa_G}}$ with $4C_{\bar\kappa}^{-1}\const^\psi_x$. Moreover, for $t\in[T'_0,T]$ we have 
\bes
W_{f_{\bar\kappa}}(\mu^{T,G}_t,\mu^{T,\hat G}_t) \leq M_1(\|G\|_\infty)+\tilde{M}_1(2\const^\psi_x)
\ees
with $M_1(\|G\|_{\infty}),\,\tilde{M}_1(2\const^{\psi}_x)$ as in \eqref{eq:moment_bound}. Finally, the turnpike estimate \eqref{eq:tp_control_lip_mild} holds for all $t\leq T_0$ replacing $\tau(G)$ with $\tau'(G)$ .
\end{enumerate}
\end{enumerate}
\end{thm}
\begin{proof}

\begin{itemize}[wide]
\item \underline{Proof of \ref{item:grad_est_mild}:} The proof is identical to the one of \Cref{prop:abstract_suff_cond_tpike_high}-\ref{item:grad_est_high}.
\item \underline{Proof of \ref{item:stability_iterate_mild}:} 
The proof being almost identical to the one of \Cref{prop:abstract_suff_cond_tpike_high}-\ref{item:stability_iterate_mild}, we only point out the differences here.
In order to ease notation, we still write $\hat{\nu}= \nu^{T,\hat{G}}$, $\nu = \nu^{T,G}$, $\hat{\psi} = \psi^{T,\hat{G}}$, $\psi = \psi^{T,G}$ and $\mu_\cdot = \mu^{T,G}_\cdot$, $\hat{\mu} = \mu^{T,\hat{G}}_\cdot$.
For the forward estimate \eqref{eq:iterate_stability_mild_1}
the bound on the gradients becomes due to \Cref{lem:basics_MF_mild}-\ref{item:stab_MF_mild} and \Cref{lem:boring_calculations_mild}
\begin{equation}
\frac{\|\nabla\hat{\psi}_t[\mu_\cdot]-\nabla\hat{\psi}_t[\hat\mu_\cdot]\|_{\infty}}{\overrightarrow{d_{\lambda,\bar{\kappa}}^{T}}(\mu_\cdot,\hat\mu_\cdot)}
\leq \frac{2\const^F_{\mu}}{C_{\bar{\kappa}}} \int_t^Te^{-\lambda s} q^{\bar{\kappa}}_{s-t} \De s
\leq  \frac{2\const^F_{\mu}}{\sqrt{\pi}C^2_{\bar{\kappa}}\sigma_0}  \bigg( e^{-\lambda t}\Big(\frac{\sqrt{\lambda_{\bar{\kappa}}}}{\lambda_{\bar{\kappa}}+\lambda}+\frac{1}{\sqrt{\lambda_{\bar{\kappa}}}}\Big)\bigg)
\end{equation}
Using this bound, we obtain after similar calculations as in the proof of \Cref{prop:abstract_suff_cond_tpike_high}-\ref{item:stability_iterate_mild}
\begin{equation}
W_{f_{\bar{\kappa}}}(\hat{\nu}_t[\mu_\cdot],\hat{\nu}_t[\hat\mu_\cdot])\leq e^{-\lambda_{\bar{\kappa}}t}W_{f_{\bar{\kappa}}}(\mu_0,\hat{\mu}_0)+ 
\frac{2\const^F_{\mu}}{\sqrt{\pi}C^2_{\bar{\kappa}}\sigma_0\rho^L_{uu}}  \bigg( e^{-\lambda t}\Big(\frac{\sqrt{\lambda_{\bar{\kappa}}}}{\lambda_{\bar{\kappa}}^2-\lambda^2}+\frac{1}{\sqrt{\lambda_{\bar{\kappa}}}(\lambda_{\bar{\kappa}}-\lambda)}\Big)\bigg)\overrightarrow{d_{\lambda,\bar{\kappa}}^{T}}(\mu_\cdot,\hat\mu_\cdot).
\end{equation}
For the second bound in $\overleftarrow{d_{\lambda,\bar{\kappa}}^{T}}$ suppose now $\mu_0 = \hat{\mu}_0$. The bound on the difference of the gradients becomes now with the help of \Cref{lem:boring_calculations_mild}
\begin{align}
\frac{\|\nabla\hat{\psi}_t[\mu_\cdot]-\nabla\hat{\psi}_t[\hat\mu_\cdot]\|_{\infty}}{\overleftarrow{d_{\lambda,\bar{\kappa}}^{T}}(\mu_\cdot,\hat\mu_\cdot)}
\leq \frac{2\const^F_{\mu}}{C_{\bar{\kappa}}} \int_t^Te^{-\lambda (T-s)}q^{\bar{\kappa}}_{s-t} \De s 
\leq \frac{2\const^F_{\mu}\sqrt{e}}{\sqrt{\pi}C^2_{\bar{\kappa}}\sigma_0}  \bigg( e^{-\lambda (T-t)}\Big(\frac{\sqrt{\lambda_{\bar{\kappa}}}}{\lambda_{\bar{\kappa}}-\lambda}+\frac{1}{\sqrt{\lambda_{\bar{\kappa}}}}\Big)\bigg).
\end{align}
Repeating the same calculations as before, we arrive at
\begin{equation}
W_{f_{\bar{\kappa}}}(\hat{\nu}_t[\mu_\cdot],\hat{\nu}_t[\hat\mu_\cdot])\leq  \frac{2\const^F_{\mu}\sqrt{e}}{\sqrt{\pi}C^2_{\bar{\kappa}}\sigma_0}  \bigg( e^{-\lambda (T-t)}\Big(\frac{\sqrt{\lambda_{\bar{\kappa}}}}{\lambda_{\bar{\kappa}}^2-\lambda^2}+\frac{1}{\sqrt{\lambda_{\bar{\kappa}}}({\lambda_{\bar{\kappa}}+\lambda)}}\Big)\bigg)\overleftarrow{d_{\lambda,\bar{\kappa}}^{T}}(\mu_\cdot,\hat\mu_\cdot)
\end{equation}
\item \underline{Proof of \ref{item:contraction_iterate_mild}:} 
The proof is identical to the the proof of item \ref{item:contraction_iterate_high} in \Cref{prop:abstract_suff_cond_tpike_high}.
\item \underline{Proof of \ref{item:mild_reg_tpike_est}:} 
For the turnpike estimates, note that thanks to the estimates we have just established in \ref{item:stability_iterate_mild}, \Cref{lem:abs_turnpike_est} gives for $\sfd = W_{f_{\bar{\kappa}}}$ and $\upsilon^\lambda(t) = e^{-\lambda t}$
\begin{equation}
W_{f_{\bar{\kappa}}}(\mu^{T,G}_t, \mu^{T,\hat{G}}_t) \leq \frac{e^{-\lambda t}}{1 - \varepsilon(\lambda)} W_{f_{\bar{\kappa}}}(\mu_0, \hat{\mu}_0) + 
\frac{e^{-\lambda(T-t)}}{1 - \varepsilon(\lambda)} \overleftarrow{d_{\lambda,\bar{\kappa}}^{T}}(\hat{\nu}_\cdot[\mu^{T,G}_\cdot],\nu_\cdot[\mu^{T,G}_\cdot]) \eqsp.
\end{equation}
We hence conclude the turnpike estimates by properly bounding $\overleftarrow{d_{\lambda,\bar{\kappa}}^{T}}(\hat{\nu}_\cdot[\mu^{T,G}_\cdot],\nu_\cdot[\mu^{T,G}_\cdot])$ as follows.
\begin{itemize}[wide]
\item Suppose first $\| G\|_{f_{\kappa_b}} := \sup_{\mu \in \cP_1(\bbR^d)}\| G(\mu,\cdot)\|_{f_{\kappa_b}}  < +\infty$.
In this case the proof of \eqref{eq:tp_flow_Lip_mild} is identical to the one of \eqref{eq:tp_flow_Lip_high}.

In order to obtain the turnpike estimate for the value functions, observe that setting $T_0 = T - \tau(G)$, we have $\|\varphi^{T,G}_t \|_{f_{\kappa_b}} \leq 2 \const^\psi_x$ for all $t \leq T_0$. Hence, invoking the dynamic programming principle and plugging the slightly better bound in \eqref{eq:tpike_better_G_Lip_small} into \Cref{lem:basics_MF_mild}-\ref{item:stab_MF_mild} gives after some basic calculations with the help of \Cref{lem:boring_calculations_mild} the turnpike estimate \eqref{eq:tp_control_lip_mild}.

\item The proof for $\| G\|_{\infty}  < +\infty$ follows the lines of proof of \Cref{prop:abstract_suff_cond_tpike_high}-\ref{item:high_reg_tpike_est}-\ref{item:hig_bdd_case} by using \eqref{eq:tp_flow_Lip_mild}, \eqref{eq:tp_control_lip_mild} instead of \eqref{eq:tp_flow_Lip_high}, \eqref{eq:tp_control_lip_high}.
\end{itemize}
\end{itemize}
\end{proof}

\begin{lemma}\label{lem:boring_calculations_mild}
Let $\lambda<\lambda_{\bar\kappa}$ and $t\leq T$. If we denote by $\Phi$ the cumulative distribution function of the standard Gaussian distribution, we have the bounds
\begin{equation}
\begin{split}
\int_t^Te^{-\lambda s}q^{\bar{\kappa}}_{s-t} \De s &\leq \frac{e^{-\lambda t}}{C_{\bar{\kappa}}\sigma_0}\Big[ \sqrt{\frac{2}{\lambda}}\big( \Phi(\sqrt{\lambda/\lambda_{\bar{\kappa}}})-\Phi(0)\big)+\frac{\sqrt{\lambda_{\bar{\kappa}}}}{\sqrt{\pi}(\lambda+\lambda_{\bar{\kappa}})}e^{-\lambda/2\lambda_{\bar{\kappa}}}\Big] \\
&\leq \frac{e^{-\lambda t}}{\sqrt{\pi}C_{\bar{\kappa}}\sigma_0} \left[\frac{1}{\sqrt{\lambda_{\bar{\kappa}}}} +\frac{\sqrt{\lambda_{\bar{\kappa}}}}{\lambda+\lambda_{\bar{\kappa}}}\right]
\end{split}
\end{equation}

\begin{equation}
\begin{split}
\int_t^Te^{-\lambda (T-s)}q^{\bar{\kappa}}_{s-t} \De s &\leq \frac{e^{-\lambda (T-t)}e^{\frac{\lambda}{2\lambda_{\bar{\kappa}}}}}{\sqrt{\pi}C_{\bar{\kappa}}\sigma_0} \left[\frac{1}{\sqrt{\lambda_{\bar{\kappa}}}} +\frac{\sqrt{\lambda_{\bar{\kappa}}}}{\lambda_{\bar{\kappa}}-\lambda}\right]
\end{split}
\end{equation}
\end{lemma}

\subsection{Low regularity}\label{sec:MF_low_reg}
In this section we redefine
\bes\label{c_psi_x_low}
\const^\psi_x:=  \frac{\const^L + \const^F}{\sqrt{\pi \lambda_{\kappa_{b}}}C_{\kappa_{b}}\sigma_0} \eqsp .
\ees

In the low regularity regime, we need to work with the total variation distance instead of the Wasserstein distance to show contraction of the fixed-point iterations. Therefore, we redefine the space of flows
\begin{equation}
\Gamma:= \left\{ \mu:[0,T] \rightarrow \cP_1(\bbR^d): 
\mu_0\in\cP_p(\bbR^d), \quad 
\sup_{s\neq t}\frac{\|\mu_s-\mu_t\|_{\TV}}{|t-s|^{1/2}}<+\infty \quad \forall s,t\in[\varepsilon,T], \, \forall \varepsilon >0\right\}
\end{equation}
and the abstract metrics for the turnpike estimate
\begin{align}\label{eq:def_metric_flow_TV}
\overrightarrow{d_{\lambda, \kappa}^{T}}(\fl,\hfl) := \sup_{0 \leq s\leq T}(q^{\kappa,\lambda}_s)^{-1}\|\mu_s-\hat\mu_s\|_{\TV}, \quad \overleftarrow{d_{\lambda,\kappa}^{T}}(\fl,\hfl) := \sup_{0 \leq s\leq T}e^{\lambda (T-t)}\|\mu_s-\hat\mu_s\|_{\TV}.
\end{align}
Let us summarize again the basic facts we deduce from \Cref{sec:fin_dim_cont} in the low regularity regime.
\begin{lemma}\label{lem:basics_MF_low}
Assume \Cref{ass:MF_drift_intro}, \Cref{ass:MF_coercivity_intro}, \Cref{ass:low_reg_in_proofs}. Let $\mu_\cdot \in \Gamma$, $G: \cP_1(\bbR^d) \times \bbR^d \rightarrow \bbR$ locally Lipschitz continuous in the second variable such that $G(\mu,\cdot)$ is of linear growth for all $\mu\in\cP_1(\bbR^d)$.
The following holds.
\begin{enumerate}[label=(\roman*)] 
\item\label{item:HJB_ok_low} There is a unique solution to \eqref{eq:linearized_PDE_system}, denote it by $(\psi^{T,G}_{\cdot}[\mu_\cdot],\nu^{T,G}_{\cdot}[\mu_\cdot])$. Moreover, setting $b_s=b$, $\ell_s(x,u) = L(x,u) + F(\mu_s,x)$, $g(x)=G(\mu_T,x)$, $\psi^{T,G}_{\cdot}[\mu_\cdot]$ is the value function of \eqref{eq:classical_control_problem} and the marginal laws of its optimally controlled dynamics are given by $\nu^{T,G}_{\cdot}[\mu_\cdot] \in \Gamma$.
\item\label{item:grad_est_MF_low} The following quantitative bounds hold true
\begin{align}
\|\psi^{T,G}_{t}[\mu_\cdot]\|_{f_{\kappa_{b}}} &\leq 2\const^\psi_x\left(2 - \sqrt{e} e^{-\lambda_{\kappa_{b}}(T -t) } \right) +\min\left\{ \| G(\mu_T,\cdot)\|_{f_{\kappa_{b}}}e^{-\lambda_{\kappa_{b}}(T-t)},\|G(\mu_T,\cdot)\|_{\infty} q^{\kappa_b}_{T-t}  \right\} \\
\|w(\cdot,\nabla\psi^{T,G}_{t}[\mu_\cdot](\cdot))\|_\infty &\leq \frac{\|\psi^{T,G}_t[\mu_\cdot]\|_{f_{\kappa_b}}+\const^{L(\cdot,0)}}{\rho^{L}_{uu}},
\end{align}
where
\begin{equation}
w(x,p) := \argmin_{u \in \bbR^d} \left\{ L(x,u)+(b(x)+u)\cdot p\right\}  \eqsp.
\end{equation}
\item\label{item:stab_MF_low} Suppose now that $G$ is Lipschitz in the space variable and that $\hat{G}: \cP_1(\bbR^d) \times \bbR^d \rightarrow \bbR$ is another terminal cost satisfying the same assumptions as $G$. Given another marginal flow $\hat{\mu}_\cdot \in \Gamma$, we have for $0 \leq t_0 \leq t \leq T$
\begin{align}
\|\psi^{T,G}_{t}[\mu_\cdot]-\psi^{T,\hat{G}}_{t}[\hat{\mu}_\cdot]\|_{f_{\tilde\kappa}}
&\leq \const^{\delta\psi_t}_x+ \|G(\mu_T,\cdot) - \hat{G}(\hat{\mu}_T,\cdot) \|_{f_{\tilde\kappa}}e^{-\lambda_{\tilde\kappa}(T-t)} \\
\|\nu^{T,G}_{\cdot}[\mu_\cdot]-\nu^{T,\hat{G}}_{\cdot}[\hat{\mu}_\cdot]\|_{\TV}
&\leq q^{\tilde\kappa}_{t-t_0}e^{-\lambda_{\tilde\kappa}t_0} W_{f_{\tilde\kappa}}(\mu_0, \hat{\mu}_0)\\
&+ q^{\tilde\kappa}_{t-t_0}\int_{0}^{t_0} e^{-\lambda_{\tilde\kappa}(t_0 - s)} \frac{\const^{\delta \psi_s}_x}{\rho^L_{uu}}\De s+ \frac{1}{\sqrt{2}\rho^L_{uu}} \left(\int_{t_0}^{t}\left(\const^{\delta \psi_s}_x\right)^2 \De s \right)^{\frac{1}{2}} \\ &+ \left(\frac{e^{-\lambda_{\tilde\kappa}(T-t)}}{2\sqrt{\lambda_{\tilde\kappa}}} + \frac{e^{-\lambda_{\tilde\kappa}T}}{2\lambda_{\tilde\kappa}}(e^{\lambda_{\tilde\kappa}t_0} - e^{-\lambda_{\tilde\kappa}t_0}) \right)\|G(\mu_T,\cdot) - \hat{G}(\hat{\mu}_T,\cdot) \|_{f_{\tilde\kappa}},
\end{align}
where 
\bes
\const^{\delta \psi_s}_x = 2\const^F_{\mu,\TV}\int_t^{T} \|\mu_u - \hat{\mu}_u \|_{\TV}q^{\tilde\kappa}_{u-t}\De u,\ees
and $\tilde{\kappa} \in \msk$ is any profile satisfying and $\tilde{\kappa} \in \msk$ is any profile satisfying \eqref{eq:tilde_kappa_cond}. In particular, if $\max\{\|\hat G(\mu_T,\cdot)\|_{f_{\kappa_\beta}},\| G(\mu_T,\cdot)\|_{f_{\kappa_\beta}}\}\leq 4\const^\psi_x$ we can choose $\tilde{\kappa}=\bar\kappa$, with $\bar\kappa$ as in \eqref{eq:kappa_low}.
\end{enumerate}
\end{lemma}
\begin{proof}
\Cref{item:HJB_ok_low} is a consequence of \Cref{prop:properties} in combination with \Cref{prop:Holder_in_time}, where we refer to \Cref{cor:grad_est_g_bdd} for the bounded case. In the same way, \ref{item:grad_est_MF_low} follows by an application of \Cref{lemma:gradient_estimate_linearized problem} and \Cref{cor:grad_est_g_bdd} and noting that
\begin{align}
\int_t^T q^{\kappa_{b}}_{s-t} \De s 
= \frac{1}{\sqrt{\pi \lambda_{\kappa_{b}}}C_{\kappa_{b}}\sigma_0} \left(1 + \sqrt{e}(e^{-1/2} - e^{-\lambda_{\kappa_{b}}(T -t) }) \right) = \frac{2}{\sqrt{\pi\lambda_{\kappa_{b}}} C_{\kappa_{b}}\sigma_0} - \frac{\sqrt{e}}{\sqrt{\pi \lambda_{\kappa_{b}}}C_{\kappa_{b}}\sigma_0} e^{-\lambda_{\kappa_{b}}(T -t) } .
\end{align}
Finally, \ref{item:stab_MF_low} follows from \Cref{lem:stability_estimate_linear_pb}.
\end{proof}
We now proceed to establish a well posedness result for the ergodic system.

\begin{thm}\label{lem:erg_MF_syst_low_reg}
Assume \Cref{ass:MF_drift_intro}, \Cref{ass:MF_coercivity_intro}, \Cref{ass:low_reg_in_proofs} and that
\begin{equation}\label{eq:suff_ergodic_sol_low}
\left(\frac{1}{\sqrt{\pi}C_{\bar{\kappa}}\sigma_0}  + \frac{1}{2} \right)\frac{4\const^F_{\mu,\TV} }{\sqrt{\pi}\lambda_{\bar{\kappa}}\rho^L_{uu} C_{\bar{\kappa}}\sigma_0} < 1
\end{equation}

Then, there there exists a unique solution $(\mu^\infty,\varphi^\infty,\eta^\infty) \in \cP_1(\bbR^d) \times C^{0,1}(\bbR^d) \times \bbR$ to the ergodic mean field PDE system \eqref{eq:MF_PDE_ergodic} satisfying $\varphi^\infty(0)=0$ . Moreover, we have 
\begin{equation}\label{eq:gr_est_tpike_low}
\|\varphi^\infty \|_{f_{\kappa_b}}\leq 4\const^\psi_x.
\end{equation}
with $\const^\psi_x$ as in \eqref{c_psi_x_low}.
\end{thm}

\begin{proof}
\bei[wide]
\item \underline{Step 1:Existence and uniqueness for the frozen ergodic system}

\noindent In this step, we show that for any $\mu \in \cP_1(\bbR^d)$ we can find a unique solution $(\eta^{\infty}[\mu],\psi^{\infty}[\mu],\nu^{\infty}[\mu])$ to the ergodic system
\begin{equation}\label{eq:linearized_ergodic_PDE_system_low}
\bec
-\eta^\infty + \frac{1}{2}\tr\left(\sigma(x)^{\top}\sigma(x)\nabla^2\psi(x)\right) + H(x,\nabla\psi(x)) + F(\mu,x)=0, \quad \\
 -\frac{1}{2}\tr(\nabla^2(\sigma^{\top}\sigma\nu)(x)) + \nabla\cdot( \partial_p H(x,\nabla\psi_s(x))\nu(x))=0.
\eec
\end{equation}
Unlike the proofs of \Cref{lem:MF_syst_high_reg} and \Cref{lem:MF_syst_mild_reg} building on uniform in time Lipschitz estimates to  apply directly Banach's fixed point theorem to a suitable invariant set,  the Lipschitz estimates from \Cref{lem:basics_MF_low}-\ref{item:grad_est_MF_low}  are not strong enough to provide with such invariant sets directly. This is why we have to modify the proof here a bit.
Fix $T>0$, $\mu \in \cP_1(\bbR^d)$ and define
\begin{align}\label{def:iterations_linearized_tpike_low}
\Phi^{\mu}_T: C^{0,1}(\bbR^d) \longrightarrow C^{0,1}(\bbR^d), \qquad &g\mapsto \psi^{T,g}_0[\mu], \\
\overline{\Phi}^{\mu}_T:  C^{0,1}_0(\bbR^d) \longrightarrow C^{0,1}_0(\bbR^d), \qquad &g\mapsto \psi^{T,g}_0[\mu] - \psi^{T,g}_0[\mu](0),
\end{align}
where with a slight abuse of notation, for a given $\mu$ we write $ \psi^{T,g}_t[\mu]$ instead of $\psi^{T,g}_t[\mu_\cdot]$, when $\mu_\cdot$ is the flow constantly equal to $\mu$. We recall that $C^{0,1}(\bbR^d)$ is the space of Lipschitz continuous functions and $C^{0,1}_0(\bbR^d)$ its subspace of functions satisfying $g(0) = 0$. 

Let $g \in C^{0,1}_0(\bbR^d)$ with $\const^g_{x} := \| g\|_{\Lip}$. Define the sequence $(\phi_n)_{n \in \N}$ by iteratively applying $\overline{\Phi}^{\mu}_T$ to $g$, \ie,
\begin{equation}
\phi_0 := g, \quad \phi_{n+1} = \overline{\Phi}^{\mu}_T(\phi_n)
\end{equation}
From the dynamic programming principle we have
\begin{equation}
\phi_n = \psi^{nT,g}_0[\mu] - \psi^{nT,g}_0[\mu](0).
\end{equation}
This implies thanks to \Cref{lem:basics_MF_low}-\ref{item:grad_est_MF_low} that
\begin{align}
\| \phi_n \|_{f_{\kappa_b}} &\leq \frac{4(\const^L + \const^F)}{\sqrt{\pi \lambda_{\kappa_{b}}}C_{\kappa_{b}}\sigma_0} - \frac{2\sqrt{e}(\const^L + \const^F)}{\sqrt{\pi\lambda_{\kappa_{b}}} C_{\kappa_{b}}\sigma_0} e^{-\lambda_{\kappa_{b}}nT } + \frac{1}{C_{\kappa_b}}\const^g_x e^{-\lambda_{\kappa_b}nT}\label{eq:grad_est_erg_low} \\
&\leq \frac{4(\const^L + \const^F)}{\sqrt{\pi \lambda_{\kappa_{b}}}C_{\kappa_{b}}\sigma_0} + \max\left\{0, \frac{1}{C_{\kappa_b}}\const^g_x - \frac{2\sqrt{e}(\const^L + \const^F)}{\sqrt{\pi\lambda_{\kappa_{b}}} C_{\kappa_{b}}\sigma_0}\right\} =: \const^{\psi[\mu],g}_x
\end{align}
By Ascoli-Arzelà's theorem there exists $\psi^{\infty,T,g}[\mu]$ and a subsequence $(n_k)_{k\in \N}$ such that
\begin{equation}
\phi_{n_k} \longrightarrow \psi^{\infty,T,g}[\mu],
\end{equation}
locally uniformly, and this convergence can be strenghtened to be w.r.t. $\| \cdot \|_{\Lip}$ by standard arguments.
To see that $\psi^{\infty,T,g}[\mu]$ does not depend neither on the subsequence or $g$, let $\hat{g} \in C^{0,1}_0(\bbR^d)$ with $\const^{\hat{g}}_x := \| \hat{g}\|_{\Lip}$. Defining analogously $\hat{\phi}_n$ for $\hat{g}$ we obtain in the same way
\begin{align}
\| \hat{\phi}_n \|_{f_{\kappa_b}} \leq \frac{4(\const^L + \const^F)}{\sqrt{\pi \lambda_{\kappa_{b}}}C_{\kappa_{b}}\sigma_0} + \max\left\{0, \frac{1}{C_{\kappa}}\const^{\hat{g}}_x - \frac{2\sqrt{e}(\const^L + \const^F)}{\sqrt{\pi\lambda_{\kappa_{b}}} C_{\kappa_{b}}\sigma_0}\right\} =: \const^{\psi[\mu],\hat{g}}_x
\end{align}
Now applying \Cref{lem:basics_MF_low}-\ref{item:stab_MF_low} for $\tilde{\kappa}(r) = \kappa_b(r) - 2\max\{ \const^{\psi[\mu],g}_x, \const^{\psi[\mu],\hat{g}}_x \}/(\rho^L_{uu}r)$, we obtain
\begin{equation}\label{eq:contr_est_seq_tp_low}
\|\phi_n - \hat{\phi}_n \|_{f_{\tilde{\kappa}}} \leq e^{-\lambda_{\tilde{\kappa}}(nT)} \|g- \hat{g} \|_{f_{\tilde{\kappa}}}.
\end{equation}
This implies that the limit does not depend on $g$ and convergence of the full sequence. We can hence set $\psi^{\infty,T}[\mu] := \psi^{\infty,T,g}[\mu]$. To see that it is a fixed point, observe that we have continuity of $\overline{\Phi}^{\mu}_T$ thanks to \eqref{eq:contr_est_seq_tp_low} and hence
\begin{equation}
\overline{\Phi}^{\mu}_T(\psi^{\infty,T}[\mu]) = \lim_{n \rightarrow \infty} \overline{\Phi}^{\mu}_T(\phi_n) = \lim_{n \rightarrow \infty} \phi_{n+1} = \psi^{\infty,T}[\mu].
\end{equation}
Finally note that $\psi^{\infty,T}[\mu]$ is also the unique fixed point in $C^{0,1}_0(\bbR^d)$, again thanks to the contraction estimate \eqref{eq:contr_est_seq_tp_low}, and we have passing $T \rightarrow \infty$ in \eqref{eq:grad_est_erg_low}
\begin{equation}\label{eq:erg_fnc_mu_Lip_est_low}
\| \psi^{\infty,T}[\mu] \|_{f_{\kappa_b}} \leq \frac{4(\const^L + \const^F)}{\sqrt{\pi \lambda_{\kappa_{b}}}C_{\kappa_{b}}\sigma_0} = 4\const^\psi_x.
\end{equation}
Set $\eta^{\infty,T}=\psi^{\infty,T}[\mu](0)$. 
The fact that setting $\psi^{\infty}[\mu]:=\psi^{\infty,1}[\mu]$, ${\eta^{\infty}[\mu]:=\eta^{\infty,1}[\mu]}$ is well-defined, existence of a unique invariant measure $\nu^\infty[\mu]$ and that $(\eta^{\infty}[\mu],\nu^{\infty}[\mu],\psi^{\infty}[\mu])$ form a solution to \eqref{eq:linearized_ergodic_PDE_system_low} is analogous to the proof in \Cref{lem:MF_syst_high_reg}.

\item \underline{Step 2: Existence and uniqueness for the mean field ergodic system \eqref{eq:MF_PDE_ergodic}} 

\noindent Let $\mu,\hat\mu \in \cP_1(\bbR^d)$. Thanks to the fixed point properties $\psi^{\infty}[\mu]=\bar\Phi^\mu_T(\psi^{\infty}[\mu])$ and ${\psi^{\infty}[\hat\mu]=\bar\Phi^{\hat\mu}_T(\psi^{\infty}[\hat\mu])}$ and the fact that
\begin{equation}
\| \psi^{\infty}[\mu]\|_{f_{\kappa_b}},\, \| \psi^{\infty}[\hat\mu] \|_{f_{\kappa_b}} \leq 
4\const^\psi_x
\end{equation}
we can apply  \Cref{lem:basics_MF_low}-\ref{item:stab_MF_low} with $\tilde{\kappa}(r) = \bar{\kappa}$ from \eqref{eq:kappa_low} to get
\begin{equation}
\|\nabla\psi^{\infty}[\mu]-\nabla\psi^{\infty}[\hat\mu]\|_{\infty} \leq 2\const^F_{\mu,\TV} \|\mu - \hat{\mu} \|_{\TV} \int_0^Tq^{\bar{\kappa}}_s \De s + e^{-\lambda_{\bar{\kappa}}T} \frac{8\const^\psi_x}{C_{\bar{\kappa}}}. 
\end{equation}
Letting $T\rightarrow+\infty$ and using \Cref{lem:boring_calculations_mild} to obtain
\begin{equation}
\int_0^{+\infty}q^{\bar{\kappa} 
}_s \De s =\frac{2}{\sqrt{\pi \lambda_{\bar{\kappa}}}C_{\bar{\kappa}}\sigma_0}, 
\end{equation}
we arrive at 
\begin{equation}\label{eq:optimal_ergodic_policy_bound}
\|w(\cdot,\nabla\psi^{\infty}[\mu](\cdot))-w(\cdot,\nabla\psi^{\infty}[\hat\mu](\cdot))\|_{\infty}\leq \frac{4\const^F_{\mu,\TV} \|\mu - \hat{\mu} \|_{\TV}}{\sqrt{\pi\lambda_{\bar{\kappa}}}\rho^L_{uu}C_{\bar{\kappa}}\sigma_0} =: \const^{\delta w}_{\TV}
\end{equation}

Now let $(X_s,\hat{X}_s)$ be an approximate coupling by reflection \eqref{eq:delta_coup_by_ref} for the optimal dynamics associated to marginal flows $\mu$ and $\hat{\mu}$, starting in the invariant laws $\nu^{\infty}[\mu]$ and $\nu^{\infty}[\hat{\mu}]$. Then by \Cref{prop:contr_delta_coup}-\ref{item:TV_contr_delta_coup} for $0 \leq t_0 < t$ (since $\cL(X_t) = \nu^{\infty}[\mu]$, $\cL(\hat{X}_t) = \nu^{\infty}[\hat{\mu}]$)
\begin{equation}
\|\nu^{\infty}[\mu] - \nu^{\infty}[\hat{\mu}] \|_{\TV} \leq q^{\bar{\kappa}}_{t-t_0} \bbE[f_{\bar{\kappa}}(|X_{t_0} - \hat{X}_{t_0}|)] + \frac{1}{\sqrt{2}} \left(\int_{t_0}^{t}\left(\const^{\delta w}_{\TV}\right)^2 \De s \right)^{\frac{1}{2}} 
\end{equation}
Now using \Cref{prop:contr_delta_coup}-\ref{item:W_1_contr_delta_coup}
\begin{align}
\bbE[f_{\bar{\kappa}}(|X_{t_0} - \hat{X}_{t_0}|)]
\leq e^{-\lambda_{\bar{\kappa}}t_0} W_{f_{\bar{\kappa}}}(\mu_0, \hat{\mu}_0) + \int_{0}^{t_0} e^{-\lambda_{\bar{\kappa}}(t_0 - s)} \const^{\Delta w}_{\TV} \De s
\end{align}
For $t \geq \frac{1}{\lambda_{\bar\kappa}}$, choosing $t_0 = t - \frac{1}{2\lambda_{\bar\kappa}}$ we get
\begin{align}
\|\nu^{\infty}[\mu] - \nu^{\infty}[\hat{\mu}] \|_{\TV} \leq q^{\bar{\kappa}}_t W_{f_{\bar{\kappa}}}(\nu^{\infty}[\mu], \nu^{\infty}[\hat{\mu}]) + \frac{1}{\sqrt{\pi\lambda_{\kappa}}C_{\kappa}\sigma_0} \const^{\delta w}_{\TV} + \frac{1}{2\sqrt{\lambda_{\kappa}}} \const^{\delta w}_{\TV}
\end{align}
Letting $t \rightarrow \infty$ we arrive at
\begin{equation}
\|\nu^{\infty}[\mu] - \nu^{\infty}[\hat{\mu}] \|_{\TV} \leq \left(\frac{1}{\sqrt{\pi\lambda_{\kappa}}C_{\kappa}\sigma_0}  + \frac{1}{2\sqrt{\lambda_{\kappa}}} \right)\frac{4\const^F_{\mu,\TV} }{\sqrt{\pi\lambda_{\bar{\kappa}}}\rho^L_{uu} C_{\bar{\kappa}}\sigma_0} \|\mu - \hat{\mu} \|_{\TV}
\end{equation}
So condition \eqref{eq:suff_ergodic_sol_low} implies that  the map
\begin{equation}
    \cP_1(\bbR^d)\ni\mu \mapsto \nu^{\infty}[\mu]
\end{equation}
is a contraction and we conclude with Banach's fixed point theorem existence and uniqueness of a fixed point $\mu^{\infty}$ in $\cP_1(\bbR^d)$. It then follows that $(\eta^{\infty}[\mu^\infty],\psi^{\infty}[\mu^{\infty}],\mu^{\infty})$ is the unique solution to the ergodic mean field PDE system \eqref{eq:MF_PDE_ergodic}. 
\eei 
\end{proof}

\begin{thm}\label{prop:abstract_suff_cond_tpike_low}
Assume \Cref{ass:MF_drift_intro}, \Cref{ass:MF_coercivity_intro}, \Cref{ass:low_reg_in_proofs}. 
Let $\hat{G}: \bbR^d \rightarrow \bbR^d$ be such that $\| \hat{G}\|_{f_{\kappa_b}}  \leq 4 \const^\psi_x$.
\ben[label=(\roman*)] 
\item\label{item:grad_est_low} For any $\mu_\cdot, \hat{\mu}_\cdot \in\Gamma$ and $T>0$ the estimate
\begin{equation}\label{eq:gest_low_last}
\|\nabla\psi^{T,\hat{G}}_t[\hat{\mu}_\cdot] \|_\infty \leq (8 - 2\sqrt{e}) \const_x^\psi 
\end{equation}
hold for all $t\in[0,T].$  Moreover, for any $\hat\mu_\cdot\in\Gamma$ such that $\hat\mu_0\in\cP_p(\bbR^d)$, we have $\nu^{T,\hat G}[\hat\mu_\cdot]\in \Gamma_{\hat\mu_0,\const_{\cdot}}\subseteq\Gamma$ where
\bes
\Gamma_{\hat\mu_0,\const_{\cdot}} =\{\tilde\mu_\cdot: \tilde\mu_0=\hat\mu_0,\eqsp\sup_{\varepsilon\leq s<t\leq T}\frac{W_1(\tilde\mu_s,\tilde\mu_t)}{|t-s|^{1/2}}\leq \const_\varepsilon \eqsp\forall \varepsilon\leq T \}
\ees
and the function $\const_\cdot$ depends only on $\Sigma,\int|x|^p\mu_0(\De x),\kappa_b,\const^\psi_x$ and $T$.
\item\label{item:stability_iterate_low} 
Let $\bar\kappa$ be given by \eqref{eq:kappa_low}.
Then for any $\mu_\cdot, \hat{\mu}_\cdot \in \Gamma$, setting $\lambda=\lambda_{\bar{\kappa}}/2$ we have
\begin{equation}\label{eq:iterate_stability_low_1}
\overrightarrow{d_{\lambda,\bar{\kappa}}^{T}}(\nu^{T,\hat{G}}_\cdot[\mu_\cdot],\nu^{T,\hat{G}}_\cdot[\hat\mu_\cdot])\leq W_{f_{\bar{\kappa}}}(\mu_0,\hat{\mu}_0)+ \varepsilon(\lambda) \overrightarrow{d_{\lambda,\bar{\kappa}}^{T}}(\mu_\cdot,\hat\mu_\cdot)
\end{equation}
with
\begin{equation}\label{eq:epsilon_lambda_low}
\varepsilon(\lambda)= \frac{\sqrt{e}\const^F_{\mu,\TV}}{\sqrt{\pi} \rho^{L}_{uu}C_{\bar{\kappa}}\sigma_0\lambda_{\bar{\kappa}}} \max \left\{9, \left(4 + \frac{7}{\sqrt{\pi}C_{\kappa}\sigma_0} \right)\right\}
\end{equation}
If $\mu_0 = \hat{\mu}_0$, then we also have
\begin{equation}\label{eq:iterate_stability_low_T}
\overleftarrow{d_{\lambda,\bar{\kappa}}^{T}}(\nu^{T,\hat{G}}_\cdot[\mu_\cdot],\nu^{T,\hat{G}}_\cdot[\hat\mu_\cdot])\leq \varepsilon(\lambda) \overleftarrow{d_{\lambda,\bar{\kappa}}^{T}}(\mu_\cdot,\hat\mu_\cdot).
\end{equation}
\item\label{item:contraction_iterate_low} Let again $\lambda = \lambda_{\bar{\kappa}}/2$. If
\begin{equation}\label{eq:suff_cond_tpik_low_reg_in_proofs}
\const^F_{\mu,\TV}< \frac{\rho^{L}_{uu}\sqrt{\pi} C_{\bar{\kappa}}\sigma_0\lambda_{\bar{\kappa}}}{\sqrt{e} \max \left\{9, \left(4 + \frac{7}{\sqrt{\pi}C_{\kappa}\sigma_0} \right)\right\}},
\end{equation}
then 
we have $ \varepsilon(\lambda)<1$. 
In particular, for any initial condition $\mu_0\in\cP_1(\bbR^d)$ the mean field PDE system \eqref{eq:MF_PDE} with terminal condition $\hat{G}$ and initial condition $\hat{\mu}_0$ has a unique solution $(\mu^{T,\hat{G}}_\cdot, \varphi^{T,\hat{G}}_\cdot)$ in $\Gamma_{\hat\mu_0,\const_\cdot}$  with $\const$ as in \ref{item:grad_est_low}. 

\item\label{item:turnpike_low} Let again $\lambda = \lambda_{\bar{\kappa}}/2$. If $(\mu^{T,G}_\cdot, \varphi^{T,G}_\cdot)$ is any solution  to \eqref{eq:MF_PDE} with terminal condition $G:\cP_1(\bbR^d) \times \bbR^d \rightarrow \bbR$ and initial condition $\mu_0\in\cP_p(\bbR^d)$ the following holds.

\ben[(a)]
\item\label{eq:low_lipschitz_case} If $\| G\|_{f_{\kappa_b}}<+\infty$ then, defining 
\bes\label{eq:mult_constants_low_3}
\tau(G)= \frac{1}{\lambda_{\kappa_b}}\log \left(\frac{\| G \|_{f_{\kappa_b}} - 2\sqrt{e}\const^\psi_x}{(4-2\sqrt{e})\const^\psi_x} \right) \ke{\vee 0}, \eqsp \kappa_G(r) = \kappa_b(r) -\frac{2(\const^{L(\cdot,0)}_u+ \max\{(8-2\sqrt{e})\const^\psi_x, \|G \|_{f_{\kappa_b}}\})}{\rho^{L}_{uu} r}
\ees
we have that 
\begin{equation}\label{eq:tp_flow_Lip_low}
\|\mu^{T,G}_t - \mu^{T,\hat{G}}_t\|_{\TV} \leq \frac{1}{(1-\varepsilon(\lambda))}\Big( W_{f_{\bar{\kappa}}}(\mu_0, \hat{\mu}_0)q^{\bar{\kappa},\lambda}_t +
\frac{e^{\lambda_{\bar\kappa}\tau(G)}}{2\sqrt{\lambda_{\bar{\kappa}}}\rho^L_{uu}C_{\kappa_G}} \|\hat{G} - G(\mu^{T,G}_T,\cdot) \|_{f_{\kappa_G}}e^{-\lambda(T-t)} \Big)
\end{equation}
holds for all $0 \leq t\leq T$. Moreover, for all $\frac{1}{2\lambda_{\bar\kappa}} \leq t \leq T-\tau(G)$ we have
\bes\label{eq:tp_control_lip_low}
\|\varphi^{T,G}_t-\varphi^{T,\hat G}_t\|_{f_{\bar\kappa}} \leq \frac{4W_{f_{\bar\kappa}}(\mu_0,\hat{\mu_0})\const^{F}_{\mu,\mathrm{TV}}}{\pi C^2_{\bar\kappa}\sigma^2_0\sqrt{\lambda_{\bar\kappa}}(1-\varepsilon(\lambda))}e^{-\lambda t}+\Big(\frac{3e^{\frac{1}{4}}\const^{F}_{\mu,\mathrm{TV}}e^{\lambda_{\bar\kappa}\tau(G)}}{\sqrt{\pi}\lambda_{\bar\kappa}C_{\bar{\kappa}}^2\sigma_0\rho^L_{uu}(1-\varepsilon(\lambda))}+1\Big){\frac{(16-4\sqrt{e})\const^\psi_x}{C_{\bar{\kappa}}}} e^{-\lambda(T-t)}
\ees
\item If  $\|G \|_{\infty} < + \infty$ and we define
\footnote{ We have the explicit bound
\bes
\tau'(G)\leq \max\left\{ \frac{1}{\lambda_{\kappa_b}}\log\const^G, \frac{1}{2\lambda_{\kappa_b}} \right\}> 0, \text{ with } \const^G = \frac{\frac{\sqrt{\lambda_{\kappa_b}e}}{\sqrt{\pi}C_{\kappa_b}\sigma_0}\| G\|_{\infty}  - 2\sqrt{e}\const^\psi_x}{(4-2\sqrt{e})\const^\psi_x}
\ees
}
\bes
\tau'(G)=\inf\{\tau>0:\|\varphi^{T,G}_{T-s}\|_{f_{\kappa_b}} \leq (8-2\sqrt{e})\const^\psi_x\eqsp\,\, \forall s\geq \tau \},\quad T'_0=T-\tau'(G), 
\ees
then the turnpike estimate for marginal flows  \eqref{eq:tp_flow_Lip_low} holds for $t\leq T'_0$ when setting ${\tau(G)=\tau'(G),}$ $\kappa_G=\bar\kappa$ and replacing $\|\hat{G} - G(\mu^{T,G}_T,\cdot) \|_{f_{\kappa_G}}$ with $(16-4\sqrt{e})C_{\bar{\kappa}}^{-1}\const^\psi_x$. Finally, the turnpike estimate \eqref{eq:tp_control_lip_high} holds for all $\frac{1}{2\lambda_{\bar\kappa}}\leq t\leq T'_0$ replacing $\tau(G)$ with $\tau'(G)$ .
\een
\end{enumerate}
\end{thm}
\begin{proof}
\begin{itemize}[wide]
\item \underline{Proof of \ref{item:grad_est_low}:} 
The first claim follows directly from  \Cref{lem:basics_MF_low}-\ref{item:grad_est_MF_low}. The second claim follows applying \Cref{prop:Holder_in_time}-\ref{item:time_Holder_TV} with $\beta_\cdot=\partial_pH(\cdot,\nabla \psi^{T,\hat G}[\hat\mu_\cdot]_\cdot)$ using the gradient estimate \eqref{eq:gest_low_last} to find a lower bound for $\kappa_{\beta}$, and referring to \Cref{prop:properties} for the Hölder continuity in time.
\item \underline{Proof of \ref{item:stability_iterate_low}:} In order to ease notation, we write $\hat{\nu}= \nu^{T,\hat{G}}$, $\nu = \nu^{T,G}$, $\hat{\psi} = \psi^{T,\hat{G}}$, $\psi = \psi^{T,G}$.
Under the current assumptions on $\|\hat{G}\|_{f_{\kappa_b}}$ and thanks to \ref{item:grad_est_low} we are allowed to apply \Cref{lem:basics_MF_low}-\ref{item:stab_MF_low} with $\tilde{\kappa}=\bar\kappa$ to obtain using \eqref{eq:def_metric_flow_TV} 
\begin{equation}
\|\nabla\hat\psi_t[\mu_\cdot]-\nabla\hat\psi_t[\hat\mu_\cdot]\|_{\infty}\leq \const^{\delta\psi_t}_x=2\const^F_{\mu,\TV}  \overrightarrow{d_{\frac{\lambda_{\bar\kappa}}{2},\bar{\kappa}}^{T}}(\mu_\cdot,\hat\mu_\cdot) \int_t^Tq^{\bar{\kappa},\lambda_{\bar\kappa}/2}_s q^{\bar{\kappa}}_{s-t} \De s
\end{equation}
and for all $0\leq t_0\leq t\leq T$
\begin{equation}\label{eq:tv_bound_iterates_raw}
\begin{split}
\|\hat\nu_t[\mu_\cdot] - \hat\nu_t[\hat{\mu}_\cdot] \|_{\TV} &\leq q^{\bar{\kappa}}_{t-t_0}e^{-\lambda_{\bar{\kappa}}t_0} W_{f_{\bar{\kappa}}}(\mu_0, \hat{\mu}_0) + q^{\bar{\kappa}}_{t-t_0}\frac{1}{\rho^L_{uu}}\int_{0}^{t_0} e^{-\lambda_{\bar{\kappa}}(t_0 - s)} \const^{\delta\psi_s}_x \De s \eqsp+ \frac{1}{\sqrt{2}\rho^L_{uu}} \left(\int_{t_0}^{t}\left(\const^{\delta\psi_s}_x\right)^2 \De s \right)^{\frac{1}{2}}
\end{split}
\end{equation}
The precise further computations on how to properly further estimate the $\TV$ distance are postponed to \Cref{lem:low_nasty_calc}, which gives that for  $t \leq \frac{2}{\lambda_{\bar\kappa}}$
\begin{align}
\|\hat\nu_t[\mu_\cdot] - \hat\nu_t[\hat{\mu}_\cdot] \|_{\TV} 
\leq q^{\bar{\kappa}}_{t}W_{f_{\bar{\kappa}}}(\mu_0, \hat{\mu}_0) + \frac{\sqrt{2}\const^F_{\mu,\TV}}{\rho^{L}_{uu}\pi \sqrt{\lambda_{\bar\kappa}}C_{\kappa}^2 \sigma_0^2}\left(2\log(4)+  \frac{25e}{9} \right)^{\frac{1}{2}} \overrightarrow{d^{T}_{\frac{\lambda_{\bar\kappa}}{2},\bar\kappa}}(\fl,\hfl)
\end{align}
Dividing by $q^{\bar{\kappa},\lambda_{\bar\kappa}/2}_{t}$ and optimizing over $t \in (0, \frac{2}{\lambda_{\bar\kappa}}]$ gives (using $q^{\bar{\kappa}}_{t} \leq q^{\bar{\kappa},\lambda_{\bar\kappa}/2}_{t}$)
\begin{align}
\sup_{t \leq \frac{1}{\lambda_{\bar\kappa}}} (q^{\bar{\kappa},\lambda_{\bar\kappa}/2}_{t})^{-1}\|\hat\nu_t[\mu_\cdot] - \hat\nu_t[\hat{\mu}_\cdot] \|_{\TV} 
\leq W_{f_{\bar{\kappa}}}(\mu_0, \hat{\mu}_0) 
+ \frac{\sqrt{2e}\const^F_{\mu,\TV}}{\rho^{L}_{uu} \sqrt{\pi}C_{\kappa} \sigma_0\lambda_{\bar{\kappa}}}\left(8\log(2) +  \frac{50e}{9} \right)^{\frac{1}{2}} \overrightarrow{d^{T}_{\frac{\lambda_{\bar\kappa}}{2},\bar\kappa}}(\fl,\hfl)
\end{align}
For $t > \frac{2}{\lambda_{\bar\kappa}}$ we have
\begin{align}
\|\hat\nu_t[\mu_\cdot] - \hat\nu_t[\hat{\mu}_\cdot] \|_{\TV} 
&\leq q^{\bar{\kappa},\lambda_{\bar{\kappa}}}_{t}W_{f_{\bar{\kappa}}}(\mu_0, \hat{\mu}_0)\\ &+ e^{-(\lambda_{\bar\kappa}/2) t} \frac{\const^F_{\mu,\TV}e}{\rho^{L}_{uu}\pi \sqrt{\lambda_{\bar\kappa}} C_{\bar{\kappa}}^2\sigma_0^2}\left(   \frac{5}{3}\sqrt{e}+ \frac{\sqrt{2}}{\sqrt{\pi}C_{\kappa}\sigma_0} \left(1 + \frac{5\sqrt{2}}{3}\right) \right)
\overrightarrow{d^{T}_{\frac{\lambda_{\bar\kappa}}{2},\bar\kappa}}(\fl,\hfl)
\end{align}
Dividing again by $q^{\bar{\kappa},\lambda}_{t}$ and optimizing over $t \in (\frac{2}{\lambda_{\bar\kappa}},T]$ gives (using $q^{\bar{\kappa}}_{t} \leq q^{\bar{\kappa},\lambda_{\bar\kappa}/2}_{t}$)
\begin{align}
\sup_{ \frac{2}{\lambda_{\bar\kappa}} < t \leq T} (q^{\bar{\kappa},\lambda_{\bar\kappa}}_{t})^{-1}\|\hat\nu_t[\mu_\cdot] - \hat\nu_t[\hat{\mu}_\cdot] \|_{\TV} 
\leq W_{f_{\bar{\kappa}}}(\mu_0, \hat{\mu}_0) 
+ \frac{\const^F_{\mu,\TV}\sqrt{e}}{\rho^{L}_{uu}\sqrt{\pi} \lambda_{\bar\kappa}C_{\bar{\kappa}}\sigma_0}\left(\frac{5\sqrt{2e}}{3} + \frac{2(3+5\sqrt{2})}{3\sqrt{\pi}C_{\kappa}\sigma_0} \right) 
\overrightarrow{d^{T}_{\frac{\lambda_{\bar\kappa}}{2},\bar\kappa}}(\fl,\hfl)
\end{align}
This gives in total
\begin{align}
d_{\TV,\bar{\kappa}}^{\lambda_{\bar\kappa}/2,T}(\hat\nu_\cdot[\fl],\hat\nu_\cdot[\hfl]) 
&\leq W_{f_{\bar{\kappa}}}(\mu_0, \hat{\mu}_0) + \frac{\sqrt{e}\const^F_{\mu,\TV}}{\rho^{L}_{uu}\sqrt{\pi} C_{\bar{\kappa}}\sigma_0\lambda_{\bar{\kappa}}} \max \left\{\left(4\log(2)^{\frac{1}{2}} + \frac{10}{3}\sqrt{e}\right), \left(\frac{5\sqrt{2e}}{3} + \frac{2(3+5\sqrt{2})}{3\sqrt{\pi}C_{\kappa}\sigma_0} \right)\right\}\overrightarrow{d^{T}_{\frac{\lambda_{\bar\kappa}}{2},\bar\kappa}}(\fl,\hfl) \\
&\leq W_{f_{\bar{\kappa}}}(\mu_0, \hat{\mu}_0) + \varepsilon(\lambda)\overrightarrow{d^{T}_{\frac{\lambda_{\bar\kappa}}{2},\bar\kappa}}(\fl,\hfl).
\end{align}
For the second bound in $\overleftarrow{d_{\frac{\lambda_{\bar\kappa}}{2},\bar{\kappa}}^{T}}$ suppose now $\mu_0 = \hat{\mu}_0$. Then, using \Cref{lem:basics_MF_low}-\ref{item:stab_MF_low} and \eqref{eq:def_metric_flow_TV} we get
\begin{equation}
\|\nabla\hat\psi_t[\mu_\cdot]-\nabla\hat\psi_t[\hat\mu_\cdot]\|_{\infty}=\const^{\delta\psi_t}_x\leq 2\const^F_{\mu,\mathrm{TV}} \overleftarrow{d_{\frac{\lambda_{\bar\kappa}}{2},\bar{\kappa}}^{T}}(\mu_\cdot,\hat\mu_\cdot) \int_t^Te^{-(\lambda_{\bar\kappa}/2)(T-s)}q^{\bar{\kappa}}_{s-t} \De s
\end{equation}
and choosing $t_0 = 0$ in \eqref{eq:tv_bound_iterates_raw}
\begin{equation}\label{eq:low_reg_iteration_stability_1}
\begin{split}
\|\hat\nu_t[\mu_\cdot] - \hat\nu_t[\hat{\mu}_\cdot] \|_{\TV} \leq q^{\bar{\kappa}}_{t} W_{f_{\bar{\kappa}}}(\mu_0, \hat{\mu}_0) + \frac{1}{\sqrt{2}\rho^L_{uu}} \left(\int_{0}^{t}\left(\const^{\delta\psi_s}_x\right)^2 \De s \right)^{\frac{1}{2}}
\end{split}
\end{equation}
A standard calculation though long calculation detailed at \Cref{lem:boring_calculations_mild} gives
\begin{equation}
\const^{\delta\psi_s}_x\leq \frac{6\sqrt{2}\const^F_{\mu,\TV}}{\sqrt{\pi}\sqrt{\lambda_{\bar\kappa}}C_{\bar{\kappa}}\sigma_0} \overleftarrow{d_{\frac{\lambda_{\bar\kappa}}{2},\bar{\kappa}}^{T}}(\mu_\cdot,\hat\mu_\cdot)   e^{-(\lambda_{\bar\kappa}/2) (T-t)} \leq \varepsilon(\lambda) \overleftarrow{d_{\frac{\lambda_{\bar\kappa}}{2},\bar{\kappa}}^{T}}(\mu_\cdot,\hat\mu_\cdot)   e^{-(\lambda_{\bar\kappa}/2) (T-t)}
\end{equation}
Using this result in \eqref{eq:low_reg_iteration_stability_1} we obtain after some simple calculations that
\bes
\|\nu_t[\mu_\cdot] - \nu_t[\hat{\mu}_\cdot] \|_{\TV} \leq \varepsilon(\lambda_{\bar\kappa}/2)\overleftarrow{d_{\frac{\lambda_{\bar\kappa}}{2},\bar{\kappa}}^{T}}(\mu_\cdot,\hat\mu_\cdot) e^{-(\lambda_{\bar\kappa}/2) (T-t)}.
\ees
Optimizing over $t \in [0,T]$ we obtain \ref{item:stability_iterate_low}. 
\item \underline{Proof of \ref{item:contraction_iterate_low}:} 
The fact that we have $\varepsilon(\lambda)<1$ follows by standard computations. For \ref{item:contraction_iterate_low}. But then, thanks to \ref{item:grad_est_low}, we have that for any $\hat\mu_0\in\cP_p(\bbR^d)$, $\nu^{T,\hat G}[\cdot]$ maps the complete metric space $\Gamma_{\hat\mu_0,C}$ onto itself and is a strict contraction for the distance $\overrightarrow{d_{\lambda,\bar{\kappa}}^{T}}$. We can therefore invoke Banach's fixed point Theorem which gives us existence of a unique fixed point.
\item \underline{Proof of \ref{item:turnpike_low}:} For the turnpike estimate we have by \Cref{lem:abs_turnpike_est} and \ref{item:stability_iterate_low}
\begin{equation}\label{eq:abstract_tpike_low}
\|\mu^{T,G}_t - \mu^{T,\hat{G}}_t\|_{\TV} \leq \frac{q^{\bar{\kappa},\lambda}_t}{1 - \varepsilon(\lambda)} W_{f_{\bar{\kappa}}}(\mu_0, \hat{\mu}_0) + 
\frac{e^{-\lambda(T-t)}}{1 - \varepsilon(\lambda)} \overleftarrow{d_{\lambda,\bar{\kappa}}^{T}}(\nu^{T,\hat{G}}_\cdot[\mu^{T,G}_\cdot],\nu^{T,G}_\cdot[\mu^{T,G}_\cdot]) \eqsp.
\end{equation}
We can hence conclude the turnpike estimates for the dynamics by properly bounding the term involving $\overleftarrow{d_{\lambda,\bar{\kappa}}^{T}}$ as follows.
\begin{itemize}[wide]
\item Let $\| G\|_{f_{\kappa_b}} < + \infty$. Then, the proof on how to bound the term involving $\overleftarrow{d_{\lambda,\bar{\kappa}}^{T}}$ is very similar to the proof in \Cref{prop:abstract_suff_cond_tpike_high}-\ref{item:contraction_iterate_high}. If $\sup_{t\leq T}\|\psi^{T,G}_t[\mu_\cdot]\|_{f_{\kappa_b}}\leq (8-2\sqrt{e})\const^\psi_x$, we apply the TV-estimate of \Cref{lem:basics_MF_low}-\ref{item:stab_MF_low} for $t_0 = 0$ and $\tilde\kappa=\bar\kappa$ to obtain 
\begin{align}
\|\nu^{T,\hat{G}}_t[\mu^{T,G}_\cdot] - \nu^{T,G}_t[\mu^{T,G}_\cdot]\|_{\TV} \leq e^{-\lambda_{\bar{\kappa}}(T-t)} \frac{1}{2\sqrt{\lambda_{\bar{\kappa}}}\rho^L_{uu}} \|\hat{G} - G(\mu^{T,G}_T,\cdot) \|_{f_{\bar\kappa}} \eqsp
\end{align}
Plugging this bound in \eqref{eq:abstract_tpike_low} gives the turnpike estimate 
\bes
\|\mu^{T,G}_t - \mu^{T,\hat{G}}_t\|_{\TV} \leq \frac{q^{\bar{\kappa},\lambda}_t}{1 - \varepsilon(\lambda)} W_{f_{\bar{\kappa}}}(\mu_0, \hat{\mu}_0) +
\frac{e^{-\lambda(T-t)}}{(1 - \varepsilon(\lambda))}\frac{1}{2\sqrt{\lambda_{\bar{\kappa}}}} \|\hat{G} - G(\mu^{T,G}_T,\cdot) \|_{f_{\bar\kappa}} 
\ees
Plugging this estimate into the bound on the value functions from \Cref{lem:basics_MF_low}-\ref{item:stab_MF_low} and computing the resulting integrals by means of \Cref{lem:boring_calculations_mild} gives the following turnpike estimate for $\frac{1}{2\lambda_{\bar\kappa}}\leq t \leq T$
\bes\label{eq:est_val_func_low_G_Lip_small}
\|\varphi^{T,G}_t-\varphi^{T,\hat G}_t\|_{f_{\bar\kappa}} \leq \frac{4W_{f_{\bar\kappa}}(\mu_0,\hat{\mu_0})\const^{F}_{\mu,\mathrm{TV}}}{\pi(1-\varepsilon(\lambda))C^2_{\bar\kappa}\sigma^2_0\sqrt{\lambda_{\bar\kappa}}}e^{\ke{-\lambda t}}+\Big(\frac{3e^{\frac{1}{4}}\const^{F}_{\mu,\mathrm{TV}}}{\sqrt{\pi}(1-\varepsilon(\lambda))\lambda_{\bar\kappa}C_{\bar{\kappa}}\sigma_0\rho^L_{uu}C_{\bar{\kappa}}}+1\Big)\|\hat{G}- G(\mu_T,\cdot) \|_{f_{\bar\kappa}}e^{\ke{-\lambda(T-t)}}
\ees
If $\|G\|_{f_{\kappa_b}} \leq 4 \const^\psi_x$, we can guarantee $\sup_{t\leq T}\|\psi^{T,G}_t[\mu_\cdot]\|_{f_{\kappa_b}}\leq (8-2\sqrt{e})\const^\psi_x$ thanks to \Cref{lem:basics_MF_low}-\ref{item:grad_est_MF_low}. Hence, \eqref{eq:est_val_func_low_G_Lip_small} becomes \eqref{eq:tp_control_lip_low} by bounding ${\|\hat{G}- G(\mu_T,\cdot) \|_{f_{\bar\kappa}}}$ with $(16-4\sqrt{e})\const^\psi_x /C_{\bar{\kappa}}$ by means of the triangular inequality and setting $\tau(G)=0$.

To work with $\|\hat G\|_{f_{\kappa_b}} > 4{\const^x_\psi}$, we start by observing that setting $T_0 = T - \tau(G)$, we can guarantee ${\sup_{t\leq T_0}\|\psi^{T,G}_{t}[\mu_\cdot]\|_{f_{\kappa_b}}\leq (8-\sqrt{2})\const^\psi_x}$ from \Cref{lem:basics_MF_low}-\ref{item:grad_est_MF_low}.
From now on, the arguments used to conclude the estimates in this case are similar to the proof of \Cref{prop:abstract_suff_cond_tpike_high}-\ref{item:high_reg_tpike_est}, the only differences being that we directly work with the information ${\sup_{t\leq T_0}\|\psi^{T,G}_{t}[\mu_\cdot]\|_{f_{\kappa_b}}\leq (8-\sqrt{2})\const^\psi_x}$  this tim and we invoke \Cref{lem:basics_MF_low} instead of \Cref{lem:basics_MF_high}.
\item Suppose now $\| G\|_{\infty}  < +\infty$.
We have thanks to \Cref{lem:basics_MF_low}-\ref{item:grad_est_MF_low} that $\|\psi_{t}[\mu^{T,G}_\cdot] \|_{f_{\kappa_b}} \leq (8 - 2\sqrt{e}) \const^\psi_x$ for $t\leq T'_0$. Thus, invoking the dynamic programming principle, we can use the results obtained under the assumption $\sup_{t\leq T'_0}\|\psi_{t}[\mu^{T,G}_\cdot] \|_{f_{\kappa_b}} \leq (8 - 2\sqrt{e}) \const^\psi_x$ 
if we restrict to the time-interval $[0,T'_0]$. This means that the bounds \eqref{eq:tp_flow_Lip_low}-\eqref{eq:tp_control_lip_low} hold setting 
\bes 
\tau(G)=\tau'(G),\,\kappa_G=\bar\kappa, \quad \text{and} \quad \hat{G}=\varphi^{T,\hat G}_{T'_0},G(\mu_T,\cdot)=\varphi^{T,G}_{T'_0}.\ees 
Since the difference $\| \varphi^{T,G}_{T'_0}-\varphi^{T,\hat{G}}_{T_0}\|_{f_{\bar\kappa}}$ can be bounded with $(16-4\sqrt{e})C^{-1}_{\bar\kappa}\const^\psi_x$ by a triangular inequality and the definition of $T'_0$, we obtain the desired turnpike estimate for the flows and the sought turnpike estimate for the gradients of the value functions in $t\in[1/(2\lambda_{\bar{\kappa}}),T'_0]$. 
\end{itemize}
\end{itemize}
\end{proof}

\begin{lemma}\label{lem:low_nasty_calc}
In the setting of \Cref{prop:abstract_suff_cond_tpike_low} we have 
\ben[(i),wide]
\item If $t \leq \frac{2}{\lambda_{\bar\kappa}}$ then
\begin{align}
\|\nu_t[\mu_\cdot] - \nu_t[\hat{\mu}_\cdot] \|_{\TV} \leq &q^{\bar{\kappa}}_{t}W_{f_{\bar{\kappa}}}(\mu_0, \hat{\mu}_0) \\
&+ \frac{\sqrt{2}\const^F_{\mu,\TV}}{\rho^{L}_{uu}\pi \sqrt{\lambda_{\bar\kappa}}C_{\bar{\kappa}}^2 \sigma_0^2}\left(2\log(4) + \frac{25}{9} e  \right)^{\frac{1}{2}} \overrightarrow{d^{T}_{\frac{\lambda_{\bar\kappa}}{2},\bar\kappa}}(\fl,\hfl)
\end{align}
\item If $t > \frac{2}{\lambda_{\bar\kappa}}$ then
\begin{align}
\|\nu_t[\mu_\cdot] - \nu_t[\hat{\mu}_\cdot] \|_{\TV} 
&\leq q^{\bar{\kappa},\lambda_{\bar{\kappa}}}_{t}W_{f_{\bar{\kappa}}}(\mu_0, \hat{\mu}_0) \\
&+ e^{-(\lambda_{\bar\kappa}/2) t} \frac{\const^F_{\mu,\TV}e}{\rho^{L}_{uu}\pi C_{\bar{\kappa}}^2\sigma_0^2}\left(   \frac{5}{3}\sqrt{\frac{e}{\lambda_{\bar{\kappa}}}}
+ \frac{\sqrt{2}}{\sqrt{\pi}C_{\bar\kappa}\sigma_0} \left(\frac{1}{\sqrt{\lambda_{\bar{\kappa}}}} + \frac{5\sqrt{2}}{3\sqrt{\lambda_{\bar{\kappa}}}}\right) \right)
\overrightarrow{d^{T}_{\frac{\lambda_{\bar\kappa}}{2},\bar\kappa}}(\fl,\hfl)
\end{align}
\een
\end{lemma}
\begin{proof}
We recall the identity \eqref{eq:tv_bound_iterates_raw}
\begin{equation}\label{eq:low_reg_TV_est_general}
\begin{split}
\|\nu_t[\mu_\cdot] - \nu_t[\hat{\mu}_\cdot] \|_{\TV} &\leq q^{\bar{\kappa}}_{t-t_0}e^{-\lambda_{\bar{\kappa}}t_0} W_{f_{\bar{\kappa}}}(\mu_0, \hat{\mu}_0) + q^{\bar{\kappa}}_{t-t_0}\frac{1}{\rho^L_{uu}}\int_{0}^{t_0} e^{-\lambda_{\bar{\kappa}}(t_0 - s)} \const^{\delta\psi_s}_x \De s \eqsp+ \frac{1}{\sqrt{2}\rho^L_{uu}} \left(\int_{t_0}^{t}\left(\const^{\delta\psi_s}_x\right)^2 \De s \right)^{\frac{1}{2}},
\end{split}
\end{equation}
which is valid  for all $t_0\leq t\leq T$ and where
\begin{equation}
\const^{\delta\psi_t}_x = 2\const^F_{\mu,\TV}  \overrightarrow{d_{\lambda_{\bar\kappa}/2,\bar{\kappa}}^{T}}(\mu_\cdot,\hat\mu_\cdot) \int_t^Tq^{\bar{\kappa},\lambda_{\bar\kappa}/2}_s q^{\bar{\kappa}}_{s-t} \De s \eqsp. 
\end{equation}
Let us define for brevity
\begin{equation}
\bar{\const}^{\delta\psi_t}_x := 2\const^F_{\mu,\TV} \int_t^Tq^{\bar{\kappa},\lambda_{\bar\kappa}/2}_s q^{\bar{\kappa}}_{s-t} \De s \eqsp. 
\end{equation}
The rest of the proof consists of three steps: we first bound $\bar{\const}^{\delta\psi_t}_x$, then use this result in \eqref{eq:low_reg_TV_est_general} to obtain the conclusion $t\leq2/\lambda_{\bar\kappa}$ and conclude by doing the same in the case $t>2/\lambda_{\bar\kappa}$.
\bei
\item \underline{Step 1: upper bound for $\bar{\const}^{\delta\psi_t}_x$}
 If $t \geq \frac{1}{\lambda_{\bar\kappa}}$ then we can apply \Cref{lem:boring_calculations_mild} to obtain
\begin{align}\label{eq:ctv_bound}
\bar{\const}^{\delta\psi_t}_x
&= \frac{\sqrt{2}\const^F_{\mu,\TV}\sqrt{\lambda_{\bar\kappa} e}}{\sqrt{\pi}C_{\bar{\kappa}}\sigma_0}\left(\int_t^{T} q^{\bar{\kappa}}_{s-t} e^{-(\lambda_{\bar\kappa}/2) s}\De s \right) 
\leq  \frac{10\sqrt{e}}{3\sqrt{2}\pi }  \frac{\const^F_{\mu,\TV}}{C_{\bar{\kappa}}^2\sigma_0^2} e^{-(\lambda_{\bar\kappa}/2) t}.
\end{align}
Whereas for $t < \frac{1}{\lambda_{\bar\kappa}}$
\begin{align}\label{eq:bound_c_bar_psi}
\bar{\const}^{\delta\psi_t}_x
&= 2\const^F_{\mu,\TV}\left(\int_t^{\frac{1}{\lambda_{\bar\kappa}}} q^{\bar\kappa}_{s-t} q^{\bar\kappa,\lambda_{\bar\kappa}/2}_s\De s + \int_{\frac{1}{\lambda_{\bar\kappa}}}^T q^{\bar\kappa}_{s-t} q^{\bar\kappa,\lambda_{\bar\kappa}/2}_s\De s \right) \\
&\leq \frac{2\const^F_{\mu,\TV}}{\pi C_{\bar{\kappa}}^2 \sigma_0^2} \Big(\sinh^{-1}\left(\sqrt{\frac{1}{\lambda_{\bar\kappa} t} - 1} \right)+  \frac{5\sqrt{e}}{3\sqrt{2} }   e^{-\lambda_{\bar\kappa}/2 t} \Big)\eqsp,
\end{align}
where for the first integral we use $q^{\bar\kappa}_{s-t} \leq \frac{1}{\sqrt{2\pi (s-t)}C_{\bar{\kappa}}\sigma_0}$. This gives
\begin{align}
\int_t^{\frac{1}{\lambda_{\bar\kappa}}} q^{\bar\kappa}_{s-t} q^{\bar\kappa,\lambda_{\bar\kappa}/2}_s\De s
\leq \frac{1}{2\pi C_{\bar\kappa}^2 \sigma_0^2}\int_0^{\frac{1}{\lambda_{\bar\kappa}}-t}\frac{1}{\sqrt{s(s+t)}} \De s = \frac{1}{\pi C_{\bar\kappa}^2 \sigma_0^2}\sinh^{-1}\left(\sqrt{\frac{1}{\lambda_{\bar\kappa} t} - 1} \right) \eqsp .
\end{align}
The second integral is computed analogously as for $t \geq \frac{1}{2\lambda_{\bar\kappa}}$ using $e^{-\frac{1}{2}} < e^{-(\lambda_{\bar\kappa}/2) t}$ for $t < \frac{1}{2\lambda_{\bar\kappa}}$.

\item \underline{Step 2: $t \leq \frac{2}{\lambda_{\bar\kappa}}$.}
In this case we set $t_0=0$. Then, \eqref{eq:low_reg_TV_est_general} reads as
\bes
\|\nu_t[\mu_\cdot] - \nu_t[\hat{\mu}_\cdot] \|_{\TV} \leq q^{\bar{\kappa},\lambda_{\bar{\kappa}}}_{t}W_{f_{\bar{\kappa}}}(\mu_0, \hat{\mu}_0) + \frac{1}{\sqrt{2}\rho^L_{uu}} \left(\int_{0}^{t}\left(\bar{\const}^{\delta\psi_s}_x\right)^2 \De s \right)^{\frac{1}{2}} \overrightarrow{d^{T}_{\frac{\lambda_{\bar\kappa}}{2},\bar\kappa}}(\fl,\hfl) \eqsp.
\ees
In order to compute the second term we split the integral as follows
\begin{align}
\int_{0}^{t}\left(\bar{\const}^{\delta\psi_s}_x\right)^2 \De s  
= \int_{0}^{\frac{1}{\lambda_{\bar\kappa}}\wedge t}\left(\bar{\const}^{\delta\psi_s}_x\right)^2 \De s + \int_{\frac{1}{\lambda_{\bar\kappa}}\wedge t}^t\left(\bar{\const}^{\delta\psi_s}_x\right)^2 \De s \eqsp .
\end{align}
For the first term 
\begin{align}
\int_{0}^{\frac{1}{\lambda_{\bar\kappa}}\wedge t}\left(\bar{\const}^{\delta\psi_s}_x\right)^2 \De s 
\stackrel{\eqref{eq:bound_c_bar_psi}}{\leq} &\,2\left(\frac{2\const^F_{\mu,\TV}}{\pi C_{\bar{\kappa}}^2 \sigma_0^2}\right)^2\Biggl[\int_{0}^{\frac{1}{\lambda_{\bar\kappa}}} \sinh^{-1}\left(\sqrt{\frac{1}{\lambda_{\bar\kappa} s} - 1} \right)^2 \De s 
+ \frac{25 e}{18}\int_{0}^{\frac{1}{\lambda_{\bar\kappa}}\wedge t}e^{-\lambda_{\bar\kappa} s} \De s\Biggr] \eqsp.
\end{align}
Now, for the first integral
\begin{align}
\int_{0}^{\frac{1}{\lambda_{\bar\kappa}}} \sinh^{-1}\left(\sqrt{\frac{1}{\lambda_{\bar\kappa} s} - 1} \right)^2 \De s
=\frac{1}{\lambda_{\bar\kappa}}\int_{0}^{1} \sinh^{-1}\left(\sqrt{\frac{1}{s} - 1} \right)^2 \De s = \frac{\log(4)}{\lambda_{\bar\kappa}} \eqsp .
\end{align}
With this we obtain
\begin{align}
\left(\int_{0}^{t}\left(\bar{\const}^{\delta\psi_s}_x\right)^2 \De s \right)^{\frac{1}{2}} 
&\leq \left(\int_{0}^{\frac{1}{\lambda_{\bar\kappa}}}\left(\bar{\const}^{\delta\psi_s}_x\right)^2 \De s + \int_{\frac{1}{\lambda_{\bar\kappa}}\wedge t}^t\left(\bar{\const}^{\delta\psi_s}_x\right)^2 \De s \right)^{\frac{1}{2}} \\
&\leq \frac{2\const^F_{\mu,\TV}}{\pi C_{\bar{\kappa}}^2 \sigma_0^2}\left(\frac{2\log(4)}{\lambda_{\bar\kappa}} + \frac{25}{9}\int_0^t  e^{-\lambda _{\bar\kappa}s} \De s \right)^{\frac{1}{2}}\\
&\leq \frac{2\left(2\log(4) +   \frac{25}{9}e\right)^{\frac{1}{2}}}{\pi}\frac{\const^F_{\mu,\TV}}{ \sqrt{\lambda_{\bar\kappa}}C_{\bar{\kappa}}^2 \sigma_0^2} \eqsp ,
\end{align}
which is the desired result.
\item \underline{ Step 3: $t \geq \frac{2}{\lambda_{\bar\kappa}}$.} In this case, we choose $t_0 = t - \frac{1}{\lambda_{\bar\kappa}}$. Thus, \eqref{eq:low_reg_TV_est_general} becomes
\begin{align}\label{eq:horror_lemma_last step}
\|\nu_t[\mu_\cdot] - \nu_t[\hat{\mu}_\cdot] \|_{\TV} 
\leq &\,q^{\bar\kappa}_t W_{f_{\bar{\kappa}}}(\mu_0, \hat{\mu}_0) + q^{\bar\kappa}_{\frac{1}{\lambda_{\bar\kappa}}} \frac{1}{\rho^L_{uu}}\int_{0}^{t - \frac{1}{\lambda_{\bar\kappa}}} e^{-\lambda_{\bar{\kappa}}(t - \frac{1}{\lambda_{\bar\kappa}} - s)} \bar{\const}^{\delta\psi_s}_x \De s \eqsp \overrightarrow{d^{T}_{\frac{\lambda_{\bar\kappa}}{2},\bar\kappa}}(\fl,\hfl) \\
&+ \frac{1}{\sqrt{2}\rho^L_{uu}} \left(\int_{t - \frac{1}{\lambda_{\bar\kappa}}}^{t}\left(\bar{\const}^{\delta\psi_s}_x\right)^2 \De s \right)^{\frac{1}{2}} d_{\TV,\bar\kappa}^{\lambda_{\bar\kappa}/2,T}(\fl,\hfl),
\end{align}
where we have used that 
\begin{equation}
q^{\bar{\kappa},\lambda_{\bar{\kappa}}}_{\frac{1}{\lambda_{\bar\kappa}}} e^{-\lambda_{\bar{\kappa}}(t - \frac{1}{\lambda_{\bar\kappa}})} = q^{\bar{\kappa},\lambda_{\bar{\kappa}}}_t \eqsp .
\end{equation}
Since $t - \frac{1}{\lambda_{\bar\kappa}} \geq \frac{1}{\lambda_{\bar\kappa}}$
\begin{align}\label{eq:first_int_bd}
\left(\int_{t - \frac{1}{\lambda_{\bar\kappa}}}^{t}\left(\bar{\const}^{\delta\psi_s}_x\right)^2 \De s \right)^{\frac{1}{2}}
\leq \frac{2\const^F_{\mu,\TV}\sqrt{e}}{\pi C_{\bar{\kappa}}^2\sigma_0^2}   \frac{5}{3\sqrt{2}}\left(\int_{t - \frac{1}{\lambda_{\bar\kappa}}}^{t}e^{-\lambda_{\bar\kappa} s} \De s \right)^{\frac{1}{2}} \leq \frac{5\sqrt{2}e^{3/2}}{3\pi} \frac{\const^F_{\mu,\TV}}{ \sqrt{\lambda_{\bar{\kappa}}}C_{\bar{\kappa}}^2\sigma_0^2}   e^{-(\lambda_{\bar\kappa}/2) t} \eqsp .
\end{align}
For the other integral,
\begin{align}\label{eq:second_int_bd}
&\,\int_{0}^{t - \frac{1}{\lambda_{\bar\kappa}}} e^{-\lambda_{\bar{\kappa}}(t - \frac{1}{\lambda_{\bar{\kappa}}} - s)} \bar{\const}^{\delta\psi_s}_x \De s \\
= &\,\int_{0}^{\frac{1}{\lambda_{\bar\kappa}}} e^{-\lambda_{\bar{\kappa}}(t - \frac{1}{\lambda_{\bar\kappa}} - s)} \bar{\const}^{\delta\psi_s}_x \De s + \int_{\frac{1}{\lambda_{\bar\kappa}}}^{t - \frac{1}{\lambda_{\bar\kappa}}} e^{-\lambda_{\bar{\kappa}}(t - \frac{1}{\lambda_{\bar\kappa}} - s)} \bar{\const}^{\delta\psi_s}_x \De s \\
\stackrel{\eqref{eq:bound_c_bar_psi}}{\leq} &\,\frac{2\const^F_{\mu,\TV}}{\pi C_{\bar{\kappa}}^2 \sigma_0^2} \left(\int_{0}^{\frac{1}{\lambda_{\bar\kappa}}} e^{-\lambda_{\bar{\kappa}}(t - \frac{1}{\lambda_{\bar{\kappa}}} - s)} \sinh^{-1}\left(\sqrt{\frac{1}{\lambda_{\bar\kappa} s} - 1} \right)\De s  +   \frac{5\sqrt{e}}{3\sqrt{2}}\int_{0}^{t-\frac{1}{\lambda_{\bar\kappa}}} e^{-\lambda_{\bar{\kappa}}(t - \frac{1}{\lambda_{\bar\kappa}} - s)}e^{-(\lambda_{\bar\kappa}/2) s} \De s\right) \\
\leq &\, \frac{2e\const^F_{\mu,\TV}}{\pi \lambda_{\bar\kappa}C_{\bar{\kappa}}^2 \sigma_0^2} \big(1+\frac{5\sqrt{2}}{3}\big)e^{-(\lambda_{\bar\kappa}/2) t} \eqsp,
\end{align}
where for the first integral we have used $e^{-\lambda_{\bar{\kappa}}(t - \frac{1}{\lambda_{\bar\kappa}} - s)} \leq e^{-\lambda_{\bar\kappa} t} e$ and
\begin{equation}
\int_{0}^{\frac{1}{\lambda_{\bar\kappa}}}  \sinh^{-1} \left(\sqrt{\frac{1}{\lambda_{\bar\kappa} s} - 1} \right)\De s = 1 \eqsp.
\end{equation}
Using \eqref{eq:first_int_bd} and \eqref{eq:second_int_bd} in \eqref{eq:horror_lemma_last step} and bounding $q^{\bar\kappa}_{1/\lambda_{\bar\kappa}}$ with $q^{\bar\kappa,\lambda_{\bar\kappa}/2}_{1/\lambda_{\bar\kappa}}$ gives the desired result.
\eei

\end{proof}

\appendix
\section{Proofs of \Cref{sec:couplings}}\label{sec:appendix_coupling}

\subsection{Proofs of \Cref{sec:intro_coupling}}

We preface this section by well-known moment estimates of the diffusion under \Cref{ass:coupling}.
Consider the generator associated with \eqref{eq:SDE_sigma} defined for any $f \in C^2(\rset^d)$, $s \geq 0$ and $x \in \rset^d$ by
\begin{equation}
    \label{eq:def_generator}
    \scrL_s f(x) = (\beta_s(x) + \alpha_s(x)) \cdot \nabla f(x) + \tr(\sigma(x) \sigma(x)^{\transpose} \nabla^2 f(x)) \eqsp.
\end{equation}

\begin{prop}
\label{prop_existence_moment_estimate}
Assume \Cref{ass:coupling} and let $p \geq 1$.
Then, setting $V_p(x) = \sqrt{1+\norm{x}^{2p}}$, there exist, $\uplambda_p >0$ and $\const_p\geq 0$ such that for any $s \geq 0$ and $x \in \rset^d$, we have
\begin{equation}
    \scrL_s V_p(x) \leq - \uplambda_p V_p(x) + \const_p \eqsp.
    \label{eq:drift_V_1}
\end{equation}
Therefore, there exists unique strong
solutions for \eqref{eq:SDE_sigma} that satisfy if $\PE[\norm{X_0}^p] < \plusinfty$,
\begin{equation}
\sup_{t \geq 0} \PE[\norm{X_t}^p] < \plusinfty \eqsp.
\end{equation}
\end{prop}

\begin{proof}
  Since $\kappaX \in \msk$, we have that there exists $\kappa_+ >0$  and $\const_{\beta}$ such that for any $x \in\rset^d$ and $s \geq 0$, it holds
  \begin{equation}
    \beta_s(x) \cdot x \leq -\kappa_+ \norm{x}^2 + \const_\beta \eqsp.
  \end{equation}
  In addition, we have
  \begin{equation}
    \nabla V_p(x) = \frac{p x \norm{x}^{2p-2}}{V_p(x)} \eqsp, \quad \nabla^2 V_p(x) - \frac{p \Idd \norm{x}^{2p-2}}{V_p(x)} + \frac{p(p-1) x x^{\transpose} \norm{x}^{2(p-2)}}{V_p(x)} - \frac{p^2 x x^{\transpose} \norm{x}^{4(p-1)}}{V_p^{2}(x)} \eqsp.
  \end{equation}
  Therefore, combining this two facts and using \eqref{eq:def_generator} and by \Cref{ass:coupling}, $\alpha$ and $\sigma$ are uniformly bounded (in time and space), we obtain that
  there exist $\uplambda_p >0$ and $\const_{1,p},\const_{2,p}\geq 0$ such that for any $s \geq 0$ and $x \in \rset^d$,
  \begin{equation}
    \label{eq:drift_proof_V}
    \scrL_s V_p(x) \leq - 2 \uplambda_p V_p(x) + \const_{1,p} \norm{x}^{p-1} + \const_{2,p} \eqsp.
  \end{equation}
  Then,  there exists $M_{\beta} \geq 0$ such that for any $s \geq 0$ and $x \in\rset^d$, $\norm{x} \geq M_\beta$, $    \scrL_s V_p(x) \leq -  \uplambda_p V_p(x)$, which completes the proof of \eqref{eq:drift_V_1} using \eqref{eq:drift_proof_V} again.

Equation~\eqref{eq:drift_V_1} and  \cite[Cor 2.6]{gyongy1996ExistenceStrongSolutions} imply that there exists unique strong
solutions for \eqref{eq:SDE_sigma}. In addition, combining Itô formula and \eqref{eq:drift_V_1}, we have for any $M \geq 0$, and $t \geq t_0 \geq 0$,
\begin{align}
  V_p(X_{t \wedge \tau_M}) &=   V_p(X_{t_0 \wedge \tau_M}) + \int_{t_0 \wedge \tau_M}^{t \wedge \tau_M} \scrL_s  V_p(X_{s\wedge \tau_M}) \rmd s + \int_{t_0 \wedge \tau_M}^{t \wedge \tau_M} \nabla V_p(X_s) \cdot \rmd B_s \\
  & \leq  V_p(X_{t_0 \wedge \tau_M}) - \uplambda_p \int_{t_0 \wedge \tau_M}^{t \wedge \tau_M}  V_p(X_{s\wedge \tau_M}) \rmd s + \const_p (t \wedge \tau_M) + \int_{t_0 \wedge \tau_M}^{t \wedge \tau_M} \nabla V_p(X_s) \cdot \rmd B_s \eqsp,
\end{align}
 where $\tau_M = \inf \{ s \geq 0 \, :\, \norm{X_s} \geq M \}$. Since $(\int_{t_0 \wedge \tau_M}^{t \wedge \tau_M} \nabla V_p(X_s) \cdot \rmd B_s)_{t \geq 0}$ is a bounded martingale, we get taking expectation using Gronwall lemma (see \Cref{lem:Gronwall_weak_form} for a suitable version) for any $M \geq 0$, and $t \geq 0$,
 \begin{equation}
\PE[ \1\{t \leq \tau_M\} V_p(X_{t \wedge \tau_M})] \leq \PE[  V_p(X_{0})] + \const_p/\uplambda_p \eqsp.
 \end{equation}
 Taking $M\to \plusinfty$ and using Fatou's Lemma complete the proof.
\end{proof}

\subsubsection{Proof of \Cref{prop:Holder_in_time}}\label{sec:proof_Holder_in_time}

The proof of \ref{item:time_Holder_Wf} is standard under our assumptions. Indeed,
by \Cref{ass:coupling}, $\alpha,\sigma$ are uniformly bounded, $\sup_{x \in\rset^d \,  s \geq 0} \{  \norm{\alpha_s(x)} + \normFr{\sigma(x)\sigma(x)^{\transpose}}\} \leq \const_1$ for $\const_1\geq 0$,  and the assumption on $\beta$ implies that there exists $\const_2 \geq 0$ such that for any $x \in\rset^d$ and $s \geq 0$, $\norm{\beta_s(x)} \leq \const_2(1+ \norm{x}^p)$. Therefore by \eqref{eq:SDE_sigma}, we have that for any $t,s \geq 0$, using Jensen inequality
\begin{align}\label{eq:proof_1_SDE_sigma}
 \norm{X_t- X_s} &=  \int_s^t \norm{\beta_u(X_u) + \alpha_{u}(X_u)} \De u +\norm{\int_s^t \sigma(X_u) \De B_u} \\
 & \leq \int_s^t \const_3(\norm{X_u}^p +1)  \De u +\norm{\int_s^t \sigma(X_u) \De B_u} \eqsp,
\end{align}
where $\const_3 = \const_1\vee \const_2$.
Taking expectation,  using the Cauchy-Schwarz inequality and Itô isometry,  we get
\begin{equation}
  \PE[  \norm{X_t- X_s} ] \leq 2 \int_s^t \const_3(\bbE[\norm{X_u}^p] +1)  \De u + \| \sigma \|^2_{\mathrm{Fr}} \sqrt{t-s}.
  \end{equation}
  \Cref{prop_existence_moment_estimate} concludes the proof.
  We now show~\ref{item:time_Holder_TV}. To this end, denote by $(P_{s,t})_{t \geq s \geq 0}$ the in-homogeneous semigroup associated with~\eqref{eq:SDE_sigma} satisfying for any measurable and bounded function $F: \rset^d \to \rset$, $P_{s,t}F(X_s) = \PE[F(X_t)\vert X_s]$. We have using that for any $s,t \geq 0$, $s \leq t$, $\mu_s P_{s,t} = P_{0,t}$,
\begin{equation}
\|\mu_s-\mu_t\|_{\mathrm{TV}} \leq \|\mu_{s-\varepsilon}P_{s-\varepsilon,s}-\mu_{t-\varepsilon}P_{s-\varepsilon,s}\|_{\mathrm{TV}} + \|\mu_{t-\varepsilon}P_{s-\varepsilon,s}-\mu_{t-\varepsilon}P_{t-\varepsilon,t}\|_{\mathrm{TV}}
\end{equation}
Now by definition of the total variation distance and using  \Cref{prop:contr_same_drift}, we have
\begin{equation}
  \label{eq:2}
 \|\mu_{s-\varepsilon}P_{s-\varepsilon,s}-\mu_{t-\varepsilon}P_{s-\varepsilon,s}\|_{\mathrm{TV}} \leq  \|\mu_{s-\varepsilon}-\mu_{t-\varepsilon}\|_{\mathrm{TV}} \leq q^\kappa_{\varepsilon} W_f(\mu_{s-\varepsilon}, \mu_{t-\varepsilon}) \eqsp,
\end{equation}
where $q^\kappa_\varepsilon$ is defined by \eqref{q_kappa_t_def}.
Regarding the second term, we apply Pinsker's inequality and the data processing inequality to get
\begin{align}
\|\mu_{t-\varepsilon}P_{s-\varepsilon,s}-\mu_{t-\varepsilon}P_{t-\varepsilon,t}\|_{\mathrm{TV}}
  &\leq \scrH(\mu_{t-\varepsilon}P_{s-\varepsilon,s} |\mu_{t-\varepsilon}P_{t-\varepsilon,t})^{\frac{1}{2}}\\
  & \leq \scrH(\mathcal{L}((\tilde{X}^s_{u})_{u \in \ccint{0,\varepsilon}}) \vert \mathcal{L}((\tilde{X}^t_{u})_{u \in \ccint{0,\varepsilon}}))^{1/2} \eqsp,
\end{align}
where $\scrH(\mu \vert \nu)$ denotes the relative entropy of $\mu$ w.r.t. $\nu$,
$\scrH(\mu \vert \nu) = \int_{\mathrm{p}_{\mu,\nu} >0}\log \mathrm{p}_{\mu,\nu} \rmd \mu$, if $\mu$ is absolutely continuous with respect to $\nu$ and $\mathrm{p}_{\mu,\nu}$ denotes the corresponding density and $\scrH(\mu \vert \nu) = \plusinfty$ otherwise, and $\mathcal{L}((\tilde{X}^s_{u})_{u \in \ccint{0,\varepsilon}})$ and $\mathcal{L}((\tilde{X}^t_{u})_{u \in \ccint{0,\varepsilon}})$ denotes the distributions of solutions of \eqref{eq:SDE_sigma} on $\ccint{s-\varepsilon,s}$ and $\ccint{t-\varepsilon,t}$ respectively starting from $\mu_{t-\varepsilon}$. Then using the Girsanov theorem \cite[VIII Thm 1.7]{revuz1999ContinuousMartingalesBrownian}, we get  
\begin{align}
\|\mu_{t-\varepsilon}P_{s-\varepsilon,s}-\mu_{t-\varepsilon}P_{t-\varepsilon,t}\|_{\mathrm{TV}} &\leq (1/\sqrt{2})\bbE\left[\int_0^{\varepsilon}|\beta_{s-\varepsilon+u}(\tilde{X}_u^s) - \beta_{t-\varepsilon+u}(\tilde{X}_u^t)|^2 \De u\right]^{1/2} \\
  \label{eq:4}
&\leq (1/\sqrt{2}) \sqrt{\varepsilon}\const_{\beta} |t-s|^{\upgamma} \eqsp,
\end{align}
where we have used that $\beta$ is $\upgamma$-Hölder and $\const_{\beta}$ denotes the corresponding constant.
Using~\eqref{eq:2} and \eqref{eq:4} completes the proof of \ref{item:time_Holder_TV}.

\subsection{Proofs of \Cref{prop:contr_same_drift}}\label{sec:proof_contr_same_drift}

Conditioning on $(X_0,\hX_0)$, we can assume without loss of generality that $(X_0,\hX_0) = (x,\hat{x})$.

\begin{itemize}[wide]
\item \underline{Proof of \ref{item_2:contraction_coup_by_ref}.}
Let us start by establishing a differential inequality satisfied by $f_{\bar{\kappa}}(|X_s-\hat{X}_s|)$. Define for any $s \geq 0$, $r_s = |X_s - \hat{X}_s|$. Then applying Itô's formula we obtain for $s < T_0$, setting $\De W_s = \rme_s \cdot \De B^1_s$
\begin{align}
\De r_s = &\,\frac{(X_s - \hat{X}_s)}{r_s} \cdot \left(\beta_s(X_s) - \beta_s(\hat{X}_s) \right) \De s + 2 \sigma_0  \De W_s + \frac{(X_s - \hat{X}_s)}{r_s} \cdot(\bar{\sigma}(X_s) - \bar{\sigma}(\hat{X}_s)) \De B^3_s \\
&+ \frac{1}{2} \tr \Bigg((2\rme_s \cdot \rme_s^{\top}, \bar{\sigma}(X_s) - \bar{\sigma}(\hat{X}_s))^{\top} \left(\frac{\mathrm{I}}{r_s} - \frac{(X_s - \hat{X}_s) (X_s - \hat{X}_s)^{\top}}{r_s^3} \right)
(2\rme_s \cdot \rme_s^{\top}, \bar{\sigma}(X_s) - \bar{\sigma}(\hat{X}_s))\Bigg) \De s\\
=&\,\frac{(X_s - \hat{X}_s)}{r_s} \cdot \left(\beta_s(X_s) - \beta_s(\hat{X}_s)\right) \De s + 2 \sigma_0  \De W_s + \frac{(X_s - \hat{X}_s)}{r_s} \cdot (\bar{\sigma}(X_s) - \bar{\sigma}(\hat{X}_s)) \De B^3_s \\
&+ \defEns{ \frac{1}{2r_s} \|\bar{\sigma}(X_s) - \bar{\sigma}(\hat{X}_s) \|_{\mathrm{Fr}}^2 - \frac{1}{2r_s^3} |(\bar{\sigma}(X_s) - \bar{\sigma}(\hat{X}_s))(X_s - \hat{X}_s)|^2 }\De s \eqsp.
\end{align}
Then applying the Itô-Tanaka formula \cite[VI Thm 1.5]{revuz1999ContinuousMartingalesBrownian} combined with the occupation formula \cite[VI Cor 1.6]{revuz1999ContinuousMartingalesBrownian} to the concave function $f_{\bar{\kappa}}$ which is continuously differentiable and such that $f'_{\bar{\kappa}}$ is absolutely continuously, for any $t \geq t_0 \geq 0$ and $M>0$,
\begin{align}
 &f_{\bar{\kappa}}(r_{t \wedge \tau_M}) - f_{\bar{\kappa}}(r_{t_0 \wedge \tau_M})\\
 & =\int_{t_0 \wedge \tau_M}^{t \wedge \tau_M }\defEns{f_{\bar{\kappa}}'(r_s) \frac{(X_s - \hat{X}_s)}{r_s} \cdot \left(\beta_s(X_s) - \beta_s(\hat{X}_s)\right) + \frac{1}{2}f_{\bar{\kappa}}''(r_s)\left(4 \sigma_0^2 + \frac{1}{(r_s)^2} |(\bar{\sigma}(X_s) - \bar{\sigma}(\hat{X}_s))(X_s - \hat{X}_s)|^2\right)}\De s \\ &  +  M_{t\wedge \tau_M} - M_{t_0\wedge \tau_M}  + \int_{t_0 \wedge \tau_M}^{t \wedge \tau_M} \defEns{ \frac{1}{2r_s} \|\bar{\sigma}(X_s) - \bar{\sigma}(\hat{X}_s) \|_{\mathrm{Fr}}^2 - \frac{1}{2r_s^3} |(\bar{\sigma}(X_s) - \bar{\sigma}(\hat{X}_s))(X_s - \hat{X}_s)|^2 }\De s \\
 &\leq  \int_{t_0 \wedge \tau_M}^{t \wedge \tau_M }\defEns{-{\bar{\kappa}_{\beta}}(r_s)r_sf_{\bar{\kappa}}'(r_s)  + 2f_{\bar{\kappa}}''(r_s) \sigma_0^2 } \rmd s  +  M_{t\wedge \tau_M} - M_{t_0\wedge \tau_M} \eqsp ,
\end{align}
where we have used the assumption on $\bar{\kappa}$ and set
\begin{equation}
  \label{eq:def_m_s-1deda}
  \rmd M_s = f_{\bar{\kappa}}'(r_s)\{ 2 \sigma_0 \rmd W_s + r_s^{-1}{(X_s - \hat{X}_s)} \cdot (\bar{\sigma}(X_s) - \bar{\sigma}(\hat{X}_s))\rmd B_s^3\}\eqsp,
\end{equation}
and $\tau_M= \inf \{ t >0 \, : \, r_t \leq 1/M \eqsp, \, r_t \geq M \}$.
Using \eqref{eq:Eberle_funct_2}, we obtain
\begin{equation}
  \label{eq:proof_1coupling_tv_1}
  f_{\bar{\kappa}}(r_{t \wedge \tau_M}) - f_{\bar{\kappa}}(r_{t_0 \wedge \tau_M}) \leq - \lambda_{\bar{\kappa}} \int_{t_0 \wedge \tau_M}^{t \wedge \tau_M} f_{\bar{\kappa}}(r_s)\rmd s + M_{t\wedge \tau_M}  - M_{t_0\wedge \tau_M} \eqsp.
\end{equation}
Then, since $( M_{t\wedge \tau_M})_{t \geq 0}$ is a martingale, we get
for any $t \geq t_0 \geq 0$ and $M>0$,

\begin{align}
\PE[f_{\bar{\kappa}}(r_{t \wedge \tau_M}) - f_{\bar{\kappa}}(r_{t_0 \wedge \tau_M}) ] \leq -\lambda_{\bar{\kappa}} \PE\left[\int_{t_0 \wedge \tau_M}^{t \wedge \tau_M} f_{\bar{\kappa}}(r_{s})\rmd s\right] \eqsp.
\end{align}
Taking $M \to \plusinfty$ and using Fatou's Lemma, and then using the proper version of Grönwall's lemma (see \Cref{lem:Gronwall_weak_form}), we obtain
\begin{equation}
  \PE[f_{\bar{\kappa}}(r_{t})] \leq \rme^{\lambda_{\bar{\kappa}}t } f_{\bar{\kappa}}(\norm{x-\hat{x}}) \eqsp.
\end{equation}

\item \underline{Proof of \ref{item:final_TV_coup_by_ref}.}
We consider first the case $t \leq  1/(2\lambda_{\bar{\kappa}})$.
Following the same lines as the proof of  \eqref{eq:proof_1coupling_tv_1}, we obtain that for $t < T_0$, almost surely
\begin{equation}
  \label{eq:proof_1coupling_tv_2}
  f_{\bar{\kappa}}(r_{t}) - f_{\bar{\kappa}}(r_0) \leq - \lambda_{\bar{\kappa}} \int_0^{t } f_{\bar{\kappa}}(r_s)\rmd s + M_{t}\eqsp.
\end{equation}
Note that we can extend $M_t$ defined in \eqref{eq:def_m_s-1deda} for $t \geq T_0$ by continuity and $\langle M\rangle_t \geq 4 f_{\bar{\kappa}}'(r_s)^2\sigma_0^2 t$, where $\langle M\rangle_t$ denotes the bracket of $(M_t)_{t \geq 0}$. In addition, since $f_{\bar{\kappa}}(r_t) =0$ for $t \geq T_0$, for any $t \geq 0$, we have
\begin{equation}
  \label{eq:proof_1coupling_tv_3}
  f_{\bar{\kappa}}(r_{t}) \leq  f_{\bar{\kappa}}(r_0)+ M_{t\wedge T_0}\eqsp.
\end{equation}
Since $f_{\bar{\kappa}}'(r)\geq C_{\bar{\kappa}}$ for any $r \geq 0$, for any $t \geq 0$,
\begin{equation}
  \label{eq:lower_bound_bracket_M}
  \langle M\rangle_t \geq 4 C_{\bar{\kappa}}^2 \sigma_0^2 t \eqsp.
\end{equation}
Then, the process $(M_t)_{t \geq 0}$ satisfies the hypothesis of the Dambis-Dubins-Schwartz Theorem  \cite[V Thm 1.6]{revuz1999ContinuousMartingalesBrownian} whose application yields that there exists a one-dimensional Brownian motion $(\tilde{B}_t)_{t\geq 0}$ such that $M_t = \tilde{B}_{\langle M\rangle_t}$.
Therefore, we obtain by~\eqref{eq:proof_1coupling_tv_3},
\begin{align}
  \PP(T_0 \geq t) &= \PP\parenthese{\tilde{B}_{\langle M\rangle_s} + f_{\bar{\kappa}}(r_0) \geq 0\, , \text{for any $s \in \ccint{0,t}$}}\\
  \label{eq:frist_estimate_tv_proof-1}
  &\leq \PP\Big(\sup_{s\in \ccint{0,4 C_{\bar{\kappa}}^2 \sigma^2 t}} \tilde{B}_s \leq f_{\bar{\kappa}}(r_0)\Big) \leq\PP(\absLigne{ \tilde{B}_{4 C_{\bar{\kappa}}^2 \sigma^2 t}} \leq  f_{\bar{\kappa}}(r_0)) = \frac{2f_{\bar{\kappa}}(|x-\hat{x}|)}{\sqrt{8\uppi C_{\bar{\kappa}}^2 \sigma_0^2 t}}\eqsp,
\end{align}
where we have used the the reflection principle asserting that if $(\tilde{B}_s)_{s\geq0}$ is a one-dimensional Brownian motion, then
\begin{equation}
  \txts \bbP(\sup_{0\leq u\leq s}\tilde{B}_u \geq a) = 2\bbP(\tilde{B}_s\geq a) \quad \forany a,s>0 \eqsp.
\end{equation}

We now consider the case  $t \geq 1/(2\lambda_{\bar{\kappa}})$. Set $\epsilon = 1/(2\lambda_{\bar{\kappa}})$. Then by combining \eqref{eq:frist_estimate_tv_proof-1} and \ref{item_2:contraction_coup_by_ref}, and using the Markov property, we have
\begin{align}
\bbP[X_t \neq \hat{X}_t] \leq \frac{\bbE[f_{\bar{\kappa}}(|X_{t-\epsilon} - \hat{X}_{t - \epsilon}|)]}{{\sqrt{2\uppi C_{\bar{\bar{\kappa}}}^2\sigma_0^2\epsilon}}}
\leq \frac{\rme^{\lambda_{\bar{\kappa}}\epsilon}}{{\sqrt{2\uppi C_{\bar{\bar{\kappa}}}^2\sigma_0^2\epsilon}}} \rme^{-\lambda_{\bar{\kappa}} t} f_{\bar{\kappa}}(|x-\hat{x}|)\eqsp.
\end{align}
Using the definition of $\epsilon$ concludes the proof.
\end{itemize}

\subsection{Proof of \Cref{prop:contr_same_drift_interp}}\label{sec:proof_contr_same_drift_interp}

The proofs of \ref{item_2:contraction_coup_by_ref_interp}-\ref{item:final_TV_coup_by_ref_interp} follow the same lines as their counterpart in \Cref{prop:contr_same_drift} and are therefore omitted.

We now deal with \ref{item:h_dominates_square}. First note that following the same arguments as for proving
\eqref{eq:proof_1coupling_tv_1} but using standard Itô's formula, we have that for any twice continuously differentiable function $f:\rset_+ \to \rset_+$ as well as for $f=f_{\bar{\kappa}}$,
\begin{align}
 &f(r_{t \wedge \tau_M}) - f(r_{t_0 \wedge \tau_M})
    \label{eq:def_m_s-1deda_0_1}
\leq  \int_{t_0 \wedge \tau_M}^{t \wedge \tau_M }\defEns{-{\bar{\kappa}}(r_s)r_sf'(r_s)  + 2f''(r_s) \sigma_0^2 } \rmd s  +  M_{t\wedge \tau_M}^f -  M_{t_0\wedge \tau_M}^f\eqsp,
\end{align}
where we have set
\begin{equation}
  \label{eq:def_m_s-1deda_1}
  \rmd M_s^f = f'(r_s)\{ 2 \sigma_0 \rmd W_s + r_s^{-1}{(X_s - \hat{X}_s)} \cdot (\bar{\sigma}(X_s) - \bar{\sigma}(\hat{X}_s))\rmd B_s^3\}\eqsp,
\end{equation}
and $\tau_M= \inf \{ t >0 \, : \, r_t \leq 1/M \eqsp, \, r_t \geq M \}$.

Then, since $\bar{\kappa} \in\msk$,
\begin{equation}
  \text{ there exists $\bar{\kappa}_+ >0$ and $R_1 \geq 1$ such that $\bar{\kappa}(r) \geq \bar{\kappa}_+$ for any $r \geq R_1$} \eqsp.
  \label{eq:def_kappa_p}
\end{equation}
We now define
\begin{equation}
  \tilde{f}_{\bar{\kappa},2}(r) = f_{\bar{\kappa}}(r) + \akappa \fRU(r) \eqsp, \quad \akappa = \frac{\lambda_{\bar{\kappa}}C_{\bar{\kappa}}R_1}{24(1+\sigma_0^2)} \eqsp, \quad
  \fRU(r)  = \frac{(r-R_1)_+^3}{1+\sigma_{0,\kappa}^2+r}\eqsp, \quad \sigma_{0,\kappa}^2 = 8 \sigma^2_0/\kappa_+ \eqsp.
\end{equation}
Note that $\fRU$ is twice continuously differentiable and for any $r \geq 0$,
\begin{align}
  \fRU'(r) &=  \frac{3 (r-R_1)_+^2}{1+\sigma_{0,\kappa}^2+r} - \frac{ (r-R_1)_+^3}{(1+\sigma_{0,\kappa}^2+r)^2} \geq \frac{2 (r-R_1)_+^2}{1+\sigma_{0,\kappa}^2+r} \eqsp,\\
  \fRU''(r) &= \frac{6 (r-R_1)_+}{1+\sigma_{0,\kappa}^2+r} - \frac{6 (r-R_1)_+^2}{(1+\sigma_{0,\kappa}^2+r)^2} + \frac{2 (r-R_1)_+^3}{(1+\sigma_{0,\kappa}^2+r)^3}\eqsp.
\end{align}
Using these results, \eqref{eq:Eberle_funct_2}, \eqref{eq:def_kappa_p} and \eqref{eq:def_m_s-1deda_0_1}, we get
\begin{align}
 \tilde{f}_{\bar{\kappa},2}(r_{t \wedge \tau_M}) - \tilde{f}_{\bar{\kappa},2}(r_{t_0 \wedge \tau_M})
  & \leq  \int_{t_0 \wedge \tau_M}^{t \wedge \tau_M }\defEns{-{\bar{\kappa}}(r_s)r_sf_{\bar{\kappa}}'(r_s) + 2f_{\bar{\kappa}}''(r_s) \sigma_0^2 - {\bar{\kappa}}(r_s)r_s \akappa \fRU'(r_s)  + 2\akappa\fRU''(r_s) \sigma_0^2 } \rmd s \\
   & \qquad +  M_{t\wedge \tau_M}^{f_{\bar{\kappa}}}+ M_{t\wedge \tau_M}^{\fRU} -M_{t_0\wedge \tau_M}^{f_{\bar{\kappa}}}- M_{t_0\wedge \tau_M}^{\fRU}\\
  \label{eq:5}
  & \leq \int_{t_0 \wedge \tau_M}^{t \wedge \tau_M }\defEns{- \lambda_{\bar{\kappa}}f_{\bar{\kappa}}(r_s)  - \bar{\kappa}_+r_s\akappa\fRU'(r_s)  + 2\akappa\fRU''(r_s) \sigma_0^2 } \rmd s \\
  &+ M_{t\wedge \tau_M}^{f_{\bar{\kappa}}}+  M_{t\wedge \tau_M}^{\fRU} -M_{t_0\wedge \tau_M}^{f_{\bar{\kappa}}}- M_{t_0\wedge \tau_M}^{\fRU} \eqsp.
\end{align}
Then, it is easy to verify, distinguishing the cases $r \in \ccint{0,R_1}$, $r\in\ccint{R_1,\bar{R}_1}$ and $r \geq \bar{R}_1$, $\bar{R}_1 = (R_1+1)\vee (12\sigma^2_0/\kappa_+)$, that there exists $\tilde{\lambda}_{\bar{\kappa},2}>0$ such that for any $r \geq 0$,
\begin{equation}
  - \lambda_{\bar{\kappa}}f_{\bar{\kappa}}(r)  - \bar{\kappa}_+\akappa r\fRU'(r)  + 2\akappa\fRU''(r) \sigma_0^2 \leq -\tilde{\lambda}_{\bar{\kappa},2} \tilde{f}_{\bar{\kappa},2}(r)\eqsp.
\end{equation}
This inequality, \eqref{eq:5} Fatou's Lemma and another application of Grönwall (\Cref{lem:Gronwall_weak_form}) conclude the proof.

\subsection{Proofs of \Cref{prop:contr_delta_coup}}\label{sec:proof_contr_delta_coup}

\begin{itemize}[wide]

\item\textbf{Step 1: SDE inequality for $f(r^{\delta}_s)$}

Define $\tau_n := \inf\{s \geq 0: |X^\delta_s| \geq n \}$ and consider the localized processes $X^{\delta,n}_s := X^{\delta}_{s\wedge\tau_n}$, $\hat{X}^{\delta,n}_s := \hat{X}^{\delta}_{s\wedge\tau_n}$. For readability, we omit the $n$ in the following computation.

Define $r^{\delta,a}_s = |X^{\delta}_s - \hat{X}^{\delta}_s|_a$, where $|x|_a = \sqrt{|x|^2 + a}$. Then applying Itô's formula
\begin{align}
\De r^{\delta,a}_s &= \frac{(X^{\delta}_s - \hat{X}^{\delta}_s)}{r^{\delta,a}_s} \cdot \left(\beta_s(X^{\delta}_s) - \hat{\beta}_s(\hat{X}^{\delta}_s) \right) \De s + 2 \sigma_0 \rcd(r^{\delta}_s) \De W_s 
+ \frac{(X^{\delta}_s - \hat{X}^{\delta}_s)}{r^{\delta,a}_s} \cdot(\bar{\sigma}(X^{\delta}_s) - \bar{\sigma}(\hat{X}^{\delta}_s)) \De B^3_s \\
&+ \frac{1}{2} \tr \Bigg((2\rcd(r^{\delta}_s)\rme_s \cdot \rme_s^{\top}, \bar{\sigma}(X^{\delta}_s) - \bar{\sigma}(\hat{X}^{\delta}_s))^{\top} \left(\frac{\mathrm{I}}{r^{\delta,a}_s} - \frac{(X^{\delta}_s - \hat{X}^{\delta}_s) (X^{\delta}_s - \hat{X}^{\delta}_s)^{\top}}{(r^{\delta,a}_s)^3} \right) (2\rcd(r^{\delta}_s)\rme_s \cdot \rme_s^{\top}, \bar{\sigma}(X^{\delta}_s) - \bar{\sigma}(\hat{X}^{\delta}_s))\Bigg) \De s\\
&=\frac{(X^{\delta}_s - \hat{X}^{\delta}_s)}{r^{\delta,a}_s} \cdot \left(\beta_s(X^{\delta}_s) - \hat{\beta}_s(\hat{X}^{\delta}_s) \right) \De s 
+ 2 \sigma_0 \rcd(r^{\delta}_s) \De W_s 
+ \frac{(X^{\delta}_s - \hat{X}^{\delta}_s)}{r^{\delta,a}_s} \cdot (\bar{\sigma}(X^{\delta}_s) - \bar{\sigma}(\hat{X}^{\delta}_s)) \De B^3_s \\
&+ 2 \rcd(r^{\delta}_s)^2\left(\frac{1}{r^{\delta,a}_s} - \frac{r^{\delta}_s}{(r^{\delta,a}_s)^3} \right)  
+ \frac{1}{2r^{\delta,a}_s} \|\bar{\sigma}(X^{\delta}_s) - \bar{\sigma}(\hat{X}^{\delta}_s) \|_{\mathrm{Fr}}^2 - \frac{1}{2(r^{\delta,a}_s)^3} |(\bar{\sigma}(X^{\delta}_s) - \bar{\sigma}(\hat{X}^{\delta}_s))(X^{\delta}_s - \hat{X}^{\delta}_s)|^2 \De s
\end{align}
Then applying again Itô's formula
\begin{align}
&\De f(r^{\delta,a}_s) \\
=&f'(r^{\delta,a}_s) \frac{(X^{\delta}_s - \hat{X}^{\delta}_s)}{r^{\delta,a}_s} \cdot \left(\beta_s(X^{\delta}_s) - \hat{\beta}_s(\hat{X}^{\delta}_s) \right) \De s + \De M_s \\
&+ f'(r^{\delta,a}_s) \Bigg(2 \rcd(r^{\delta}_s)^2\left(\frac{1}{r^{\delta,a}_s} - \frac{r^{\delta}_s}{(r^{\delta,a}_s)^3} \right) + \frac{1}{2r^{\delta,a}_s} \|\bar{\sigma}(X^{\delta}_s) - \bar{\sigma}(\hat{X}^{\delta}_s) \|_{\mathrm{Fr}}^2 
- \frac{1}{2(r^{\delta,a}_s)^3} |(\bar{\sigma}(X^{\delta}_s) - \bar{\sigma}(\hat{X}^{\delta}_s))(X^{\delta}_s - \hat{X}^{\delta}_s)|^2\Bigg) \De s \\
&+ \frac{1}{2}f''(r^{\delta,a}_s)\left(4 \rcd(r^{\delta}_s)^2\sigma_0^2 + \frac{1}{(r^{\delta,a}_s)^2} |(\bar{\sigma}(X^{\delta}_s) - \bar{\sigma}(\hat{X}^{\delta}_s))(X^{\delta}_s - \hat{X}^{\delta}_s)|^2\right) \De s\\
\leq &f'(r^{\delta,a}_s)\left(-\rcd(r^{\delta}_s)^2\bar{\kappa}(r^{\delta}_s)\frac{r^{\delta}_s}{r^{\delta,a}_s} +  \const^{\delta\beta}_s\right) + 2f''(r^{\delta,a}_s) \rcd(r^{\delta}_s)^2\sigma_0^2 \De s + \De M_s\\
&+ f'(r^{\delta,a}_s) \left(2 \rcd(r^{\delta}_s)^2\left(\frac{1}{r^{\delta,a}_s} - \frac{r^{\delta}_s}{(r^{\delta,a}_s)^3} \right) \right) \De s + (1 - \rcd(r^{\delta}_s)^2)f'(r^{\delta,a}_s) \frac{(X^{\delta}_s - \hat{X}^{\delta}_s)}{r^{\delta,a}_s} \cdot \left(\beta_s(X^{\delta}_s) - \hat{\beta}_s(\hat{X}^{\delta}_s) + \frac{L_{\sigma}^2(r^{\delta}_s)^2}{2r^{\delta,a}_s}  \right) \De s,
\end{align}
since $\bar{\sigma}$ is Lipschitz continuous with constant $L_{\sigma}$ thanks to \Cref{lem:sigmabar_Lip}.
Observe that by standard dominated convergence arguments as in the proof of \cite[Lem 7]{durmus2020elementary} we can take $a \rightarrow 0$ to get
\begin{align}
\De f(r^{\delta}_s)
\leq &f'(r^{\delta}_s)\left(-\rcd(r^{\delta}_s)^2\bar{\kappa}(r^{\delta}_s) +  \const^{\delta\beta}_s \right) + 2f''(r^{\delta}_s) \rcd(r^{\delta}_s)^2\sigma_0^2 \De s + (1 - \rcd(r^{\delta}_s)^2)f'(r^{\delta}_s)  \left(|\beta_s(X^{\delta}_s) - \hat{\beta}_s(\hat{X}^{\delta}_s)| + \frac{L_\sigma^2}{2}r^{\delta}_s  \right) \De s \eqsp .
\end{align}
Now, using the functional inequality satisfied by $f = f_{\bar{\kappa}}$ given in \eqref{eq:Eberle_funct_2}
\begin{align}
\De f_{\bar{\kappa}}(r^{\delta}_s)
\leq &-\rcd(r^{\delta}_s)^2\lambda f_{\bar{\kappa}}(r^{\delta}_s) +  \const^{\delta\beta}_s  \De s + \De M_s + (1 - \rcd(r^{\delta}_s)^2)f_{\bar{\kappa}}'(r^{\delta}_s)  \left(|\beta_s(X^{\delta}_s) - \hat{\beta}_s(\hat{X}^{\delta}_s)| + \frac{L_\sigma^2}{2}r^{\delta}_s \right) \De s\\
\leq &-\lambda f_{\bar{\kappa}}(r^{\delta}_s) + \const^{\delta\beta}_s + \lambda \delta \De s + \De M_s + (1 - \rcd(r^{\delta}_s)^2)f_{\bar{\kappa}}'(r^{\delta}_s)  \left(|\beta_s(X^{\delta}_s) - \hat{\beta}_s(\hat{X}^{\delta}_s)| + \frac{L_\sigma^2}{2}r^{\delta}_s \right) \De s \eqsp
\end{align}

\item\textbf{Step 2: Conclusion of contraction in Wasserstein distance}

Recall that we are working with the stopped processes $X^{\delta,n}_s = X^{\delta}_{s\wedge\tau_n}$, $\hat{X}^{\delta,n}_s = \hat{X}^{\delta}_{s\wedge\tau_n}$, define the distance process $r^{\delta,n}_s  := |X^{\delta,n}_s - \hat{X}^{\delta,n}_{s}|$. The last inequality means rigorously
\begin{align}
f_{\bar{\kappa}}(r^{\delta, n}_t) - f_{\bar{\kappa}}(r^{\delta, n}_{t_0})
\leq &\int_{t_0 \wedge \tau_n}^{t \wedge \tau_n}-\lambda f_{\bar{\kappa}}(r^{\delta, n}_s) + \lambda\delta +  \const^{\delta\beta}_s  + (1 - \rcd(r^{\delta, n}_s)^2)|\beta_s(X^{\delta, n}_s) - \beta_s(\hat{X}^{\delta, n}_s)| + \frac{L_{\sigma}^2}{2} r^{\delta,n}_s\De s  \\
&+ \int_{t_0 \wedge \tau_n}^{t \wedge \tau_n}\De M_s
\end{align}
Using Grönwall's lemma and taking expectation we obtain
\begin{align}
\bbE[f_{\bar{\kappa}}(r^{\delta, n}_{t})]
&\leq \bbE[e^{-\lambda(t\wedge \tau_n-t_0\wedge \tau_n)}f_{\bar{\kappa}}(r^{\delta, n}_{t_0})] \\
& \quad + \bbE\left[\int_{t_0 \wedge \tau_n}^{t \wedge \tau_n} e^{-\lambda(t \wedge \tau_n-s)} \left(\lambda\delta +  \const^{\delta\beta}_s + (1 - \rcd(r^{\delta, n}_s)^2)|\beta_s(X^{\delta, n}_s) - \beta_s(\hat{X}^{\delta, n}_s)| +  \frac{L_{\sigma}^2}{2}  r^{\delta,n}_s\right) \De s \right]\\
&\leq \bbE[e^{-\lambda(t\wedge \tau_n-t_0\wedge \tau_n)}f_{\bar{\kappa}}(r^{\delta, n}_{t_0})] + \bbE\left[\int_{t_0 \wedge \tau_n}^{t \wedge \tau_n} e^{-\lambda(t \wedge \tau_n-s)} \left(\lambda\delta  +  \const^{\delta\beta}_s \right) \De s \right] \\
& \quad + \bbE\left[\int_{0}^{t}  (1 - \rcd(r^{\delta, n}_s)^2)|\beta_s(X^{\delta, n}_s) - \beta_s(\hat{X}^{\delta, n}_s)| + \frac{L_{\sigma}^2}{2}  r^{\delta,n}_s \De s \right]
\end{align}
We proceed to prove that the last term goes to $0$ as $\delta \rightarrow 0$.
Since $\beta_s$ is continuous, it is locally uniformly continuous on $B_{n+1}(0)$ with say rate $\epsilon_{n+1}, \epsilon_{n+1}(r) \rightarrow 0$ as $r \rightarrow 0$, \ie, for $x,\hat{x} \in B_{n+1}(0)$
\begin{equation}
    |\beta_s(x) - \beta_s(\hat{x})| \leq  \epsilon_{n+1}(|x - \hat{x}|).
\end{equation}
Using this and $\rcd(r) =1$ for $r\geq \delta$, we get for $\delta \leq 1$
\begin{align}
\bbE\left[  (1 - \rcd(r^{\delta, n}_s)^2)|\beta_s(X^{\delta, n}_s) - \beta_s(\hat{X}^{\delta, n}_s)| \right]
&\leq \bbE[\IND_{\{r^{\delta,n}_s \leq \delta\}}
|\beta_s(X^{\delta,n}_s) - \beta_s(\hat{X}^{\delta,n}_s)|] \\
&\leq \bbE[\IND_{\{r^{\delta,n}_s \leq \delta\}} \epsilon_R(r^{\delta,n}_s)]
\end{align}
From this we get
\begin{align}
&\lim_{\delta \rightarrow 0}\bbE[(1 - \rcd(r^{\delta,n}_s)^2)|\beta_s(X^{\delta,n}_s) - \beta_s(\hat{X}^{\delta,n}_s)|] = 0
\end{align}
Using dominated convergence, we arrive at
\begin{align}
\limsup_{\delta \rightarrow 0}\bbE[f_{\bar{\kappa}}(r^{\delta, n}_{t})]
\leq \limsup_{\delta \rightarrow 0}\bbE[e^{-\lambda(t\wedge \tau_n-t_0\wedge \tau_n)}f_{\bar{\kappa}}(r^{\delta, n}_{t_0})] + \bbE\left[\int_{t_0 \wedge \tau_n}^{t \wedge \tau_n} e^{-\lambda(t \wedge \tau_n-s)}  (\const^{\delta\beta}_s ) \De s \right]
\end{align}
Using an optimal intial coupling at $t_0$ and letting $n \rightarrow \infty$ we conclude
\begin{align}
W_f(\cL(X^{\delta}_t),\cL(\hat{X}^{\delta}_t)) \leq e^{-\lambda(t-t_0)}W_f(\cL(X^{\delta}_{t_0}),\cL(\hat{X}^{\delta}_{t_0})) + \int_{t_0}^t e^{-\lambda(t-s)}  (\const^{\delta\beta}_s ) \De s
\end{align}

\item\textbf{Step 3: Estimate in TV}
 Denote by $(P_{s,t})_{t\geq s \geq 0}$ and $(\hat{P}_{s,t})_{t\geq s \geq 0}$ the non-homogeneous Markov semigroup associated with the first SDE in \eqref{eq:delta_coup_by_ref} and the second SDE respectively.

Let $0 \leq t_0 < t$. By the triangular inequality
\begin{align}
\| \cL(X^{\delta}_t) - \cL(\hat{X}^{\delta}_t) \|_{\TV}
\leq \| \cL(X^{\delta}_t) - \cL(\hat{X}^{\delta}_{t_0})P_{t,t_0} \|_{\TV} + \| \cL(\hat{X}^{\delta}_{t_0})P_{t,t_0} - \cL(\hat{X}^{\delta}_t) \|_{\TV}
\end{align}
Then by \Cref{prop:contr_same_drift}-\ref{item:final_TV_coup_by_ref}
\begin{align}
\| \cL(X^{\delta}_t) - \cL(\hat{X}^{\delta}_{t_0})P_{t,t_0} \|_{\TV}
= \| \cL(X^{\delta}_{t_0})P_{t,t_0} - \cL(\hat{X}^{\delta}_{t_0})P_{t,t_0} \|_{\TV} 
\leq q^{\bar{\kappa}}_{t-t_0}\bbE[f(r^{\delta}_{t_0})]
\end{align}
For the second term, we obtain by Pinsker's inequality and Girsanov's theorem
\begin{align}
\| \cL(\hat{X}^{\delta}_{t_0})P_{t,t_0} - \cL(\hat{X}^{\delta}_t) \|_{\TV}
&= \| \cL(\hat{X}^{\delta}_{t_0})P_{t,t_0} - \cL(\hat{X}^{\delta}_{t_0})\hat{P}_{t,t_0} \|_{\TV} 
\leq \scrH(\cL(\hat{X}^{\delta}_{t_0})P_{t,t_0} | \cL(\hat{X}^{\delta}_{t_0})\hat{P}_{t,t_0})^\frac{1}{2} \\
&= \frac{1}{\sqrt{2}} \bbE\left[\int_{t_0}^t |\beta_s(X_s) - \hat{\beta}_s(X_s)|^2 \De s \right]^\frac{1}{2} \leq \frac{1}{\sqrt{2}} \left(\int_{t_0}^{t} (\const^{\delta\beta_s})^2 \De s \right)^{\frac{1}{2}}.
\end{align}
Putting the estimates together finishes the proof.
\end{itemize}

We need the following technical lemma.
\begin{lemma}\label{lem:sigmabar_Lip}
Assume \Cref{ass:coupling}. The map $x \mapsto \bar{\sigma}(x)$ is Lipschitz continuous, \ie, there exists $L_{\sigma} \geq 0$ such that
\begin{equation}
\| \bar{\sigma}(x) - \bar{\sigma}(\hat{x}) \|_{\mathrm{Fr}} \leq L_{\sigma} |x-\hat{x}|
\end{equation}
\end{lemma}
\begin{proof}
For $\cM = \{A \in \bbR^{d \times d}: \Sigma^2 \mathrm{I} \succeq AA^{\top} \succeq 2\sigma_0^2 \mathrm{I} \}$
define the map
$$P: \cM \rightarrow \bbR^{d \times d}, \, A \mapsto \sqrt{A A^{\top} - \sigma_0^2 \mathrm{I}}\eqsp .$$
This map is differentiable on $\cM$ and for $A \in \cM$, $H \in \bbR^{d \times d}$
\begin{equation}
DP(A)H = \frac{1}{2}\left( \sqrt{AA^{\top} - \sigma_0^2 \mathrm{I}} \right)^{-1} \left( AH^{\top} + H A^{\top}\right)
\end{equation}
Using $AA^{\top} - \sigma_0^2 \mathrm{I} \succeq  \sigma_0^2 \mathrm{I}$ for all $A \in \cM$, we get for $A, \hat{A} \in \cM$
\begin{equation}
\|P(A) -P(\hat{A}) \|_{\mathrm{Fr}} \leq \frac{2\Sigma}{\sigma_0}\|A - \hat{A} \|_{\mathrm{Fr}}
\end{equation}
The statement now follows because $\bar{\sigma} = P \circ \sigma$ and using Lipschitz continuity of $\sigma$.
\end{proof}

\section{Estimates Hamiltonian}
\begin{lemma}\label{lem:Est_Contr_Hamiltonian}
Suppose \Cref{ass-coercivity}.
The function $(s,x,p) \in \ccint{0,T}\times \rset^d\times \rset^d \mapsto w_s(x,p)$ in \eqref{eq:optimal_policy_def}, associated with the Hamiltonian given in  \eqref{eq:Hamiltonian_fin_dim} is well-defined and satisfies  for any $(s,x,p) \in \ccint{0,T}\times \rset^d\times \rset^d$,
\begin{equation}\label{eq:Est_Optcontr}
|w_s(x,p)| \leq (1/\rho^\ell_{uu}) (\const^{\ell(\cdot,0)}_u + |p|) \eqsp .
\end{equation}
  In addition,
the function $p \mapsto w_s(x,p)$ is Lipschitz continuous with constant $\frac{1}{\rho^{\ell}_{uu}}$, \ie,
\begin{equation}\label{eq:Est_Lip_contr}
|\partial_p w_s(x,p)| \leq 1/{\rho^{\ell}_{uu}} \eqsp.
\end{equation}
If in addition $\|\partial_{ux}\ell_s \|_{\infty} \leq \const^\ell_{ux}$, then $|\partial_x w_s(x,p)| \leq \const^\ell_{ux}/\rho^{\ell}_{uu}$.
\end{lemma}
\begin{proof}
  Note that under  \Cref{ass-coercivity},  for any $x,p \in \rset^d$ and $s \in\ccint{0,T}$, $u \mapsto \ell_s(x,u) + (b_s(x)+u) \cdot p$ is strongly convex and twice continuously differentiable. Therefore, it admits a unique minimizer that corresponds to $w_s(x,p)$. In addition, it should satisfy by the first order optimality condition that for any $x,p \in \rset^d$ and $s \in\ccint{0,T}$.
\begin{equation}\label{eq:OptCont_contr}
\partial_u \ell_s(x,w_s(x,p)) + p =0 \eqsp.
\end{equation}
Using that under  \Cref{ass-coercivity} for any $u \in\rset^d$, $|\{\partial_u \ell_s(x,u) - \partial_u \ell_s(x,0)\}\cdot u | \geq \rho^\ell_{uu} |u |^2$, we get using  \eqref{eq:OptCont_contr} and the Cauchy-Schwarz inequality, $|\partial_u \ell_s(x,0) + p| \geq \rho^\ell_{uu} |w_s(x,p)|$,
which implies \eqref{eq:Est_Optcontr}.
Furthermore, applying \cite[Theorem 1B.1]{Dontchev:rocka:2014} to the function $(p,u) \mapsto \partial_u \ell_s(x,u) + p$, for fixed $x$ and $s$, we get $\partial_{uu} \ell_s(x,w_s(x,p))\partial_{p}w_s(x,p) + \Id =0$,
which yields \eqref{eq:Est_Lip_contr} by \Cref{ass-coercivity}.
Regarding the last part of the statement, applying \cite[Theorem 1B.1]{Dontchev:rocka:2014} again to the function $(x,u) \mapsto \partial_u \ell_s(x,u) + p$, for fixed $p$ and $s$ gives
\begin{equation}
\partial^2_{ux} \ell_s(x,w_s(x,p)) + \partial^2_{uu} \ell_s(x,w_s(x,p)) \partial_x w_s(x,p) = 0 \eqsp,
\end{equation}
which concludes the proof using \Cref{ass-coercivity}.
\end{proof}

\section{Miscellaneous}
\subsection{Proof of \Cref{lem:moment_bound}}\label{sec:proof_lem:moment_bound}
\begin{proof}
Let us write $\mu_t = \mu^{T,G}_t$. By a simple triangle inequality
\bes
\int_{\bbR^d}|x|\mu_t \leq \int_{\bbR^d}|x|\mu^b+C_{\kappa_b}^{-1}W_{f_{\kappa_b}}(\mu_t,\mu^b).
\ees
To bound the Wasserstein distance, we apply $\delta$-coupling by reflection \eqref{eq:delta_coup_by_ref} with 
\bes 
\hat\beta_s = b(x), \quad \beta_s = b(x)+w_s(x,\nabla\varphi^{T,G}_s(x)),
\ees
and initial distributions $\mu_0,\mu^b$. We have from \Cref{prop:contr_delta_coup}-\ref{item:W_1_contr_delta_coup} and \Cref{lem:basics_MF_high}-\ref{item:grad_est_MF_high} that 
\bes
\begin{aligned}
W_{f_{\kappa_b}}(\mu_t,\mu^b)
&\leq \exp(-\lambda_{\kappa_b}t)W_{f_{\kappa_b}t}(\mu_0,\mu^b)+\frac{1}{\rho^L_{uu}}\int_0^te^{-\lambda_{\kappa_b}(t-s)}[\|\varphi^{T,g}_s\|_{f_{\kappa_{b}}} + \const^{L(\cdot,0)}_u]\De s.
\end{aligned}
\ees
If $\| G \|_{f_{\kappa_b}} < + \infty$, we can apply \Cref{lem:basics_MF_high}-\ref{item:grad_est_MF_high} to bound
\begin{align}
W_{f_{\kappa_b}}(\mu_t,\mu^b)
&\leq \exp(-\lambda_{\kappa_b}t)W_{f_{\kappa_b}t}(\mu_0,\mu^b)+\frac{1}{\rho^L_{uu}}\int_0^te^{-\lambda_{\kappa_b}(t-s)}[(1-e^{-\lambda_{\kappa_b}(T-s)})\const^{\psi}_{x}+e^{-\lambda_{\kappa_b}(T-s)}\|G\|_{f_{\kappa_b}} + \const^{L(\cdot,0)}_u]\De s\\
&\leq \int_{\bbR^d}|x|\mu^b + \int_{\bbR^d}|x|\mu_0 + \frac{1}{\rho^L_{uu}\lambda_{\kappa_b}} (\const^{\psi}_{x}+\|G\|_{f_{\kappa_b}} + \const^{L(\cdot,0)}_u).
\end{align}
Now if $\| G \|_{\infty} < + \infty$, this gives, again referring to \Cref{lem:basics_MF_high}-\ref{item:grad_est_MF_high} for the bound on $\|\varphi^{T,G}_{t}\|_{f_{\kappa_{b}}}$
\begin{align}
W_{f_{\kappa_b}}(\mu_t,\mu^b)
&\leq \exp(-\lambda_{\kappa_b}t)W_{f_{\kappa_b}t}(\mu_0,\mu^b)+\frac{1}{\rho^L_{uu}}\int_0^te^{-\lambda_{\kappa_b}(t-s)}[(1-e^{-\lambda_{\kappa_b}(T-s)})\const^{\psi}_{x}+q^{\kappa_b}_{T-s}\|G\|_\infty + \const^{L(\cdot,0)}_u]\De s\\
&\leq W_{f_{\kappa_b}}(\mu_0,\mu^b)+\frac{1}{\rho^L_{uu}\lambda_{\kappa_b}}(\const^\psi_x + \const^{L(\cdot,0)}_u)+\frac{\|G\|_\infty}{\rho^L_{uu}}\int_T^{T+t}q^{\kappa_b}_{s-t}e^{-\lambda_{\kappa_b}(s-T)}\De s \\
&\leq \int_{\bbR^d}|x|\mu^b + \int_{\bbR^d}|x|\mu_0 +\frac{1}{\rho^L_{uu}\lambda_{\kappa_b}}(\const^\psi_x + \const^{L(\cdot,0)}_u)+\|G\|_\infty \frac{3}{2\rho^L_{uu}\sqrt{\pi\lambda_{\kappa_b}}C_{\kappa_b}\sigma_0},
\end{align}
where in the last inequality we have used \Cref{lem:boring_calculations_mild}.
\end{proof}

\subsection{On Grönwall's lemma}

\begin{lemma}\label{lem:Gronwall_weak_form}[Grönwall's lemma - integral form]
Suppose that $f: \bbR_+ \rightarrow \bbR_+$ is continuous and that there exists $\lambda \in \bbR$ such that for all $0\leq s \leq t$
\begin{equation}
    f(t) - f(s) \leq \lambda \int_s^t f(u) \De u.
\end{equation}
Then for all $t \geq 0$
\begin{equation}
    f(t) \leq e^{\lambda t}f(0).
\end{equation}
\end{lemma}
\begin{proof}
Let us take a sequence of mollifiers indexed by $n \in \N$, i.e. $\rho_n \in C^\infty_c(\bbR)$ with $\mathrm{spt}(\rho_n) \subset (-\frac{1}{2n}, \frac{1}{2n})$, $\rho_n \geq 0$ $\int_{\bbR}\rho_n = 1$. Set for $t \geq \frac{1}{n}$
\begin{equation}
    f_n(t) = \rho_n \ast f (t), \ F_n(t) = \rho_n \ast F (t),
\end{equation}
where $F(t) = \int_0^t f(u) \De u$. Then since convolution preserves positivity we have for $t \geq s \geq \frac{1}{n}$
\begin{equation}
    \frac{f_n(t) - f_n(s)}{t-s} \leq \lambda \frac{F_n(t) - F_n(s)}{t-s}.
\end{equation}
Now observe that $F_n'(t) = \rho_n \ast F'(t) = f_n(t)$, so that taking $t \rightarrow s$ we obtain
\begin{equation}
    f_n'(s) \leq \lambda f_n(s).
\end{equation}
At this point we deduce with the standard Grönwall's lemma
\begin{equation}
    f_n(s) \leq e^{\lambda (s- \frac{1}{n})} f_n(\frac{1}{n})
\end{equation}
The result now follows by taking $n \rightarrow \infty$.
\end{proof}

\bibliographystyle{plain}
\bibliography{refs_tpike}

\begin{thebibliography}{10}

\bibitem{bardi2024long}
Martino Bardi and Hicham Kouhkouh.
\newblock Long-time behavior of deterministic mean field games with nonmonotone
  interactions.
\newblock {\em SIAM Journal on Mathematical Analysis}, 56(4):5079--5098, 2024.

\bibitem{bogachev2007UniquenessSolutionsWeak}
V.~I. Bogachev, G.~Da~Prato, M.~R{\"o}ckner, and W.~Stannat.
\newblock Uniqueness of solutions to weak parabolic equations for measures.
\newblock {\em Bulletin of the London Mathematical Society}, 39(4):631--640,
  August 2007.

\bibitem{cardaliaguet2013KAM}
Pierre Cardaliaguet.
\newblock Long time average of first order mean field games and weak {KAM}
  theory.
\newblock {\em Dynamic Games and Applications}, 3(4):473--488, 2013.

\bibitem{cardaliaguet2019master}
Pierre Cardaliaguet, Fran{c}ois Delarue, Jean-Michel Lasry, and Pierre-Louis
  Lions.
\newblock {\em The master equation and the convergence problem in mean field
  games:(ams-201)}.
\newblock Princeton University Press, 2019.

\bibitem{cardaliaguet2012long}
Pierre Cardaliaguet, Jean-Michel Lasry, Pierre-Louis Lions, and Alessio
  Porretta.
\newblock Long time average of mean field games.
\newblock {\em Networks and Heterogeneous Media}, 7(2):279--301, 2012.

\bibitem{cardaliaguet2013long}
Pierre Cardaliaguet, Jean-Michel Lasry, Pierre-Louis Lions, and Alessio
  Porretta.
\newblock Long time average of mean field games with a nonlocal coupling.
\newblock {\em SIAM Journal on Control and Optimization}, 51(5):3558--3591,
  2013.

\bibitem{Cardaliaguet20203255}
Pierre Cardaliaguet and Marco Masoero.
\newblock Weak kam theory for potential mean field games.
\newblock {\em Journal of Differential Equations}, 268(7):3255 – 3298, 2020.

\bibitem{cardaliaguet2021ergodic}
Pierre Cardaliaguet and Cristian Mendico.
\newblock Ergodic behavior of control and mean field games problems depending
  on acceleration.
\newblock {\em Nonlinear Analysis}, 203:112185, 2021.

\bibitem{cardaliaguet2019long}
Pierre Cardaliaguet and Alessio Porretta.
\newblock Long time behavior of the master equation in mean field game theory.
\newblock {\em Analysis \& PDE}, 12(6):1397--1453, 2019.

\bibitem{carmona2013control}
Ren{\'e} Carmona, Fran{\c{c}}ois Delarue, and Aim{\'e} Lachapelle.
\newblock Control of mckean--vlasov dynamics versus mean field games.
\newblock {\em Mathematics and Financial Economics}, 7:131--166, 2013.

\bibitem{carmona2015forward}
René Carmona and François Delarue.
\newblock Forward–backward stochastic differential equations and controlled
  mckean–vlasov dynamics.
\newblock {\em The Annals of Probability}, 43(5):2647--2700, 2015.

\bibitem{carmona2024probabilisticapproachdiscountedinfinite}
René Carmona, Ludovic Tangpi, and Kaiwen Zhang.
\newblock A probabilistic approach to discounted infinite horizon and invariant
  mean field games.
\newblock {\em Preprint arXiv 2407.03642}, 2024.

\bibitem{carmona2024leveragingturnpikeeffectmean}
René Carmona and Claire Zeng.
\newblock Leveraging the turnpike effect for mean field games numerics.
\newblock {\em Preprint arXiv 2402.18725}, 2024.

\bibitem{cesaroni2024stationary}
Annalisa Cesaroni and Marco Cirant.
\newblock Stationary equilibria and their stability in a kuramoto mfg with
  strong interaction.
\newblock {\em Communications in Partial Differential Equations},
  49(1-2):121--147, 2024.

\bibitem{CHEN1997287}
Mu-Fa Chen and Feng-Yu Wang.
\newblock Estimates of logarithmic sobolev constant: An improvement of
  bakry–emery criterion.
\newblock {\em Journal of Functional Analysis}, 144(2):287--300, 1997.

\bibitem{chen1997estimation}
Mu-Fa Chen and Feng-Yu Wang.
\newblock Estimation of spectral gap for elliptic operators.
\newblock {\em Transactions of the American mathematical society},
  349(3):1239--1267, 1997.

\bibitem{chen1997general}
Mufa Chen and Fengyu Wang.
\newblock General formula for lower bound of the first eigenvalue on riemannian
  manifolds.
\newblock {\em Science in China Series A: Mathematics}, 40(4):384--394, 1997.

\bibitem{cirant2021long}
Marco Cirant and Alessio Porretta.
\newblock Long time behavior and turnpike solutions in mildly non-monotone mean
  field games.
\newblock {\em ESAIM: Control, Optimisation and Calculus of Variations}, 27:86,
  2021.

\bibitem{conforti2022coupling}
Giovanni Conforti.
\newblock Coupling by reflection for controlled diffusion processes:
  {{Turnpike}} property and large time behavior of
  {{Hamilton}}--{{Jacobi}}--{{Bellman}} equations.
\newblock {\em The Annals of Applied Probability}, 33(6A):4608--4644, December
  2023.

\bibitem{Dontchev:rocka:2014}
A.L. Dontchev and R.T. Rockafellar.
\newblock {\em Implicit Functions and Solution Mappings: A View from
  Variational Analysis}.
\newblock Springer Series in Operations Research and Financial Engineering.
  Springer New York, 2014.

\bibitem{dorfman1987linear}
Robert Dorfman, Paul~Anthony Samuelson, and Robert~M Solow.
\newblock {\em Linear programming and economic analysis}.
\newblock Courier Corporation, 1987.

\bibitem{durmus2020elementary}
Alain Durmus, Andreas Eberle, Arnaud Guillin, and Raphael Zimmer.
\newblock An elementary approach to uniform in time propagation of chaos.
\newblock {\em Proceedings of the American Mathematical Society},
  148(12):5387--5398, 2020.

\bibitem{eberle2016reflection}
Andreas Eberle.
\newblock Reflection couplings and contraction rates for diffusions.
\newblock {\em Probability {T}heory and {R}elated {F}ields}, 166(3-4):851--886,
  2016.

\bibitem{eberle2019couplings}
Andreas Eberle, Arnaud Guillin, and Raphael Zimmer.
\newblock Couplings and quantitative contraction rates for {L}angevin dynamics.
\newblock {\em The Annals of Probability}, 47(4):1982--2010, 2019.

\bibitem{eberle2019quantitative}
Andreas Eberle, Arnaud Guillin, and Raphael Zimmer.
\newblock Quantitative harris-type theorems for diffusions and mckean--vlasov
  processes.
\newblock {\em Transactions of the American Mathematical Society},
  371(10):7135--7173, 2019.

\bibitem{eberle2019sticky}
Andreas Eberle and Raphael Zimmer.
\newblock Sticky couplings of multidimensional diffusions with different
  drifts.
\newblock {\em Annales de l'Institut Henri Poincar{\'e}, Probabilit{\'e}s et
  Statistiques}, 55(4):2370--2394, 2019.

\bibitem{fleming2006controlled}
Wendell~H Fleming and Halil~Mete Soner.
\newblock {\em Controlled Markov processes and viscosity solutions}, volume~25.
\newblock Springer Science {\&} Business Media, 2006.

\bibitem{geshkovski2022turnpike}
Borjan Geshkovski and Enrique Zuazua.
\newblock Turnpike in optimal control of {PDE}s, {R}es{N}ets, and beyond.
\newblock {\em arXiv preprint arXiv:2202.04097}, 2022.

\bibitem{gyongy1996ExistenceStrongSolutions}
Istv{\'a}n Gy{\"o}ngy and Nicolai Krylov.
\newblock Existence of strong solutions for {{It{\^o}}}'s stochastic equations
  via approximations.
\newblock {\em Probability Theory and Related Fields}, 105(2):143--158, June
  1996.

\bibitem{huang2006LargePopulationStochastic}
Minyi Huang, Roland~P. Malham{\'e}, and Peter~E. Caines.
\newblock Large population stochastic dynamic games: Closed-loop
  {{McKean-Vlasov}} systems and the {{Nash}} certainty equivalence principle.
\newblock {\em Communications in Information \& Systems}, 6(3):221--252,
  January 2006.

\bibitem{Krylov:holder}
Nikolay~V. Krylov.
\newblock {\em Lectures on Elliptic and Parabolic Equations in H\"older
  Spaces}, volume~12 of {\em Graduate Studies in Mathematics}.
\newblock Americal Mathematical Society, Providence, RI, 1996.

\bibitem{lasry2006JeuxChampMoyen}
Jean-Michel Lasry and Pierre-Louis Lions.
\newblock Jeux {\`a} champ moyen. {{I}} -- {{Le}} cas stationnaire.
\newblock {\em Comptes Rendus Mathematique}, 343(9):619--625, November 2006.

\bibitem{lasry2006JeuxChampMoyena}
Jean-Michel Lasry and Pierre-Louis Lions.
\newblock Jeux {\`a} champ moyen. {{II}} -- {{Horizon}} fini et contr{\^o}le
  optimal.
\newblock {\em Comptes Rendus Mathematique}, 343(10):679--684, November 2006.

\bibitem{lasry2007mean}
Jean-Michel Lasry and Pierre-Louis Lions.
\newblock Mean field games.
\newblock {\em Japanese journal of mathematics}, 2(1):229--260, 2007.

\bibitem{lindvall1986coupling}
Torgny Lindvall and L~Cris~G Rogers.
\newblock Coupling of multidimensional diffusions by reflection.
\newblock {\em The Annals of Probability}, pages 860--872, 1986.

\bibitem{luo2016ExponentialConvergenceWasserstein}
Dejun Luo and Jian Wang.
\newblock Exponential convergence in -{{Wasserstein}} distance for diffusion
  processes without uniformly dissipative drift.
\newblock {\em Mathematische Nachrichten}, 289(14-15):1909--1926, 2016.

\bibitem{majka2017coupling}
Mateusz~B Majka.
\newblock Coupling and exponential ergodicity for stochastic differential
  equations driven by l{\'e}vy processes.
\newblock {\em Stochastic processes and their applications},
  127(12):4083--4125, 2017.

\bibitem{Masoero2019}
Marco Masoero.
\newblock On the long time convergence of potential mfg.
\newblock {\em Nonlinear Differential Equations and Applications}, 26(2), 2019.

\bibitem{mckenzie1976turnpike}
Lionel~W McKenzie.
\newblock Turnpike theory.
\newblock {\em Econometrica: Journal of the Econometric Society}, pages
  841--865, 1976.

\bibitem{porretta2013GlobalLipschitzRegularizing}
A.~Porretta and E.~Priola.
\newblock Global {{Lipschitz}} regularizing effects for linear and nonlinear
  parabolic equations.
\newblock {\em Journal de Math{\'e}matiques Pures et Appliqu{\'e}es},
  100(5):633--686, November 2013.

\bibitem{porretta2024decay}
Alessio Porretta.
\newblock Decay rates of convergence for fokker-planck equations with confining
  drift.
\newblock {\em Advances in Mathematics}, 436:109393, 2024.

\bibitem{priola2006gradient}
Enrico Priola and Feng-Yu Wang.
\newblock Gradient estimates for diffusion semigroups with singular
  coefficients.
\newblock {\em Journal of Functional Analysis}, 236(1):244--264, 2006.

\bibitem{revuz1999ContinuousMartingalesBrownian}
Daniel Revuz and Marc Yor.
\newblock {\em Continuous {{Martingales}} and {{Brownian Motion}}}, volume 293
  of {\em Grundlehren Der Mathematischen {{Wissenschaften}}}.
\newblock Springer, Berlin, Heidelberg, 1999.

\bibitem{rubio2011ExistenceUniquenessCauchy}
Gerardo Rubio.
\newblock Existence and uniqueness to the {{Cauchy}} problem for linear and
  semilinear parabolic equations with local conditions.
\newblock {\em ESAIM: Proceedings}, 31:73--100, January 2011.

\bibitem{sun2024turnpike}
Jingrui Sun and Jiongmin Yong.
\newblock Turnpike properties for mean-field linear-quadratic optimal control
  problems.
\newblock {\em SIAM Journal on Control and Optimization}, 62(1):752--775, 2024.

\bibitem{trelat2015turnpike}
Emmanuel Tr{\'e}lat and Enrique Zuazua.
\newblock The turnpike property in finite-dimensional nonlinear optimal
  control.
\newblock {\em Journal of Differential Equations}, 258(1):81--114, 2015.

\bibitem{wang1994application}
Feng-Yu Wang.
\newblock Application of coupling methods to the neumann eigenvalue problem.
\newblock {\em Probability Theory and Related Fields}, 98:299--306, 1994.

\bibitem{wang2020exponential}
Feng-Yu Wang.
\newblock Exponential contraction in wasserstein distances for diffusion
  semigroups with negative curvature.
\newblock {\em Potential Analysis}, 53(3):1123--1144, 2020.

\end{thebibliography}
\end{document}